\documentclass[11pt]{article}

\usepackage[square,numbers,sort&compress]{natbib}

\usepackage{tablefootnote}
\usepackage[utf8]{inputenc} 
\usepackage[T1]{fontenc}    
\usepackage{hyperref}       
\usepackage{url}            
\usepackage{booktabs}       
\usepackage{amsfonts}       
\usepackage{nicefrac}       
\usepackage{microtype}      
\usepackage{tikz}
\usetikzlibrary{decorations.pathreplacing,calc}

\usepackage[algo2e,ruled]{algorithm2e}
\LinesNumbered

\usepackage{verbatim}
\usepackage{amsthm}
\usepackage{amsfonts}
\usepackage[nonamelimits]{amsmath}
\usepackage{amssymb}
\usepackage{sansmath}
\usepackage{color}
\usepackage{xcolor}
\usepackage{bm}
\usepackage{bbm}
\usepackage{fullpage}
\usepackage[bottom]{footmisc}

\theoremstyle{plain}
\newtheorem{theorem}{Theorem}[section]
\newtheorem{assumption}[theorem]{Assumption}
\newtheorem{lemma}[theorem]{Lemma}
\newtheorem{fact}[theorem]{Fact}
\newtheorem{proposition}[theorem]{Proposition}
\newtheorem{remark}[theorem]{Remark}
\newtheorem{corollary}[theorem]{Corollary}

\newtheorem{problem}[theorem]{Problem}
\theoremstyle{definition}
\newtheorem{definition}[theorem]{Definition}
\newtheorem{question}{Question}
\usepackage{subfig}
\usepackage{graphicx}
\usepackage{listings}
\usepackage{enumitem}
\usepackage{tabularx}
\newtheoremstyle{named}{}{}{\itshape}{}{\bfseries}{.}{.5em}{\thmnote{#3}}
\theoremstyle{named}

\usepackage{tikz}
\usetikzlibrary{fit,calc}

\usepackage{tikz}
\usetikzlibrary{cd}

\DeclareSymbolFont{extraup}{U}{zavm}{m}{n}
\DeclareMathSymbol{\varheart}{\mathalpha}{extraup}{86}
\DeclareMathSymbol{\vardiamond}{\mathalpha}{extraup}{87}

\usepackage{mathtools}

\makeatletter
\newcommand*{\inlineequation}[2][]{%
  \begingroup
    \refstepcounter{equation}%
    \ifx\\#1\\%
    \else
      \label{#1}%
    \fi
    \relpenalty=10000 %
    \binoppenalty=10000 %
    \ensuremath{%
      #2%
    }%
    ~\@eqnnum
  \endgroup
}
\makeatother
\usepackage[table]{xcolor}

\usepackage{threeparttable}

\usepackage{todonotes}

\usepackage{xspace}

\xspaceaddexceptions{[]\{\}}

\usepackage{hyperref}

\usepackage{thm-restate}

\usepackage{bm}
\makeatletter
\newcommand*{\rom}[1]{\expandafter\@slowromancap\romannumeral #1@}

\usepackage{multirow}

\usepackage{tablefootnote}
\usepackage{float}
\usepackage{makecell}
\usepackage{authblk}

\allowdisplaybreaks

\title{Finding Stationary Points under Higher-Order Smoothness: Trapping Algorithms and First-Order Lower Bounds}
\date{\vspace{-5ex}}
\begin{document}
\author[1]{Huanjian Zhou\thanks{zhou-huanjian185@g.ecc.u-tokyo.ac.jp}}
\author[2,1]{Masashi Sugiyama\thanks{
sugi@k.u-tokyo.ac.jp}} 
\author[1,2]{Taiji Suzuki\thanks{taiji@mist.i.u-tokyo.ac.jp }}
\affil[1]{The University of Tokyo}
\affil[2]{RIKEN Center for Advanced Intelligence Project}

\maketitle

\begin{abstract}
Smooth nonconvex optimization arises throughout machine learning,
statistics, and scientific computing, where finding an approximate
stationary point is a fundamental algorithmic task. We study the
query complexity of this task under higher-order smoothness.
Deterministic gradient descent has the classical
\(\mathcal{O}(\varepsilon^{-2})\) complexity, whereas the faster deterministic 
\(\mathcal{O}(\varepsilon^{-\frac{p+1}{p}})\) rate of \(p\)-th-order regularization
is typically obtained using higher-order derivatives. We ask whether
the trapping framework can exploit higher-order smoothness using only
function values or only gradients, and study this question across
dimensions, oracle models, randomization and constraints. Our contributions are threefold.

For constant dimension \(d\) with objectives in
\(C^{p,\beta_p}\), we
give deterministic function-value trapping algorithms using
\(\widetilde{\mathcal{O}}(\varepsilon^{-(d-1)})\) queries in the
unconstrained setting and
\(\widetilde{\mathcal{O}}(\varepsilon^{-\frac{d-1}{p+\beta_p}})\) queries over
\([0,1]^d\). We
then implement trapping randomly using gradients alone through
high-order line integration. 
In particular, for two-dimensional \(C^{1,1}\)
objectives, this gives
\(\widetilde{\mathcal{O}}(\varepsilon^{-\frac{4}{3}})\) gradient queries in the
unconstrained setting.
This breaks the barrier of
$\Theta(\varepsilon^{-2})$ deterministic complexity by gradient descent.

For general dimension, we introduce a trapping-or-descent framework.
A boundary approximation either produces a point of certified
objective decrease or certifies that the enclosed region contains
an approximate stationary point. By lazily reusing boundary models
across successive moves, we obtain, for \(C^{p,\beta_p}\) objectives
with \(p\ge2\) and \(\nu:=p+\beta_p\), a deterministic function-value
complexity of
\(\mathcal O(d^p+d^{p-1+\frac{\beta_p}{\nu-1}}
\varepsilon^{-\frac{\nu}{\nu-1}})\)
and a deterministic gradient-only complexity of
\(\mathcal O(d^{p-1}+d^{p-2+\frac{\beta_p}{\nu-1}}
\varepsilon^{-\frac{\nu}{\nu-1}})\).
Both bounds apply to finding approximate stationary points on
\(\mathbb R^d\) and approximate KKT points on \([0,1]^d\).
Their accuracy exponent is independent of dimension and reduces
to the classical AR\(p\) exponent \(\frac{p+1}{p}\) when \(\beta_p=1\),
highlighting the potential of trapping as a general framework
for nonconvex optimization using either function values alone
or gradients alone.

For sufficiently
high dimension, we prove randomized first-order lower bounds. For \(C^{1,1}\cap C^{p,1}\) objectives, even an
exact value-and-gradient oracle requires
\(\Omega_p(\varepsilon^{-(\frac{3}{2}+\frac{1}{2p})})\) queries in the
high-accuracy regime. This matches known upper bounds for \(p=2\)
and, up to logarithmic factors, for \(p=3\). Without a prescribed
gradient-Lipschitz bound, we further obtain an
\(\Omega_p(\varepsilon^{-\frac{5}{3}})\) lower bound for every
\(p\geq3\).
\end{abstract}

\newpage
\setcounter{tocdepth}{2}
\tableofcontents

\newpage

\section{Introduction}

Smooth nonconvex optimization arises throughout machine learning,
statistics, and scientific computing
\cite{BCN18,LW15,NW06}. In the absence of additional
global structure, finding a global minimizer is generally
intractable, and approximate first-order stationarity is therefore
the standard algorithmic objective. Given a differentiable function
\(f\), the goal is to find a stationary point \(\boldsymbol{x}\) satisfying
\(\|\nabla f(\boldsymbol{x})\|_2\leq\varepsilon\), or the
corresponding approximate Karush--Kuhn--Tucker (KKT) point in the constrained setting.
The query complexity of these tasks captures how smoothness,
oracle type, constraints and randomness jointly govern the
difficulty of nonconvex optimization.

Randomization can reduce query complexity even when the
oracle and objective class remain unchanged.
For one-dimensional $C^{1,1}$ objectives under the standard
normalization, Chewi, Bubeck, and Salim
\cite{chewi2022complexity} established a polynomial separation:
deterministic gradient-only algorithms require
$\Theta(\varepsilon^{-2})$ queries, whereas randomized
algorithms achieve the optimal complexity
$\Theta(\varepsilon^{-1})$.
However they conjectured that randomized gradient-only
algorithms require $\Omega(\varepsilon^{-2})$ queries 
in dimension two, suggesting the advantage of randomization disappears in dimensions two and higher.

Other settings also admit lower bounds that limit the benefit
of randomization.
For finding approximate KKT points on the two-dimensional
cube, the randomized lower bound of Bubeck and
Mikulincer~\cite{bubeck2020trap} matches, up to logarithmic
factors, the deterministic value-query upper bound of
Hollender and Zampetakis~\cite{hollender2023computational}.
In sufficiently high dimensions, Carmon et
al.~\cite{carmon2020lowerI} proved an
$\Omega(\varepsilon^{-2})$ lower bound under
Lipschitz-gradient assumptions even for randomized algorithms
with access to derivatives of arbitrary order,
matching the deterministic gradient-descent upper bound.

Under higher-order smoothness, sharper deterministic
first-order barriers are known.
For algorithms accessing function values and gradients,
In sufficiently high dimensions, Zhou~\cite{zhou2026sharp}
proved deterministic first-order lower bounds of
$\Omega(\varepsilon^{-\frac{7}{4}})$ for $p=2$ and
$\Omega(\varepsilon^{-\frac{5}{3}})$ for any fixed $p\ge3$,
under Lipschitz bounds on derivatives of orders $1,\ldots,p$.
These match deterministic upper bounds up to logarithmic
factors~\cite{carmon2017convex}, but leave open whether
randomization can improve the corresponding rates.
Together, these results motivate the following question:
\begin{question}[\bf Randomness]
\label{question:randomness}
When does randomization reduce query complexity across dimensions?
\end{question}

Higher-order smoothness can improve the query complexity
of finding first-order stationary points.
Starting with cubic regularization
\cite{NP06}, higher-order regularization methods attain
$\mathcal O(\varepsilon^{-\frac{p+1}{p}})$ evaluation
complexity under Lipschitz continuity of the $p$th derivative
\cite{BGM+17}.
For $C^{p,\beta_p}$ objectives, universal regularization
extends this bound to
$\mathcal O(\varepsilon^{-\frac{p+\beta_p}{p+\beta_p-1}})$,
including projected-gradient guarantees under convex
constraints \cite{cartis2019universal}.
These methods access derivatives through order $p$,
so their guarantees combine higher-order smoothness
with higher-order oracle information.

Lower-order methods can also benefit from additional
smoothness.
Accelerated first-order methods improve the accuracy
exponent under Lipschitz Hessians and third derivatives
\cite{carmon2017convex}, while finite-difference and lazy
regularization methods obtain further guarantees with
explicit dimension dependence
\cite{doikov2025first,doikov2026lowerorder}.
Nevertheless, deterministic first-order complexity has
an $\Omega(\varepsilon^{-\frac{5}{3}})$ barrier under
Lipschitz bounds through any fixed order $p\ge3$ in
sufficiently high dimensions \cite{zhou2026sharp}.
The benefit of smoothness therefore depends on both
oracle access and dimension.
Existing fixed-dimensional trapping guarantees use
Lipschitz-gradient assumptions
\cite{bubeck2020trap,hollender2023computational,
huanjian2025adaptive}, leaving their dependence on
higher-order regularity to be understood.
This motivates the following question:
\begin{question}[\bf Smoothness]
\label{question:smoothness}
How much can higher-order smoothness reduce query
complexity under low-order access?
\end{question}

Access to full local information can substantially reduce query complexity.
For one-dimensional $C^{1,1}$ objectives under the standard
normalization, Chewi, Bubeck, and Salim
\cite{chewi2022complexity} proved that gradient-only
algorithms require
$\Theta(\varepsilon^{-2})$ deterministic queries or
$\Theta(\varepsilon^{-1})$ randomized queries.
Allowing function-value queries in addition to gradients
reduces the deterministic complexity to
$\Theta(\log(\frac{1}{\varepsilon}))$.

Conversely, gradient-only methods can preserve the guarantees
of methods using full first-order information in some other regimes.
Under Lipschitz gradients and Hessians, Li and Lin
\cite{li2023restarted} obtained a deterministic gradient-only
complexity of $\mathcal O(\varepsilon^{-\frac{7}{4}})$,
matching the accuracy dependence of the deterministic
joint-oracle lower bound in sufficiently high dimensions
\cite{zhou2026sharp}.
Allowing explicit dimension dependence, gradient-only methods
also achieve fast rates as
$\mathcal O(d^{\frac{1}{4}}\varepsilon^{-\frac{13}{8}})$
\cite{jiang2025improved} and
$\widetilde{\mathcal O}(d+d^{\frac{1}{3}}
\varepsilon^{-\frac{3}{2}})$
\cite{adil2025balancing}.
Together, these results motivate understanding when gradients
alone suffice to recover function-value guarantees.
\begin{question}[\bf Oracle]
\label{question:oracle-information}
How closely can gradient-only algorithms match
function-value guarantees?
\end{question}

Constraints change both the search domain and the stationarity
condition.
They also affect the computational classification.
In the standard white-box models, finding an approximate
stationary point of a smooth bounded function on $\mathbb R^2$
is $\mathsf{PLS}$-complete
\cite{hollender2023computational},
whereas finding an approximate KKT point on $[0,1]^2$
is complete for
$\mathsf{CLS}=\mathsf{PPAD}\cap\mathsf{PLS}$
\cite{fearnley2021complexity}.
Here $\mathsf{PLS}$ and $\mathsf{CLS}$ denote Polynomial
Local Search and Continuous Local Search, respectively.
Computing a KKT point remains $\mathsf{CLS}$-complete
even for explicitly given quadratic objectives on $[0,1]^d$
when the dimension is part of the input
\cite{fearnley2024complexity}.

A quantitative distinction also appears in query complexity.
For two-dimensional $C^{1,1}$ objectives under the standard
normalization, the optimal deterministic function-value
complexity is $\Theta(\varepsilon^{-1})$ for finding an
$\varepsilon$-stationary point on $\mathbb R^2$, compared
with $\Theta(\varepsilon^{-\frac{1}{2}})$ for finding an
$\varepsilon$-KKT point over $[0,1]^2$
\cite{hollender2023computational}.
These bounds show that restricting the domain and allowing
boundary stationarity can reduce query complexity.
They motivate understanding how this comparison changes
under higher-order smoothness and gradient-only access.
\begin{question}[\bf Constraints]
\label{question:constraints}
How do constraints change the complexity of finding
stationary points?
\end{question}

We make progress on these questions by developing trapping
algorithms and establishing randomized first-order lower
bounds in high dimensions.
We first develop higher-order trapping algorithms
in fixed dimension in Section~\ref{sec:3} and show that randomized gradient-only implementation
breaks the classical gradient-descent barrier already in dimension
two in Section~\ref{sec:4}. 
We then introduce a trapping-or-descent framework for arbitrary
dimension, obtaining function-value and gradient-only analogues of
the classical AR\(p\) rate in Secton~\ref{sec:trapping-or-descent}. 
Finally, we establish randomized
first-order lower bounds in high dimension, identifying regimes in
which the known upper bounds are optimal in Section~\ref{sec:lb}. Our results are stated
precisely below.

\subsection{Our Results}

\begin{table}[ht]
\caption{
Comparison of existing and our \emph{zeroth-order} query-complexity bounds for
finding an $\varepsilon$-stationary point over $\mathbb{R}^d$ or an
$\varepsilon$-KKT point over $[0,1]^d$, using exact function values only.
Here $\nu:=p+\beta_p$, and $d\ge2$ is fixed independently of
$\varepsilon$ unless its dependence is displayed.
When $p=1$, $C^{p,\beta_p}=C^{1,\beta_1}$; for $p\ge2$, our results
require no additional first-order smoothness.
Deterministic upper bounds also apply to randomized algorithms.
}
\centering
\renewcommand{\arraystretch}{1.9}
\small
\resizebox{\textwidth}{!}{
\begin{tabular}{|l|l|l|c|c|}
\hline
Problem setup
& Smoothness
& Algorithm
& Existing results
& Our results \\
\hline


\multirow{5}{*}{
\makecell{
Unconstrained\\
($\Omega=\mathbb{R}^d$)\\
function-value oracle
}}
& \multirow{2}{*}{$C^{1,1}$}
& deterministic
& \makecell[l]{
$\Theta(\varepsilon^{-1})$, $d=2$
~\cite{hollender2023computational}\\
$\mathcal{O}_d(\varepsilon^{-(d-1)})$, fixed $d\ge2$
~\cite{hollender2023computational,huanjian2025adaptive}\\
$\mathcal{O}(d\varepsilon^{-2})$
~\cite{grapiglia2023quadratic,chaudhry2026complexity}
}
& \makecell[l]{$\mathcal{O}_d(\varepsilon^{-(d-1)})$, fixed $d\ge2$\\
\hfill (Theorem~\ref{the:main1})
}
\\
\cline{3-5}

&
& randomized
& \makecell[l]{
$\Theta(\varepsilon^{-1})$, $d=2$
~\cite{hollender2023computational}\\
$\widetilde{\mathcal{O}}(d\varepsilon^{-2})$
~\cite{gratton2015direct,roberts2023direct,ghadimi2013stochastic,nesterov2017random,
bergou2020stochastic}
}
& --
\\
\cline{2-5}

& \multirow{2}{*}{
\makecell{
$C^{p,\beta_p}$\\
($p\ge1$)
}}
& deterministic
& \makecell[l]{
$\mathcal{O}(d^{\frac{3}{2}}\varepsilon^{-\frac{3}{2}})$
for $p=2$, $\beta_p=1$
~\cite{doikov2025first}\\
$\Omega(\varepsilon^{-1})$ for $C^{2,1}$~\cite{kontogiannis2024computational}
}
& \makecell[l]{
$\mathcal{O}_{d,p}\!\left(
\varepsilon^{-(d-1)}
\right)$, fixed $d\ge2$\\
\hfill (Theorem~\ref{the:main1})\\
$\widetilde{\mathcal{O}}_p\!\left(
d^p+d^{p-1+\frac{\beta_p}{\nu-1}}
\varepsilon^{-\frac{\nu}{\nu-1}}
\right)$\\
\hfill (Theorem~\ref{the:lazy-trapping-or-descent})\\
}
\\
\cline{3-5}

&
& randomized
& \makecell[l]{
$\mathcal{O}\!\big(
d^{\,2+\frac{1-\beta_p}{2\beta_p}}
\varepsilon^{-\frac{2}{\beta_p}}
\big)$ for $p=1$,
~\cite{shibaev2022zeroth}\\
$\mathcal{O}(d\varepsilon^{-2})$ for
$p=1$, $\beta_p=1$
~\cite{ghadimi2013stochastic,nesterov2017random}
}
& --
\\
\hline


\multirow{4}{*}{
\makecell{
Constrained\\
($\Omega=[0,1]^d$)\\
function-value oracle
}}
& \multirow{2}{*}{$C^{1,1}$}
& deterministic
& \makecell[l]{
$\mathcal{O}_d(\varepsilon^{-\frac{2d}{d+2}})$, fixed $d\ge2$
~\cite{vavasis1993black}\\
$d^{\mathcal{O}(d)}
\widetilde{\mathcal{O}}(\varepsilon^{-\frac{d-1}{2}})$
~\cite{bubeck2020trap}\\
$\mathcal{O}_d(\varepsilon^{-\frac{d-1}{2}})$, fixed $d\ge2$
~\cite{hollender2023computational,huanjian2025adaptive}\\
$\mathcal{O}(d\varepsilon^{-2})$
~\cite{liu2026unified}
}
& \makecell[l]{
$\widetilde{\mathcal{O}}_d
(\varepsilon^{-\frac{d-1}{2}})$ \\
\hfill (Theorem~\ref{the:main1})
}
\\
\cline{3-5}

&
& randomized
& \makecell[l]{
$\widetilde{\Theta}(\varepsilon^{-\frac{1}{2}})$, $d=2$
~\cite{bubeck2020trap}
}
& --
\\
\cline{2-5}

& \multirow{2}{*}{
\makecell{
$C^{p,\beta_p}$\\
($p\ge1$)
}}
& deterministic
& \makecell[l]{
$\mathcal{O}(d^{\frac{3}{2}}\varepsilon^{-\frac{3}{2}})$
for $p=2$, $\beta_p=1$
~\cite{doikov2025first}\\
}
& \makecell[l]{
$\widetilde{\mathcal{O}}_{d,p}\!\left(
\varepsilon^{-\frac{d-1}{p+\beta_p}}
\right)$,
fixed $d\ge2$\\
\hfill (Theorem~\ref{the:main1})\\
$\widetilde{\mathcal{O}}_p\!\left(
d^p+d^{p-1+\frac{\beta_p}{\nu-1}}
\varepsilon^{-\frac{\nu}{\nu-1}}
\right)$,\\
\hfill (Theorem~\ref{the:lazy-trapping-or-descent})\\
}
\\
\cline{3-5}

&
& randomized
& --
& --
\\
\hline

\end{tabular}
}
\label{tab:function-value-comparison}
\end{table}

\paragraph{Trapping with higher-order smoothness (results in Section~\ref{sec:3}).}
Following \cite{hollender2023computational}, we study the
value-query complexity of trapping-based methods. For two-dimensional
objectives in $C^{1,1}$, they developed deterministic zeroth-order
trapping algorithms requiring
$\mathcal{O}(\varepsilon^{-1})$ queries in the unconstrained setting and
$\mathcal{O}(\varepsilon^{-1/2})$ queries in the constrained setting.
Matching lower bounds show that both rates are optimal in two-dimensional case~\cite{bubeck2020trap,hollender2023computational}. They further
posed the problem of characterizing the query complexity under
higher-order smoothness in both settings. 
We make progress on this question by
extending the trapping argument to objectives in
$C^{p,\beta_p}$ and exploiting higher-order smoothness.
Writing $\nu:=p+\beta_p$, our first result is summarized below.

\begin{theorem}[\bf informal, see Theorem~\ref{the:main1}]
\label{thm:informal-low-dimensional-trapping}
For objective in $C^{p,\beta_p}$, there exist $0$-th order
deterministic trapping algorithms that find an
$\varepsilon$-stationary point in the unconstrained setting using $\mathcal{O}_{d,p}
    \bigl(\varepsilon^{-(d-1)}\bigr)$
queries and an $\varepsilon$-KKT point over $[0,1]^d$ using $    \mathcal{O}_{d,p}
    \bigl(\varepsilon^{-(d-1)/\nu}\bigr)$
queries.
\end{theorem}

These results  make progress on Question~\ref{question:constraints} by strengthening the known contrast between constrained and
unconstrained stationarity. In dimension two, our constrained
guarantee improves from $\mathcal{O}(\varepsilon^{-\frac{1}{2}})$ under
$C^{1,1}$ smoothness to $\mathcal{O}(\varepsilon^{-\frac{1}{3}})$ under
$C^{2,1}$ smoothness. By comparison, finding an unconstrained
$\varepsilon$-stationary point still requires
$\Omega(\varepsilon^{-1})$ deterministic queries under
$C^{2,1}$ smoothness, even with access to Hessian evaluations
\cite[Appendix~F]{kontogiannis2024computational}.
Thus, our result establishes a stronger separation between the two
settings under Hessian Lipschitzness, while the unconstrained
accuracy exponent remains optimal in dimension two.
Regarding Question~\ref{question:smoothness}, these results
show that higher-order smoothness improves the cube-constrained
accuracy exponent to $\frac{d-1}{p+\beta_p}$ using only
function-value queries.

\paragraph{Fast trapping with a gradient-only oracle~(results in Section~\ref{sec:4}).}
Following~\cite{chewi2022complexity}, we study the power of
randomization in the gradient-only oracle model.
For one-dimensional $C^{1,1}$ objectives,
\cite{chewi2022complexity} established a separation between
deterministic and randomized algorithms, showing that randomization
is necessary to improve upon gradient descent when function values
are unavailable.
They further conjectured that randomized gradient-only algorithms
require $\Omega(\varepsilon^{-2})$ queries already in dimension two.
We refute this conjecture: our randomized trapping algorithm finds
an $\varepsilon$-stationary point of a $C^{1,1}$ objective on
$\mathbb R^2$ using
$\widetilde{\mathcal O}(\varepsilon^{-\frac{4}{3}})$ gradient queries.
Together with their deterministic
$\Omega(\varepsilon^{-2})$ lower bound, this establishes a
polynomial separation in dimension two. We also develop a deterministic gradient-only variant and extend both algorithms to $C^{p,\beta_p}$ objectives.

\begin{theorem}[\bf Informal, see Theorem~\ref{the:main_gradient}]
Let $\nu = p+\beta_p$. For $C^{p,\beta_p}(L_p)$ objectives, there exist trapping algorithms
using only exact gradient queries with the following guarantees.
\begin{enumerate}
\item \textbf{Deterministic algorithms.}
An $\varepsilon$-stationary point in the unconstrained setting
can be found using
$\widetilde{\mathcal O}\!\left(
\varepsilon^{-\max\{d-1,\,\nu/(\nu-1)\}}\right)$ queries,
and an $\varepsilon$-KKT point over $[0,1]^d$ can be found using
$\widetilde{\mathcal O}\!\left(
\varepsilon^{-\max\{(d-1)/\nu,\,1/(\nu-1)\}}\right)$ queries.

\item \textbf{Randomized algorithms.}
With high probability, an $\varepsilon$-stationary point in the
unconstrained setting can be found using
$\widetilde{\mathcal O}\!\left(
\varepsilon^{-\max\{d-1,\,2\nu/(2\nu-1)\}}\right)$ queries,
and an $\varepsilon$-KKT point over $[0,1]^d$ can be found using
$\widetilde{\mathcal O}\!\left(
\varepsilon^{-\max\{(d-1)/\nu,\,2/(2\nu-1)\}}\right)$ queries.
\end{enumerate}
\end{theorem}

For Question~\ref{question:randomness}, our theorem suggests the power of randomness in low dimension. Specifically, for two-dimensional $C^{1,1}$ objectives, our deterministic
unconstrained bound is $\mathcal O(\varepsilon^{-2})$, matching
the gradient-only lower bound of~\cite{chewi2022complexity}.
Our randomized bound
$\widetilde{\mathcal O}(\varepsilon^{-\frac{4}{3}})$ therefore establishes
a polynomial separation between the two algorithm classes.
This extends the one-dimensional separation
of~\cite{chewi2022complexity} to dimension two and shows that,
even there, randomization permits gradient-only algorithms to
improve upon gradient descent.

For Question~\ref{question:smoothness}, our theorem also implies that higher-order regularity further improves both two-dimensional
guarantees.
For $C^{2,1}$ objectives, our deterministic bounds become
$\mathcal O(\varepsilon^{-\frac{3}{2}})$ in the unconstrained setting
and $\mathcal O(\varepsilon^{-\frac{1}{2}})$ over $[0,1]^2$ while
randomization improves these bounds to
$\widetilde{\mathcal O}(\varepsilon^{-\frac{6}{5}})$ and
$\widetilde{\mathcal O}(\varepsilon^{-\frac{2}{5}})$, respectively.
Existing gradient-only methods for unconstrained optimization
under Lipschitz gradients and Hessians achieve
$\mathcal O(\varepsilon^{-\frac{7}{4}})$
\cite{li2023restarted},
$\mathcal O(d^{\frac{1}{4}}\varepsilon^{-\frac{13}{8}})$
\cite{jiang2025improved}, and
$\widetilde{\mathcal O}(d^{\frac{1}{3}}\varepsilon^{-\frac{3}{2}}+d)$
\cite{adil2025balancing}.
Our deterministic two-dimensional guarantee matches the
$\varepsilon^{-\frac{3}{2}}$ accuracy exponent of the latter bound,
while our randomized guarantee improves it to
$\varepsilon^{-\frac{6}{5}}$.
More generally, for each fixed $\nu=p+\beta_p>1$, our
two-dimensional deterministic exponents are $\frac{\nu}{\nu-1}$
and $\frac{1}{\nu-1}$, and the randomized exponents are
$\frac{2\nu}{2\nu-1}$ and $\frac{2}{2\nu-1}$.
All four exponents decrease with $\nu$.

For Question~\ref{question:oracle-information}, our theorem also shows that function values also provide a useful benchmark for this separation.
For two-dimensional $C^{1,1}$ objectives, deterministic
value-query trapping uses $ \Theta(\varepsilon^{-1})$ queries
in the unconstrained setting and
$\Theta(\varepsilon^{-\frac{1}{2}})$ queries over $[0,1]^2$
\cite{hollender2023computational}.
Our randomized gradient-only bounds,
$\widetilde{\mathcal O}(\varepsilon^{-\frac{4}{3}})$ and
$\widetilde{\mathcal O}(\varepsilon^{-\frac{2}{3}})$, move closer to
these benchmarks.
For every fixed $d\ge3$, however, our randomized algorithms already
match the accuracy exponents of the preceding value-query trapping
bounds under $C^{p,\beta_p}$ smoothness.
The deterministic variants also match these exponents whenever
$d-1\ge\frac{\nu}{\nu-1}$, which holds for every $d\ge3$ if $p\ge2$
or if the objective is $C^{1,1}$.


For Question~\ref{question:constraints}, the theorem also sharpens the contrast between constrained and unconstrained stationarity in
the gradient-only model.
For two-dimensional $C^{1,1}$ objectives, our deterministic
algorithm finds an $\varepsilon$-KKT point over $[0,1]^2$
using $\mathcal O(\varepsilon^{-1})$ gradient queries,
whereas deterministic gradient-only algorithms require
$\Omega(\varepsilon^{-2})$ queries in the unconstrained
setting~\cite{chewi2022complexity}.
The randomized guarantee sharpens this contrast:
$\widetilde{\mathcal O}(\varepsilon^{-\frac{2}{3}})$ queries suffice
in the constrained setting, polynomially fewer than the
$\Omega(\varepsilon^{-1})$ queries required by randomized
gradient-only algorithms in the unconstrained
setting~\cite{chewi2022complexity}.

\paragraph{Trapping-or-descent framework~(results in Section~\ref{sec:trapping-or-descent}).}
The preceding methods are effective in fixed low dimension, but their
accuracy exponents deteriorate as the dimension grows. To avoid this
deterioration, we introduce a trapping-or-descent framework and its lazy version.

\begin{theorem}[\bf Informal; see
Theorem~\ref{the:lazy-trapping-or-descent}]
\label{thm:informal-lazy-trapping}
There are deterministic trapping-or-descent algorithms that return an
$\varepsilon$-stationary point when $\mathcal X=\mathbb R^d$, and an
$\varepsilon$-KKT point when $\mathcal X=[0,1]^d$, using respectively
$$
\begin{aligned}
N_{\mathrm{fval}}
&=
\mathcal O_{p,\beta_p}\!\left(
d^p+d^{p-1+\frac{\beta_p}{\nu-1}}
\varepsilon^{-\frac{\nu}{\nu-1}}
\right),\qquad
N_{\mathrm{grad}}
=
\mathcal O_{p,\beta_p}\!\left(
d^{p-1}+d^{p-2+\frac{\beta_p}{\nu-1}}
\varepsilon^{-\frac{\nu}{\nu-1}}
\right)
\end{aligned}
$$
function-value queries and strict gradient-only queries.
\end{theorem}
\begin{table}[t]
\caption{
Comparison of representative existing and our oracle-complexity bounds.
The main model is the \emph{strict gradient-only oracle} and deterministic upper bounds also apply to randomized
algorithms.
Here $\nu:=p+\beta_p$.
``Fixed $d$'' means that $d\ge2$ is independent of $\varepsilon$.
Our bounds that improve upon or match the best previously known
bounds in the corresponding settings are highlighted in
\textcolor{red}{red} and \textcolor{blue}{blue}, respectively.
}
\centering
\renewcommand{\arraystretch}{1.55}
\scriptsize
\resizebox{\textwidth}{!}{
\begin{tabular}{|l|l|l|c|c|}
\hline
Problem setup
& Smoothness
& Algorithm
& Existing results
& Our results \\
\hline


\multirow{20}{*}{
\makecell{
Unconstrained\\
($\Omega=\mathbb{R}^d$, $f\geq0$)\\
gradient-only
}}
& \multirow{2}{*}{$C^{1,1}$}
& deterministic
& \makecell[l]{
$\Theta(\varepsilon^{-2})$,  every $d\ge1$
~\cite{chewi2022complexity}
}
& \makecell[l]{
$\widetilde{\mathcal O}_d(\varepsilon^{-(d-1)})$,
fixed $d\ge2$ (Theorem~\ref{the:main_gradient})
}
\\
\cline{3-5}

&
& randomized
& \makecell[l]{
$\Omega(\varepsilon^{-1})$
~\cite{chewi2022complexity};\\
$\Omega(\varepsilon^{-2})$,
$d\gtrsim\varepsilon^{-4}$,
even with $f$ and all derivatives
~\cite{carmon2020lowerI}
}
& \makecell[l]{
$\textcolor{red}{\widetilde{\mathcal O}(\varepsilon^{-\frac{4}{3}})}$, $d=2$
(Theorem~\ref{the:main_gradient});\\
$\widetilde{\mathcal O}_d(\varepsilon^{-(d-1)})$,
fixed $d\ge3$ (Theorem~\ref{the:main_gradient})
}
\\
\cline{2-5}

& \multirow{4}{*}{$C^{2,1}$}
& deterministic
& \makecell[l]{
$\mathcal O(\varepsilon^{-\frac{7}{4}})$,
~\cite{li2023restarted,patitucci2026improving};\\
$\widetilde{\mathcal O}(d+d^{\frac{1}{3}}\varepsilon^{-\frac{3}{2}})$,
\cite{adil2025balancing,chen2026faster};\\
$\mathcal O(\varepsilon^{-\frac{3}{2}})$, given
$f,\nabla f,\nabla^2f$~\cite{cartis2020sharp,gratton2023convergence};\\
$\Omega(\varepsilon^{-\frac{7}{4}})$,
$d\gtrsim\varepsilon^{-\frac{7}{4}}$, even with $(f,\nabla f)$ and additional $C^{1,1}$
~\cite{zhou2026sharp}
}
& \makecell[l]{
\textcolor{blue}{$\widetilde{\mathcal O}(\varepsilon^{-\frac{3}{2}})$}, $d=2$
(Theorem~\ref{the:main_gradient});\\
$\mathcal O\!\left(
d+d^{\frac{1}{2}}\varepsilon^{-\frac{3}{2}}
\right)$~(Theorem~\ref{the:lazy-trapping-or-descent})\\
}
\\
\cline{3-5}

&
& randomized
& \makecell[l]{
$\Omega(\varepsilon^{-\frac{3}{2}})$,
$d\gtrsim\varepsilon^{-3}$
even with $f$ and all derivatives
~\cite{carmon2020lowerI}
}
& \makecell[l]{
$\textcolor{red}{\widetilde{\mathcal O}(\varepsilon^{-\frac{6}{5}})}$, $d=2$~(Theorem~\ref{the:main_gradient});\\
$\widetilde{\mathcal O}_d(\varepsilon^{-(d-1)})$,
fixed $d\ge3$ (Theorem~\ref{the:main_gradient});\\
$\textcolor{red}{\Omega(\varepsilon^{-\frac{7}{4}})}$,
$d\gtrsim\varepsilon^{-5}$, even with $(f,\nabla f)$ and additional $C^{1,1}$\\
~(Theorem~\ref{thm:rand-fo-lower-bound} and Theorem~\ref{thm:drhs-without-l1})
}
\\
\cline{2-5}

& \multirow{4}{*}{$C^{3,1}$}
& deterministic
& \makecell[l]{
$\widetilde{\mathcal O}(\varepsilon^{-\frac{5}{3}})$,
using $(f,\nabla f)$, assuming additionally $C^{1,1}$
~\cite{carmon2017convex,xiong2026parameterfree};\\
$\mathcal O(\varepsilon^{-\frac{4}{3}})$, given
$f,\nabla f,\nabla^2f,\nabla^3f$
~\cite{cartis2020sharp,gratton2023convergence};\\
$\Omega(\varepsilon^{-\frac{5}{3}})$,
$d\gtrsim\varepsilon^{-\frac{5}{3}}$, even with $(f,\nabla f)$ and additional $C^{1,1}$
~\cite{zhou2026sharp}
}
& \makecell[l]{
$\textcolor{red}{\widetilde{\mathcal O}(\varepsilon^{-\frac{4}{3}})}$, $d=2$
(Theorem~\ref{the:main_gradient});\\
$\textcolor{blue}{\mathcal O\!\big(
d^2+d^{\frac{4}{3}}\varepsilon^{-\frac{4}{3}}
\big)}$,~(Theorem~\ref{the:lazy-trapping-or-descent})
}
\\
\cline{3-5}

&
& randomized
& \makecell[l]{
$\widetilde{\mathcal O}(d\varepsilon^{-\frac{4}{3}})$,
using $(f,\nabla f)$, assuming additionally $C^{1,1}$
~\cite{hinder2018cutting};\\
$\Omega(\varepsilon^{-\frac{4}{3}})$,
$d\gtrsim\varepsilon^{-\frac{8}{3}}$,
even with $f$ and all derivatives
~\cite{carmon2020lowerI}
}
& \makecell[l]{
$\textcolor{red}{\widetilde{\mathcal O}(\varepsilon^{-\frac{8}{7}})}$, $d=2$~(Theorem~\ref{the:main_gradient});\\
$\widetilde{\mathcal O}_d(\varepsilon^{-(d-1)})$,
fixed $d\ge3$~(Theorem~\ref{the:main_gradient});\\
$\textcolor{red}{\Omega(\varepsilon^{-\frac{5}{3}})}$,
$d\gtrsim\varepsilon^{-\frac{14}{3}}$, even with $(f,\nabla f)$ and additional $C^{1,1}$\\
~(Theorem~\ref{thm:rand-fo-lower-bound} and Theorem~\ref{thm:drhs-without-l1})
}
\\
\cline{2-5}

& \multirow{5}{*}{$C^{p,\beta_p}$}
& deterministic
& \makecell[l]{
$\mathcal O\!\big(
\varepsilon^{-\frac{1+\beta_p}{\beta_p}}
\big)$ for $p=1$, gradients only, under coercivity
~\cite{yashtini2016global};\\
$\mathcal O\!\big(
\varepsilon^{-\frac{\nu}{\nu-1}}
\big)$, given $f,\nabla f,\ldots,\nabla^pf$
~\cite{cartis2020sharp};\\
$\mathcal O\!\big(
\varepsilon^{-\frac{\nu}{\nu-1}}
\big)$, given $\nabla f,\ldots,\nabla^pf$ \cite{gratton2023convergence};\\
$\mathcal O\!\big(
d^{\frac{1}{p}}\varepsilon^{-\frac{p+1}{p}}+d
\big)$, given $f,\nabla f,\ldots,\nabla^{p-1}f$, $\beta_p=1$, $p\ge2$
~\cite{doikov2026lowerorder};\\
$\Omega(\varepsilon^{-\frac{5}{3}})$ for $p\ge3$, $\beta_p=1$, $d\gtrsim\varepsilon^{-\frac{5}{3}}$, even with $(f,\nabla f)$ and $C^{1,1}\cap\cdots\cap C^{p,1}$
~\cite{zhou2026sharp}
}
& \makecell[l]{
$\textcolor{blue}{\widetilde{\mathcal O}(\varepsilon^{-\frac{\nu}{\nu-1}})}$, $d=2$
(Theorem~\ref{the:main_gradient});\\
\textcolor{red}{$\mathcal O_p\!\big(
d^{p-1}
+d^{p-2+\frac{\beta_p}{\nu-1}}
\varepsilon^{-\frac{\nu}{\nu-1}}
\big)$}, $\beta_p=1$, ~(Theorem~\ref{the:lazy-trapping-or-descent})
}
\\
\cline{3-5}

&
& randomized
& \makecell[l]{
$\Omega\!\big(\varepsilon^{-\frac{p+1}{p}}\big)$,
$\beta_p=1$,
$d\gtrsim\varepsilon^{-\frac{2(p+1)}{p}}$,
even with $f$ and all derivatives
~\cite{carmon2020lowerI}
}
& \makecell[l]{
\textcolor{red}{$\widetilde{\mathcal O}_{p}(\varepsilon^{-\frac{2\nu}{2\nu-1}})$},
$d=2$, $p\ge1$~(Theorem~\ref{the:main_gradient});\\
$\widetilde{\mathcal O}_{d,p}(\varepsilon^{-(d-1)})$,
fixed $d\ge3$, $p\ge1$~(Theorem~\ref{the:main_gradient});\\
\textcolor{red}{$\Omega_p(\varepsilon^{-\frac{5}{3}})$},
$d\gtrsim\varepsilon^{-\big(\frac{13}{3}+\frac{1}{p}\big)}$,\\
$\beta_p=1$, $p\ge3$, under $C^{p,1}$ alone,even with $(f,\nabla f)$\\
~(Theorem~\ref{thm:drhs-without-l1});\\
\textcolor{red}{$\Omega_p\!\big(
\varepsilon^{-(\frac32+\frac{1}{2p})}
\big)$},
$d\gtrsim\varepsilon^{-(4+2/p)}$,\\
$\beta_p=1$, $p\ge2$, even under additional $C^{1,1}$,
even with $(f,\nabla f)$\\
~(Theorem~\ref{thm:rand-fo-lower-bound})
}
\\
\hline


\multirow{10}{*}{
\makecell{
Constrained\\
($\Omega=[0,1]^d$)\\
gradient-only
}}
& \multirow{2}{*}{$C^{1,1}$}
& deterministic
& \makecell[l]{
$\mathcal O(\varepsilon^{-2})$,
projected gradient, gradients only
~\cite{ghadimi2016minibatch}
}
& \makecell[l]{
\textcolor{red}{$\widetilde{\mathcal O}_d(\varepsilon^{-\max\{\frac{d-1}{2},1\}})$},
(Theorem~\ref{the:main_gradient});\\
}
\\
\cline{3-5}

&
& randomized
& --
& \makecell[l]{
\textcolor{red}{$\widetilde{\mathcal O}_d\varepsilon^{-\max\{\frac{d-1}{2},\frac{2}{3}\}})$},
(Theorem~\ref{the:main_gradient});\\
}
\\
\cline{2-5}

& \multirow{2}{*}{$C^{2,1}$}
& deterministic
& \makecell[l]{
$\mathcal O(d+d^{\frac{1}{2}}\varepsilon^{-\frac{3}{2}})$,
using $(f,\nabla f)$,~\cite{doikov2025first};\\
$\mathcal O(\varepsilon^{-\frac{3}{2}})$,
given $f,\nabla f,\nabla^2f$
~\cite{cartis2020sharp}
}
& \makecell[l]{
\textcolor{red}{$\widetilde{\mathcal O}_d(\varepsilon^{-\max\{\frac{1}{2},\frac{d-1}{3}\}})$}, 
(Theorem~\ref{the:main_gradient})\\
\textcolor{blue}{$\mathcal O\!\big(d+d^{\frac{1}{2}}\varepsilon^{-\frac{3}{2}}\big)$},
~(Theorem~\ref{the:lazy-trapping-or-descent})
}
\\
\cline{3-5}

&
& randomized
& --
& \makecell[l]{
\textcolor{red}{$\widetilde{\mathcal O}_d(\varepsilon^{-\max\{\frac{2}{5},\frac{d-1}{3}\}})$}, ~(Theorem~\ref{the:main_gradient});\\
}
\\
\cline{2-5}

& \multirow{2}{*}{$C^{3,1}$}
& deterministic
& \makecell[l]{
$\mathcal O_d(\varepsilon^{-\frac{4}{3}})$,
given $f,\nabla f,\nabla^2f,\nabla^3f$~\cite{cartis2020sharp}
}
& \makecell[l]{
\textcolor{red}{$\widetilde{\mathcal O}_d(\varepsilon^{-\max\{\frac{1}{3},\frac{d-1}{4}\}})$}, 
(Theorem~\ref{the:main_gradient})\\
\textcolor{blue}{$\mathcal O\!\big(d^2+d^{\frac{4}{3}}\varepsilon^{-\frac{4}{3}}\big)$}~(Theorem~\ref{the:lazy-trapping-or-descent}),
}
\\
\cline{3-5}

&
& randomized
& --
& \makecell[l]{
\textcolor{red}{$\widetilde{\mathcal O}_d(\varepsilon^{-\max\{\frac{2}{7},\frac{d-1}{4}\}})$, ~(Theorem~\ref{the:main_gradient})};\\
}
\\
\cline{2-5}

& \multirow{2}{*}{$C^{p,\beta_p}$}
& deterministic
& \makecell[l]{
$\mathcal O\!\big(
\varepsilon^{-\frac{\nu}{\nu-1}}
\big)$,
given $f,\nabla f,\ldots,\nabla^pf$~\cite{cartis2020sharp}
}
& \makecell[l]{
\textcolor{red}{$\widetilde{\mathcal O}_{d,p}(\varepsilon^{-\max\{\frac{d-1}{\nu},\,\frac{1}{\nu-1}\}})$},
(Theorem~\ref{the:main_gradient})\\
\textcolor{blue}{$\mathcal O_p\!\big(
d^{p-1}
+d^{p-2+\frac{\beta_p}{\nu-1}}
\varepsilon^{-\frac{\nu}{\nu-1}}
\big)$}, ~(Theorem~\ref{the:lazy-trapping-or-descent})
}
\\
\cline{3-5}

&
& randomized
& --
& \makecell[l]{
\textcolor{red}{$\widetilde{\mathcal O}_{d,p}
(\varepsilon^{-\max\{\frac{2}{2\nu-1},\frac{d-1}{\nu}}\})$}, $p\ge1$~(Theorem~\ref{the:main_gradient});\\
}
\\
\hline

\end{tabular}
}
\label{tab:gradient-only-comparison}
\end{table}
Together, these two guarantees suggest that trapping admits efficient
implementations in high dimensions using
either function values alone or gradients alone.
Both implementations exploit arbitrary-order H\"older smoothness
without querying higher derivatives, retaining the
$\varepsilon^{-\frac{\nu}{\nu-1}}$ accuracy dependence of higher-order
regularization methods~\citep{cartis2019universal}, with explicit
polynomial dependence on dimension. When $\beta_p=1$, this
specializes to the classical AR$p$ exponent $\frac{p+1}{p}$.

For $p=2$ and $\beta_2=1$, the leading function-value term is
$\mathcal O(d^{\frac{3}{2}}
\varepsilon^{-\frac{3}{2}})$,
matching the zeroth-order complexity of regularized Newton methods
with lazy finite-difference Hessian
updates~\citep{doikov2025first}.
The strict gradient-only implementation has leading complexity
$\mathcal O(d^{\frac{1}{2}}
\varepsilon^{-\frac{3}{2}})$.
In the unconstrained gradient-only setting, \cite{adil2025balancing} obtain a better dependence on dimension, with complexity
$\widetilde{\mathcal O}
(d^{\frac{1}{3}}
\varepsilon^{-\frac{3}{2}}+d)$. For $p=3$, our framework achieves an accuracy exponent of
$\frac{3+\beta_3}{2+\beta_3}$ using either function values
alone or gradients alone.
In particular, under Lipschitz third derivatives, it attains
a new query complexity of $\mathcal O(d^{\frac{4}{3}}\varepsilon^{-\frac{4}{3}}+d^2)$ without Hessian or third-derivative queries.
Both implementations apply to unconstrained and
box-constrained optimization.

Our contribution is a unified trapping framework that combines arbitrary-order H\"older smoothness and low-order oracle implementations.
The framework applies to both constrained and unconstrained
optimization and yields new query complexity upper bounds
for objectives with H\"older-continuous third derivatives. 
These results further make progress on
Question~\ref{question:smoothness} by achieving an accuracy
exponent of $\frac{\nu}{\nu-1}$ that improves with smoothness,
and on Question~\ref{question:oracle-information} by preserving
the function-value accuracy exponent under strict
gradient-only access.

\paragraph{High-dimensional lower bounds for randomized gradient-only
algorithms~(results in Section~\ref{sec:lb}).}
The preceding fixed-dimensional upper bounds show that randomization
can substantially improve gradient-only optimization in low dimension.
We next ask whether this advantage persists in the high-dimensional
regime, where the target query rate is independent of \(d\). We answer
this question completely for second-order smoothness and establish a
general lower bound for every finite smoothness order.

\begin{theorem}[\bf Informal; see
Theorems~\ref{thm:rand-fo-lower-bound} and \ref{thm:drhs-without-l1}]
\label{thm:informal-drhs}
In sufficiently high dimension, every randomized algorithm using only
exact gradient queries requires
$$
\begin{cases}
\Omega(\varepsilon^{-\frac{7}{4}}),
    & f\in C^{2,1},\\[2mm]
\Omega_p(\varepsilon^{-\frac{5}{3}}),
    & f\in C^{p,1},\quad p\geq3,\\[2mm]
\Omega_p\!\left(
    \varepsilon^{-(\frac{3}{2}+\frac{1}{2p})}
\right),
    & f\in C^{1,1}\cap C^{p,1},\quad p\geq2,
\end{cases}
$$
queries to return an \(\varepsilon\)-stationary point with constant
success probability. The hard instances can be chosen \(C^\infty\)
and coercive.
\end{theorem}

Without gradient Lipschitzness, our lower bound is
\(\Omega(\varepsilon^{-\frac{7}{4}})\) for \(p=2\) and
\(\Omega_p(\varepsilon^{-\frac{5}{3}})\) for \(p\geq3\). In comparison, the
optimal randomized complexity for algorithms accessing derivatives
through order \(p\) is
\(\Theta_p(\varepsilon^{-\frac{p+1}{p}})\)
\cite{BGM+17,cartis2020sharp,carmon2020lowerI}. Thus our result separates
gradient-only algorithms from higher-order methods. Previous
deterministic value-and-gradient lower bounds gave the exponents
\(\frac{12}{7}\) for \(p=2\) and \(\frac{8}{5}\) for \(p\geq3\)
\cite{carmon2021lowerII} and these were recently improved to \(\frac{7}{4}\) and
\(\frac{5}{3}\), respectively, by \cite{zhou2026sharp}. Our theorem establishes
the same exponents for randomized strict gradient-only algorithms. 

With gradient Lipschitzness, our lower bound interpolates between the
classical \(\Omega(\varepsilon^{-2})\) rate and
\(\Omega_p(\varepsilon^{-(\frac{3}{2}+\frac{1}{2p})})\).
For \(p=2\), the resulting \(\Omega(\varepsilon^{-\frac{7}{4}})\) bound matches
the dimension-free deterministic strict gradient-only upper bound of
\cite{li2023restarted}. Hence randomization does not improve the
high-dimensional strict gradient-only complexity for \(p=2\).
For \(p=3\), our exponent \(\frac{5}{3}\) matches, up to logarithmic factors,
the deterministic value-and-gradient upper bound of
\cite{carmon2017convex}. 
A recent parameter-free method attains the
same exponent but still uses function values
\cite{xiong2026parameterfree} but the strict gradient-only complexity therefore
remains open between the exponents \(\frac{5}{3}\) and \(\frac{7}{4}\). For \(p\geq4\),
the deterministic value-and-gradient complexity saturates at
\(\widetilde\Theta_p(\varepsilon^{-\frac{5}{3}})\)
\cite{carmon2017convex,zhou2026sharp}, whereas our randomized
Closing this gap remains open and, because
the comparison also differs in access to function values, it cannot
be attributed solely to randomization. Finally, dimension-dependent
gradient-only bounds such as
\(O(d^{\frac{1}{4}}\varepsilon^{-\frac{13}{8}})\)
\cite{jiang2025improved} apply in smaller dimensions and do not
conflict with the high-dimensional regime considered here.

These results make progress on
Question~\ref{question:randomness} by establishing lower bounds
against randomized gradient-only algorithms,
and on Question~\ref{question:smoothness} by showing that
higher-order smoothness $C^{p,1}$  alone cannot reduce the accuracy
exponent of randomized algorithm below $\frac{5}{3}$ for  any fixed $p\ge3$
in sufficiently high dimensions.

Tables~\ref{tab:function-value-comparison}
and~\ref{tab:gradient-only-comparison} summarize the function-value and
strict gradient-only comparisons separately.

\section{Preliminaries}
\label{sec:problem}

In Sections~\ref{sec:3} and~\ref{sec:4}, the dimension \(d\ge 2\)
and the integer smoothness order \(p\ge 1\) are fixed independently
of the target accuracy \(\varepsilon\).
In Sections~\ref{sec:trapping-or-descent} and~\ref{sec:lb}, the dimension \(d\) is allowed to be arbitrarily large relative to any polynomial in \(\varepsilon^{-1}\).

We identify \(\nabla^p f(x)\) with the corresponding symmetric
\(p\)-linear form. For such a form \(T\), define its operator norm by
\[
    \|T\|_{\mathrm{op}}
    :=
    \sup_{\|v_1\|_2,\ldots,\|v_p\|_2\le 1}
    \bigl|T[v_1,\ldots,v_p]\bigr|.
\]

\begin{definition}[\bf \((p,\beta)\)-H\"older smoothness]
\label{def:holder}
Let \(p\ge 1\) be an integer, \(\beta\in(0,1]\), and \(L_p\ge 0\).
A function \(f:\mathbb R^d\to\mathbb R\) belongs to
\(C^{p,\beta}(L_p)\) if \(f\in C^p(\mathbb R^d)\) and
\[
    \bigl\|\nabla^p f(x)-\nabla^p f(y)\bigr\|_{\mathrm{op}}
    \le
    L_p\|x-y\|_2^\beta
    \qquad
    \text{for all }x,y\in\mathbb R^d.
\]
We call \(p+\beta\) the corresponding fractional smoothness order and
write \(C^{p,\beta}\) when \(L_p\) is clear from context.
\end{definition}

\begin{assumption}[\bf Lower boundedness]
\label{ass:lower-bounded}
The objective has a finite infimum:
\[
    f_\star
    :=
    \inf_{x\in\mathbb R^d} f(x)
    >
    -\infty.
\]
\end{assumption}

Adding a constant to \(f\) changes neither its derivatives nor its
stationary points. Throughout the unconstrained analysis, we assume
that \(f\) is nonnegative. Equivalently, the algorithm is given the
known global lower bound \(0\). The actual value of \(f_\star\) need
not be known.

\paragraph{Unconstrained problem and stationary points.}
The algorithm is given an initial point \(x^0\in\mathbb R^d\) and a
known number \(\Delta>0\) satisfying
\(f(x^0)-f_\star\le \Delta\).
Its goal is to return a point \(\widehat x\in\mathbb R^d\) such that
\[
    \|\nabla f(\widehat x)\|_2\le \varepsilon.
\]
This is the standard first-order stationarity objective considered in
fixed-dimensional oracle-complexity work; see, for example,
\cite{chewi2022complexity,hollender2023computational}.

\paragraph{Constrained problem and KKT points.}
For the constrained problem on \([0,1]^d\), fix a stepsize
\(\gamma>0\), independently of \(\varepsilon\), and define the
projected-gradient mapping
\[
    G_\gamma(x)
    :=
    \frac{1}{\gamma}
    \left(
        x-
        \operatorname{proj}_{[0,1]^d}
        \bigl(x-\gamma\nabla f(x)\bigr)
    \right),
    \qquad x\in[0,1]^d,
\]
where \(\operatorname{proj}_{[0,1]^d}\) denotes Euclidean projection
onto \([0,1]^d\). We call \(x\in[0,1]^d\) an
\(\varepsilon\)-KKT point, with respect to \(\gamma\), if
\(\|G_\gamma(x)\|_2\le \varepsilon\).

\paragraph{Oracle models.}
A value-oracle query at \(x\) returns the exact value \(f(x)\).
A gradient-only query at \(x\) returns the exact vector
\(\nabla f(x)\), but does not return \(f(x)\).
When both quantities are returned by one query, we refer to a
value-and-gradient oracle.
In both the unconstrained and constrained problems, queries may be
made at arbitrary points of \(\mathbb R^d\).

A randomized algorithm has access to an internal random seed \(U\)
that is independent of the objective \(f\).
For a prescribed failure probability \(\eta\in(0,1)\), the algorithm
must satisfy
\[
    \mathbb P_U
    \left(
        \text{the returned point is an \(\varepsilon\)-solution}
    \right)
    \ge 1-\eta
\]
for every admissible objective \(f\).
Here, \(\varepsilon\)-solution means an
\(\varepsilon\)-stationary point in the unconstrained problem and an
\(\varepsilon\)-KKT point in the constrained problem.

Our query bounds are worst-case bounds over both the objective and
the internal randomness. More precisely, a query-complexity bound
\(Q\) means that
\[
    \sup_f\sup_u N(f,u)\le Q,
\]
where \(N(f,u)\) denotes the number of oracle calls made on objective
\(f\) when the random seed equals \(u\).
Thus \(Q\) is a deterministic upper bound for every realization of
the internal randomness, rather than an expected or
high-probability bound.

\section{$p$-th Order Gradient Flow Trapping}
\label{sec:3}

Gradient-flow trapping provides a way to find approximate stationary
points using only function evaluations while requiring few adaptive
rounds. Existing implementations, however, have focused primarily on
objectives with first-order smoothness
\cite{bubeck2020trap,hollender2023computational,huanjian2025adaptive}.
It remains open whether higher-order smoothness can reduce the
number of function evaluations required. In this section, we answer this question by developing a deterministic
\(p\)-th order gradient-flow trapping method for functions in
\(C^{p,\beta_p}(L_p)\). For fixed dimension, the
result retains logarithmic adaptive depth and gives different
function-value complexities for unconstrained optimization and
optimization over a fixed bounded domain. In particular, higher-order
smoothness leads to an improved dependence on the target accuracy in
the constrained setting. The following theorem states the main result
of this section.

\begin{theorem}[\bf{$p$-th order trapping via function value oracle}]
\label{the:main1}
Assume target function $f$ lies in $C^{p,\beta_p}(L_p)$.
For $\varepsilon>0$, $d = \Theta(1)$, there is a deterministic algorithm finding $\varepsilon$-stationary points, running within $k = \mathcal{O}(\log (\frac{1}{\varepsilon}))$-round, with
\begin{itemize}[leftmargin =14pt]
    \item $\mathcal{O}(\varepsilon^{-(d-1)})$ function value queries per iteration when $f$ is unconstrained but lower bounded,
    \item $\mathcal{O}(\varepsilon^{ - \frac{d-1}{p+\beta_p}})$ function value queries per iteration when $f$ is constrained on $[0,1]^d$.
\end{itemize}   
\end{theorem}

The theorem extends deterministic trapping to objectives~\citep{hollender2023computational,huanjian2025adaptive} with
higher-order H\"older smoothness using only function evaluations.
Since the number of rounds is logarithmic in $1/\varepsilon$,
the total query complexity is
$\widetilde{\mathcal O}(\varepsilon^{-(d-1)})$
in the unconstrained setting and
$\widetilde{\mathcal O}(\varepsilon^{-\frac{d-1}{p+\beta_p}})$
on $[0,1]^d$.
Higher-order smoothness therefore improves the accuracy dependence
in the constrained setting, whereas the stated unconstrained bound
has an exponent independent of the smoothness order.

\begin{remark}[\textbf{No additional $C^{1,1}$ assumption}]
For $p\ge2$, Theorem~\ref{the:main1} requires only
$f\in C^{p,\beta_p}(L_p)$, without an additional
$C^{1,1}(L_1)$ assumption, and its query bounds do not depend on $L_1$.
This distinction matters because the highest-order H\"older
condition does not control the gradient Lipschitz constant.
\end{remark}

The rest of this section is organized as follows.
Section~\ref{sec:trapping-to-stationarity} develops the trapping
certificates and shows how to compute approximate stationary points
without an additional $C^{1,1}$ assumption.
Section~\ref{subsec:trapping-algorithm} presents the resulting
algorithm for both unconstrained and cube-constrained optimization.
Finally, Section~\ref{sec:main1-analysis} outlines the proof of
Theorem~\ref{the:main1}, including correctness, query complexity,
and adaptive depth.



\subsection{From Trapping to Stationarity}
\label{sec:trapping-to-stationarity}

In this section, we remove the additional $C^{1,1}$ assumption from
the terminal step in standard gradient-flow trapping~\cite{hollender2023computational,huanjian2025adaptive}. A standard termination
argument first shows that a small trapped region $R$ contains an
approximate stationary point $\boldsymbol z$, and then uses gradient
Lipschitzness to certify the current pivot $\boldsymbol x$:
$\|\nabla f(\boldsymbol x)\|_2
\le\|\nabla f(\boldsymbol z)\|_2+L_1\operatorname{diam}(R)$.
The constrained argument similarly uses continuity of $G_\gamma$.


In the following, Section~\ref{subsec:trapping-certificates} recalls the boundary
certificates and their implications for the existence of stationary
and KKT points to a prescribed accuracy.
Section~\ref{subsec:stationarity-from-models} shows how to compute
stationary points or KKT points by minimizing a gradient-accurate scalar
model, without an additional $C^{1,1}$ assumption.

\subsubsection{Unreachability Certificates}
\label{subsec:trapping-certificates}

\begin{definition}[\bf Hyperrectangles]
\label{def:full-dimensional-hyperrectangle}
A hyperrectangle in $\mathbb R^d$ is a set
$R=\prod_{i=1}^d[a_i,b_i]$ with $a_i\le b_i$.
Its dimension is the number $k$ of indices satisfying $a_i<b_i$.
It is \emph{full-dimensional} if $k=d$.
\end{definition}

\begin{definition}[\bf Unreachability and trapping function]
\label{def:trapping-certificate}
For $\boldsymbol x,\boldsymbol y\in\mathbb R^d$ and $\eta>0$, define
\begin{equation}
\label{eq:trapping-potential}
    \phi_{\boldsymbol x,\eta}(\boldsymbol y)
    :=f(\boldsymbol y)-f(\boldsymbol x)
      +\eta\|\boldsymbol y-\boldsymbol x\|_2.
\end{equation}
We say that $\boldsymbol y$ is \emph{$\eta$-unreachable from
$\boldsymbol x$} if $\phi_{\boldsymbol x,\eta}(\boldsymbol y)>0$,
and \emph{$\eta$-reachable} otherwise.
\end{definition}

\noindent\textbf{Physical and artificial boundaries.}
For a compact set $R\subseteq [0,1]^d$, write $\operatorname{int}_{[0,1]^d}(R)$
and $\partial_{[0,1]^d} R$ for its interior and boundary in the subspace
topology of the cube. Thus
$\partial_{[0,1]^d} R=R\setminus\operatorname{int}_{[0,1]^d}(R)$.
For a full-dimensional hyperrectangle $R=\prod_i[a_i,b_i]\subseteq [0,1]^d$,
let $F_i^-(R):=\{\boldsymbol y\in R:y_i=a_i\}$ and
$F_i^+(R):=\{\boldsymbol y\in R:y_i=b_i\}$.
The lower facet is physical if $a_i=0$, and the upper facet is
physical if $b_i=1$; the remaining facets are artificial.
Its artificial boundary is
\begin{equation}
\label{eq:trapping-box-boundary}
    \partial_{\mathrm{art}}R
    :=\bigcup_{i:a_i>0}F_i^-(R)
       \;\cup\;
       \bigcup_{i:b_i<1}F_i^+(R)
    =\partial_{[0,1]^d} R.
\end{equation}
Each artificial facet includes its intersections with physical facets.
Indeed, a point on an artificial facet admits arbitrarily close points
of the cube outside $R$, whereas a point on no artificial facet has a
neighborhood in which $R$ and the cube coincide.
In particular, $\partial_{\mathrm{art}}([0,1]^d)=\varnothing$.
For the ball $Q_{\boldsymbol x}(r):=\boldsymbol x+rB_2^d$, where
$B_2^d$ is the closed Euclidean unit ball and $r>0$, the feasible cap
satisfies
$\partial_{[0,1]^d}(Q_{\boldsymbol x}(r)\cap [0,1]^d)
\subseteq\Gamma_{\boldsymbol x}(r)$, where
$\Gamma_{\boldsymbol x}(r):=
\{\boldsymbol y\in [0,1]^d:\|\boldsymbol y-\boldsymbol x\|_2=r\}$.

For a closed convex set $C$ and $\boldsymbol z\in C$, let
$N_C(\boldsymbol z):=
\{\boldsymbol n:\langle\boldsymbol n,\boldsymbol y-\boldsymbol z\rangle
\le0\text{ for every }\boldsymbol y\in C\}$.
We will use the following consequence of boundary unreachability as an analog of Lemma 3 in \cite{hollender2023computational}.

\begin{lemma}[\bf Trapping implies existence of stationarity]
\label{lem:trapping-implies-stationarity}
Let $R\subseteq\mathbb R^d$ be nonempty, compact, and convex,
$\boldsymbol x\in R$, and $\eta>0$.
If every point of $\partial R$ is $\eta$-unreachable from
$\boldsymbol x$, then $R$ contains a point $\boldsymbol z$ with
$\|\nabla f(\boldsymbol z)\|_2\le\eta$.
In the constrained setting, if $R\subseteq [0,1]^d$ and every point of
$\partial_{[0,1]^d} R$ is $\eta$-unreachable from $\boldsymbol x$, then $R$
contains a point $\boldsymbol z$ with
$\|G_\gamma(\boldsymbol z)\|_2\le\eta$ for every $\gamma>0$.
\end{lemma}

\begin{proof}
Let $\boldsymbol z$ minimize $\phi_{\boldsymbol x,\eta}$ over $R$ and existence follows from compactness and continuity.
Its value is at most $\phi_{\boldsymbol x,\eta}(\boldsymbol x)=0$.
The boundary certificate therefore gives
$\boldsymbol z\in\operatorname{int}(R)$ in the unconstrained case,
and $\boldsymbol z\in\operatorname{int}_{[0,1]^d}(R)$ in the constrained case.
First-order optimality gives
$\nabla f(\boldsymbol z)+\eta\boldsymbol v+\boldsymbol n=0$
for some $\boldsymbol v\in
\partial\|\cdot-\boldsymbol x\|_2(\boldsymbol z)$ and
$\boldsymbol n\in N_R(\boldsymbol z)$.
The norm is continuous and convex, so the subgradient sum rule
applies, and $\|\boldsymbol v\|_2\le1$.

In the unconstrained case, $N_R(\boldsymbol z)=\{\boldsymbol 0\}$,
which gives $\|\nabla f(\boldsymbol z)\|_2\le\eta$.
For the constrained case, we verify that
$N_R(\boldsymbol z)=N_{[0,1]^d}(\boldsymbol z)$ whenever
$\boldsymbol z\in\operatorname{int}_{[0,1]^d}(R)$.
The inclusion $N_{[0,1]^d}(\boldsymbol z)\subseteq N_R(\boldsymbol z)$ follows
from $R\subseteq [0,1]^d$.
Conversely, if $\boldsymbol n\in N_R(\boldsymbol z)$ and
$\boldsymbol y\in [0,1]^d$, then for sufficiently small $t>0$,
$\boldsymbol z+t(\boldsymbol y-\boldsymbol z)\in R$ by convexity
of the cube and local equality of $R$ and the cube.
Hence $t\langle\boldsymbol n,\boldsymbol y-\boldsymbol z\rangle\le0$,
which proves the reverse inclusion.

Consequently, $\boldsymbol n\in N_{[0,1]^d}(\boldsymbol z)$ and
$\operatorname{proj}_{[0,1]^d}(\boldsymbol z+\gamma\boldsymbol n)=\boldsymbol z$.
For this $\boldsymbol n$, nonexpansiveness of projection gives
\begin{equation}
\label{eq:gradient-mapping-normal-residual}
\begin{split}
    \|G_\gamma(\boldsymbol z)\|_2
    &=\gamma^{-1}
      \|\operatorname{proj}_{[0,1]^d}(\boldsymbol z+\gamma\boldsymbol n)
      -\operatorname{proj}_{[0,1]^d}(\boldsymbol z-\gamma\nabla f(\boldsymbol z))\|_2\\
    &\le\|\nabla f(\boldsymbol z)+\boldsymbol n\|_2
     \le\eta.
\end{split}
\end{equation}
The first inequality holds for any $\boldsymbol z\in [0,1]^d$ and
$\boldsymbol n\in N_{[0,1]^d}(\boldsymbol z)$, independently of the trapping
certificate.
\end{proof}

\paragraph{Unreachability Certification via Global Minimization.}
Since the boundary contains infinitely many points,
we instead use a \emph{global} approximate minimizer of the
trapping function to certify unreachability, as the following
lemma shows.
\begin{lemma}
\label{lmm:to_local_minima}
Let $\boldsymbol{y}\in E$ be an $\varepsilon_0$-approximate minimizer of $\phi_{\boldsymbol{x},\varepsilon}$ over $E$, \emph{i.e.},
\[
\phi_{\boldsymbol{x},\varepsilon}(\boldsymbol{y}) \le \min\limits_{\boldsymbol{z}\in E}\phi_{\boldsymbol{x}, \varepsilon}(\boldsymbol{z})+\varepsilon_0.
\]
If $\boldsymbol{y}$ is $\varepsilon$-unreachable from $\boldsymbol{x}$, then every point $\boldsymbol{z}\in E$ is $\varepsilon'$-unreachable from $\boldsymbol{x}$, where
\[
\varepsilon'=\varepsilon+\frac{\varepsilon_0}{\mathrm{dist}(\boldsymbol{x},E)}.
\]
\end{lemma}

\begin{proof}
For any point $\boldsymbol{z} \in E$, since $\phi_{\boldsymbol{x},\varepsilon}(\boldsymbol{y}) \le \min\limits_{\boldsymbol{z}\in E}\phi_{\boldsymbol{x},\varepsilon}(\boldsymbol{z})+\varepsilon_0$, we have 
\[f(\boldsymbol{z})  + \varepsilon\|\boldsymbol{z} - \boldsymbol{x}\| + \varepsilon_0 \geq f(\boldsymbol{y}) + \varepsilon\|\boldsymbol{y} - \boldsymbol{x}\|.  \]
Combining the fact that $\boldsymbol{y}$ is $\varepsilon$-unreachable from $\boldsymbol{x}$, we have
\begin{align*}
f(\boldsymbol{z}) ~\geq~& f(\boldsymbol{y})  + \varepsilon\|\boldsymbol{y} - \boldsymbol{x}\| - \varepsilon\|\boldsymbol{z} - \boldsymbol{x}\| - \varepsilon_0   \\
~\geq~& f(\boldsymbol{x})  - \varepsilon\|\boldsymbol{z} - \boldsymbol{x}\| - \varepsilon_0   \\
~\geq~& f(\boldsymbol{x})  - \varepsilon\|\boldsymbol{z} - \boldsymbol{x}\| -\frac{\varepsilon_0}{\mathrm{dist}(\boldsymbol{x},E)}  \|\boldsymbol{z} - \boldsymbol{x}\|.
\end{align*}
\end{proof}

\begin{remark}[\bf{Comparison with Lemma~5 in
\cite{hollender2023computational}}]
Lemma~5 in \cite{hollender2023computational} is a special case of
Lemma~\ref{lmm:to_local_minima}. In their setting, a nice
\(\delta\)-net and \(L\)-smoothness yield an
\(\varepsilon_0\)-approximate minimizer of
\(\phi_{\boldsymbol{x},\varepsilon}\) over \(E\), where
\[
\varepsilon_0
=
\frac{\delta^2}{2}
\left(
L+\frac{2\varepsilon}{\mathrm{dist}(\boldsymbol{x},E)}
\right).
\]
Lemma~\ref{lmm:to_local_minima} then recovers their unreachability
parameter exactly. More generally, our lemma reduces the geometric
task of certifying that every point of \(E\) is unreachable to the
global minimization of the trapping function over \(E\), after which
the smoothness of \(f\) can be exploited to construct an accurate
approximate minimizer.
\end{remark}

By Lemma~\ref{lmm:to_local_minima}, the verification of unreachability for all boundary points reduces to find an approximate global minimizer of $\phi_{\boldsymbol{x}, \varepsilon}$. 
In the gradient flow trapping literature~\cite{bubeck2020trap,hollender2023computational,huanjian2025adaptive}, the authors focused on the first-order smooth setting with access only to a function-value oracle.
To obtain an approximate global minimum over a bounded domain, they searched over a $\delta$-net and returned the minimizer among the $\delta$-net points.

A standard approach to approximate global minimization from
function values is to reconstruct the objective and minimize
the resulting approximation
\cite{ivanov1971optimum,novak2012tractability,rudi2020finding}.
Theorem~\ref{thm:global-minimization-known-term} in Appendix~\ref{app:approximation} formalizes this
approach while allowing an additional known continuous term.
Specifically, let $\mathcal Y$ be a $k$-dimensional
hyperrectangle with $\operatorname{diam}(\mathcal Y)\le D$,
let $f\in C^{p,\beta}(L_p)$ on a neighborhood of $\mathcal Y$,
and let $q\in C(\mathcal Y)$ be known.
For any $\delta>0$, one nonadaptive batch of
\[
    \mathcal O_{k,p,\beta}\!\left(
        1+D^k
        \left(\frac{L_p}{\delta}\right)^{k/(p+\beta)}
    \right)
\]
exact function-value queries returns a point
$\hat{\boldsymbol y}\in\mathcal Y$ satisfying
$(f+q)(\hat{\boldsymbol y})
\le\min_{\boldsymbol y\in\mathcal Y}(f+q)(\boldsymbol y)+\delta$.
For the trapping subproblem, we take
$q(\boldsymbol y)
=-f(\boldsymbol x)+\eta\|\boldsymbol y-\boldsymbol x\|_2$.
Each rectangular separator has intrinsic dimension $k=d-1$,
and an isometric parameterization preserves the
$C^{p,\beta_p}(L_p)$ smoothness bound.

\begin{remark}[\bf Query versus computational complexity]
In this work, we only concerns oracle complexity. No claim is made about the computational
cost of polynomial reconstruction or the subsequent global search.
This follows the standard evaluation-complexity convention for
\(\mathrm{AR}p\) methods
\cite{BGM+17,cartis2020sharp}, which counts evaluations of
\(f\) and its derivatives but excludes the internal cost of
approximately minimizing each regularized Taylor model.
\end{remark}

\subsubsection{Stationarity from Approximate Models}
\label{subsec:stationarity-from-models}

The preceding lemma guarantees the existence of an approximate
stationary or KKT point in $R$. We now compute such a point from
an explicit continuously differentiable scalar model
$P:\mathbb R^d\to\mathbb R$ whose gradient approximates $\nabla f$
uniformly on $R$. The required model constructions are given in
Appendix~\ref{app:approximation}.

For a target accuracy $\varepsilon>0$ and a pivot
$\boldsymbol x\in R$, define the penalized model
\begin{equation}
\label{eq:penalized-scalar-model}
    \Psi(\boldsymbol y)
    :=P(\boldsymbol y)-P(\boldsymbol x)
      +\frac{3\varepsilon}{4}\|\boldsymbol y-\boldsymbol x\|_2.
\end{equation}
The distance penalty converts the trapping certificate into a
strict boundary barrier. Minimizing $P$ alone may return a point
on an artificial boundary, where constrained optimality over $R$
need not imply the required stationarity or KKT guarantee.

\begin{lemma}[\bf Stationarity from an approximate model]
\label{lem:stationarity-from-model}
Let $R\subseteq\mathbb R^d$ be nonempty, compact, and convex,
let $\boldsymbol x\in R$, and let $\varepsilon>0$.
Suppose that $f$ and $P$ are continuously differentiable on a
neighborhood of $R$ and satisfy
$\|\nabla P-\nabla f\|_{\infty,R}\le\frac{\varepsilon}{8}$.
Every minimizer
$\boldsymbol z\in
\operatorname*{argmin}\limits_{\boldsymbol y\in R}\Psi(\boldsymbol y)$
has the following properties:
\begin{enumerate}[leftmargin =* ]
\item If every point of $\partial R$ is $(\frac{\varepsilon}{2})$-unreachable
from $\boldsymbol x$, then $\boldsymbol z\in\operatorname{int}(R)$
and $\|\nabla f(\boldsymbol z)\|_2\le\varepsilon$.

\item If $R\subseteq[0,1]^d$ and every point of
$\partial_{[0,1]^d}R$ is $(\frac{\varepsilon}{2})$-unreachable from
$\boldsymbol x$, then
$\boldsymbol z\in\operatorname{int}_{[0,1]^d}(R)$ and
$\|G_\gamma(\boldsymbol z)\|_2\le\varepsilon$ for every $\gamma>0$.
\end{enumerate}
\end{lemma}

\begin{proof}
Since $\Psi$ is continuous and $R$ is nonempty and compact,
a global minimizer exists. Fix any such minimizer $\boldsymbol z$.
Because $\boldsymbol x\in R$ and $\Psi(\boldsymbol x)=0$,
we have $\Psi(\boldsymbol z)\le0$.

For every $\boldsymbol y\in R$, convexity keeps the segment from
$\boldsymbol x$ to $\boldsymbol y$ in $R$. Integrating
$\nabla P-\nabla f$ along this segment gives
\begin{equation}
\label{eq:model-increment-error}
\begin{split}
 &|P(\boldsymbol y)-P(\boldsymbol x)-f(\boldsymbol y)+f(\boldsymbol x)|=
 \left|\int_0^1
 \left\langle
 (\nabla P-\nabla f)
 (\boldsymbol x+t(\boldsymbol y-\boldsymbol x)),
 \boldsymbol y-\boldsymbol x
 \right\rangle\,dt\right|
 \le\frac{\varepsilon}{8}\|\boldsymbol y-\boldsymbol x\|_2.
\end{split}
\end{equation}
Consequently, on the certified boundary,
$\Psi(\boldsymbol y)\ge
\phi_{\boldsymbol x,\frac{\varepsilon}{2}}(\boldsymbol y)
+\frac{\varepsilon}{8}\|\boldsymbol y-\boldsymbol x\|_2>0$.
Since $\Psi(\boldsymbol z)\le0$, the minimizer cannot lie on
that boundary. Thus $\boldsymbol z\in\operatorname{int}(R)$
in case~(i), and
$\boldsymbol z\in\operatorname{int}_{[0,1]^d}(R)$ in case~(ii).

The constrained first-order necessary condition for minimizing
the sum of the continuously differentiable function $P$ and
the convex distance penalty supplies vectors
$\boldsymbol v\in\partial\|\cdot-\boldsymbol x\|_2(\boldsymbol z)$
and $\boldsymbol n\in N_R(\boldsymbol z)$ such that
$\nabla P(\boldsymbol z)
+(3\varepsilon/4)\boldsymbol v+\boldsymbol n=\boldsymbol0$.
Here $\|\boldsymbol v\|_2\le1$, including when
$\boldsymbol z=\boldsymbol x$. Therefore,
\begin{equation}
\label{eq:model-normal-residual}
    \|\nabla f(\boldsymbol z)+\boldsymbol n\|_2
    \le
    \|\nabla f(\boldsymbol z)-\nabla P(\boldsymbol z)\|_2
    +\frac{3\varepsilon}{4}\|\boldsymbol v\|_2
    \le\frac{7\varepsilon}{8}
    \le\varepsilon.
\end{equation}
In case~(i), $\boldsymbol z\in\operatorname{int}(R)$ implies
$N_R(\boldsymbol z)=\{\boldsymbol0\}$, so
\eqref{eq:model-normal-residual} proves the claimed gradient bound.

For case~(ii), we first show that
$N_R(\boldsymbol z)=N_{[0,1]^d}(\boldsymbol z)$.
The inclusion
$N_{[0,1]^d}(\boldsymbol z)\subseteq N_R(\boldsymbol z)$
follows from $R\subseteq[0,1]^d$.
Conversely, take $\boldsymbol n\in N_R(\boldsymbol z)$ and
$\boldsymbol y\in[0,1]^d$.
Since $\boldsymbol z\in\operatorname{int}_{[0,1]^d}(R)$,
the point $\boldsymbol z+t(\boldsymbol y-\boldsymbol z)$
belongs to $R$ for all sufficiently small $t>0$.
The normal-cone inequality gives
$t\langle\boldsymbol n,\boldsymbol y-\boldsymbol z\rangle\le0$.
Dividing by $t$ shows that
$\langle\boldsymbol n,\boldsymbol y-\boldsymbol z\rangle\le0$.
As $\boldsymbol y$ was arbitrary,
$\boldsymbol n\in N_{[0,1]^d}(\boldsymbol z)$, proving the identity.

For every $\gamma>0$, the projection characterization now gives
$\operatorname{proj}_{[0,1]^d}
(\boldsymbol z+\gamma\boldsymbol n)=\boldsymbol z$.
Hence, by nonexpansiveness of projection,
\[
\begin{split}
 \|G_\gamma(\boldsymbol z)\|_2
 &=
 \frac{1}{\gamma}
 \left\|
 \operatorname{proj}_{[0,1]^d}
     (\boldsymbol z+\gamma\boldsymbol n)
 -
 \operatorname{proj}_{[0,1]^d}
     (\boldsymbol z-\gamma\nabla f(\boldsymbol z))
 \right\|_2\\
 &\le\|\nabla f(\boldsymbol z)+\boldsymbol n\|_2
 \le\varepsilon,
\end{split}
\]
where the last inequality follows from
\eqref{eq:model-normal-residual}.
This proves case~(ii).
\end{proof}

\subsection{Gradient Flow Trapping Framework}
\label{subsec:trapping-algorithm}
\label{subsec:separator-minimization}

We now apply the results of Section~\ref{sec:trapping-to-stationarity}
to construct a function-value algorithm for
$f\in C^{p,\beta_p}(L_p)$, with $\nu:=p+\beta_p>1$.
The unconstrained and cube-constrained versions use the same
trisection procedure and terminal model minimization.
At iteration $t$, the algorithm maintains a full-dimensional
hyperrectangle $R_t=\prod_{i=1}^d[a_i^t,b_i^t]$, a pivot
$\boldsymbol x^t\in R_t$ with known value $f(\boldsymbol x^t)$,
and a trapping parameter $\varepsilon_t$.
Write $r^t:=\max_i(b_i^t-a_i^t)$.
Each iteration approximately minimizes the trapping function on
two separators, updates the pivot according to reachability, and
removes one third of a longest coordinate interval.
The geometric update follows \cite{hollender2023computational,huanjian2025adaptive}.
Algorithm~\ref{alg:main0} states the complete procedure.

\begin{algorithm2e}[ht]
\caption{Higher-Order Gradient-Flow Trapping from Function Values}
\label{alg:main0}
\SetAlgoLined
\DontPrintSemicolon
\SetKwInOut{Input}{Input}
\Input{Accuracy $\varepsilon>0$; parameters $p,\beta_p,L_p$;
an exact function-value oracle; and $R_0$, $\boldsymbol x^0$,
$f(\boldsymbol x^0)$ from the initialization above.}

Set $t\leftarrow0$, $\varepsilon_0\leftarrow\frac{\varepsilon}{4}$,
and $r^0\leftarrow\max_i(b_i^0-a_i^0)$\;
Choose $\alpha$ as above and $r_\varepsilon$ by
\eqref{eq:trapping-terminal-radius}\;

\While{$r^t>r_\varepsilon$}{
  Choose $j\in\operatorname*{argmax}_i(b_i^t-a_i^t)$\;
  Set $\Delta_t$ and $\varepsilon_0^t$ by
  \eqref{eq:trapping-error-budget}, and form $\mathcal E^t$
  by \eqref{eq:trisection-separators}\;

  Find an $\varepsilon_0^t$-approximate global minimizer
  $\boldsymbol z^\star$ of
  $\phi_{\boldsymbol x^t,\varepsilon_t}$ over $\mathcal E^t$
  using Theorem~\ref{thm:global-minimization-known-term}\;
  Query $f(\boldsymbol z^\star)$\;

  \eIf{$\phi_{\boldsymbol x^t,\varepsilon_t}
          (\boldsymbol z^\star)\le0$}{
    Set $\boldsymbol x^{t+1}\leftarrow\boldsymbol z^\star$\;
  }{
    Set $\boldsymbol x^{t+1}\leftarrow\boldsymbol x^t$\;
  }

  Set $a_i^{t+1}\leftarrow a_i^t$ and
  $b_i^{t+1}\leftarrow b_i^t$ for every $i\ne j$\;
  \eIf{$[\boldsymbol x^{t+1}]_j\ge a_j^t+\frac{r^t}{2}$}{
    Set $a_j^{t+1}\leftarrow a_j^t+\frac{r^t}{3}$ and
    $b_j^{t+1}\leftarrow b_j^t$\;
  }{
    Set $a_j^{t+1}\leftarrow a_j^t$ and
    $b_j^{t+1}\leftarrow b_j^t-\frac{r^t}{3}$\;
  }

  Set $R_{t+1}\leftarrow\prod_i[a_i^{t+1},b_i^{t+1}]$
  and $r^{t+1}\leftarrow\max_i(b_i^{t+1}-a_i^{t+1})$\;
  Set $\varepsilon_{t+1}\leftarrow\varepsilon_t+\Delta_t$
  and $t\leftarrow t+1$\;
}

Set $T\leftarrow t$ and construct $P_T$ and $\Psi_T$ as Lemma~\ref{lem:stationarity-from-model}\;
Compute $\boldsymbol z^T\in
\operatorname*{argmin}_{\boldsymbol y\in R_T}\Psi_T(\boldsymbol y)$\;
\Return{$\boldsymbol z^T$}\;
\end{algorithm2e}
\paragraph{Initialization.}
Set $\varepsilon_0:=\frac{\varepsilon}{4}$ and query $f(\boldsymbol x^0)$.
In the unconstrained case, use the normalization $f\ge0$, obtained
by subtracting a known lower bound if necessary.
If $f(\boldsymbol x^0)=0$, return $\boldsymbol x^0$.
Otherwise, set $r^0:=\frac{16f(\boldsymbol x^0)}{\varepsilon}$ and
$R_0:=\{\boldsymbol y\in\mathbb R^d:
\|\boldsymbol y-\boldsymbol x^0\|_\infty\le \frac{r^0}{2}\}$.
The boundary to be certified is $\partial R_t$.
In the cube-constrained case, take any $\boldsymbol x^0\in[0,1]^d$
and set $R_0:=[0,1]^d$ and $r^0:=1$.
Here the boundary to be certified is
$\partial_{\mathrm{art}}R_t=\partial_{[0,1]^d}R_t$, as defined
in Section~\ref{subsec:trapping-certificates} and physical facets
remain constraints of the original problem.
This initialization does not require a lower bound on $f$.

\paragraph{Error Budgets.}
Choose
$\alpha:=\frac{1+\max\left\{\frac{2}{3},
\left(\frac{2}{3}\right)^{\nu-1}\right\}}{2}$.
At iteration $t$, use the parameter increment $\Delta_t$ and
minimization tolerance $\varepsilon_0^t$ given by
\begin{equation}
\label{eq:trapping-error-budget}
    \Delta_t:=\frac{\varepsilon(1-\alpha)}{4d}
              \alpha^{\lfloor t/d\rfloor},
    \qquad
    \varepsilon_0^t:=\frac{r^t\Delta_t}{12},
    \qquad
    \varepsilon_{t+1}:=\varepsilon_t+\Delta_t.
\end{equation}
The parameter $\varepsilon_t$ defines reachability at the current
pivot, whereas $\varepsilon_0^t$ specifies the required value
accuracy of the separator minimization.
After cutting, the algorithm increases the trapping parameter
by $\Delta_t$ to account for this minimization error.
\begin{figure}[t]
    \centering
    \includegraphics[width=0.7\linewidth]
    {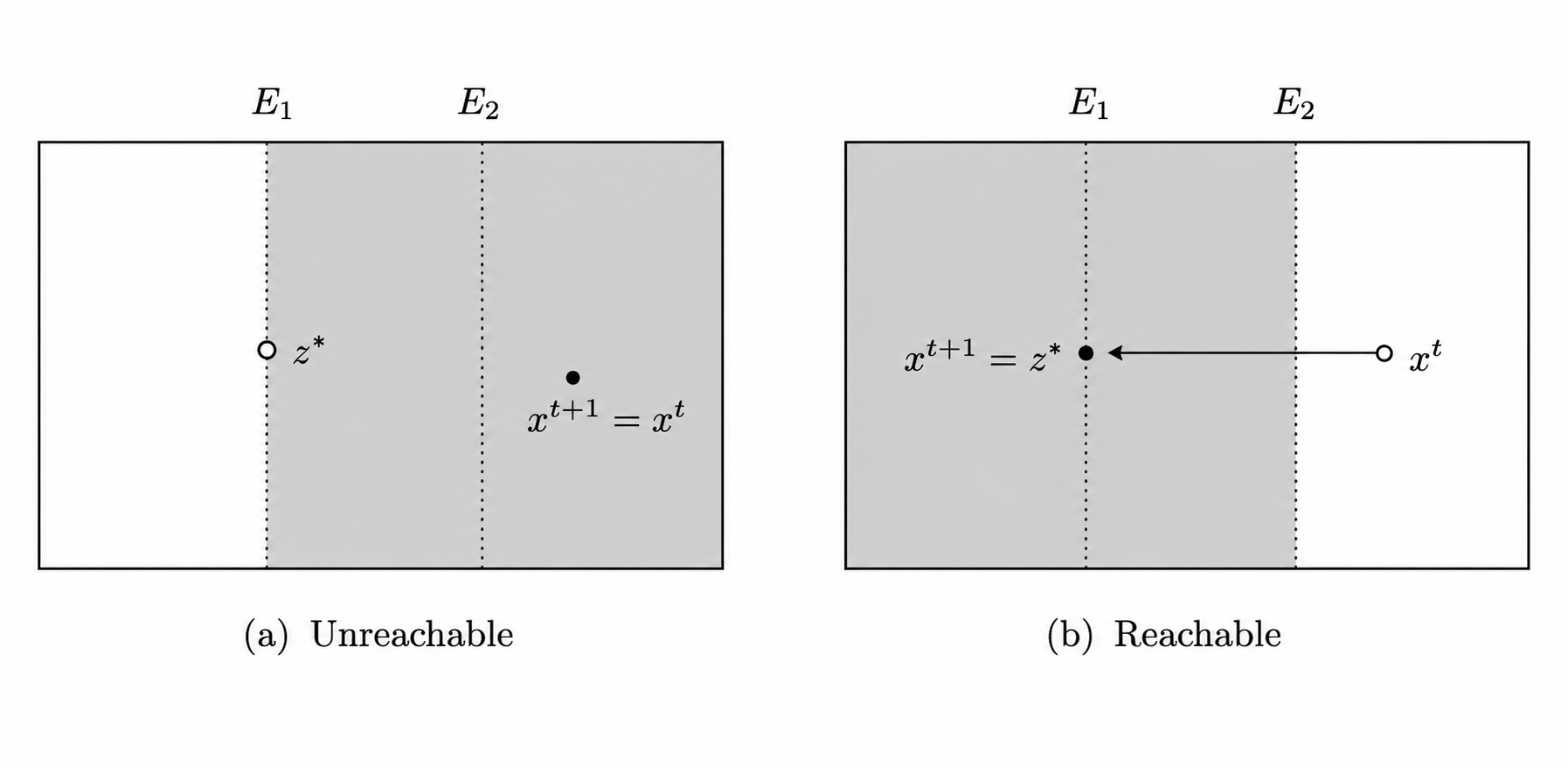}
    \caption{
    The two cases of the parallel-trapping update. Left:
    $\boldsymbol{z}^\star$ is $\varepsilon_t$-unreachable from
    $\boldsymbol{x}^t$, so
    $\boldsymbol{x}^{t+1}=\boldsymbol{x}^t$. Right:
    $\boldsymbol{z}^\star$ is $\varepsilon_t$-reachable, so
    $\boldsymbol{x}^{t+1}=\boldsymbol{z}^\star$. In both cases, the
    shaded two-thirds subrectangle represents $R_{t+1}$. Adapted from
    Figures~4--5 of
    \cite{hollender2023computational}.
    }
    \label{fig:parallel-trap-update}
\end{figure}
\paragraph{Trapping Function Minimization.}
Choose $j\in\operatorname*{argmax}_i(b_i^t-a_i^t)$ and place two
separators at the trisection points of the selected interval:
\begin{equation}
\label{eq:trisection-separators}
    E_m^t:=\left\{\boldsymbol y\in R_t:
        [\boldsymbol y]_j=a_j^t+\frac{mr^t}{3}\right\},
    \quad m\in\{1,2\},
    \qquad \mathcal E^t:=E_1^t\cup E_2^t.
\end{equation}
Using Theorem~\ref{thm:global-minimization-known-term} in
Appendix~\ref{app:approximation}, compute an
$\varepsilon_0^t$-approximate global minimizer
$\boldsymbol z^\star\in\mathcal E^t$ of
$\phi_{\boldsymbol x^t,\varepsilon_t}$.

\paragraph{Pivot Update and Cutting.}
If $\phi_{\boldsymbol x^t,\varepsilon_t}(\boldsymbol z^\star)\le0$,
set $\boldsymbol x^{t+1}:=\boldsymbol z^\star$ and retain its
queried function value. Otherwise, set
$\boldsymbol x^{t+1}:=\boldsymbol x^t$ and retain the known value
at the current pivot.
If $[\boldsymbol x^{t+1}]_j\ge a_j^t+\frac{r^t}{2}$, discard the leftmost
third of the selected interval by setting
$a_j^{t+1}:=a_j^t+\frac{r^t}{3}$ and $b_j^{t+1}:=b_j^t$.
Otherwise, discard the rightmost third by setting
$a_j^{t+1}:=a_j^t$ and $b_j^{t+1}:=b_j^t-\frac{r^t}{3}$.
All other coordinate intervals are unchanged.
Set $R_{t+1}:=\prod_i[a_i^{t+1},b_i^{t+1}]$, recompute $r^{t+1}$,
and update $\varepsilon_{t+1}$ according to
\eqref{eq:trapping-error-budget}. We illustrate this in
Figure~\ref{fig:parallel-trap-update}.

\paragraph{Stopping and Output.}
We use the terminal model of Lemma~\ref{lem:stationarity-from-model}
at the following scale.
Recall $a_{p-1,\beta_p}:=
\frac{\Gamma(\beta_p+1)}{\Gamma(p+\beta_p)}$
and, for $L_p>0$, define
\begin{equation}
\label{eq:trapping-terminal-radius}
    r_\varepsilon:=\frac{2}{\sqrt d}
       \left(\frac{\varepsilon}
       {12a_{p-1,\beta_p}L_p}\right)^{1/(\nu-1)}.
\end{equation}
Set $r_\varepsilon:=+\infty$ when $L_p=0$.
Stop at the first iteration $T$ with $r^T\le r_\varepsilon$.
Let $\boldsymbol c_T$ be the geometric center of $R_T$ and apply
Proposition~\ref{prop:local-reconstruction} with center
$\boldsymbol c_T$ to construct $P_T\in\Pi_p^d$ with
$\|\nabla P_T-\nabla f\|_{\infty,R_T}\le\frac{\varepsilon}{8}$.
The reconstruction is centered at $\boldsymbol c_T$, while the
distance penalty remains centered at the final pivot
$\boldsymbol x^T$.
Following \eqref{eq:penalized-scalar-model}, define
$\Psi_T(\boldsymbol y):=P_T(\boldsymbol y)-P_T(\boldsymbol x^T)
+\frac{3\varepsilon}{4}\|\boldsymbol y-\boldsymbol x^T\|_2$
and return a  minimizer
$\boldsymbol z^T\in\operatorname*{argmin}\limits_{\boldsymbol y\in R_T}
\Psi_T(\boldsymbol y)$.

\subsection{Analysis of Algorithm~\ref{alg:main0}: Proof Sketch of Theorem~\ref{the:main1}}
\label{sec:main1-analysis}

We prove both parts of Theorem~\ref{the:main1} through a common
boundary invariant and the terminal model construction of
Section~\ref{subsec:stationarity-from-models}.
The two settings differ in their initialization and in the boundary
on which unreachability is maintained.
The complete proof appears in Appendix~\ref{app:full-proof-main1}.

\paragraph{Correctness.}
By Lemma~\ref{lem:stationarity-from-model}, it suffices to
establish boundary unreachability and gradient accuracy of
the terminal model.
We maintain $\boldsymbol x^t\in R_t$ and
$\phi_{\boldsymbol x^t,\varepsilon_t}>0$ on $\partial R_t$
in the unconstrained case, or on
$\partial_{\mathrm{art}}R_t=\partial_{[0,1]^d}R_t$
in the cube-constrained case.

Suppose the invariant holds at iteration $t$, and let $F_t$
be the new facet created by the cut.
The cutting rule ensures
$\boldsymbol x^{t+1}\in R_{t+1}$ and
$\operatorname{dist}(\boldsymbol x^{t+1},F_t)
\ge\frac{r^t}{6}$.
Every certified boundary point of $R_{t+1}$ lies on a
retained old boundary or on $F_t$.
The pivot rule gives
$\phi_{\boldsymbol x^t,\varepsilon_t}(\boldsymbol x^{t+1})
=\min\{0,\phi_{\boldsymbol x^t,\varepsilon_t}(\boldsymbol z^\star)\}
\le0$.
\textbf{Case (1):} for every $\boldsymbol y\in R_t$: the triangle inequality gives
\[
\begin{aligned}
 \phi_{\boldsymbol x^{t+1},\varepsilon_{t+1}}(\boldsymbol y)
 &\ge
 \phi_{\boldsymbol x^t,\varepsilon_t}(\boldsymbol y)
 -\phi_{\boldsymbol x^t,\varepsilon_t}(\boldsymbol x^{t+1})
 +\Delta_t\|\boldsymbol y-\boldsymbol x^{t+1}\|_2>0.
\end{aligned}
\]
\textbf{Case (2):} for $\boldsymbol y\in F_t\subseteq\mathcal E^t$,
approximate optimality of $\boldsymbol z^\star$ and the pivot
rule give
$\phi_{\boldsymbol x^t,\varepsilon_t}(\boldsymbol y)
-\phi_{\boldsymbol x^t,\varepsilon_t}(\boldsymbol x^{t+1})
\ge-\varepsilon_0^t$.
Consequently,
\[
\begin{aligned}
 \phi_{\boldsymbol x^{t+1},\varepsilon_{t+1}}(\boldsymbol y)
 &\ge-\varepsilon_0^t
      +\Delta_t\|\boldsymbol y-\boldsymbol x^{t+1}\|_2\ge-\frac{r^t\Delta_t}{12}
      +\frac{r^t\Delta_t}{6}
  =\frac{r^t\Delta_t}{12}>0,
\end{aligned}
\]
where the second inequality uses the prescribed tolerance
and the distance to $F_t$.
This proves preservation of the invariant.

The longest-side cuts reach $r^T\le r_\varepsilon$ in
finitely many iterations, as quantified below.
The parameter budget gives
$\varepsilon_T
=\frac{\varepsilon}{4}+\sum_{t=0}^{T-1}\Delta_t
\le\frac{\varepsilon}{2}$,
so the invariant implies
$\phi_{\boldsymbol x^T,\frac{\varepsilon}{2}}>0$
on the certified terminal boundary.
Moreover, the stopping scale and
Proposition~\ref{prop:local-reconstruction} give
$\|\nabla P_T-\nabla f\|_{\infty,R_T}
\le\frac{\varepsilon}{8}$.
Lemma~\ref{lem:stationarity-from-model} therefore applies
to the returned minimizer $\boldsymbol z^T$ of $\Psi_T$:
$\|\nabla f(\boldsymbol z^T)\|_2\le\varepsilon$
in the unconstrained case, and
$\|G_\gamma(\boldsymbol z^T)\|_2\le\varepsilon$
for every $\gamma>0$ in the cube-constrained case.
\paragraph{Total Query Complexity.}
Both initial regions are cubes, and each iteration cuts a longest
side, giving $r^t=r^0\left(\frac{2}{3}\right)^{\lfloor t/d\rfloor}$.
Applying Theorem~\ref{thm:global-minimization-known-term} to the two
separators and substituting the minimization tolerance gives,
including the exact query at $\boldsymbol z^\star$,
\begin{equation}
\label{eq:main1-iteration-cost-sketch}
\begin{aligned}
Q_t\le C_{d,p,\beta_p}\Biggl[
1
+\left(\frac{L_p}{\varepsilon}\right)^{\frac{d-1}{\nu}}
(r^0)^{\frac{(d-1)(\nu-1)}{\nu}}
\times
\left(
\frac{\left(\frac{2}{3}\right)^{\nu-1}}{\alpha}
\right)^{
\frac{d-1}{\nu}\lfloor t/d\rfloor}
 \Biggr].
\end{aligned}
\end{equation}
Since $(2/3)^{\nu-1}<\alpha$ and each value of
$\lfloor t/d\rfloor$ occurs at most $d$ times, the scale-dependent
terms form a summable geometric sequence.
The constant terms contribute $\mathcal O(T)$ queries.
As shown in Appendix~\ref{app:full-proof-main1}, the stopping scale
ensures that
$T\le C_{d,p,\beta_p}
\left(\frac{L_p}{\varepsilon}\right)^{\frac{d-1}{\nu}}
(r^0)^{\frac{(d-1)(\nu-1)}{\nu}}$,
so this contribution is absorbed into the same bound.

Including the initial query and the $\binom{d+p}{p}$ terminal
reconstruction queries, and substituting the two initial side
lengths, we obtain
\begin{equation}
\label{eq:main1-total-cost-sketch}
Q\le1+\binom{d+p}{p}+C_{d,p,\beta_p}
\begin{cases}
L_p^{\frac{d-1}{\nu}}
f(\boldsymbol x^0)^{\frac{(d-1)(\nu-1)}{\nu}}
\varepsilon^{-(d-1)},
&\text{unconstrained},\\
\left(\frac{L_p}{\varepsilon}\right)^{\frac{d-1}{\nu}},
&\text{cube-constrained}.
\end{cases}
\end{equation}
For fixed problem parameters, the total query complexity is
$\mathcal O(\varepsilon^{-(d-1)})$ in the unconstrained case
and $\mathcal O(\varepsilon^{-\frac{d-1}{p+\beta_p}})$
in the cube-constrained case.

\section{Higher-Order Gradient-Only Trapping}
\label{sec:4}

In this section, we implement the trapping framework of
Section~\ref{sec:3} using exact gradient queries. Since the separator
subproblems depend on $f$ only through increments
$f(\boldsymbol y)-f(\boldsymbol x^t)$, we estimate these increments
by line integration and reconstruct a common relative-value model. Deterministic and randomized integration
give two versions of the algorithm. At termination, a scalar model
constructed from gradients supplies the output certificate.
In this section, $d\ge2$, $p\ge1$, and $\beta_p\in(0,1]$ are fixed,
and we write $\nu:=p+\beta_p>1$.

\begin{theorem}[\bf Higher-order gradient-only trapping]
\label{the:main_gradient}
Suppose that $f\in C^{p,\beta_p}(L_p)$. For $\varepsilon\in(0,1)$, the following
\emph{total} gradient-query bounds hold, with $d,p,\beta_p,L_p$
and the initial-gap bound $\Delta$ fixed independently of
$\varepsilon$ whenever applicable.

\begin{enumerate}[leftmargin = 2em]
\item \textbf{Unconstrained optimization.}
Given $\boldsymbol x^0$ and a known bound
$f(\boldsymbol x^0)-\inf f\le\Delta<\infty$, an
$\varepsilon$-stationary point can be found deterministically using
$\widetilde O\!\left(
\varepsilon^{-\max\{d-1,\frac{\nu}{\nu-1}\}}\right)$
gradient queries, or with randomization using
$\widetilde O\!\left(
\varepsilon^{-\max\{d-1,\frac{2\nu}{2\nu-1}\}}\right)$
gradient queries.

\item \textbf{Cube-constrained optimization.}
A point $\boldsymbol z\in[0,1]^d$ satisfying
$\|G_\gamma(\boldsymbol z)\|_2\le\varepsilon$ for every $\gamma>0$
can be found deterministically using
$\widetilde O\!\left(
\varepsilon^{-\max\{\frac{d-1}{\nu},\ \frac{1}{\nu-1}\}}\right)$
gradient queries, or with randomization using
$\widetilde O\!\left(
\varepsilon^{-\max\{\frac{d-1}{\nu},\ \frac{2}{2\nu-1}\}}\right)$
gradient queries.
\end{enumerate}
For any $\rho\in(0,1)$, the randomized algorithms succeed with
probability at least $1-\rho$ and their query bounds hold for every
random seed, with only logarithmic dependence on $1/\rho$.
Both implementations use $O(\log(1/\varepsilon))$ adaptive rounds.
\end{theorem}

Each query exponent is the maximum of a dimension-dependent term
and a dimension-independent smoothness term.
Randomization reduces the latter while leaving the former unchanged,
and therefore improves the overall exponent whenever the smoothness
term strictly dominates the deterministic bound.
For example, when $d=2$ and $\nu=2$, the unconstrained query bound
improves from $\widetilde{\mathcal O}(\varepsilon^{-2})$
to $\widetilde{\mathcal O}(\varepsilon^{-\frac{4}{3}})$,
and the cube-constrained bound improves from
$\widetilde{\mathcal O}(\varepsilon^{-1})$
to $\widetilde{\mathcal O}(\varepsilon^{-\frac{2}{3}})$.
Both versions retain logarithmically many adaptive rounds.
The randomized query bounds hold for every random seed,
rather than only in expectation, while the output satisfies
the stationarity guarantee with probability at least $1-\rho$.

\subsection{Line Integration from Gradient Queries}
\label{subsec:randomized-line-integration}

We recover function increments by integrating the gradient along
line segments. Piecewise polynomial interpolation gives a
deterministic quadrature rule and independent residual samples give
an unbiased randomized rule with a better accuracy exponent.
The conditional formulation below permits segments and accuracy
targets selected by the trapping algorithm.

\begin{lemma}[\bf High-order line integration]
\label{lem:randomized-line-integration}
\label{lem:randomized-high-order-line-integration}
Let $\mathcal F$ be the sigma-field generated by the past
transcript, and let $u,v\in\mathbb R^d$ and $\varepsilon_0>0$
be $\mathcal F$-measurable. Suppose that the fixed objective $f$
belongs to $C^{p,\beta_p}(L_p)$ on a neighborhood of $[u,v]$.
Write $\ell:=\|v-u\|_2$.
Each of the following constructions uses one nonadaptive batch
of exact gradient queries on $[u,v]$.
\begin{enumerate}[leftmargin =*]
\item \textbf{Deterministic integration.}
An estimate $I^{\mathrm{det}}$ satisfying
\[|I^{\mathrm{det}}-(f(v)-f(u))|\le\varepsilon_0\] can be computed
using $\mathcal O_p(1+(L_p\ell^\nu/\varepsilon_0)^{\frac{1}{\nu-1}})$
queries.

\item \textbf{Randomized integration.}
For any $\mathcal F$-measurable $\vartheta\in(0,1)$, an estimate
$\widehat I$ can be computed using
$\mathcal O_p(1+(L_p\ell^\nu/\varepsilon_0)^{\frac{2}{2\nu-1}}
[\log(2/\vartheta)]^{\frac{1}{2\nu-1}})$ queries such that
$\mathbb E[\widehat I\mid\mathcal F]=f(v)-f(u)$ and
\begin{equation}
\label{eq:apx-line-integration-accuracy}
 \mathbb P\!\left(
 |\widehat I-(f(v)-f(u))|\le\varepsilon_0
 \mid\mathcal F\right)\ge1-\vartheta.
\end{equation}
The random seeds are independent of one another and of
$\mathcal F$.
\end{enumerate}
If $L_p=0$, both estimates are exact. If $\ell=0$, both return
zero without queries.
\end{lemma}

\begin{proof}[Proof sketch]
Condition on $\mathcal F$ and assume $\ell>0$.
The function $g(t):=\langle\nabla f(u+t(v-u)),v-u\rangle$
satisfies $\int_0^1g(t)\,dt=f(v)-f(u)$, and its
$(p-1)$st derivative is $\beta_p$-H\"older with constant
$L_p\ell^\nu$.

Partition $[0,1]$ into $n$ equal subintervals $J_j$ and let
$P_j$ be the local degree-$(p-1)$ interpolant of $g$, using
left-endpoint constants when $p=1$. Taylor approximation and
interpolation stability give
$\|g-P_j\|_{\infty,J_j}
=\mathcal O_p(L_p\ell^\nu n^{-(\nu-1)})$.
Integrating the interpolants gives a deterministic estimate
with error of the same order. For $p=1$, the sharper bound is
$L_p\ell^\nu n^{-\beta_p}/(1+\beta_p)$.

For the randomized rule, draw independent uniform points
$U_j\in J_j$, using fresh seeds, and set
\[
 \widehat I_n
 :=\sum_{j=1}^n\left[
 \int_{J_j}P_j(t)\,dt+\frac{g(U_j)-P_j(U_j)}{n}
 \right].
\]
The residual corrections make this estimate conditionally
unbiased. Conditional Hoeffding's inequality bounds its error
by $\mathcal O_p(L_p\ell^\nu
n^{-(\nu-1/2)}\sqrt{\log(2/\vartheta)})$ with probability at
least $1-\vartheta$.
Choosing $n$ to meet the accuracy target gives the two claimed
query bounds. Each construction uses $\mathcal O_p(n)$
gradient queries, with all nodes selected before any new answer
is observed. The full proof, including the degenerate cases and
constants uniform in $\beta_p$, is given in
Appendix~\ref{app:4.1}.
\end{proof}
\subsection{Gradient-Only Implementation}
\label{subsec:gradient-only-implementation}

We retain the trisection loop of Algorithm~\ref{alg:main0}.
The changes are a gradient-only separator
subroutine (Line 6 and Line 8), and a gradient-only terminal model (Lines 22-23). All optimization of
explicit models is offline.

\paragraph{Initialization and parameter schedule.}
Set $\varepsilon_0:=\frac{\varepsilon}{4}$. For the unconstrained problem
with $\Delta>0$, take
$R_0:=\{\boldsymbol y:\|\boldsymbol y-\boldsymbol x^0\|_\infty
\le\frac{8\Delta}{\varepsilon}\}$ and
$r^0:=\frac{16\Delta}{\varepsilon}$.
For the constrained problem, take $R_0=[0,1]^d$, $r^0=1$, and
any $\boldsymbol x^0\in R_0$. Neither initialization requires a
function-value query.

Choose
$\alpha:=\frac12+\frac12\max\!\left\{\frac23,
\left(\frac23\right)^{\nu-1}\right\}$.
At iteration $t$, use the same parameter schedule as in
Section~\ref{sec:3}:
\begin{equation}
\label{eq:gradient-only-outer-schedule}
 \Delta_t:=\frac{\varepsilon(1-\alpha)}{4d}
            \alpha^{\lfloor t/d\rfloor},\qquad
\varepsilon_0^t:=\frac{r^t\Delta_t}{12},\qquad
 \varepsilon_{t+1}:=\varepsilon_t+\Delta_t.
\end{equation}
Each initial region is a cube, so the
longest-side rule gives
$r^t=r^0\left(\frac23\right)^{\lfloor t/d\rfloor}$.
Also, $\sum_{t\ge0}\Delta_t=\frac{\varepsilon}{4}$ and hence
$\varepsilon_t\le\frac{\varepsilon}{2}$.

\paragraph{Relative-value models on the separators.}
Following section~\ref{sec:3}, we write $R_t=\prod_i[a_i^t,b_i^t]$, choose
$j\in\operatorname*{argmax}_i(b_i^t-a_i^t)$, and let
$E_q^t:=\{\boldsymbol y\in R_t:
[\boldsymbol y]_j=a_j^t+\frac{qr^t}{3}\}$ for $q=1,2$.
Set $\mathcal E^t:=E_1^t\cup E_2^t$ and
$F_t(\boldsymbol y):=f(\boldsymbol y)-f(\boldsymbol x^t)$.
Both facets will use this same reference value.
Let $\lambda_p$ denote the univariate
interpolation Lebesgue constant defined in
Appendix~\ref{app:approximation}, and write
$\Lambda_{{d-1},p}:=\lambda_p^{d-1}$ for the corresponding
tensor-product interpolation constant.
With the Taylor remainder constant
$a_{p,\beta_p}
:=
\frac{\Gamma(\beta_p+1)}
{\Gamma(p+\beta_p+1)}$,
define $A_{{d-1},p,\beta_p}:=(1+\Lambda_{{d-1},p})a_{p,\beta_p}$.

Apply Lemma~\ref{lem:interpolation-error-decomposition}
on each separator cell using degree-$p$ tensor-product
interpolation. Its polynomial space contains
$\Pi_p^{d-1}$, and its value stability constant is
$\Lambda_{d-1,p}$.
Shared face nodes use the same supplied values, so the
cellwise interpolants agree on common faces.
Thus, cells of diameter at most $H$ and nodal errors
at most $\tau$ yield a continuous piecewise interpolant
satisfying
\begin{equation}
\label{eq:gradient-only-stable-recovery}
 \|\widehat F_t-F_t\|_{\infty,\mathcal E^t}
 \le
 A_{d-1,p,\beta_p}L_pH^\nu
 +\Lambda_{d-1,p}\tau.
\end{equation}
We allocate $\frac{\varepsilon_0^t}{8}$ to each term,
giving a model error of at most $\frac{\varepsilon_0^t}{4}$,
and reserve $\frac{\varepsilon_0^t}{2}$ for approximate
surrogate minimization.

\paragraph{Nested grids.}
Choose the smallest refinement level $J_t\in Z_{\ge0}$ for which
the interpolation approximation error is at most $\frac{\varepsilon_0^t}{8}$:
\begin{equation}
\label{eq:gradient-only-refinement-level}
 J_t:=\min\left\{\ell\in\mathbb Z_{\ge0}:
 A_{d-1,p,\beta_p}L_p
 \left(\frac{\sqrt{d-1}\,r^t}{2^\ell}\right)^\nu
 \le\frac{\varepsilon_0^t}{8}\right\}.
\end{equation}
For each facet and $0\le\ell\le J_t$, use the grid
\begin{equation}
\label{eq:nested-facet-grid}
 \mathcal V_{q,\ell}^t:=
 \left\{\boldsymbol y\in E_q^t:
 [\boldsymbol y]_i=a_i^t+
 \frac{k_i(b_i^t-a_i^t)}{p2^\ell},\
 k_i\in\{0,\ldots,p2^\ell\},\ i\ne j\right\}.
\end{equation}
The grids are nested and have $(p2^\ell+1)^{d-1}$ nodes each.
At level $\ell$, every free coordinate interval is partitioned
into $2^\ell$ macro-intervals. The $(p+1)^{d-1}$ tensor nodes on each
macro-cell determine its degree-$p$ interpolant in each coordinate.
At the finest level the cell diameter is at most
$\sqrt{d-1}\,r^t2^{-J_t}$. The nested grids are illustrated in Figure~\ref{fig:shared-root-reconstruction}(a).

\paragraph{Shared-root reconstruction.}
On each facet, retain each grid node once and connect all
level-zero nodes to a common root at $\boldsymbol x^t$.
A node first appearing at level $\ell\ge1$ with indices $k_i$
has as its parent the node with indices $\lfloor k_i/2\rfloor$
in $\mathcal V_{q,\ell-1}^t$.%
\footnote{A node coinciding with $\boldsymbol x^t$ remains
distinct and is connected directly to the root by a zero-length edge.}
See Figure~\ref{fig:shared-root-reconstruction}(b).

Orient edges from parent to child.
Let $\mathcal A_{t,\ell}$ denote edges whose children first
appear at level $\ell$, set
$\mathcal A_t:=\bigcup_{\ell=0}^{J_t}\mathcal A_{t,\ell}$,
and write $h_e:=\|\boldsymbol v-\boldsymbol u\|_2$
for $e=(\boldsymbol u,\boldsymbol v)$.
Each root-to-node path contains at most one edge per level.
Summing estimated increments along this path recovers
$f(\boldsymbol v)-f(\boldsymbol x^t)$ approximately,
with each edge estimate reused by all descendants.

\IfFileExists{figures/shared-root-reconstruction.pdf}{%
\begin{figure}[t]
\centering
\includegraphics[width=\linewidth]{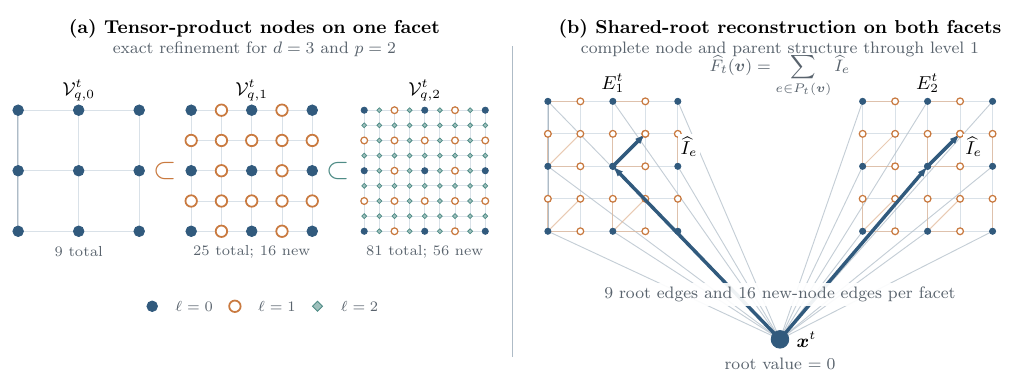}
\caption{Nested interpolation grids and shared-root reconstruction
for $d=3$ and $p=2$. Both facetwise trees share the pivot as their
root. A node value is reconstructed by summing increments along its
root path, which uses at most one edge from each creation level.}
\label{fig:shared-root-reconstruction}
\end{figure}
}{}

\paragraph{Integration budgets and buffered update.}
Fix the integration mode
$\mathsf m\in\{\mathsf{det},\mathsf{rand}\}$ and let
$\kappa:=\kappa_{\mathsf m}$ denote its query-complexity exponent,
where $\kappa_{\mathsf{det}}:=\frac{1}{\nu-1}$ and
$\kappa_{\mathsf{rand}}:=\frac{2}{2\nu-1}$.
Define the normalization constant $S_t$, the level weights
$w_{t,\ell}$, and the per-edge integration tolerances
$\xi_{t,\ell}$ by
\begin{equation}
\label{eq:edge-error-budget}
 S_t:=\sum_{\ell=0}^{J_t}2^{\frac{(d-1-\nu\kappa)\ell}{1+\kappa}},
 \qquad
 w_{t,\ell}:=\frac{2^{\frac{(d-1-\nu\kappa)\ell}{1+\kappa}}}{S_t},
 \qquad
 \xi_{t,\ell}:=\frac{\varepsilon_0^t w_{t,\ell}}{8\Lambda_{d-1,p}},
 \quad 0\le\ell\le J_t.
\end{equation}
Integrate each level-$\ell$ edge to accuracy $\xi_{t,\ell}$ using
Lemma~\ref{lem:randomized-line-integration}. Since
$\sum_\ell w_{t,\ell}=1$, simultaneous edge accuracy bounds every
nodal error by $\frac{\varepsilon_0^t}{8\Lambda_{d-1,p}}$. The choice of weights
also balances the query cost across levels, as shown in
Section~\ref{subsubsec:gradient-only-complexity}.
In randomized mode, set
$\rho_t:=\frac{6\rho}{\pi^2(t+1)^2}$ and give
each positive-length edge failure budget
$\frac{\rho_t}{|\mathcal A_t|}$.

\begin{algorithm2e}[t]
\caption{Shared-Root Reconstruction and Buffered Update}
\label{alg:gradient-only-dyadic}
\SetAlgoLined
\DontPrintSemicolon
\SetKwInOut{Input}{Input}
\SetKwInOut{Output}{Output}
\Input{$R_t,\boldsymbol x^t,\varepsilon_t,r^t,\varepsilon_0^t$;
separators $E_1^t,E_2^t$; parameters $p,\beta_p,L_p$;
mode $\mathsf m\in\{\mathsf{det},\mathsf{rand}\}$; and $\rho_t$ in randomized mode.}
\Output{The next pivot $\boldsymbol x^{t+1}$.}
Choose $J_t$ by~\eqref{eq:gradient-only-refinement-level};
construct the nested grids and their shared-root tree\;
Set $\xi_{t,\ell}$ by~\eqref{eq:edge-error-budget}\;
\ForEach{$e=(\boldsymbol u,\boldsymbol v)\in\mathcal A_{t,\ell}$,
 $0\le\ell\le J_t$, \textnormal{in parallel}}{
 \uIf{$\boldsymbol u=\boldsymbol v$}{Set $\widehat I_e:=0$\;}
 \Else{Estimate $f(\boldsymbol v)-f(\boldsymbol u)$ by $\widehat I_e$
using Lemma~\ref{lem:randomized-line-integration} with
\[
\left[
\begin{array}{ll}
\text{tolerance }\xi_{t,\ell},
& \text{if }\mathsf m=\mathsf{det},\\[2pt]
\text{tolerance }\xi_{t,\ell}
\text{ and failure budget }\frac{\rho_t}{|\mathcal A_t|},
& \text{if }\mathsf m=\mathsf{rand}.
\end{array}
\right.
\]\;}
}
Assign value zero to the root and sum increments along every
root-to-node path\;
Interpolate the finest nodal values to obtain $\widehat F_t$\;
Set $\widehat\phi_t(\boldsymbol y):=\widehat F_t(\boldsymbol y)
+\varepsilon_t\|\boldsymbol y-\boldsymbol x^t\|_2$\;
Compute $\widehat{\boldsymbol z}^t\in\mathcal E^t$ with
$\widehat\phi_t(\widehat{\boldsymbol z}^t)
\le
\min_{\mathcal E^t}\widehat\phi_t
+\frac{\varepsilon_0^t}{2}$\;
\uIf{$\widehat\phi_t(\widehat{\boldsymbol z}^t)
\le -\frac{\varepsilon_0^t}{4}$}{
 Set $\boldsymbol x^{t+1}:=\widehat{\boldsymbol z}^t$\;
}
\Else{Set $\boldsymbol x^{t+1}:=\boldsymbol x^t$\;}
\Return{$\boldsymbol x^{t+1}$}\;
\end{algorithm2e}

The subroutine replaces the separator minimization, candidate-value
query, and reachability test in Algorithm~\ref{alg:main0}.
After it returns $\boldsymbol x':=\boldsymbol x^{t+1}$, retain
$\left[a_j^t+\frac{r^t}{3},b_j^t\right]$ in coordinate $j$ if
$[\boldsymbol x']_j\ge a_j^t+\frac{r^t}{2}$., and retain
$\left[a_j^t,\; b_j^t-\frac{r^t}{3}\right]$ otherwise. Keep all other coordinate
intervals and update $\varepsilon_t$ by
\eqref{eq:gradient-only-outer-schedule}.

\begin{remark}[\textbf{Nonadaptivity}]
Conditional on the past transcript, each iteration selects all
query points before observing any new answers and thus uses
one nonadaptive batch.
Reconstruction and minimization require no further queries,
and all integration segments lie in $R_t$ by convexity.
\end{remark}

\paragraph{Stopping and output.}
Let $c_{1,\beta_p}:=1$ and
$c_{p,\beta_p}:=2a_{p-1,\beta_p}+a_{p-2,\beta_p}$ for $p\ge2$.
Define
\begin{equation}
\label{eq:gradient-only-terminal-radius}
 r_\varepsilon^\nabla:=
 \frac2{\sqrt d}
 \left(\frac{\varepsilon}{8c_{p,\beta_p}L_p}\right)^{1/(\nu-1)}.
\end{equation}
Stop at the first $T$ with $r^T\le r_\varepsilon^\nabla$.
At the geometric center of $R_T$, apply
Corollary~\ref{lem:gradient-only-scalar-reconstruction} to construct
$P_T\in\Pi_p^d$ using
$N_{d,p}^\nabla:=\binom{d+p-1}{p-1}$ gradient queries in one batch.
Set \[\Psi_T(\boldsymbol y):=
P_T(\boldsymbol y)-P_T(\boldsymbol x^T)
+\frac{3\varepsilon}{4}\|\boldsymbol y-\boldsymbol x^T\|_2\]
and
$\zeta_T:=\frac{\varepsilon r^T}{16}$.
Compute offline a $\zeta_T$-approximate global minimizer
$\boldsymbol z^T$ with
\begin{equation}
\label{eq:gradient-only-terminal-certificates}
 \begin{gathered}
 \Psi_T(\boldsymbol z^T)\le0,\qquad
 \left\|\nabla P_T(\boldsymbol z^T)
       +\frac{3\varepsilon}{4}\boldsymbol v+\boldsymbol n\right\|_2
 \le\frac\varepsilon8,\\
 \boldsymbol v\in\partial\|\cdot-\boldsymbol x^T\|_2(\boldsymbol z^T),
 \qquad \boldsymbol n\in N_{R_T}(\boldsymbol z^T).
 \end{gathered}
\end{equation}
Return $\boldsymbol z^T$.
Lemma~\ref{lem:gradient-only-terminal-output} proves that these
certificates can be obtained by finite offline computation.

\subsection{Proof Sketch of Theorem~\ref{the:main_gradient}}
\label{subsec:gradient-only-analysis}

We sketch the correctness and query-complexity arguments.
The proof combines the boundary-invariant argument of
Section~\ref{sec:main1-analysis} with gradient-only reconstruction.
Detailed proofs of correctness and complexity are provided in
Appendices~\ref{subsec:gradient-only-analysis}
and~\ref{subsubsec:gradient-only-complexity}, respectively.

\paragraph{Correctness.}
Write $\phi_t:=\phi_{\boldsymbol x^t,\varepsilon_t}$.
When all edge estimates meet their prescribed tolerances,
the shared-root construction and stable interpolation give
$\|\widehat\phi_t-\phi_t\|_{\infty,\mathcal E^t}
\le\frac{\varepsilon_0^t}{4}$:
each root path contains at most one edge from each level.
The surrogate minimization tolerance and the buffered
acceptance test then ensure
\[
\begin{aligned}
 \phi_t(\boldsymbol x^{t+1})&\le0,\\
 \phi_t(\boldsymbol y)-\phi_t(\boldsymbol x^{t+1})
 &\ge-\varepsilon_0^t,
 \qquad \boldsymbol y\in\mathcal E^t.
\end{aligned}
\]
These are precisely the update guarantees used in
Section~\ref{sec:main1-analysis} to preserve the boundary invariant.
The invariant holds initially by the gap bound in the
unconstrained case and vacuously in the cube-constrained case.
At termination,
Lemma~\ref{lem:gradient-only-terminal-output}
certifies the returned point using finite offline computation.
Deterministic integration always meets the edge tolerances.
In randomized mode,
Lemma~\ref{lem:randomized-line-integration}
and the edgewise union bound give conditional failure
probability at most $\rho_t$ in iteration $t$.
Since $\sum_{t\ge0}\rho_t\le\rho$, all iterations succeed
simultaneously with probability at least $1-\rho$.

\paragraph{Query complexity per iteration.}
Set
$B_t:=\frac{L_p(r^t)^\nu}{\varepsilon_0^t}$.
The refinement rule produces
$\mathcal O_{d,p,\beta_p}
(1+B_t^{\frac{d-1}{\nu}})$ tree edges.
At level $\ell$, there are
$\mathcal O_{d,p}(2^{(d-1)\ell})$ edges of length
$\mathcal O_{d,p}(r^t2^{-\ell})$.
Combining these bounds with the line-integration costs
and the prescribed levelwise error allocation gives
\[
 Q_t^{\mathsf m}
 =\widetilde{\mathcal O}_{d,p,\beta_p}
 \left(1+B_t^{\chi_{\mathsf m}}\right),
 \qquad
 \chi_{\mathsf m}
 :=\max\left\{\frac{d-1}{\nu},\kappa_{\mathsf m}\right\},
\]
where
$\kappa_{\mathrm{det}}=\frac{1}{\nu-1}$ and
$\kappa_{\mathrm{rand}}=\frac{2}{2\nu-1}$.
The two exponents correspond to the number of fine-grid
nodes and the cost of integrating coarse-level edges.
The suppressed factors are logarithmic, including the
confidence dependence in randomized mode.

\paragraph{Total complexity.}
The geometric schedules imply $B_t\le B_0$ and
$T=\mathcal O_{d,p,\beta_p}(1+\log(1+B_0))$.
Summing the per-iteration bounds and adding the terminal
reconstruction therefore gives
$Q^{\mathsf m}
=\widetilde{\mathcal O}_{d,p,\beta_p}
(1+B_0^{\chi_{\mathsf m}})$ with  $ B_0=
 \begin{cases}
 \mathcal O_{d,p,\beta_p}\!\left(
 L_p\Delta^{\nu-1}\varepsilon^{-\nu}
 \right),
 &\text{unconstrained},\\
 \mathcal O_{d,p,\beta_p}\!\left(
 \frac{L_p}{\varepsilon}
 \right),
 &\text{cube-constrained}.
 \end{cases}$
Thus the accuracy exponents are
$\nu\chi_{\mathsf m}$ and $\chi_{\mathsf m}$, respectively,
as claimed.

\section{Trapping or Descent}
\label{sec:trapping-or-descent}

Sections~\ref{sec:3} and~\ref{sec:4} show that trapping can
provably outperform gradient descent based methods in low dimensions.
However, the accuracy exponents grow with $d$, limiting the
efficiency of these implementations in high dimensions.
This raises the question of whether such dimensional
deterioration is intrinsic to the trapping framework.

In this section, we introduce ``Trapping or Descent,'' a trapping
framework that works in arbitrary dimension using low-order oracles.
At each iteration, we construct a local polynomial model on a
local ball.
Approximate minimization of this model over the ball boundary either
yields a guaranteed decrease in the objective or certifies that the
ball contains an approximate stationary point. 
In the latter case, with such unreachable certification, minimizing the local model yields such a target point through a finite terminal computation.
Furthermore, a bounded initial gap yields an explicit
bound on the number of iterations before termination.



\begin{theorem}[\bf Trapping or descent]
\label{the:fresh-trapping-or-descent}
Under assumption~\ref{def:holder} and~\ref{ass:lower-bounded}, each of the following deterministic
algorithms returns $\boldsymbol z\in\mathcal X$ satisfying
$\|\nabla f(\boldsymbol z)\|_2\le\varepsilon$ when
$\mathcal X=\mathbb R^d$, or
$\|G_\gamma(\boldsymbol z)\|_2\le\varepsilon$ for every $\gamma>0$
when $\mathcal X=[0,1]^d$.
\begin{enumerate}[leftmargin =*]
\item \textbf{Function values.}
The total number of function-value queries is
\[
\mathcal O_{p,\beta_p}\!\left(
d^p\left[
1+L_p^{\frac{1}{\nu-1}}
\Delta\,
\varepsilon^{-\frac{\nu}{\nu-1}}
\right]
\right).
\]
\item \textbf{Gradients.}
The total number of gradient queries is
\[
\mathcal O_{p,\beta_p}\!\left(
d^{p-1}\left[
1+L_p^{\frac{1}{\nu-1}}
\Delta\,
\varepsilon^{-\frac{\nu}{\nu-1}}
\right]
\right).
\]
\end{enumerate}
All queries lie in $\mathcal X$. If $\Delta=0$, both algorithms
return $\boldsymbol x^0$ without queries.
\end{theorem}
The theorem gives the same query bounds for unconstrained and
cube-constrained optimization. In both oracle models, the accuracy
exponent $\frac{\nu}{\nu-1}$ depends only on the smoothness
$\nu=p+\beta_p$ and is independent of the dimension.
The dimension enters through the polynomial prefactor
$d^p$ for function values and $d^{p-1}$ for gradients.
Thus gradient access saves a factor of $d$ in the query bound,
while both bounds have the same dependence on
$L_p$, $\Delta$, and $\varepsilon$.
Greater smoothness improves the accuracy exponent, which
approaches one as $\nu$ increases.

A fresh implementation reconstructs the model after every accepted
move. The following implementation instead fixes its interpolation
nodes for a phase of $d$ model calls and reconstructs only a
lower-degree correction between refreshes.

\begin{theorem}[\bf Lazy trapping or descent]
\label{the:lazy-trapping-or-descent}
Under the assumptions of Theorem~\ref{the:fresh-trapping-or-descent},
the following deterministic query bounds achieve the same output
guarantees.
\begin{enumerate}[leftmargin =*]
\item \textbf{Function values.}
The total number of function-value queries is
\[
\mathcal O_{p,\beta_p}\!\left(
d^p
+d^{p-1+\frac{\beta_p}{\nu-1}}
L_p^{\frac{1}{\nu-1}}
\Delta\,
\varepsilon^{-\frac{\nu}{\nu-1}}
\right).
\]
\item \textbf{Gradients.}
The total number of gradient queries is
\[
\mathcal O_{p,\beta_p}\!\left(
d^{p-1}
+d^{p-2+\frac{\beta_p}{\nu-1}}
L_p^{\frac{1}{\nu-1}}
\Delta\,
\varepsilon^{-\frac{\nu}{\nu-1}}
\right).
\]
\end{enumerate}
\end{theorem}

The repeated dimension factors improve by
$d^{1-\frac{\beta_p}{\nu-1}}=d^{\frac{p-1}{\nu-1}}$.
The essential observation is that fixed linear interpolation makes
the rolling corrections telescope: the error depends on the net
displacement from the refresh center, through its $\beta_p$th power.
For $p=2$, the gradient bound becomes
$\mathcal O_{\beta_2}\!\left(
d+d^{\frac{\beta_2}{1+\beta_2}}
\varepsilon^{-\frac{2+\beta_2}{1+\beta_2}}
\right).$
At $\beta_p=1$, both the fresh and lazy bounds recover the
corresponding Lipschitz highest-derivative rates.

The remainder of this section is organized as follows.
Section~\ref{subsec:trapping-framework} presents the framework
and terminal localization procedure.
Sections~\ref{subsec:trapping-fresh-implementations}
and~\ref{subsec:lazy-boundary-reuse} develop the fresh and lazy
implementations, respectively.
Section~\ref{subsec:trapping-proof-sketch} sketches the proofs
of correctness and query complexity.
Finally, comparisons with regularized tensor methods, including the
distinction between their derivative oracles and the strict oracles
used here, are given in
Section~\ref{app:trapping-comparisons}.
All omitted proofs are provided in
Appendix~\ref{app:trapping-or-descent}.

\subsection{The Trapping-or-Descent Framework}
\label{subsec:trapping-framework}

For initial point $\boldsymbol x\in\mathcal X$ and local search radius $r>0$ $r>0$, write the local search ball as 
$Q_{\boldsymbol x}(r):=\mathcal X\cap(\boldsymbol x+r\mathbb B)$
and its boundary as $\Gamma_{\boldsymbol x}(r):=
\{\boldsymbol y\in\mathcal X:
\|\boldsymbol y-\boldsymbol x\|_2=r\}$.
Thus $\partial_{\mathcal X}Q_{\boldsymbol x}(r)
\subseteq\Gamma_{\boldsymbol x}(r)$.
In the unconstrained case these sets are the ball and its sphere.
Recall the trapping function is defined as $\phi_{\boldsymbol x,\eta}(\boldsymbol y)
:=f(\boldsymbol y)-f(\boldsymbol x)
+\eta\|\boldsymbol y-\boldsymbol x\|_2$~(Definition~\ref{def:trapping-certificate}).

The boundary procedure retains a scalar polynomial
$P_{\boldsymbol x}\in\Pi_p^d$ and returns
$\widehat\phi_{\boldsymbol x}(\boldsymbol y)
:=P_{\boldsymbol x}(\boldsymbol y)-P_{\boldsymbol x}(\boldsymbol x)
+\frac{\varepsilon}{4}\|\boldsymbol y-\boldsymbol x\|_2$
on $\Gamma_{\boldsymbol x}(r)$ in Line 3.
The outer loop uses only the condition that
\begin{equation}
\label{eq:boundary-approximation}
 \|\widehat\phi_{\boldsymbol x}
       -\phi_{\boldsymbol x,\varepsilon/4}\|_{\infty,
       \Gamma_{\boldsymbol x}(r)}
 \le\frac{\varepsilon r}{16}.
\end{equation}
If the sphere is nonempty, compute an
$\frac{\varepsilon r}{32}$-approximate minimizer
$\widehat{\boldsymbol y}$ of $\widehat\phi_{\boldsymbol x}$ in Line 7.
When $\widehat\phi_{\boldsymbol x}(\widehat{\boldsymbol y})
>\frac{\varepsilon r}{32}$, every relative-boundary point is
$\frac{3\varepsilon}{8}$-unreachable from $\boldsymbol x$.
Otherwise,
$f(\widehat{\boldsymbol y})\le f(\boldsymbol x)-\frac{5\varepsilon r}{32}$.
Lemma~\ref{lem:boundary-trapping-or-descent} proves this dichotomy.

\begin{algorithm2e}[t]
\caption{Trapping or Descent}
\label{alg:unconstrained-boundary-lazy}
\SetAlgoLined
\DontPrintSemicolon
\SetKw{Return}{return}
\SetKwInOut{Input}{Input}
\SetKwInOut{Output}{Output}
\Input{$\boldsymbol x^0\in\mathcal X$, $\varepsilon>0$, local search radius $r>0$,
a fresh or lazy boundary-model procedure.}
\Output{An $\varepsilon$-stationary point or an $\varepsilon$-KKT point.}
Set $\boldsymbol x\gets\boldsymbol x^0$ and initialize the model state\;
\For{$t=0,1,2,\ldots$}{
 Construct $(P_{\boldsymbol x},\widehat\phi_{\boldsymbol x})$
 using the selected boundary-model procedure\;
 \If{$\Gamma_{\boldsymbol x}(r)=\varnothing$}{
  \Return $\mathsf{TrapLocalize}
  (Q_{\boldsymbol x}(r),\boldsymbol x,P_{\boldsymbol x},\varepsilon,r)$\;
 }
 Compute $\widehat{\boldsymbol y}\in\Gamma_{\boldsymbol x}(r)$ with
 $\widehat\phi_{\boldsymbol x}(\widehat{\boldsymbol y})
 \le\inf_{\Gamma_{\boldsymbol x}(r)}\widehat\phi_{\boldsymbol x}
 +\frac{r\varepsilon }{32}$\;
 \eIf{$\widehat\phi_{\boldsymbol x}(\widehat{\boldsymbol y})
       >\frac{r \varepsilon }{32}$}{
  \Return $\mathsf{TrapLocalize}
  (Q_{\boldsymbol x}(r),\boldsymbol x,P_{\boldsymbol x},\varepsilon,r)$\;
 }{
  Set $\boldsymbol x\gets\widehat{\boldsymbol y}$\;
 }
}
\end{algorithm2e}

Every accepted move has length $r$ and decreases $f$ by at least
$\frac{\varepsilon r}{8}$. Hence
$N_{\mathrm{mov}}\le\left\lfloor\frac{8\Delta}{\varepsilon r}\right\rfloor$,
and there are exactly $N_{\mathrm{mov}}+1$ model calls.
If $\Gamma_{\boldsymbol x}(r)=\varnothing$, the cap equals
$\mathcal X$ and its relative boundary is empty, so localization
can be invoked immediately.

We first justify the boundary test in
Algorithm~\ref{alg:unconstrained-boundary-lazy} in the following lemma with proof deferred in Appendix~\ref{app:trapping-framework-proofs}.
For a fixed pivot $x\in X$, write $Q:=Q_x(r)$ and
$\Gamma:=\Gamma_x(r)$.
The following lemma shows that the prescribed approximation
tolerances allow this test to certify either trapping in $Q$
or an objective decrease of at least $\frac{\varepsilon r}{8}$.
\begin{lemma}[\bf Trapping-or-descent dichotomy]
\label{lem:boundary-trapping-or-descent}
\label{lem:box-trapping-or-descent}
Suppose that $\Gamma\ne\varnothing$ and that a continuous
surrogate $\widehat\phi_x$ on $\Gamma$ satisfies
$\|\widehat\phi_x-\phi_{x,\frac{\varepsilon}{4}}\|_{\infty,\Gamma}
\le\frac{\varepsilon r}{16}$.
Let $\widehat y\in\Gamma$ satisfy
$\widehat\phi_x(\widehat y)
\le\min_\Gamma\widehat\phi_x+\frac{\varepsilon r}{32}$.
Then the following alternatives are exhaustive and disjoint.
\begin{enumerate}
\item If $\widehat\phi_x(\widehat y)>\frac{\varepsilon r}{32}$, then
every point of $\partial_XQ$ is $\frac{3\varepsilon}{8}$-unreachable
from $x$.
\item If $\widehat\phi_x(\widehat y)\le\frac{\varepsilon r}{32}$, then
$f(\widehat y)\le f(x)-\frac{5\varepsilon r}{32}
\le f(x)-\frac{\varepsilon r}{8}$.
\end{enumerate}
If $\Gamma=\varnothing$, the trapping certificate in the first
alternative holds vacuously.
\end{lemma}

\begin{algorithm2e}[ht]
\caption{Finite trap Localization procedure $\mathsf{TrapLocalize}
  (Q_{\boldsymbol x}(r),\boldsymbol x,P_{\boldsymbol x},\varepsilon,r)$}
\label{alg:trap-localize}
\SetAlgoLined
\DontPrintSemicolon
\SetKw{Return}{return}
\SetKwInOut{Input}{Input}
\SetKwInOut{Output}{Output}
\Input{$Q=Q_{\boldsymbol x}(r)$, $\boldsymbol x$,
$\varepsilon>0$, $r>0$, a polynomial $P_{\boldsymbol x}$
satisfying~\eqref{eq:terminal-ball-gradient-accuracy},
and a $\frac{3\varepsilon}{8}$ trapping certificate
on $\partial_{\mathcal X}Q$.}
\Output{A point $\boldsymbol z$ satisfying
\eqref{eq:trapping-terminal-certificates}.}

Set $\zeta\gets\frac{\varepsilon r}{32}$ and form
$\Psi_{\boldsymbol x}$ by~\eqref{eq:local_terminal}\;

Compute bounds from the coefficients of $P_{\boldsymbol x}$:
\[
M_P\ge
\sup_{\boldsymbol y\in Q}
\|\nabla P_{\boldsymbol x}(\boldsymbol y)\|_2
+\frac{\varepsilon}{2},
\qquad
L_P\ge
\sup_{\boldsymbol y\in Q}
\|\nabla^2P_{\boldsymbol x}(\boldsymbol y)\|_{\mathrm{op}}
\]
Set $\tau\gets\frac{1}{L_P+1}$\;

Choose a finite $\frac{\zeta}{M_P}$-net
$\mathcal V\subseteq Q$ containing $\boldsymbol x$\;
Compute $\boldsymbol z^0\in
\arg\min_{\boldsymbol y\in\mathcal V}
\Psi_{\boldsymbol x}(\boldsymbol y)$\;

\For{$k=0,1,2,\ldots$}{
 Compute $ \boldsymbol z^{k+1}
 :=
 \arg\min_{\boldsymbol y\in Q}
 \left\{
 \left\langle\nabla P_{\boldsymbol x}(\boldsymbol z^k),
 \boldsymbol y-\boldsymbol z^k\right\rangle
 +\frac{\|\boldsymbol y-\boldsymbol z^k\|_2^2}{2\tau}
 +\frac{\varepsilon}{2}
 \|\boldsymbol y-\boldsymbol x\|_2
 \right\}$
 Set $ \boldsymbol e^{k+1}
 :=
 \nabla P_{\boldsymbol x}(\boldsymbol z^{k+1})
 -\nabla P_{\boldsymbol x}(\boldsymbol z^k)
 +\frac{\boldsymbol z^k-\boldsymbol z^{k+1}}{\tau}$\;
 \If{$\|\boldsymbol e^{k+1}\|_2
       \le\frac{\varepsilon}{16}$}{
  \Return $\boldsymbol z^{k+1}$\;
 }
}
\end{algorithm2e}
In Line~9, we find the target point via a subroutine $\mathsf{TrapLocalize}$ which optimizes a local model.
The constructions below also guarantee
\begin{equation}
\label{eq:terminal-ball-gradient-accuracy}
 \|\nabla P_{\boldsymbol x}-\nabla f\|_{\infty,Q_{\boldsymbol x}(r)}
 \le\frac{\varepsilon}{16}.
\end{equation}
Set local model 
\begin{equation}
\label{eq:local_terminal}
\Psi_{\boldsymbol x}(\boldsymbol y)
:=P_{\boldsymbol x}(\boldsymbol y)-P_{\boldsymbol x}(\boldsymbol x)
+\frac{\varepsilon}{2}\|\boldsymbol y-\boldsymbol x\|_2
\end{equation}
and
$\zeta:=\frac{\varepsilon r}{32}$.
The terminal point is a $\zeta$-approximate global minimizer
of $\Psi_{\boldsymbol x}$ on $Q_{\boldsymbol x}(r)$ with
\begin{equation}
\label{eq:trapping-terminal-certificates}
 \begin{gathered}
 \Psi_{\boldsymbol x}(\boldsymbol z)\le0,\qquad
 \left\|\nabla P_{\boldsymbol x}(\boldsymbol z)
       +\frac{\varepsilon}{2}\boldsymbol v+\boldsymbol n\right\|_2
 \le\frac{\varepsilon}{16},\\
 \boldsymbol v\in\partial\|\cdot-\boldsymbol x\|_2(\boldsymbol z),
 \qquad \boldsymbol n\in N_{Q_{\boldsymbol x}(r)}(\boldsymbol z).
 \end{gathered}
\end{equation}
The residual bound expresses approximate first-order optimality
for $\Psi_{\boldsymbol x}$ on $Q_{\boldsymbol x}(r)$. Specifically, 
$\boldsymbol v$ accounts for the distance penalty, with
$\|\boldsymbol v\|_2\le1$, and $\boldsymbol n$ accounts for
the constraint.
The nonpositive model value excludes the trapped relative boundary,
so $\boldsymbol n\in N_{\mathcal X}(\boldsymbol z)$.
Combining the residual and model-gradient bounds gives
$\|\nabla f(\boldsymbol z)+\boldsymbol n\|_2
\le\frac{5\varepsilon}{8}$.
In the unconstrained case, $\boldsymbol n=0$, yielding
$\varepsilon$-stationarity.
In the cube-constrained case, nonexpansiveness of projection
gives $\|G_\gamma(\boldsymbol z)\|_2\le\varepsilon$
for every $\gamma>0$.


Lemma~\ref{lem:trapped-ball-localization} establishes finite
termination and correctness of Algorithm~\ref{alg:trap-localize} with proof is deferred to Appendix~\ref{app:trapping-framework-proofs}.
The procedure uses only the known polynomial and requires
no additional oracle queries.

\begin{lemma}[\bf Correctness of terminal localization]
\label{lem:trapped-ball-localization}
\label{lem:box-trapped-ball-localization}
Let $Q:=Q_x(r)$.
Suppose that $\phi_{x,\frac{3\varepsilon}{8}}>0$ on
$\partial_XQ$, and that the known scalar polynomial $P_x$
satisfies
$\|\nabla P_x-\nabla f\|_{\infty,Q}
\le\frac{\varepsilon}{16}$.
Then Algorithm~\ref{alg:trap-localize} terminates after
finitely many offline operations, without further oracle queries,
and returns a point $z\in Q$ satisfying
\[
\begin{cases}
 \|\nabla f(z)\|_2\le\varepsilon,
 &X=\mathbb R^d,\\
 \|G_\gamma(z)\|_2\le\varepsilon
 \quad\text{for every }\gamma>0,
 &X=[0,1]^d.
\end{cases}
\]
\end{lemma}

\subsection{Fresh Local Model Construction}
\label{subsec:trapping-fresh-implementations}
In this section, we give two implementations that reconstruct a polynomial
model at every pivot, using function values or gradients.
A suitable local radius ensures uniform gradient accuracy on
the feasible cap. This also controls relative function-value
errors on the boundary, so the same model supports both the
trapping-or-descent test and terminal localization.

For $q\ge0$, write $N_{d,q}:=\binom{d+q}{q}$.
For $L_p>0$, use the common radius
\begin{equation}
\label{eq:function-value-radius}
 r_{\mathrm F}:=\frac12
 \left(\frac{\varepsilon}{48L_p}\right)^{\frac{1}{\nu-1}}.
\end{equation}
At each pivot $\boldsymbol x$, take
$S:=Q_{\boldsymbol x}(r_{\mathrm F})$ and
$\boldsymbol c\in\operatorname{int}(S)$.
The point $\boldsymbol c=\boldsymbol x$ is admissible in the
unconstrained case and for a cap case, use any computable interior point.
Then $\rho:=\max\limits_{\boldsymbol y\in S}
\|\boldsymbol y-\boldsymbol c\|_2\le2r_{\mathrm F}$.

\begin{enumerate}[leftmargin=*]
\item \textbf{Function values.}
Apply Proposition~\ref{prop:local-reconstruction} to $f$ on $S$
with reconstruction center $\boldsymbol c$.
Using $N_{d,p}$ function-value queries, it constructs
$P_{\boldsymbol x}\in\Pi_p^d$ satisfying
\[
\begin{aligned}
\|\nabla P_{\boldsymbol x}-\nabla f\|_{\infty,S}
&\le\frac{3}{2}a_{p-1,\beta_p}L_p\rho^{\nu-1}\le\frac{\varepsilon}{32}.
\end{aligned}
\]

\item \textbf{Gradients.}
Apply Corollary~\ref{lem:gradient-only-scalar-reconstruction}
on $S$ with reconstruction center $\boldsymbol c$.
Using $N_{d,p-1}$ gradient queries, it constructs
$P_{\boldsymbol x}\in\Pi_p^d$ satisfying
\[
\begin{aligned}
\|\nabla P_{\boldsymbol x}-\nabla f\|_{\infty,S}
&\le
(2a_{p-1,\beta_p}+a_{p-2,\beta_p})L_p\rho^{\nu-1}
\le\frac{\varepsilon}{16}.
\end{aligned}
\]
\end{enumerate}

Integrating their gradient error along
$[\boldsymbol x,\boldsymbol y]\subseteq S$ gives
\[|P_{\boldsymbol x}(\boldsymbol y)-P_{\boldsymbol x}(\boldsymbol x)
-f(\boldsymbol y)+f(\boldsymbol x)|
\le\frac{\varepsilon}{16}\|\boldsymbol y-\boldsymbol x\|_2.\]
Thus Eq.~\eqref{eq:boundary-approximation} holds even when
$\boldsymbol c\ne\boldsymbol x$, and the retained scalar polynomial
also serves at termination.



\subsection{Lazy Local Model Construction}
\label{subsec:lazy-boundary-reuse}

In this subsection, we reduce the cost of repeated reconstruction by reusing a model
over a phase of $m$ calls. Each phase begins with a full
reconstruction but subsequent calls update the model with a
lower-degree correction that approximates the change caused by
moving the center. These corrections require much fewer queries.
By keeping the interpolation nodes fixed, the corrections
telescope, allowing us to control the accumulated error through
the net displacement from the refresh center.

Assume $L_p>0$ and fix a refresh period $m\ge1$.
Fix a refresh period $m\ge1$ and a mode
$\mathsf m \in\{\mathrm{val},\mathrm{grad}\}$,
corresponding to function-value and gradient queries,
respectively.
The construction in this subsection queries the whole ball
$\boldsymbol x+r\mathbb B$ at each pivot and uses the search radius
\begin{equation}
\label{eq:lazy-radius}
 r=r_m:=
 \left(\frac{\varepsilon}{96m^{\beta_p}L_p}\right)^{\frac{1}{\nu-1}}.
\end{equation}
We represent local functions on the unit ball through
$\boldsymbol y=\boldsymbol x+r\boldsymbol u$.
For $\boldsymbol u\in\mathbb B$, define
\[
F_{\boldsymbol x}(\boldsymbol u):=f(\boldsymbol x+r\boldsymbol u),
\qquad
A_{\boldsymbol x}(\boldsymbol u):=
F_{\boldsymbol x}(\boldsymbol u)-F_{\boldsymbol x}(\boldsymbol 0),
\qquad
B_{\boldsymbol x}(\boldsymbol u):=
r\nabla f(\boldsymbol x+r\boldsymbol u).
\]
Here $F_{\boldsymbol x}$ is the objective in normalized coordinates,
and $A_{\boldsymbol x}$ measures its increment relative to the
current pivot, with $A_{\boldsymbol x}(\boldsymbol 0)=0$.
The vector field $B_{\boldsymbol x}$ is the gradient with respect
to $\boldsymbol u$.


Choose fixed nested interpolation nodes
$\{\boldsymbol 0\}=\mathcal U_0\subseteq\cdots
\subseteq\mathcal U_p\subseteq\mathbb B$,
with $|\mathcal U_q|=N_{d,q}$, as in
Proposition~\ref{prop:local-reconstruction}.
Let $I_q$ denote the corresponding linear interpolation operator
of total degree at most $q$, acting componentwise on vector data.
Use a common contraction toward the origin so that the
proposition's estimates hold for the required degrees,
with exponent $\beta_p$ for fresh models and exponent $1$
for rolling corrections.
These normalized nodes and operators remain fixed throughout
the run and at center $\boldsymbol x$, the physical query locations
are $\boldsymbol x+r\mathcal U_q$.

The retained model is a scalar polynomial
$\widehat A_{\boldsymbol x}$ approximating $A_{\boldsymbol x}$
when $\mathsf m =\mathrm{val}$, or a vector polynomial
$V_{\boldsymbol x}$ approximating $B_{\boldsymbol x}$
when $\mathsf m =\mathrm{grad}$.
The cached state is
$\mathsf S=(j,\boldsymbol z,\mathcal C_{\boldsymbol z})$,
where $j$ counts the model calls in the current phase,
$\boldsymbol z$ is the preceding center, and
\[
\mathcal C_{\boldsymbol z}:=
\begin{cases}
\bigl(\widehat A_{\boldsymbol z},
F_{\boldsymbol z}|_{\mathcal U_{p-1}}\bigr),
& \mathsf m =\mathrm{val},\\[2pt]
\bigl(V_{\boldsymbol z},
B_{\boldsymbol z}|_{\mathcal U_{p-2}}\bigr),
& \mathsf m =\mathrm{grad}.
\end{cases}
\]
Thus the cache contains the retained model and the exact
data at the lower-degree nodes.
An empty state is represented by $j=0$.

Now we are ready to state our lazy algorithm (Algorithm~\ref{alg:lazy-approx-boundary}).
Algorithm~\ref{alg:lazy-approx-boundary} constructs the local
polynomial $P_{\boldsymbol x}$ and the boundary model
$\widehat\phi_{\boldsymbol x}$, while updating the cached state
for the next call.

\begin{algorithm2e}[ht]
\caption{Lazy Model}
\label{alg:lazy-approx-boundary}
\SetAlgoLined
\DontPrintSemicolon
\SetKw{Return}{return}
\SetKwInOut{Input}{Input}
\SetKwInOut{Output}{Output}
\Input{$\boldsymbol x$, $r=r_m$, $\varepsilon$, $\mathsf m$, $m$,
and $\mathsf S=(j,\boldsymbol z,\mathcal C_{\boldsymbol z})$.}
\Output{$P_{\boldsymbol x}$ and $\widehat\phi_{\boldsymbol x}$,
with $\mathsf S$ updated.}

$\ell\gets p-\mathbf 1_{\{\mathsf m=\mathrm{grad}\}}
-\mathbf 1_{\{0<j<m\}}$\;
Query the selected oracle on $\boldsymbol x+r\mathcal U_\ell$\;
Update the retained model according to
\[
\begin{cases}
\widehat A_{\boldsymbol x}\gets I_pA_{\boldsymbol x},
& \mathsf m=\mathrm{val},\ j\in\{0,m\},\\[2pt]
\widehat A_{\boldsymbol x}\gets
\widehat A_{\boldsymbol z}
+I_{p-1}(A_{\boldsymbol x}-A_{\boldsymbol z}),
& \mathsf m=\mathrm{val},\ 0<j<m,\\[2pt]
V_{\boldsymbol x}\gets I_{p-1}B_{\boldsymbol x},
& \mathsf m=\mathrm{grad},\ j\in\{0,m\},\\[2pt]
V_{\boldsymbol x}\gets
V_{\boldsymbol z}
+I_{p-2}(B_{\boldsymbol x}-B_{\boldsymbol z}),
& \mathsf m=\mathrm{grad},\ 0<j<m.
\end{cases}
\]
Define, for $\boldsymbol u\in\mathbb B$,
\[
P_{\boldsymbol x}(\boldsymbol x+r\boldsymbol u):=
\begin{cases}
\widehat A_{\boldsymbol x}(\boldsymbol u),
& \mathsf m=\mathrm{val},\\[2pt]
\displaystyle\int_0^1
\langle V_{\boldsymbol x}(t\boldsymbol u),\boldsymbol u\rangle\,dt,
& \mathsf m=\mathrm{grad}.
\end{cases}
\]
$\widehat\phi_{\boldsymbol x}\gets
P_{\boldsymbol x}-P_{\boldsymbol x}(\boldsymbol x)
+\frac{\varepsilon}{4}\|\cdot-\boldsymbol x\|_2$
on $\Gamma_{\boldsymbol x}(r)$\;
$j\gets1+j\mathbf 1_{\{0<j<m\}}$\;
$\mathsf S\gets(j,\boldsymbol x,\mathcal C_{\boldsymbol x})$\;
\Return $(P_{\boldsymbol x},\widehat\phi_{\boldsymbol x})$\;
\end{algorithm2e}

\begin{figure}[ht]
    \centering
    \includegraphics[width=0.7\linewidth]{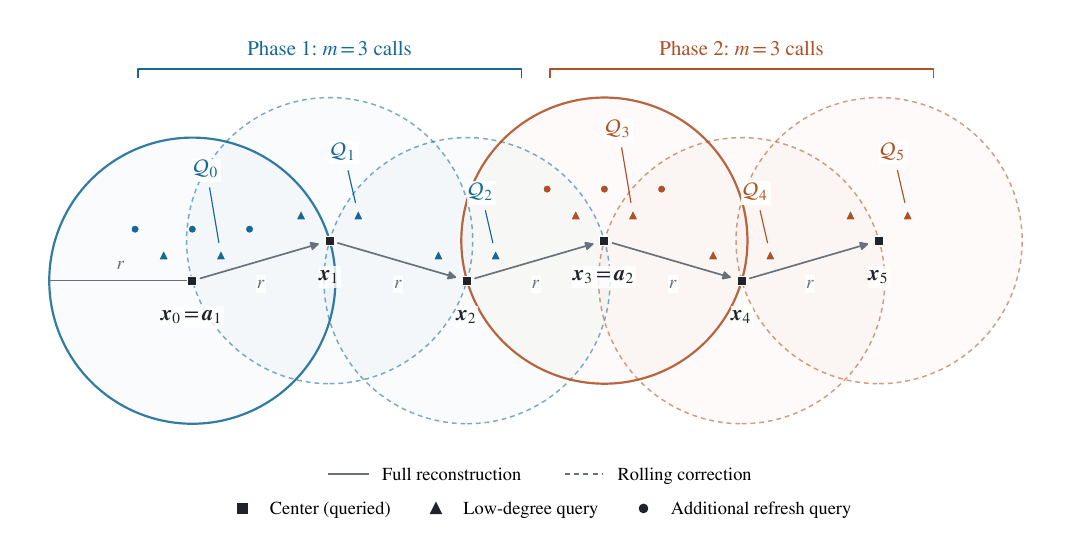}
    \caption{
    Two phases of lazy local-model construction ($m=3$).
    }
    \label{fig:lazy-local-models}
\end{figure}

The retained models are expressed in the same normalized
variable $\boldsymbol u$.
Line~1 selects the full interpolation degree at a refresh
and one lower degree otherwise.
When the center moves from $\boldsymbol z$ to $\boldsymbol x$,
the reference domain remains $\mathbb B$, while the target
changes from $A_{\boldsymbol z}$ to $A_{\boldsymbol x}$,
or from $B_{\boldsymbol z}$ to $B_{\boldsymbol x}$.
In Line~3, the rolling correction approximates this change
and adds it to the retained model.
Thus the highest-degree part remains fixed within a phase,
while the lower-degree coefficients are updated.
The counter in Line~4 schedules a new full reconstruction
after $m$ model calls.

To construct a correction, the algorithm pairs nodes with
the same normalized coordinate $\boldsymbol u$.
Their physical locations are
$\boldsymbol z+r\boldsymbol u$ and
$\boldsymbol x+r\boldsymbol u$.
The former value was cached in Line~5 of the preceding call,
so Line~2 only queries the latter.
These paired data determine the interpolated difference
used in Line~3.

Because all rolling corrections in Line~3 use the same
linear interpolation operator, their sum is the interpolant
of the sum of the successive differences.
The intermediate functions cancel, giving
\begin{equation}
\label{eq:trapping-lazy-telescoping}
\begin{aligned}
\widehat A_{\boldsymbol x}
&=I_pA_{\boldsymbol a}
  +I_{p-1}(A_{\boldsymbol x}-A_{\boldsymbol a}),\\
V_{\boldsymbol x}
&=I_{p-1}B_{\boldsymbol a}
  +I_{p-2}(B_{\boldsymbol x}-B_{\boldsymbol a}),
\end{aligned}
\end{equation}
where $\boldsymbol a$ is the refresh center.
The current model therefore consists of the initial
reconstruction and a single interpolated net change.
This identity allows the accuracy analysis to depend on
the net displacement from $\boldsymbol a$.

Finally, the scalar reconstruction uses
$\widehat A_{\boldsymbol x}$ directly in function-value mode.
In gradient mode, it follows the identity
$A_{\boldsymbol x}(\boldsymbol u)
=\int_0^1\langle B_{\boldsymbol x}(t\boldsymbol u),
\boldsymbol u\rangle\,dt$,
with $B_{\boldsymbol x}$ replaced by $V_{\boldsymbol x}$.
This conversion requires no additional queries.
Line~6 returns the resulting polynomial and its boundary model. Figure~\ref{fig:lazy-local-models} illustrates the center movements
and query locations over two consecutive phases.

\subsection{Proof Sketch of Correctness and Complexity}
\label{subsec:trapping-proof-sketch}

\paragraph{Correctness.}
The reconstruction bounds imply the boundary dichotomy:
either every relative-boundary point is unreachable, or the
objective decreases by at least $\varepsilon r/8$.
The gap bound limits the number of moves, and the retained local model
makes the terminal computation finite.
The boundary barrier, the polynomial-gradient accuracy, and
Eq.~\eqref{eq:trapping-terminal-certificates} then certify the output.


\paragraph{Complexity.}
We consider $L_p>0$ and $\Delta>0$.
Each accepted move decreases the objective by at least
$\frac{\varepsilon r}{8}$.
Summing these decreases gives
$N_{\mathrm{mov}}\frac{\varepsilon r}{8}\le\Delta$.
Thus the number of model calls satisfies
$K:=N_{\mathrm{mov}}+1\le1+\frac{8\Delta}{\varepsilon r}$,
where the final call supplies the polynomial used at termination.
The terminal computation requires no additional oracle queries.

For fresh reconstruction, every model call uses $N_{d,p}$
function-value queries or $N_{d,p-1}$ gradient queries.
Substituting $r=r_{\mathrm F}$ from
\eqref{eq:function-value-radius} gives
$K\le1+C_{p,\beta_p}
L_p^{\frac{1}{\nu-1}}\Delta
\varepsilon^{-\frac{\nu}{\nu-1}}$.
Multiplying this bound by the corresponding query count and
using $N_{d,q}=\mathcal O_p(d^q)$ yields the two bounds in
Theorem~\ref{the:fresh-trapping-or-descent}.

For lazy reconstruction, a refresh occurs once every $m$
model calls.
Hence there are $R:=\lceil K/m\rceil$ refreshes and $K-R$
rolling updates, giving
\[
Q_{\mathrm{val}}=RN_{d,p}+(K-R)N_{d,p-1},
\qquad
Q_{\mathrm{grad}}=RN_{d,p-1}+(K-R)N_{d,p-2}.
\]
The radius $r_m$ in \eqref{eq:lazy-radius} gives
$K\le1+C_{p,\beta_p}m^{\frac{\beta_p}{\nu-1}}
L_p^{\frac{1}{\nu-1}}\Delta
\varepsilon^{-\frac{\nu}{\nu-1}}$.
Using $R\le1+K/m$, we obtain
\[
\begin{aligned}
Q_{\mathrm{val}}
&\le C_p\left[
d^p+K\left(\frac{d^p}{m}+d^{p-1}\right)\right]\\
&\le C_{p,\beta_p}\left\{
d^p+
\left[
d^p m^{-\frac{p-1}{\nu-1}}
+d^{p-1}m^{\frac{\beta_p}{\nu-1}}
\right]
L_p^{\frac{1}{\nu-1}}\Delta
\varepsilon^{-\frac{\nu}{\nu-1}}
\right\}.
\end{aligned}
\]
The last inequality uses
$\frac{\beta_p}{\nu-1}-1=-\frac{p-1}{\nu-1}$ and
the constant term in the bound on $K$ is absorbed into
$C_{p,\beta_p}d^p$.
Similarly,
$Q_{\mathrm{grad}}\le
N_{d,p-1}+K(\frac{N_{d,p-1}}{m}+N_{d,p-2})$,
so its bound has each dimension power reduced by one.

The two terms involving $m$ account for refreshes and
rolling updates, respectively.
Increasing $m$ reduces the refresh contribution, but the
smaller search radius increases the number of model calls.
Choosing $m=d$ balances these contributions and yields
\[
\begin{aligned}
Q_{\mathrm{val}}
&=\mathcal O_{p,\beta_p}\!\left(
d^p+d^{p-1+\frac{\beta_p}{\nu-1}}
L_p^{\frac{1}{\nu-1}}\Delta
\varepsilon^{-\frac{\nu}{\nu-1}}
\right),\quad
Q_{\mathrm{grad}}
=\mathcal O_{p,\beta_p}\!\left(
d^{p-1}+d^{p-2+\frac{\beta_p}{\nu-1}}
L_p^{\frac{1}{\nu-1}}\Delta
\varepsilon^{-\frac{\nu}{\nu-1}}
\right).
\end{aligned}
\]
The full calculation is given in
Proposition~\ref{prop:lazy-boundary-reuse-complexity}.

\subsection{Comparison with Regularized Tensor Methods}
\label{app:trapping-comparisons}
In this section, we compare our trapping constructions with
regularized tensor methods~\citep{doikov2026lowerorder,gratton2023convergence},
focusing on their local models, certified objective decrease,
and oracle requirements.
We first explain the relation between the models, then compare
their query costs, and finally discuss their assumptions on
smoothness information.
The fixed-radius identity below holds for every
$\nu=p+\beta_p>1$, whereas the literature comparisons are
restricted to unconstrained minimization with $\beta_p=1$
and $p\ge2$.
Throughout this subsection,
$\Delta\ge f(\boldsymbol x^0)-\inf f$.


\paragraph{Relation between the local models.}
On a sphere of prescribed radius $r$, the boundary surrogate
coincides with a regularized polynomial model.
Indeed, whenever $\|\boldsymbol y-\boldsymbol x\|_2=r$,
\[
\widehat\phi_{\boldsymbol x}(\boldsymbol y)
=
P_{\boldsymbol x}(\boldsymbol y)-P_{\boldsymbol x}(\boldsymbol x)
+\frac{\varepsilon}{4r^{\nu-1}}
 \|\boldsymbol y-\boldsymbol x\|_2^\nu,
\]
because the regularization term equals $\varepsilon r/4$,
as does the linear penalty defining
$\widehat\phi_{\boldsymbol x}$.
At the fresh radius (Eq.~\eqref{eq:function-value-radius}),
the regularization coefficient is
$\frac{\varepsilon}{4r_{\mathrm F}^{\nu-1}}
=12\cdot2^{\nu-1}L_p$.
Thus the radius satisfies
$L_pr_{\mathrm F}^{\nu-1}\asymp_{p,\beta_p}\varepsilon$,
with certified decrease proportional to
$\varepsilon r_{\mathrm F}$.
When $\beta_p=1$, the regularization power is $p+1$,
as in AR$p$.

However, the two constructions use this model differently.
Our algorithm fixes the radius and approximately minimizes
the boundary surrogate to obtain either an accepted move
or a trapping certificate.
An AR$p$ step also determines its length by optimizing
the regularized model.
Since our reconstruction guarantee holds on the prescribed
ball, the identity above does not by itself justify an
unrestricted regularized-model step.

\paragraph{Descent progress and query complexity for $\beta_p=1$.}
The fresh and lazy trapping bounds recover the progress scales
of the AR$p$ schemes compared in~\citep{doikov2026lowerorder}.
Table~\ref{tab:boundary-versus-regularized-tensor} distinguishes
these progress scales from the oracle information used to
construct models. A lazy boundary phase contains at most $m$
moves and its displayed decrease applies to a completed $m$-move
phase. Every individual accepted boundary move decreases $f$ by
$\Omega_p (
m^{-\frac{1}{p}}
L_p^{-\frac{1}{p}}
\varepsilon^{1+\frac{1}{p}}
)$.

\begin{table*}[t]
\centering
\small
\renewcommand{\arraystretch}{1.15}
\caption{Comparison for $\beta_p=1$ and $p\ge2$.
The last two columns give model-reconstruction costs for the
indicated unit. Daggered entries count the information needed
for polynomial reconstruction at the prescribed accuracy and
they are not strict-oracle complexity guarantees for AR$p$.
All costs are deterministic.}
\label{tab:boundary-versus-regularized-tensor}
\begin{tabular}{@{}lccc@{}}
\toprule
Method (unit) & Objective decrease & Function values & Gradients\\
\midrule
Fresh AR$p$ (step)
& $\Omega_p(L_p^{-1/p}\varepsilon^{1+1/p})$
& $\mathcal O_p(d^p)^\dagger$
& $\mathcal O_p(d^{p-1})^\dagger$\\
Fresh boundary (move)
& $\Omega_p(L_p^{-1/p}\varepsilon^{1+1/p})$
& $\mathcal O_p(d^p)$
& $\mathcal O_p(d^{p-1})$\\
Lazy AR$p$ ($m$-step block)
& $\Omega_p(m^{1-1/p}L_p^{-1/p}\varepsilon^{1+1/p})$
& $\mathcal O_p(d^{p-1}(d+m))^\dagger$
& $\mathcal O_p(d^{p-2}(d+m))^\dagger$\\
Lazy boundary ($m$-move phase)
& $\Omega_p(m^{1-1/p}L_p^{-1/p}\varepsilon^{1+1/p})$
& $\mathcal O_p(d^{p-1}(d+m))$
& $\mathcal O_p(d^{p-2}(d+m))$\\
\bottomrule
\end{tabular}
\end{table*}

Standard AR$p$ accesses derivatives through order $p$ but the lazy
scheme in~\citep{doikov2026lowerorder} instead uses an
order-$(p-1)$ oracle. The daggered columns do not claim an exact
simulation of these derivative oracles by finitely many values
or gradients. In contrast, Theorems~
\ref{the:fresh-trapping-or-descent} and
\ref{the:lazy-trapping-or-descent} analyze the strict function-value
and gradient-only implementations directly.

For $p=2$, our lazy gradient-only bound is
$\mathcal O(d+\sqrt d\,L_2^{1/2}\Delta\varepsilon^{-3/2})$.
The $\sqrt d$ factor matches the lazy cubic-regularization
comparison, but does not improve on the $d^{1/3}$-type
dependences obtained by gradient--Hessian balancing
\citep{adil2025balancing} and NALEN~\citep{chen2026faster}.
Those methods use different progress analyses; the comparison
above concerns the displayed model costs and certified decrease.

\paragraph{Objective-function-free regularization.}
\label{rem:comparison-offarp}
OFFAR$p$~\citep{gratton2023convergence} attains the
$\varepsilon^{-(p+1)/p}$ dependence without function values,
using the full derivative tuple
$\{\nabla^if(\boldsymbol x)\}_{i=1}^p$ at each iteration.
The stochastic and random-subspace extensions use approximate
or sketched higher-order derivatives
\citep{gratton2025stochastic,bellavia2026optimally}.
These oracle models differ from gradient-only access for $p>1$;
the purely first-order OFFO result discussed in
\citep{gratton2026complexity} has the classical
$\mathcal O(\varepsilon^{-2})$ dependence.

Our fresh gradient-only bound for $\beta_p=1$ is
$\mathcal O_p(d^{p-1}[1+L_p^{1/p}\Delta
\varepsilon^{-(p+1)/p}])$. For $p=2$, it uses $d+1$ gradient
queries per reconstructed scalar model. OFFAR$2$ instead uses
gradient and Hessian evaluations, and additionally assumes a
uniform lower bound on the Hessian eigenvalues. It does not
require $L_2$ as an algorithmic input. Our fixed-radius method
uses a known upper bound on $L_2$, exact oracle answers, and
uncharged optimization of explicit models. Thus matching the
accuracy exponent does not identify the oracle models or all
of their assumptions.

\paragraph{Dependence on the smoothness bound.}
The stated finite-output guarantees use a known upper bound on
$L_p$. With function values, a trial move can instead be checked
directly by evaluating whether
$f(\boldsymbol y)\le f(\boldsymbol x)-\varepsilon r/8$.
This certifies the decrease of an accepted move even if the
radius was chosen using an underestimated smoothness constant.
It does not certify the uniform model-error bound needed for
the trapping branch.

One may increase the smoothness estimate after a failed trial
or an uncertified trapping decision, but the present analysis
does not establish the same total complexity for this adaptive
variant. In particular, generating an approximately stationary
candidate and identifying a valid stopping time are different
requirements. This distinction also arises in parameter-free
zeroth-order cubic regularization~\citep{doikov2025first}.
A parameter-free implementation with a certified stopping rule
and the stated strict-oracle bounds remains a separate question.

\section{Randomized First-Order Lower Bounds in High Dimension}
\label{sec:lb}

In this section, we prove a randomized lower bound for finding approximate stationary points
with an exact first-order oracle.  The hard family is a chain of independently
rotated quadratic systems.  The quadratic matrix action in every block is
present from the outset, whereas the explicit disclosure of a later
right-hand side is delayed until the preceding block is reached.  Before that
time, the ordinary oracle transcript is independent of the hidden right-hand
side unless the current phase is prematurely outer-hit, excluding precisely
this event is part of the information argument.  

For a differentiable function \(f:\mathbb R^d\to\mathbb R\), the exact
value-and-gradient oracle returns  $ \mathsf{FO}_f(x)=(f(x),\nabla f(x)).$
A randomized method with query budget \(T\) makes at most \(T\) oracle calls
for every input function and every realization of its internal randomness.
Its \(j\)-th query may be any Borel function of its random seed and the first
\(j-1\) oracle responses.  After its final response, the method may return any
point measurable with respect to the seed and the transcript. 
We define the complexity
\(T_{\varepsilon,\mathrm{FO}}^{\mathrm{rand}}(\mathcal F)\) as the least
\(T\) queries for which one algorithm satisfies
\[
 \inf_{f\in\mathcal F}
 \Pr\bigl(\|\nabla f(x^{\mathrm{out}})\|_2\leq\varepsilon\bigr)
 \geq\frac23.
\]
Fix an integer \(p\geq2\) and parameters
\(\Delta,L_1,L_p,\varepsilon>0\).  Define the effective curvature scale
\begin{equation}
\label{eq:drhs-effective-scale}
 \gamma_p(\varepsilon)
 :=
 \min\left\{
   L_1,
   \bigl(L_p\varepsilon^{p-1}\bigr)^{\frac{1}{p}}
 \right\}.
\end{equation}
All constants below that depend on the fixed cutoff profile are absorbed into
constants depending only on \(p\). For $\Delta,L_1,L_p>0$, define
\[
\mathcal F_{1,p}^{(d)}(\Delta,L_1,L_p)
:=
\left\{
f\in C^{1,1}(L_1)\cap C^{p,1}(L_p):
f\ge0,\quad f(0)-f_\star\le\Delta
\right\},
\]
where $f_\star=\inf_{x\in\mathbb R^d}f(x)$. Without gradient Lipschitz bound, we consider the class
\[
\mathcal F_p^{(d)}(\Delta,L_p)
:=
\left\{
f\in C^{p,1}(L_p):
f\ge0,\quad f(0)-f_\star\le\Delta
\right\}.
\]

\begin{theorem}[\bf Randomized first-order lower bound]
\label{thm:rand-fo-lower-bound}
For every integer \(p\geq2\), there is a constant \(c_p>0\) such that the
following holds.  If $ 0<\varepsilon
 \leq
 c_p\min\big\{
   (\Delta L_1)^{\frac{1}{2}},
\Delta^{\frac{p}{p+1}}L_p^{\frac{1}{p+1}}
 \big\},$
then there is a finite constant \(C_{p,\Delta,L_1,L_p}\) such that, for every
$ d\geq
 C_{p,\Delta,L_1,L_p}\varepsilon^{-\big(4+\frac{2}{p}\big)},$
one has
\begin{equation}
\label{eq:drhs-main-rate}
 T_{\varepsilon,\mathrm{FO}}^{\mathrm{rand}}
 \bigl(\mathcal F_{1,p}^{(d)}(\Delta,L_1,L_p)\bigr)
 \geq
 c_p\Delta\sqrt{L_1\gamma_p(\varepsilon)}\,
 \varepsilon^{-2}.
\end{equation}
\end{theorem}

\begin{remark}
The hard functions may be chosen \(C^\infty\) and coercive.  The conclusion
holds for the exact value-and-gradient oracle and for an arbitrary
transcript-measurable terminal output.  It therefore also holds for the
weaker gradient-only oracle.
\end{remark}

For sufficiently small $\varepsilon$, the effective curvature scale
$\gamma_p(\varepsilon)$ is determined by the higher-order smoothness
constraint. Substituting its value into
Theorem~\ref{thm:rand-fo-lower-bound} yields the following explicit
dependence on the accuracy and smoothness parameters.
\begin{corollary}
\label{cor:drhs-high-accuracy-rate}
Under the assumptions of Theorem~\ref{thm:rand-fo-lower-bound}, suppose in
addition that $ \varepsilon
 \leq
 \big(\frac{L_1^p}{L_p}\big)^{\frac{1}{p-1}}.$
Then
\[
 T_{\varepsilon,\mathrm{FO}}^{\mathrm{rand}}
 \bigl(\mathcal F_{1,p}^{(d)}(\Delta,L_1,L_p)\bigr)
 \geq
 c_p\Delta L_1^{\frac{1}{2}}L_p^{\frac{1}{2p}}
\varepsilon^{-\big(\frac{3}{2}+\frac{1}{2p}\big)}.
\]
In particular, the exponent of $\varepsilon^{-1}$ is \(\frac 7 4\) for \(p=2\) and \(\frac 5 3\) for \(p=3\).
\end{corollary}


\begin{remark}[\bf Comparison with deterministic lower bounds]
For the class $\bigcap_{q=1}^p C^{q,1}(L_q)$,
\cite{zhou2026sharp} establishes a deterministic first-order
lower bound of order
$\Delta\sqrt{L_1\Gamma_p(\varepsilon)}\,\varepsilon^{-2}$, where
\[
\Gamma_p(\varepsilon)
:=
\sup_{r\ge1}
\min\left\{
L_1,\,
\frac{L_1}{\ell_1r^2},\,
\min_{2\le q\le p}
\left(
\frac{L_q\varepsilon^{q-1}r^{q-3}}{\ell_q}
\right)^{\frac1q}
\right\},
\]
with constants $\ell_q$ for $q\ge1$.
The parameter $r$ controls the hard function and is
optimized to balance the derivative bounds.

Theorem~\ref{thm:rand-fo-lower-bound} applies to randomized
first-order methods under
$C^{1,1}(L_1)\cap C^{p,1}(L_p)$.
For fixed positive smoothness constants and sufficiently small
$\varepsilon$, one has
$\Gamma_p(\varepsilon)\sim \gamma_p(\varepsilon)$
when $p=2,3$.
Thus our randomized lower bound matches the corresponding
deterministic scales in these cases.
For $p\ge4$, however,
$\Gamma_p(\varepsilon)
=\Theta_{p,L_1,\ldots,L_p}(\varepsilon^{\frac23})$,
whereas
$\gamma_p(\varepsilon)
=L_p^{\frac1p}\varepsilon^{\frac{p-1}{p}}$.

Theorem~\ref{thm:rand-fo-lower-bound} gives a sufficient
dimension requirement of order
$\varepsilon^{-(4+\frac2p)}$ while the deterministic construction of \cite{zhou2026sharp} can be realized in
dimension $\mathcal O(\varepsilon^{-\frac74})$ for $p=2$
and $\mathcal O(\varepsilon^{-\frac53})$ for $p\ge3$.
These dimension estimates follow from its block sizes and phase count.
Whether the sufficient dimension threshold for our randomized
lower bound can be reduced remains open.
\end{remark}

The loss for \(p\geq4\) in
Corollary~\ref{cor:drhs-high-accuracy-rate} comes from the prescribed
\(C^{1,1}(L_1)\) budget, which fixes the upper spectral scale of the
quadratic blocks.  If this condition \(C^{1,1}(L_1)\) is removed, the quadratic scale can
instead be chosen as part of the hard instance.  This additional freedom
allows the construction to attain the \(\frac{5}{3}\) exponent for every
\(p\geq3\). 
We have the following theorem with same construction.

\begin{theorem}[\bf Lower bounds without gradient Lipschitzness]
\label{thm:drhs-without-l1}
For every integer \(p\geq2\), there are constants \(c_p,C_p>0\) such
that the following holds whenever
\(0<\varepsilon\leq
c_p\Delta^{p/(p+1)}L_p^{1/(p+1)}\):
\begin{equation*}
\label{eq:drhs-without-l1-rates}
\begin{aligned}
p=2,\quad
d&\geq
C_p\left(
 \frac{\Delta^{2/3}L_2^{1/3}}{\varepsilon}
\right)^5
&&\Longrightarrow&
T_{\varepsilon,\mathrm{FO}}^{\mathrm{rand}}
\bigl(\mathcal F_2^{(d)}(\Delta,L_2)\bigr)
&\geq
c_p\left(
 \frac{\Delta^{2/3}L_2^{1/3}}{\varepsilon}
\right)^{7/4},\\
p\geq3,\quad
d&\geq
C_p\left(
 \frac{\Delta^{p/(p+1)}L_p^{1/(p+1)}}{\varepsilon}
\right)^{13/3+1/p}
&&\Longrightarrow&
T_{\varepsilon,\mathrm{FO}}^{\mathrm{rand}}
\bigl(\mathcal F_p^{(d)}(\Delta,L_p)\bigr)
&\geq
c_p\left(
 \frac{\Delta^{p/(p+1)}L_p^{1/(p+1)}}{\varepsilon}
\right)^{5/3}.
\end{aligned}
\end{equation*}
The hard functions may be chosen \(C^\infty\) and coercive.
\end{theorem}

For $p=2$, Theorem~\ref{thm:drhs-without-l1} establishes a
randomized first-order lower bound with the sharp accuracy
exponent $\frac{7}{4}$ under Hessian Lipschitzness, matching
the upper bound of \cite[Theorem 1]{marumo2025universal} (under the additional assumption $f\in C^{1,1}(L_1)$) and the
deterministic lower bound of \cite{zhou2026sharp}
in their dependence on $\varepsilon$.
For every \(p\geq3\), the second conclusion gives
\[
 T
 \gtrsim_p
 \Delta^{\frac{5p}{3(p+1)}}
 L_p^{\frac{5}{3(p+1)}}
 \varepsilon^{-\frac{5}{3}}.
\]
Thus the lower bound matches the known \(\frac{5}{3}\) first-order exponent under
third-order smoothness, up to the logarithmic factor in the
dimension-free upper bound of
\cite[Theorem~2]{carmon2017convex} (under the additional assumption $f\in C^{1,1}(L_1)$).  The saturation of this exponent at
every finite order \(p\geq3\) is also consistent with the deterministic
lower bounds of \cite{zhou2026sharp}.

For $p\ge4$, removing the $L_1$-budget strengthens the
randomized query lower bound from
$\Omega_p(\varepsilon^{-(\frac32+\frac1{2p})})$ in
Corollary~\ref{cor:drhs-high-accuracy-rate} to
$\Omega_p(\varepsilon^{-\frac53})$, with the remaining
parameters fixed.
The resulting construction has a larger sufficient dimension
threshold, increasing from order
$\varepsilon^{-(4+\frac2p)}$ to
$\varepsilon^{-(\frac{13}{3}+\frac1p)}$.
Thus the dimension exponent increases by
$\frac{p-3}{3p}$.

\subsection{The Delayed-Information Chain and Proof of the Theorem~\ref{thm:rand-fo-lower-bound}}
\label{subsec:drhs-main-proof}

We now construct the hard distribution and prove
Theorem~\ref{thm:rand-fo-lower-bound} using key lemmas
whose proofs are deferred to Appendix~\ref{app:lb}. 

Our construction links a sequence of quadratic problems in
separate blocks of variables. Each phase has a target
$\boldsymbol x_s^\star=A_s^{-1}\boldsymbol b_s$, and the
construction ensures that every $\varepsilon$-stationary point
must approach all these targets.
For each phase after the first, a smooth coupling activates
the linear term containing $\boldsymbol b_s$ only when the
preceding block approaches its target.
This creates a sequential dependence that information about the
next target remains hidden unless the algorithm reaches the
preceding target or happens to approach the next target
prematurely.

The information argument establishes two properties of this
construction. First, premature hits are unlikely in sufficiently
high dimension. Second, even when $\boldsymbol b_s$ is explicitly
revealed, approaching $A_s^{-1}\boldsymbol b_s$ requires
additional matrix-vector queries despite the information
collected before disclosure.
Thus advance exploration of the quadratic terms cannot
eliminate the subsequent cost of solving each phase, and
these costs accumulate along the chain.
All phase data are sampled in advance from a distribution
$\mathcal D_d$ that is independent of the algorithm.
The resulting distributional bound for deterministic algorithms
then extends to randomized algorithms by averaging over their
internal randomness.

More precisely, the first lemma shows that the phasewise costs accumulate
sequentially
(Lemma~\ref{lem:drhs-sequential}).  The second provides the one-phase
hardness estimate underlying this accumulation
(Lemma~\ref{lem:drhs-one-phase}).  The third shows that every
\(\varepsilon\)-stationary point must complete all phases
(Lemma~\ref{lem:drhs-certification}).  Finally, the fourth establishes the
regularity and initial-gap bounds of the chain
(Lemma~\ref{lem:drhs-regularity}).

\subsubsection{Calibrated Scales and Phase Data}
Let $C_{\mathrm{sm}},C_{\mathrm{gap}},C_p$ be the constants in
Lemma~\ref{lem:drhs-regularity}, $c_0,\kappa_0$ those in
Lemma~\ref{lem:drhs-spectrum}, and $C_0$ the constant in
Lemma~\ref{lem:drhs-one-phase}.
Choose positive constants $c_{\mu,p},c_N$ satisfying
\begin{equation}
\label{eq:drhs-constant-choice}
 c_{\mu,p}\leq
 \min\left\{
   \frac{1}{2C_{\mathrm{sm}}},
   C_p^{-\frac{1}{p}},
   \frac{1}{2\kappa_0},
   \frac{c_0^2}{8}
 \right\},
 \qquad
 c_N\leq\frac{1}{C_{\mathrm{gap}}}.
\end{equation}
Using the effective curvature scale
$\gamma_p(\varepsilon)$ from
Eq.~\eqref{eq:drhs-effective-scale}, set
\begin{equation}
\label{eq:drhs-calibrated-scales}
 \mu:=c_{\mu,p}\gamma_p(\varepsilon),
 \qquad
 \rho:=\frac{\varepsilon}{\mu},
 \qquad
 \kappa:=\frac{L_1}{2\mu}.
\end{equation}
The number of phases $N$, the delay per phase $t$, and the
total delay $T$ are
\begin{equation}
\label{eq:drhs-calibrated-depth}
 N:=\left\lfloor
      c_N\frac{\Delta\mu}{\varepsilon^2}
    \right\rfloor,
 \qquad
 t:=\left\lfloor c_0\sqrt{\kappa}\right\rfloor,
 \qquad
 T:=Nt.
\end{equation}
Choose
$0<c_{\mathrm{acc},p}\leq
\min\{1,(c_Nc_{\mu,p}/2)^{\frac{p}{p+1}}\}$
and require
\begin{equation}
\label{eq:drhs-calibrated-small-accuracy}
 \varepsilon\leq c_{\mathrm{acc},p}
 \min\left\{
   (\Delta L_1)^{\frac{1}{2}},
   \Delta^{\frac{p}{p+1}}L_p^{\frac{1}{p+1}}
 \right\}.
\end{equation}
For $C_{\mathrm{dim}}\geq\max\{C_0,2\}$, define the
dimension of each phase and the total required dimension by
\begin{equation}
\label{eq:drhs-calibrated-dimension}
 M:=\left\lceil
 C_{\mathrm{dim}}
 \left\{
   T^2+\kappa\left[
     T+\log\bigl(12NT(T+2)\bigr)
   \right]
 \right\}
 \right\rceil,
 \qquad
 d_0:=NM.
\end{equation}

Set $B_0:=512$.
For each $s\in[N]$, take a symmetric matrix
$A_s\in\mathbb R^{M\times M}$ and a vector
$b_s\in\mathbb R^M$, with target $x_s^\star:=A_s^{-1}b_s$,
satisfying
\begin{equation}
\label{eq:drhs-admissible-data}
 \mu I\preceq A_s\preceq\frac{L_1}{2}I,
 \qquad
 \|b_s\|_2=B_0\varepsilon,
 \qquad
 16\rho\leq\|x_s^\star\|_2\leq B_0\rho.
\end{equation}
We first define the chain for fixed admissible phase data
$(A_s,b_s)_{s=1}^N$, then specify their joint distribution.

\subsubsection{The Delayed-information Chain}
Fix once and for all a nonincreasing function
$\eta\in C^\infty(\mathbb R;[0,1])$ such that $\begin{cases}
\eta(a)=1, & a\leq4,\\
0<\eta(a)<1, & 4<a<9,\\
\eta(a)=0, & a\geq9.
\end{cases}$\footnote{One
concrete choice is $\eta(a)=1$ for $a\leq4$,
$\eta(a)=\frac{\exp(-1/(9-a))}
{\exp(-1/(9-a))+\exp(-1/(a-4))}$ for $4<a<9$, and
$\eta(a)=0$ for $a\geq9$.}.  Define the
clipping profile from the same cutoff by $ \chi(r):=\eta(5r-1).$
Thus $\chi(r)=1$ for $r\leq1$ and $\chi(r)=0$ for $r\geq2$.  Set
\begin{equation}
\label{eq:drhs-clipping-constants}
 a_\chi:=1+2\|\chi'\|_\infty,
 \qquad
 C_R:=(a_\chi+1)B_0+4,
 \qquad
 R:=C_R\rho,
 \qquad
 H:=4B_0\varepsilon R.
\end{equation}
For a block variable $w\in\mathbb R^M$, define its smoothly
clipped version as $\pi_R(w):=\chi\!\left(\frac{\|w\|_2}{R}\right)w.$
For every phase $s$, its clipped linear forcing and its positive shifted
forcing are
\begin{equation}
\label{eq:drhs-forcing-definition}
 \psi_s(w):=\langle b_s,\pi_R(w)\rangle,
 \qquad
 \Phi_s(w):=H+\psi_s(w),
\end{equation}
and its gate is $ G_s(w):=
 \eta\!\left(\frac{\|w-x_s^\star\|_2^2}{\rho^2}\right)$ with target $x_s^\star:=A_s^{-1}b_s$. For $z=(z_1,\ldots,z_N)\in(\mathbb R^M)^N$, define a chaining function
\begin{equation}
\label{eq:drhs-hard-function}
 F_{\mathbf A,\mathbf b}(z)
 :=
 \sum_{s=1}^N\frac12z_s^\top A_sz_s
 -\Phi_1(z_1)
 -\sum_{s=2}^N
   G_{s-1}(z_{s-1})\Phi_s(z_s),
\end{equation}
where the final sum is empty when $N=1$.  Equivalently, the initiating
term is $-H-\psi_1(z_1)$.  


For $d\geq d_0$, write
$x=(z,u)\in\mathbb R^{NM}\times\mathbb R^{d-NM}$ and set
\begin{equation}
\label{eq:drhs-ambient-extension}
 F_{\mathbf A,\mathbf b}^{(d)}(z,u)
 :=F_{\mathbf A,\mathbf b}(z)+\frac\mu2\|u\|_2^2,
\end{equation}
with the second term absent when $d=NM$.
For the ambient extension, these are the first $N$ gradient blocks and
\(\nabla_uF_{\mathbf A,\mathbf b}^{(d)}(z,u)=\mu u\).

We first record the elementary geometry of each components that explains the construction (See Figure~\ref{fig:drhs-local-geometry}).   The clipping map $\pi_R$ is the identity on $B(0,R)$
and vanishes outside $B(0,2R)$.  Moreover,
\begin{equation}
\label{eq:drhs-clipping-separation}
 \|x_s^\star\|_2+3\rho
 \leq(B_0+3)\rho<R,
 \qquad
 C_R-a_\chi B_0=B_0+4>2.
\end{equation}
Thus clipping is inactive throughout every gate ball.  It changes none of
the local quadratic--linear geometry and serves only to control the forcing
globally.  Indeed, $\|\pi_R(w)\|_2\leq2R$, while direct differentiation on
the transition annulus gives
$\|D\pi_R(w)\|_{\mathrm{op}}\leq a_\chi$.  Hence
\begin{equation}
\label{eq:drhs-forcing-positive}
 2B_0\varepsilon R
 \leq\Phi_s(w)\leq
 6B_0\varepsilon R
 \qquad(w\in\mathbb R^M).
\end{equation}
In particular, every shifted forcing is strictly positive. The
gate $G_s$ equals one on the closed ball
$\overline B(x_s^\star,2\rho)$ and vanishes, together with all of its
derivatives, outside the open ball $B(x_s^\star,3\rho)$. 


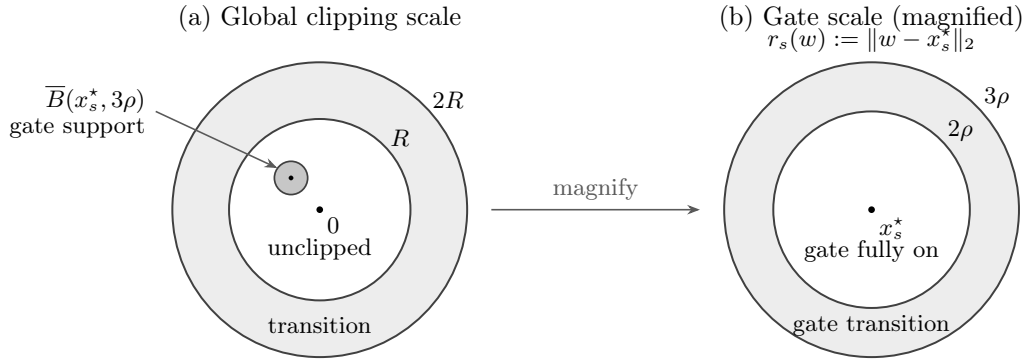
\begin{figure}[ht]
\centering
\begin{tikzpicture}[
  x=1cm,
  y=1cm,
  every node/.style={font=\footnotesize},
  inner boundary/.style={draw=black!75,line width=0.75pt},
  outer boundary/.style={draw=black!75,line width=0.75pt},
  leader/.style={-{Stealth[length=1.7mm]},black!65,line width=0.6pt},
  region name/.style={font=\footnotesize,align=center},
  key/.style={font=\footnotesize,align=left,inner sep=0pt},
]
  \coordinate (clip) at (-3.65,0.65);
  \coordinate (gate) at ( 3.65,0.65);

  \node[font=\small] at (-3.65,3.18)
    {\textup{(a) Global clipping scale}};
  \node[font=\small] at (3.65,3.18)
    {\textup{(b) Gate scale (magnified)}};
  \node at (3.65,2.89)
    {$r_s(w):=\|w-x_s^\star\|_2$};

  \fill[black!7] (clip) circle[radius=1.95];
  \filldraw[fill=white,inner boundary]
    (clip) circle[radius=1.20];
  \draw[outer boundary]
    (clip) circle[radius=1.95];

  \fill (clip) circle[radius=1.25pt];
  \node[anchor=north west,inner sep=1.5pt]
    at ($(clip)+(0.04,-0.04)$) {$0$};

  \node[region name]
    at ($(clip)+(0,-0.52)$) {unclipped};
  \node[region name]
    at ($(clip)+(0,-1.53)$) {transition};

  \node[anchor=south west,inner sep=1pt]
    at ($(clip)+(42:1.20)$) {$R$};
  \node[anchor=south west,inner sep=1pt]
    at ($(clip)+(42:1.95)$) {$2R$};

  \coordinate (smallgate)
    at ($(clip)+(-0.38,0.42)$);

  \filldraw[
    fill=black!22,
    draw=black!75,
    line width=0.65pt
  ]
    (smallgate) circle[radius=0.22];

  \fill (smallgate) circle[radius=0.9pt];

  \node[anchor=east,align=right]
    (smallgate-label) at (-5.82,1.95)
    {$\overline B(x_s^\star,3\rho)$\\[-1pt]
     \textup{gate support}};

  \draw[leader]
    (smallgate-label.east)
    --($(smallgate)+(-0.18,0.13)$);

  \draw[leader]
    (-1.38,0.65)--(1.38,0.65)
    node[
      midway,
      above=2pt,
      fill=white,
      inner sep=1.5pt
    ] {magnify};

  \fill[black!7] (gate) circle[radius=1.95];
  \filldraw[fill=white,inner boundary]
    (gate) circle[radius=1.30];
  \draw[outer boundary]
    (gate) circle[radius=1.95];

  \fill (gate) circle[radius=1.25pt];
  \node[anchor=north west,inner sep=1.5pt]
    at ($(gate)+(0.04,-0.04)$)
    {$x_s^\star$};

  \node[region name]
    at ($(gate)+(0,-0.55)$)
    {gate fully on};
  \node[region name]
    at ($(gate)+(0,-1.53)$)
    {gate transition};

  \node[anchor=south west,inner sep=1pt]
    at ($(gate)+(42:1.30)$) {$2\rho$};
  \node[anchor=south west,inner sep=1pt]
    at ($(gate)+(42:1.95)$) {$3\rho$};




\end{tikzpicture}
\caption{A two-dimensional section of the set geometry for one phase block. Panel~(a) shows the global clipping $\Phi_s$'s scale and the
gate support inside the no-clipping ball. Panel~(b) magnifies the gate $G_s$'s scale.}
\label{fig:drhs-local-geometry}
\end{figure}

\begin{remark}[\textbf{Information hiding}]
The ungated term $-\Phi_1(z_1)$ initiates the first phase.
For $s\geq2$, the coupling
$-G_{s-1}(z_{s-1})\Phi_s(z_s)$ activates the forcing in
phase $s$ only when the preceding block approaches its target.
The quadratic terms remain present throughout, allowing
earlier queries to collect information about $A_s$.

Before $z_{s-1}$ enters $B(x_{s-1}^\star,3\rho)$,
both $G_{s-1}$ and $\nabla G_{s-1}$ vanish.
Hence the incoming coupling containing $\Phi_s$ contributes
nothing to either the function value or the gradient.
The vector $b_s$ also determines the center of the outgoing
gate $G_s$ when $s<N$.
This second information channel likewise vanishes while
$z_s\notin B(x_s^\star,3\rho)$.

Thus, before the preceding target is approached, these
couplings can reveal information about $b_s$ only through
a premature hit on the current target neighborhood.
The information argument below shows that such hits are
unlikely and that prior exploration of $A_s$ cannot eliminate
the query cost after $b_s$ is revealed.
\end{remark}

Differentiating
Eq.~\eqref{eq:drhs-hard-function}, the $s$th block of the
gradient is
\begin{equation}
\label{eq:drhs-hard-gradient}
\begin{aligned}
 \nabla_sF_{\mathbf A,\mathbf b}(z)
 ={}&A_sz_s
 -\begin{cases}
    \nabla\Phi_1(z_1),&s=1,\\
    G_{s-1}(z_{s-1})\nabla\Phi_s(z_s),&s\geq2,
  \end{cases}\\
 &-\begin{cases}
    \Phi_{s+1}(z_{s+1})\nabla G_s(z_s),&s<N,\\
    0,&s=N.
  \end{cases}
\end{aligned}
\end{equation}

\begin{remark}[\textbf{Stationarity requires all phases}]
Let $s$ be the first phase with
$\|z_s-x_s^\star\|_2\geq2\rho$.
The incoming forcing in Eq.~\eqref{eq:drhs-hard-gradient}
is then fully active.
When $\|z_s\|_2\leq R$, its first two terms combine into
$A_s(z_s-x_s^\star)$.
The last term is a nonnegative multiple of
$z_s-x_s^\star$, since $\eta'\leq0$ and $\Phi_{s+1}>0$,
so it cannot cancel the outward radial component.
Hence
$\|\nabla_sF_{\mathbf A,\mathbf b}(z)\|_2
\geq\mu\|z_s-x_s^\star\|_2\geq2\varepsilon$.
When $\|z_s\|_2>R$, the last term vanishes and $A_sz_s$
dominates the bounded forcing by the choice of $R$.
Thus every $\varepsilon$-stationary point must approach
all targets within distance $2\rho$;
see Lemma~\ref{lem:drhs-certification}.
\end{remark}

\subsubsection{Spectral Construction of the Hard Distribution}
Every $\varepsilon$-stationary point must approach each target
$x_s^\star=A_s^{-1}b_s$, so each phase requires approximately
solving the linear system $A_sx=b_s$.
To make this task difficult, we choose the spectrum of $A_s$
so that low-degree polynomials cannot accurately approximate
the inverse map $\lambda\mapsto\frac{1}{\lambda}$.

Set \(K:=t+2\).  There are eigenvalues
\(\mu\leq\lambda_1<\cdots<\lambda_K\leq\frac{L_1}{2}\) and positive integers
\(n_1,\ldots,n_K\), summing to \(M\), such that
\begin{equation}
\label{eq:drhs-phase-spectrum-property}
 n_i\geq\frac{M}{2\kappa}\quad(i\in[K]),
 \qquad
 \sum_{i=1}^K\frac{n_i}{M}
 \left|\frac1{\lambda_i}-q(\lambda_i)\right|^2
 \geq\frac{3}{512\mu^2}
\end{equation}
for every real polynomial \(q\) of degree at most \(t\)
(Lemma~\ref{lem:drhs-spectrum}).
Fix once and for all one such spectral witness and define
\begin{equation}
\label{eq:drhs-phase-spectrum}
 D:=\operatorname{diag}
 \bigl(\lambda_1I_{n_1},\ldots,\lambda_KI_{n_K}\bigr).
\end{equation}
This choice depends only on the calibrated problem parameters and is made
before the algorithm is chosen.

Define the unconditioned one-phase law \(\mathsf P_0\) by drawing
\(U\) Haar-uniformly from \(O(M)\), drawing
\(\xi\) independently and uniformly from the unit sphere in
\(\mathbb R^M\), and setting
\begin{equation}
\label{eq:drhs-flat-base-law}
 A:=U^\top D U,
 \qquad
 b:=B_0\varepsilon\xi.
\end{equation}
Let
\begin{equation}
\label{eq:drhs-phase-conditioning-event}
 \mathcal G
 :=
 \left\{
   16\rho\leq\|A^{-1}b\|_2\leq B_0\rho
 \right\}.
\end{equation}
This event already has positive \(\mathsf P_0\)-probability.  Indeed,
putting \(q=0\) in
Eq.~\eqref{eq:drhs-phase-spectrum-property} gives
\[
 \sum_{i=1}^K\frac{n_i}{M}\lambda_i^{-2}
 \geq\frac{3}{512\mu^2}
 >\frac1{1024\mu^2}.
\]
Hence some \(i_0\) satisfies
\(\lambda_{i_0}^{-1}>1/(32\mu)\).  Conditional on any \(U\), every unit
vector \(\xi\) in a sufficiently small spherical neighborhood of the
corresponding eigenspace satisfies
\(\|A^{-1}b\|_2> B_0\varepsilon/(32\mu)=16\rho\), and that neighborhood
has positive surface measure.  On the other hand,
\(\|A^{-1}b\|_2\leq\|A^{-1}\|_{\mathrm{op}}\|b\|_2\leq B_0\rho\)
deterministically.  Thus \(\mathsf P_0(\mathcal G)>0\).

We may therefore define the algorithm-independent conditioned phase law
\begin{equation}
\label{eq:drhs-phase-law}
 \mathsf P_{\mathrm{ph}}(E)
 :=
 \mathsf P_0(E\mid\mathcal G)
 =
 \frac{\mathsf P_0(E\cap\mathcal G)}
      {\mathsf P_0(\mathcal G)}
 \qquad\text{for every Borel set }E.
\end{equation}
The stronger normalization bound
\(\mathsf P_0(\mathcal G)\geq7/8\), as well as the
information-theoretic hardness of this same fixed law, holds
(Lemma~\ref{lem:drhs-one-phase}).

Define the product phase law and then draw the phase data according to
\begin{equation}
\label{eq:drhs-chain-law}
 \mathsf P_{\mathrm{chain}}
 :=
 \mathsf P_{\mathrm{ph}}^{\otimes N},
 \qquad
 (A_s,b_s)_{s=1}^N
 \sim\mathsf P_{\mathrm{chain}}.
\end{equation}
Every realization in the support of
\(\mathsf P_{\mathrm{chain}}\) satisfies the admissibility conditions
Eq.~\eqref{eq:drhs-admissible-data}. Hence the deterministic chain above is
defined on the entire support of this product law.

For \(d\geq d_0\), let
\(\mathfrak F_d((A_s,b_s)_{s=1}^N):=
F_{\mathbf A,\mathbf b}^{(d)}\), and define the hard distribution on
functions by the pushforward $ \mathcal D_d
 :=
 (\mathfrak F_d)_\#
 \mathsf P_{\mathrm{chain}}.$
Thus the randomized hard instance used in the lower bound is
\(F\sim\mathcal D_d\).  The notation
\(F_{\mathbf A,\mathbf b}^{(d)}\) refers to one deterministic realization
of this distribution.  We suppress the dependence of \(\mathcal D_d\) on
\((p,\Delta,L_1,L_p,\varepsilon)\); as these parameters and \(d\) vary,
the resulting collection is the hard distribution family used in the
randomized lower bound.

\begin{remark}[{\bf Strengthened oracle}]
\label{remark:str_oracle}
For the lower-bound proof, we additionally reveal \(b_1\) before the first candidate is
selected and, for \(s\geq2\), reveal \(b_s\) after the first candidate
that outer-hits phase \(s-1\) has been frozen but before its
value-and-gradient response is returned.  This only strengthens the
oracle, so a lower bound under this disclosure protocol also applies to
the original model.
\end{remark}

\subsection{Key Properties and Proof of the Randomized Lower Bound}
We now present four properties of hard distribution that together prove
Theorem~\ref{thm:rand-fo-lower-bound}.
Before each statement, we will explain its role in the argument
and the main idea behind its proof.

\paragraph{Why advancing one disclosed phase is hard (Section~\ref{subsec:drhs-one-phase}).}
At the time \(b_s\) is disclosed, the algorithm may already have used its
previous first-order queries to collect arbitrary adaptive products with
\(A_s\).  We therefore strengthen the algorithm by allowing an arbitrary
exact
matrix-vector products stage before disclosure.  After disclosure,
outer-hitting phase \(s\) requires naming a vector within \(3\rho\) of
\(x_s^\star=A_s^{-1}b_s\).  The replicated hard spectrum prevents
low-degree polynomial approximation of
\(\lambda\mapsto\lambda^{-1}\), while the delayed-input Krylov
representation shows that the component of every vector produced with
fewer than \(t\) post-disclosure products is governed by such a low-degree
polynomial in the fresh direction.  The argument permits an adaptive
disclosure time and includes candidates named before disclosure, the
candidate frozen at disclosure, and a terminal output that is never queried.

To state the result, suppress the phase index and write
$(A,b)\sim\mathsf P_{\mathrm{ph}}$, with target
$x^\star:=A^{-1}b\in\mathbb R^M$.
Recall that $M$ is the phase dimension,
$\rho=\frac{\varepsilon}{\mu}$, and
$\kappa=\frac{L_1}{2\mu}$.
Let $c_0,\kappa_0$ be the constants from
Lemma~\ref{lem:drhs-spectrum}, and assume
$\kappa\geq\kappa_0$ and
$t:=\lfloor c_0\sqrt{\kappa}\rfloor\geq1$.

Consider a deterministic procedure that accesses $A$
through exact queries $v\mapsto Av$.
Let the integer $Q\geq t$ bound the number of queries
before $b$ is revealed. Afterward, the procedure may
make fewer than $t$ further queries.
The disclosure time may be chosen adaptively using only
the pre-disclosure transcript.

A candidate is a vector $y\in\mathbb R^M$ recorded by
the procedure as a possible approximation to $x^\star$.
It may be a query vector or an unqueried output.
Let the integer $\mathsf L\geq1$ bound the total number
of candidates, and let $\mathcal Y$ denote their set.
Candidates may be recorded at any stage, and every query
and candidate must be a Borel-measurable function of the
information available when selected.
Fix a probability tolerance $\delta\in(0,\frac{1}{4})$.
The following lemma bounds the probability that any
recorded candidate approaches the target.

\begin{lemma}[\textbf{One-Phase Hardness with Delayed Disclosure}]
\label{lem:drhs-one-phase}
There is a numerical constant $C_0>0$ such that, if the dimension of each phase block
\begin{equation}
\label{eq:drhs-one-phase-dimension}
 M\geq C_0\left(
 Q^2+\kappa\left[
 Q+\log\frac{(Q+2)\mathsf L}{\delta}
 \right]\right),
\end{equation}
every procedure described above satisfies
\begin{equation}
\label{eq:drhs-one-phase-conclusion}
 \Pr_{(A,b)\sim\mathsf P_{\mathrm{ph}}}
 \left(
 \exists\,y\in\mathcal Y:
 \|y-x^\star\|_2<3\rho
 \right)\leq\delta.
\end{equation}
Moreover, $\mathsf P_0(\mathcal G)\geq\frac{7}{8}$,
and $\mathsf P_{\mathrm{ph}}$ is well defined and supported
on admissible pairs satisfying
Eq.~\eqref{eq:drhs-admissible-data}.
\end{lemma}

\paragraph{Why the right-hand sides are disclosed sequentially (Section~\ref{subsec:proof-drhs-sequential}).}
We analyze the chain through a strengthened oracle (Remark~\ref{remark:str_oracle}).  It reveals \(b_1\)
before the first candidate is selected.  For each \(s\geq2\), it reveals
\(b_s\) only after the algorithm has selected and frozen the first candidate
that outer-hits phase \(s-1\), and before returning the value and gradient
at that candidate.  Thus the response at the disclosure candidate may
depend on \(b_s\), but the candidate itself, including its \(s\)th block,
is chosen without access to \(b_s\).  If one candidate triggers several
consecutive disclosures, all of them are made in increasing phase order
after that same candidate has been frozen.

The exact dead zones make this protocol compatible with the complete
value-and-gradient transcript.  Before \(b_s\) is disclosed, the
predecessor gate and all of its derivatives vanish, so the entire incoming
link containing \(b_s\) contributes zero to both the function value and
the gradient.  Before phase \(s\) outer-hits, the outgoing gate \(G_s\)
and all of its derivatives also vanish, hiding the dependence of its
center \(x_s^\star=A_s^{-1}b_s\) on \(b_s\).  Consequently, after fixing
the other phases, the execution up to the selection of the first
phase-\(s\) outer-hitting candidate can be simulated using only products
with \(A_s\) before \(b_s\) is disclosed.  At the first such hit, the
response may contain new information about \(b_s\), but the successful
candidate has already been frozen and named.

Applying Lemma~\ref{lem:drhs-one-phase} to each phase with
$Q=\mathsf L=T$ and $\delta=\frac{1}{12N}$, followed by
a union bound over the $N$ phases, yields the following
sequential accumulation bound.

\begin{lemma}[\bf Sequential accumulation]
\label{lem:drhs-sequential}
Let
\((\mathbf A,\mathbf b)\sim
\mathsf P_{\mathrm{chain}}=\mathsf P_{\mathrm{ph}}^{\otimes N}\),
and let \(C_0\) be the constant furnished by
Lemma~\ref{lem:drhs-one-phase}.  Set $ Q=T,
 \mathsf L=T,
 \delta=\frac1{12N}.$
If 
\[ M\geq C_0\left\{
   T^2+\kappa\left[
     T+\log\bigl(12NT(T+2)\bigr)
   \right]
 \right\},\]
then every deterministic exact value-and-gradient algorithm making at most
\(T-1\) oracle calls and returning a
point
\(z^{\mathrm{out}}\) satisfies
\begin{equation}
\label{eq:drhs-sequential-conclusion}
 \Pr_{(\mathbf A,\mathbf b)\sim\mathsf P_{\mathrm{chain}}}
 \left(
   \|z_s^{\mathrm{out}}-x_s^\star\|_2<3\rho
   \text{ for every }s\in[N]
 \right)
 \leq\frac1{12}.
\end{equation}
The output need not have been queried.  The same assertion holds for the
ambient extension \(F_{\mathbf A,\mathbf b}^{(d)}\).
\end{lemma}

\paragraph{Why stationarity requires full chain completion~(Section~\ref{subsec:drhs-certification-proof}).}
It remains to connect phase completion to approximate stationarity.  
%
%
Suppose that a point fails to inner-complete at
least one phase, and let \(s\) be the smallest index such that
\(\|z_s-x_s^\star\|_2\geq2\rho\).  If \(s=1\), its incoming forcing is
ungated.  If \(s\geq2\), the minimality of \(s\) gives
\(\|z_{s-1}-x_{s-1}^\star\|_2<2\rho\), and hence
\(G_{s-1}(z_{s-1})=1\).  Thus phase \(s\) receives its full incoming
forcing in either case. 

\noindent \textbf{Case 1 (\(\|z_s\|_2\leq R\)):} Set \(e_s:=z_s-x_s^\star\). 
 Clipping is inactive, so
\(\nabla\Phi_s(z_s)=b_s=A_sx_s^\star\).  Therefore
\[
 \nabla_sF_{\mathbf A,\mathbf b}(z)
 =
 A_se_s
 -\mathbf 1_{\{s<N\}}
  \Phi_{s+1}(z_{s+1})\nabla G_s(z_s).
\]
Since
\(\nabla G_s(z_s)
 =2\rho^{-2}\eta'(\|e_s\|_2^2/\rho^2)e_s\),
the inequalities \(\eta'\leq0\) and \(\Phi_{s+1}>0\) imply
\[
 \bigl\langle e_s,\nabla_sF_{\mathbf A,\mathbf b}(z)\bigr\rangle
 \geq e_s^\top A_se_s
 \geq\mu\|e_s\|_2^2.
\]
Cauchy--Schwarz and \(\|e_s\|_2\geq2\rho\) then give
\[
 \|\nabla_sF_{\mathbf A,\mathbf b}(z)\|_2
 \geq\mu\|e_s\|_2
 \geq2\mu\rho
 =2\varepsilon.
\]

\noindent \textbf{Case 1 (\(\|z_s\|_2>R\)):} we have
\[
 \|z_s-x_s^\star\|_2
 \geq\|z_s\|_2-\|x_s^\star\|_2
 >R-B_0\rho
 >3\rho.
\]
The outgoing gate is therefore in its dead zone, so
\(\nabla G_s(z_s)=0\).  Because the incoming forcing is fully active,
\(\nabla_sF_{\mathbf A,\mathbf b}(z)
 =A_sz_s-\nabla\Phi_s(z_s)\).  Using
\(\|\nabla\Phi_s(z_s)\|_2\leq a_\chi B_0\varepsilon\) and
\(\mu R=C_R\varepsilon\), we obtain
\[
 \|\nabla_sF_{\mathbf A,\mathbf b}(z)\|_2
 \geq\mu\|z_s\|_2-a_\chi B_0\varepsilon
 >
 (C_R-a_\chi B_0)\varepsilon
 =(B_0+4)\varepsilon
 >\varepsilon.
\]
Thus failure to inner-complete even one phase forces one hard-coordinate
gradient block to have norm greater than \(\varepsilon\).  Finally,
\(\nabla F_{\mathbf A,\mathbf b}^{(d)}(z,u)
 =(\nabla F_{\mathbf A,\mathbf b}(z),\mu u)\), so the ambient coordinates
cannot decrease the gradient norm.  Taking the contrapositive proves the
following lemma.

\begin{lemma}[\bf Stationarity certification]
\label{lem:drhs-certification}
For every \(M,N\geq1\), every admissible phase realization, every
\(d\geq NM\), and every \((z,u)\in\mathbb R^d\),
\begin{equation}
\label{eq:drhs-certification-stationarity}
 \bigl\|\nabla F_{\mathbf A,\mathbf b}^{(d)}(z,u)\bigr\|_2
 \leq\varepsilon
 \quad\Longrightarrow\quad
 \|z_s-x_s^\star\|_2<2\rho
 \quad\text{for every }s\in[N].
\end{equation}
In particular, every \(\varepsilon\)-stationary point lies in the inner,
and hence also the outer, ball of every phase.
\end{lemma}

\paragraph{Why the chain is regular~(Section~\ref{subsec:drhs-regularity-proof}).}
The local structure of the chain also gives the required smoothness and
initial-gap bounds.  For every \(j\geq1\), a \(j\)th derivative of a gate is
of order \(O_j(\rho^{-j})\), while a \(j\)th derivative of a shifted forcing
term is of order \(O_j(\varepsilon\rho^{1-j})\).  The Leibniz rule therefore
shows that every \(r\)th derivative of a coupling term
\(G_s\Phi_{s+1}\) is of order
\(O_r(\varepsilon\rho^{1-r})\).  Each coupling term involves only two
adjacent blocks, so a blockwise generalized H\"older inequality combines
these local estimates with a constant independent of \(N\).  Finally, the
uniformly positive definite quadratic terms, together with the bounded
forcing terms, make the full function coercive.  The same bounds show that
the initial gap is of order \(N\varepsilon^2/\mu\).  The following lemma
records these properties precisely.

\begin{lemma}[\bf Regularity and initial gap]
\label{lem:drhs-regularity}
There exist constants \(C_{\mathrm{sm}},C_{\mathrm{gap}}<\infty\),
and
a constant \(C_q<\infty\) for each integer \(q\geq2\)
with the following property.  For every \(M,N\geq1\), every admissible phase
realization, and every \(d\geq NM\), the function
\(F_{\mathbf A,\mathbf b}^{(d)}\) is \(C^\infty\) and coercive, and
\[
\operatorname{Lip}\bigl(\nabla F_{\mathbf A,\mathbf b}^{(d)}\bigr)\leq\frac{L_1}{2}+C_{\mathrm{sm}}\mu,\qquad \operatorname{Lip}\bigl(\nabla^qF_{\mathbf A,\mathbf b}^{(d)}\bigr)\leq C_q\frac{\mu^q}{\varepsilon^{q-1}},\qquad F_{\mathbf A,\mathbf b}^{(d)}(0)-\inf_{x\in\mathbb R^d}F_{\mathbf A,\mathbf b}^{(d)}(x)\leq C_{\mathrm{gap}}N\frac{\varepsilon^2}{\mu}.
\]
These constants are independent of
\(M,N,d,\mu,\frac{L_1}{2},\varepsilon\), and the phase realization.  For the fixed
order \(p\) in the theorem, we denote the corresponding constant in the
second inequality by \(C_p\).
\end{lemma}

These four lemmas now give the main theorem with no further structural
input.  

\begin{proof}[Proof of Theorem~\ref{thm:rand-fo-lower-bound}]
Choose \(c_p>0\) so that $ c_p\leq
 \min\left\{
   c_{\mathrm{acc},p},
   \frac{c_Nc_0\sqrt{c_{\mu,p}}}{4\sqrt2}
 \right\}.$
Then the small-accuracy assumption of the theorem implies
Eq.~\eqref{eq:drhs-calibrated-small-accuracy}, and hence the calibration above
gives \(N,t\geq1\).

We first record the routine consequences of the calibration.  Since
\(\gamma_p(\varepsilon)
 =\Theta_{p,\Delta,L_1,L_p}
 (\varepsilon^{(p-1)/p})\), the definitions of
\(N,\kappa,t,T\), and \(M\) give
\[
 N=O\bigl(\varepsilon^{-(1+1/p)}\bigr),
 \qquad
 T=O\bigl(\varepsilon^{-(3/2+1/(2p))}\bigr),
 \qquad
 M=O\bigl(\varepsilon^{-(3+1/p)}\bigr).
\]
Consequently,
\[
 d_0=NM
 =O_{p,\Delta,L_1,L_p}
   \bigl(\varepsilon^{-(4+2/p)}\bigr).
\]
Thus \(C_{p,\Delta,L_1,L_p}\) is chosen so that \(d\geq d_0\).

Now use the algorithm-independent hard distribution
\(\mathcal D_d\) .
Because \(C_{\mathrm{dim}}\geq C_0\), the choice of \(M\) in
Eq.~\eqref{eq:drhs-calibrated-dimension} satisfies the dimension hypothesis
of Lemma~\ref{lem:drhs-sequential}.  Lemma~\ref{lem:drhs-one-phase}
ensures that the phase law is supported on admissible pairs.  Moreover,
Lemma~\ref{lem:drhs-regularity}, together with
Eq.~\eqref{eq:drhs-constant-choice} and the definition of \(N\), immediately
gives
\[
 \operatorname{Lip}(\nabla F)\leq L_1,
 \qquad
 \operatorname{Lip}(D^pF)\leq L_p,
 \qquad
 F(0)-\inf F\leq\Delta.
\]
Every realization is also \(C^\infty\) and coercive.  Hence
\(\mathcal D_d\) is supported on
\(\mathcal F_{1,p}^{(d)}(\Delta,L_1,L_p)\).

Fix a deterministic value-and-gradient algorithm making at most \(T-1\)
queries, and let
\(x^{\mathrm{out}}=(z^{\mathrm{out}},u^{\mathrm{out}})\) be its output.
Lemma~\ref{lem:drhs-certification} gives
\[
 \left\{
   \|\nabla F(x^{\mathrm{out}})\|_2\leq\varepsilon
 \right\}
 \subseteq
 \bigcap_{s=1}^N
 \left\{
   \|z_s^{\mathrm{out}}-x_s^\star\|_2<2\rho
 \right\}
 \subseteq
 \bigcap_{s=1}^N
 \left\{
   \|z_s^{\mathrm{out}}-x_s^\star\|_2<3\rho
 \right\}.
\]
Lemma~\ref{lem:drhs-sequential} and the pushforward definition of
\(\mathcal D_d\) therefore imply
\begin{equation}
\label{eq:drhs-distributional-conclusion}
 \Pr_{F\sim\mathcal D_d}
 \left(
   \|\nabla F(x^{\mathrm{out}})\|_2\leq\varepsilon
 \right)
 \leq\frac1{12}.
\end{equation}

It remains to calculate the query rate.  Both quantities inside the floors
in Eq.~\eqref{eq:drhs-calibrated-depth} are at least two, so
\(\lfloor a\rfloor\geq a/2\) gives
\begin{align}
 T=Nt
 &\geq
 \frac{c_Nc_0}{4}
 \frac{\Delta\mu}{\varepsilon^2}
 \sqrt{\frac{L_1}{2\mu}}=
 \frac{c_Nc_0\sqrt{c_{\mu,p}}}{4\sqrt2}
 \Delta\sqrt{L_1\gamma_p(\varepsilon)}\,
 \varepsilon^{-2}.
\label{eq:drhs-query-rate}
\end{align}

Finally, consider a randomized algorithm with hard query budget \(T-1\).
For every fixed realization \(\omega\) of its random seed, the resulting
algorithm is deterministic, so Eq.~\eqref{eq:drhs-distributional-conclusion}
applies.  Averaging over \(\omega\) gives
\[
 \mathbb E_{F\sim\mathcal D_d}
 \Pr_\omega\left(
   \|\nabla F(x^{\mathrm{out}}(F,\omega))\|_2
   \leq\varepsilon
 \right)
 \leq\frac1{12}.
\]
Therefore, for this randomized algorithm, there is a fixed
\(f\in\mathcal F_{1,p}^{(d)}(\Delta,L_1,L_p)\) such that
\[
 \Pr_\omega\left(
   \|\nabla f(x^{\mathrm{out}}(f,\omega))\|_2
   \leq\varepsilon
 \right)
 \leq\frac1{12}<\frac23.
\]
Thus no randomized algorithm with budget \(T-1\) satisfies the uniform
\(2/3\)-success requirement, and hence
\[
 T_{\varepsilon,\mathrm{FO}}^{\mathrm{rand}}
 \bigl(\mathcal F_{1,p}^{(d)}(\Delta,L_1,L_p)\bigr)
 \geq T.
\]
Combining this with Eq.~\eqref{eq:drhs-query-rate} and the choice of \(c_p\)
proves Eq.~\eqref{eq:drhs-main-rate}.
\end{proof}

\section{Conclusions and Future Works}

We studied how higher-order smoothness can be exploited using only
function values or only gradients. In fixed dimension, our trapping
algorithms extend to H\"older smoothness, and their randomized
gradient-only implementation establishes a polynomial separation from
deterministic algorithms for $C^{1,1}$ objectives on $\mathbb R^2$.
For general dimension, deterministic trapping-or-descent algorithms
combine local boundary certificates with model reuse to attain the
accuracy exponent $\nu/(\nu-1)$ under $C^{p,\beta_p}$ smoothness,
where $p\ge2$ and $\nu=p+\beta_p$. This exponent is independent of
dimension and coincides with the classical ARp exponent when
$\beta_p=1$. Our randomized first-order lower bounds complement these
algorithms by establishing limitations in sufficiently high dimension.
Together, these results show how dimension, oracle information, and
randomization shape the benefits of higher-order smoothness. There are several interesting directions for future exploration and we list some as follows.
\paragraph{Optimal complexity in fixed dimension.}
The optimal randomized gradient-only complexity remains open even
for $C^{1,1}$ objectives on $\mathbb R^2$. Is our
$\widetilde{\mathcal O}(\varepsilon^{-4/3})$ bound optimal, or can it
be improved toward the $\mathcal O(\varepsilon^{-1})$ deterministic
function-value benchmark? More generally, it remains to determine the
optimal gradient-only rates under H\"older smoothness, both in the
unconstrained setting and on the cube. Such results would clarify the
extent to which randomization can compensate for the absence of
function values.

\paragraph{Joint dependence on dimension and accuracy.}
Our general-dimensional upper bounds have polynomial dependence on
$d$, whereas our randomized lower bounds require $d$ to grow
polynomially with $1/\varepsilon$. Determining the optimal joint
dependence on these parameters would connect the fixed-dimensional
and high-dimensional regimes. Concrete goals include improving the
dimension factors in the trapping-or-descent bounds and reducing the
dimension required by the lower-bound construction. These improvements
would help identify the dimensional regimes in which randomization
improves gradient-only complexity.


\section*{AI Disclosure}
AI tools were used to assist with literature searches, the development and verification of proof arguments, the presentation of proofs, the polishing of the exposition, and typesetting. The authors independently verified all mathematical arguments and take full responsibility for the content of the paper.

\section*{Acknowledgment}
We thank Andi Han and Chengchang Liu for helpful discussions.
TS was partially supported by JST CREST (PMJCR2015) and JSPS KAKENHI (24K02905).
This research is supported by the National Research Foundation, Singapore and the Ministry of Digital Development and Information under the AI Visiting Professorship Programme (award number AIVP-2024-004). Any opinions, findings and conclusions or recommendations expressed in this material are those of the author(s) and do not reflect the views of National Research Foundation, Singapore and the Ministry of Digital Development and Information.

\newpage
\bibliography{references}
\bibliographystyle{alpha}

\newpage

\appendix

\section{Approximation and Global Minimization}
\label{app:approximation}

We study function reconstruction from finite data
(Problem~\ref{prob:apx-function-reconstruction})
and approximate global minimization with a known additive term
(Problem~\ref{prob:apx-global-minimization}).
We establish reconstruction guarantees for exact and inexact
data, then derive the minimization guarantee from stable
reconstruction.

\paragraph{Notation.}
Let $S\subset\mathbb R^d$ be a known nonempty compact convex
set with nonempty interior.
For a scalar- or vector-valued map $F$, write $D^jF$ for its
$j$th Fr\'echet derivative, with $D^0F:=F$.
For scalar $f$, we identify $D^jf$ with the derivative tensor
$\nabla^jf$ and in particular,
$Df(\boldsymbol x)[\boldsymbol v]
=\langle\nabla f(\boldsymbol x),\boldsymbol v\rangle$.
Vector norms are Euclidean, derivative norms are the induced
multilinear operator norms, and
$\|D^jF\|_{\infty,S}
:=\sup\limits_{\boldsymbol y\in S}\|D^jF(\boldsymbol y)\|$.

For an integer $p\ge0$, $\beta\in(0,1]$, and $L_p\ge0$,
we write $F\in C^{p,\beta}(L_p)$ if $F$ is $p$ times
continuously differentiable and satisfies
$\|D^pF(\boldsymbol x)-D^pF(\boldsymbol y)\|
\le L_p\|\boldsymbol x-\boldsymbol y\|_2^\beta$.
All regularity assumptions hold on a neighborhood of $S$.

Let $\Pi_p^d$ denote the space of real polynomials of total
degree at most $p$ in $d$ variables.
For integers $j\ge0$, define the Taylor constants
$a_{j,\beta}:=\prod\limits_{i=1}^j\frac{1}{\beta+i}$,
with $a_{0,\beta}=1$.

\paragraph{Data.}
Let $f\in C^{p,\beta}(L_p)$ be the unknown scalar objective,
and let $V\subset S$ be a finite set of sampling nodes.
Function-value data are supplied as
$\widetilde f:V\to\mathbb R$ with a known error bound
$\tau\ge0$ satisfying
\[
 \max_{\boldsymbol v\in V}
 |\widetilde f(\boldsymbol v)-f(\boldsymbol v)|
 \le\tau.
\]
Exact function values correspond to $\tau=0$.
We also consider exact gradient data
$\{\nabla f(\boldsymbol v):\boldsymbol v\in V\}$.
Each sampled function value or gradient vector counts as
one oracle query.

\begin{problem}[\textbf{Function Reconstruction}]
\label{prob:apx-function-reconstruction}
Given $S$, the smoothness parameters $p,\beta,L_p$,
function-value or gradient data on $V$, and target
accuracies $\varepsilon_0,\varepsilon_1>0$, construct
an explicit model $\widehat f$ satisfying
\[
 \|\widehat f-f\|_{\infty,S}\le\varepsilon_0,
 \qquad
 \|\nabla\widehat f-\nabla f\|_{\infty,S}
 \le\varepsilon_1,
\]
where the second guarantee is required only when
derivative accuracy is requested.
For gradient-only data, a reference point
$\boldsymbol c\in S$ is specified, and the value guarantee
is instead
$\|\widehat f-f+f(\boldsymbol c)\|_{\infty,S}
\le\varepsilon_0$.
\end{problem}

\begin{problem}[\textbf{Approximate Global Minimization from Data}]
\label{prob:apx-global-minimization}
Given $S$, the smoothness parameters $p,\beta,L_p$,
function-value data $(V,\widetilde f)$ with error bound
$\tau$, a known continuous function $q:S\to\mathbb R$,
and a target accuracy $\delta>0$, compute
$\widehat{\boldsymbol y}\in S$ satisfying
\[
 (f+q)(\widehat{\boldsymbol y})
 -\min_{\boldsymbol y\in S}(f+q)(\boldsymbol y)
 \le\delta.
\]
The output must be computed from the supplied data
and known inputs, without further queries to $f$.
\end{problem}

For both problems, we specify sufficient sampling and
data-accuracy conditions and bound the required number
of nodes $|V|$.
All sampling nodes are chosen nonadaptively, so the required
oracle queries can be issued in a single batch.
Subsequent computations with explicit models and known
functions are not charged.


\subsection{Useful Approximation Estimates}
\label{app:common-approximation-estimates}

The Taylor remainder provides a comparison polynomial with
controlled error. The interpolation estimate then transfers
this error, together with errors in the supplied values, to
the polynomial constructed from the data.
Both estimates apply to vector-valued maps without additional
dependence on the output dimension.

\begin{lemma}[\textbf{H\"older Taylor Remainder}]
\label{lem:holder-taylor-remainder}
Let $p\in\mathbb Z_{\ge0}$, let $S\subset\mathbb R^d$ be convex,
and let $F\in C^{p,\beta}(L_p)$ on a neighborhood of $S$.
For $\boldsymbol c\in S$, let $T_p$ be the degree-$p$
Taylor polynomial of $F$ at $\boldsymbol c$.
Then
\begin{equation}
\label{eq:apx-holder-taylor}
 \|D^j(F-T_p)(\boldsymbol c+\boldsymbol v)\|
 \le a_{p-j,\beta}L_p\|\boldsymbol v\|_2^{p-j+\beta},
 \qquad
 0\le j\le p,\quad \boldsymbol c+\boldsymbol v\in S.
\end{equation}
\end{lemma}

\begin{proof}
Fix $0\le j\le p$ and put $k:=p-j$.
If $k=0$, then $D^pT_p=D^pF(\boldsymbol c)$, so the
claim follows directly from the H\"older condition on $D^pF$.

Suppose that $k\ge1$, and set
$\Delta_p(t):=D^pF(\boldsymbol c+t\boldsymbol v)
-D^pF(\boldsymbol c)$.
Convexity ensures that the segment
$[\boldsymbol c,\boldsymbol c+\boldsymbol v]$ lies in $S$.
Applying Taylor's integral formula to $D^jF$ and subtracting
$D^jT_p$ gives
\[
 D^j(F-T_p)(\boldsymbol c+\boldsymbol v)
 =\frac{1}{(k-1)!}\int\limits_0^1(1-t)^{k-1}
 \Delta_p(t)[\boldsymbol v^{\otimes k},\,\cdot^{\,j}]\,dt.
\]
Here the last $j$ arguments remain free.
The H\"older condition gives
$\|\Delta_p(t)\|\le L_pt^\beta\|\boldsymbol v\|_2^\beta$.
Consequently,
\[
\begin{aligned}
 \|D^j(F-T_p)(\boldsymbol c+\boldsymbol v)\|
 &\le
 \frac{1}{(k-1)!}
 \int\limits_0^1(1-t)^{k-1}
 \|\Delta_p(t)\|\,\|\boldsymbol v\|_2^k\,dt\\
 &\le
 \frac{L_p\|\boldsymbol v\|_2^{k+\beta}}{(k-1)!}
 \int\limits_0^1(1-t)^{k-1}t^\beta\,dt\\
 &=a_{k,\beta}L_p\|\boldsymbol v\|_2^{k+\beta}.
\end{aligned}
\]
For the last equality, repeated integration by parts gives
$\frac{1}{(k-1)!}
\int\limits_0^1(1-t)^{k-1}t^\beta\,dt
=\prod\limits_{i=1}^k\frac{1}{\beta+i}=a_{k,\beta}$.
Substituting $k=p-j$ proves the claim.
\end{proof}

\paragraph{Interpolation from inexact data.}
Let $C\subset\mathbb R^d$ be nonempty, compact, and convex,
and let $\mathcal P$ be an $n$-dimensional real polynomial
space, where $n\ge1$.
Choose nodes $\boldsymbol z_1,\ldots,\boldsymbol z_n\in C$
that are unisolvent for $\mathcal P$: every assignment of
scalar values at these nodes determines a unique polynomial
in $\mathcal P$.
The associated Lagrange basis
$\ell_1,\ldots,\ell_n\in\mathcal P$ satisfies
$\ell_i(\boldsymbol z_k)=\mathbf1_{\{i=k\}}$.

Let $F$ be a continuous scalar- or vector-valued map on $C$.
Suppose the supplied values $\widetilde F_i$ satisfy
$\|\widetilde F_i-F(\boldsymbol z_i)\|\le\tau$, where $\tau\ge0$.
Their componentwise interpolant is
$\mathcal I(\boldsymbol y)
:=\sum\limits_{i=1}^n\widetilde F_i\ell_i(\boldsymbol y)$.
Define the value and derivative stability constants by
\[
 \Lambda_0:=\sup\limits_{\boldsymbol y\in C}
 \sum\limits_{i=1}^n|\ell_i(\boldsymbol y)|,
 \qquad
 \Lambda_1:=\sup\limits_{\boldsymbol y\in C}
 \sum\limits_{i=1}^n\|\nabla\ell_i(\boldsymbol y)\|_2.
\]
These constants control how errors in the supplied values
affect the interpolant and its derivative.
For a comparison polynomial map $T$ whose components belong
to $\mathcal P$, write
$r_T:=\max\limits_{1\le i\le n}
\|F(\boldsymbol z_i)-T(\boldsymbol z_i)\|$
for its maximum nodal error.

\begin{lemma}[\textbf{Interpolation Error and Stability}]
\label{lem:interpolation-error-decomposition}
Under the preceding setup, every such comparison polynomial
$T$ satisfies
\begin{equation}
\label{eq:apx-common-interpolation-error}
\begin{aligned}
 \|\mathcal I-F\|_{\infty,C}
 &\le\|T-F\|_{\infty,C}+\Lambda_0(r_T+\tau),\\
 \|D\mathcal I-DF\|_{\infty,C}
 &\le\|DT-DF\|_{\infty,C}+\Lambda_1(r_T+\tau),
\end{aligned}
\end{equation}
where the second inequality additionally requires $F$
to be continuously differentiable on a neighborhood of $C$.

If $\Pi_p^d\subseteq\mathcal P$ and
$F\in C^{p,\beta}(L_p)$ on a neighborhood of $C$, then
\begin{equation}
\label{eq:apx-stable-local-error}
 \|\mathcal I-F\|_{\infty,C}
 \le(1+\Lambda_0)a_{p,\beta}L_p
 \operatorname{diam}(C)^{p+\beta}+\Lambda_0\tau.
\end{equation}
This includes $p=0$.
\end{lemma}

\begin{proof}
Fix $T$ and set
$\boldsymbol e_i:=\widetilde F_i-T(\boldsymbol z_i)$.
By the triangle inequality,
$\|\boldsymbol e_i\|
\le\|\widetilde F_i-F(\boldsymbol z_i)\|
+\|F(\boldsymbol z_i)-T(\boldsymbol z_i)\|
\le\tau+r_T$.
Since interpolation reproduces each component of $T$, we have
$T=\sum\limits_{i=1}^nT(\boldsymbol z_i)\ell_i$.
Therefore,
\[
 \mathcal I-F
 =T-F+\sum\limits_{i=1}^n\boldsymbol e_i\ell_i.
\]
For every $\boldsymbol y\in C$, the second term has norm at most
$\sum\limits_{i=1}^n
\|\boldsymbol e_i\|\,|\ell_i(\boldsymbol y)|
\le(r_T+\tau)\Lambda_0$.
Taking the supremum over $\boldsymbol y\in C$ proves the
value bound.

For the derivative bound, differentiate the same identity.
Each $\boldsymbol e_i$ is constant, so
$D(\boldsymbol e_i\ell_i)
=\boldsymbol e_i(\nabla\ell_i)^{\mathsf T}$.
The induced operator norm of this rank-one map is
$\|\boldsymbol e_i\|\,\|\nabla\ell_i\|_2$.
Thus, for every $\boldsymbol y\in C$,
\[
\begin{aligned}
 \|D\mathcal I(\boldsymbol y)-DF(\boldsymbol y)\|
 &\le
 \|DT(\boldsymbol y)-DF(\boldsymbol y)\|
 +\sum\limits_{i=1}^n
 \|\boldsymbol e_i\|\,\|\nabla\ell_i(\boldsymbol y)\|_2\\
 &\le
 \|DT-DF\|_{\infty,C}+(r_T+\tau)\Lambda_1.
\end{aligned}
\]
Taking the supremum proves the derivative bound.

For \eqref{eq:apx-stable-local-error}, choose $T$ to be the
degree-$p$ Taylor polynomial of $F$ at any $\boldsymbol c\in C$.
Its components belong to $\mathcal P$ because
$\Pi_p^d\subseteq\mathcal P$.
Lemma~\ref{lem:holder-taylor-remainder} and
$\|\boldsymbol y-\boldsymbol c\|_2\le\operatorname{diam}(C)$ give
$r_T\le\|F-T\|_{\infty,C}
\le a_{p,\beta}L_p\operatorname{diam}(C)^{p+\beta}$.
Substituting both bounds into the value estimate yields
\eqref{eq:apx-stable-local-error}.
For $p=0$, take $T=F(\boldsymbol c)$ and use the
H\"older continuity of $F$.
\end{proof}

\paragraph{Errors in function increments.}
We also record how value and derivative errors control
increments.
If $E$ is continuously differentiable on a neighborhood
of a convex set $S$, then
\begin{equation}
\label{eq:apx-increment-control}
 \|E(\boldsymbol y)-E(\boldsymbol x)\|
 \le\min\left\{
 2\|E\|_{\infty,S},\,
 \|DE\|_{\infty,S}\|\boldsymbol y-\boldsymbol x\|_2
 \right\},
 \qquad \boldsymbol x,\boldsymbol y\in S.
\end{equation}
The first bound follows from
$\|E(\boldsymbol y)-E(\boldsymbol x)\|
\le\|E(\boldsymbol y)\|+\|E(\boldsymbol x)\|
\le2\|E\|_{\infty,S}$.
For the second, convexity keeps the segment
$[\boldsymbol x,\boldsymbol y]$ in $S$, and the fundamental
theorem of calculus gives
\[
\begin{aligned}
 \|E(\boldsymbol y)-E(\boldsymbol x)\|
 &=
 \left\|
 \int\limits_0^1
 DE(\boldsymbol x+t(\boldsymbol y-\boldsymbol x))
 (\boldsymbol y-\boldsymbol x)\,dt
 \right\|\\
 &\le
 \int\limits_0^1
 \|DE(\boldsymbol x+t(\boldsymbol y-\boldsymbol x))\|
 \|\boldsymbol y-\boldsymbol x\|_2\,dt\\
 &\le
 \|DE\|_{\infty,S}\|\boldsymbol y-\boldsymbol x\|_2.
\end{aligned}
\]
If $E(\boldsymbol c)=0$ and $\boldsymbol x=\boldsymbol c$,
the first bound improves to $\|E\|_{\infty,S}$.


\subsection{Reconstruction from Sampled Data}
\label{app:exact-data-reconstruction}
\label{app:inexact-data-reconstruction}
\label{app:function-value-polynomial-reconstruction}
\label{app:gradient-polynomial-reconstruction}

In this section, we address Problem~\ref{prob:apx-function-reconstruction}
using sampled function values or gradients with a known
error bound $\tau\ge0$ and exact data correspond to $\tau=0$.
We first construct a polynomial and separate its approximation
error from the effect of errors in the supplied values.
Under a corresponding condition on the data error, we obtain
prescribed accuracies by combining such polynomials.
We also construct scalar models from gradient data.
For value approximation on rectangles, we finally give a
grid construction whose stability constant is independent
of the cell size.

Throughout, let $p\ge1$, write $\nu:=p+\beta$, and let
$S\subset\mathbb R^d$ be compact and convex with nonempty
interior. Fix $\boldsymbol c\in\operatorname{int}(S)$ and set
$\rho:=\max\limits_{\boldsymbol y\in S}
\|\boldsymbol y-\boldsymbol c\|_2$.
For a scalar- or vector-valued map $F$, the supplied values
at nodes $\boldsymbol z_i$ satisfy
$\|\widetilde F_i-F(\boldsymbol z_i)\|\le\tau$.
For gradient data, this condition is imposed on the supplied
vectors approximating $\nabla f(\boldsymbol z_i)$.
All node sets are chosen nonadaptively.

\subsubsection{Local Polynomial Reconstruction}
\label{subsec:local-polynomial-reconstruction}
We first construct one polynomial from supplied function values,
which may have error at most $\tau\ge0$ at each node.
For the chosen nodes, let $\Lambda_0,\Lambda_1$ denote the
value and derivative stability constants from
Lemma~\ref{lem:interpolation-error-decomposition}, evaluated
on $S$ using the Lagrange basis in the original coordinates.
The case $\tau=0$ gives reconstruction from exact values.

\begin{proposition}[\textbf{Local Polynomial Reconstruction}]
\label{prop:local-reconstruction}
Suppose that $F\in C^{p,\beta}(L_p)$ on a neighborhood of $S$.
There is a nonadaptively chosen node set $X_p\subset S$,
containing $\boldsymbol c$ and having
$|X_p|=\binom{d+p}{p}$, with the following property.
Supplied values satisfying
$\|\widetilde F_i-F(\boldsymbol z_i)\|\le\tau$
at its nodes determine a polynomial map $P$, with each
component in $\Pi_p^d$, such that
$P(\boldsymbol c)=\widetilde F_{\boldsymbol c}$ and
\begin{equation}
\label{eq:apx-local-vector-value-bounds}
 \|D^j(P-F)\|_{\infty,S}
 \le
 \frac{3}{2}a_{p-j,\beta}L_p\rho^{p-j+\beta}
 +\Lambda_j\tau,
 \qquad j=0,1.
\end{equation}
Here $\widetilde F_{\boldsymbol c}$ is the supplied value
at $\boldsymbol c$, so
$\|P(\boldsymbol c)-F(\boldsymbol c)\|\le\tau$.

For fixed $S,\boldsymbol c$ and exponents
$\beta_1,\ldots,\beta_p\in(0,1]$, the constructions admit
nested node sets
$X_0=\{\boldsymbol c\}\subset X_1\subset\cdots\subset X_p$
with $|X_j|=\binom{d+j}{j}$.
At degree $j\ge1$, the corresponding bound uses
$(j,\beta_j,L_j)$ whenever
$F\in C^{j,\beta_j}(L_j)$, and the stability constants
of the degree-$j$ node set.

At degree zero, the constant model
$P\equiv\widetilde F_{\boldsymbol c}$ satisfies
$\|P-F\|_{\infty,S}\le L_0\rho^\beta+\tau$
whenever $F\in C^{0,\beta}(L_0)$.
Its interpolation stability constants are
$\Lambda_0=1$ and $\Lambda_1=0$.
\end{proposition}

\begin{proof}
We construct a fixed interpolation pattern, contract its
nodes toward $\boldsymbol c$, and then bound separately
the approximation error and the effect of the supplied-value
errors.

\noindent\textbf{Reference nodes.}
Let
$A_p:=\{\boldsymbol\alpha\in\mathbb Z_{\ge0}^d:
|\boldsymbol\alpha|\le p\}$, where
$|\boldsymbol\alpha|:=\sum\limits_{k=1}^d\alpha_k$,
and set
$\Sigma_p:=\{\frac{\boldsymbol\alpha}{p}:
\boldsymbol\alpha\in A_p\}$.
These $\binom{d+p}{p}$ nodes include $\boldsymbol0$
and lie in the closed unit ball because
$\|\frac{\boldsymbol\alpha}{p}\|_2
\le\frac{|\boldsymbol\alpha|}{p}\le1$.

To verify unisolvence, define
$B_{\boldsymbol\alpha}(\boldsymbol w)
:=\prod\limits_{k=1}^d\binom{pw_k}{\alpha_k}$,
where
$\binom{t}{r}:=\frac{1}{r!}
\prod\limits_{i=0}^{r-1}(t-i)$ for $r\ge1$ and
$\binom{t}{0}:=1$.
At a node $\frac{\boldsymbol\zeta}{p}$,
$B_{\boldsymbol\alpha}(\frac{\boldsymbol\zeta}{p})
=\prod\limits_{k=1}^d\binom{\zeta_k}{\alpha_k}$.
This value vanishes unless
$\boldsymbol\alpha\le\boldsymbol\zeta$ coordinatewise,
and equals one when
$\boldsymbol\alpha=\boldsymbol\zeta$.
Ordering the indices by total degree gives a triangular
evaluation matrix with unit diagonal: among indices
of the same total degree, a nonzero entry requires
$\boldsymbol\alpha=\boldsymbol\zeta$.
The matrix is therefore invertible.
Since its size is $\dim\Pi_p^d$, these nodes are unisolvent.

Write $\Sigma_p=\{\boldsymbol\sigma_i\}_i$ and let
$\ell_i(\boldsymbol w)
=\sum\limits_{\boldsymbol\alpha\in A_p}
\lambda_{i,\boldsymbol\alpha}
\boldsymbol w^{\boldsymbol\alpha}$
be the associated Lagrange basis.
Define
$K_{d,p}:=\max\{1,\,
p\sum\limits_{i,\boldsymbol\alpha}
|\lambda_{i,\boldsymbol\alpha}|\}$.
For $R\ge1$ and $\|\boldsymbol w\|_2\le R$, we have
$|\boldsymbol w^{\boldsymbol\alpha}|
\le R^{|\boldsymbol\alpha|}$ and
$\|\nabla\boldsymbol w^{\boldsymbol\alpha}\|_2
\le|\boldsymbol\alpha|R^{|\boldsymbol\alpha|-1}$,
with zero derivative when $\boldsymbol\alpha=\boldsymbol0$.
The derivative estimate follows by bounding the Euclidean
norm of the coordinate derivatives by their absolute sum.
Summing over the coefficients gives
\[
 \sup\limits_{\|\boldsymbol w\|_2\le R}
 \sum\limits_i\|D^j\ell_i(\boldsymbol w)\|
 \le K_{d,p}R^{p-j},
 \qquad j=0,1.
\]

\noindent\textbf{Sampling nodes and stability.}
Choose $\eta>0$ satisfying
\begin{equation}
\label{eq:apx-interpolation-spacing}
 \eta\le
 \min\left\{
 \operatorname{dist}(\boldsymbol c,\partial S),\,
 \rho(2K_{d,p})^{-\frac{1}{\beta}}
 \right\}.
\end{equation}
Set $\boldsymbol z_i:=\boldsymbol c+\eta\boldsymbol\sigma_i$.
The first condition places these nodes in $S$.
Their physical Lagrange basis consists of the functions
$\boldsymbol y\mapsto
\ell_i((\boldsymbol y-\boldsymbol c)/\eta)$.
Since $\eta\le\rho$, the reference basis bound with
$R=\rho/\eta\ge1$ yields
\begin{equation}
\label{eq:apx-contracted-stability}
 \Lambda_j
 \le
 \eta^{-j}K_{d,p}
 \left(\frac{\rho}{\eta}\right)^{p-j}
 =K_{d,p}\rho^{p-j}\eta^{-p},
 \qquad j=0,1.
\end{equation}
Construct the interpolant of the supplied values by
$P(\boldsymbol c+\boldsymbol v)
:=\sum\limits_i\widetilde F_i\ell_i(\boldsymbol v/\eta)$.
The origin node ensures
$P(\boldsymbol c)=\widetilde F_{\boldsymbol c}$.

\noindent\textbf{Approximation and data errors.}
Let $P_{\mathrm{ex}}$ be the interpolant at the same nodes
of the exact values $F(\boldsymbol z_i)$; it is used only
in the analysis.
Let $T_p$ be the degree-$p$ Taylor polynomial of $F$
at $\boldsymbol c$ and put
$r_T:=\max\limits_i
\|(F-T_p)(\boldsymbol z_i)\|$.
Lemma~\ref{lem:holder-taylor-remainder} gives
$r_T\le a_{p,\beta}L_p\eta^{p+\beta}$.
Since interpolation reproduces $T_p$, the difference
$P_{\mathrm{ex}}-T_p$ interpolates these nodal residuals.
The scaled basis bound gives
\begin{equation}
\label{eq:apx-coefficient-bounds}
\begin{aligned}
 \|D^j(P_{\mathrm{ex}}-T_p)\|_{\infty,S}
 &\le
 \eta^{-j}K_{d,p}
 \left(\frac{\rho}{\eta}\right)^{p-j}r_T\\
 &\le
 K_{d,p}a_{p,\beta}L_p\eta^\beta\rho^{p-j}\\
 &\le
 \frac{1}{2}a_{p-j,\beta}L_p\rho^{p-j+\beta},
 \qquad j=0,1.
\end{aligned}
\end{equation}
The last inequality uses
$K_{d,p}\eta^\beta\le\frac{1}{2}\rho^\beta$
and $a_{p,\beta}\le a_{p-j,\beta}$.
Adding the Taylor bound
$\|D^j(T_p-F)\|_{\infty,S}
\le a_{p-j,\beta}L_p\rho^{p-j+\beta}$
gives
$\|D^j(P_{\mathrm{ex}}-F)\|_{\infty,S}
\le\frac{3}{2}a_{p-j,\beta}L_p\rho^{p-j+\beta}$.

On the other hand,
$P-P_{\mathrm{ex}}
=\sum\limits_i(\widetilde F_i-F(\boldsymbol z_i))
\ell_i((\cdot-\boldsymbol c)/\eta)$.
The triangle inequality for values, and the rank-one
operator norm identity for derivatives, therefore give
$\|D^j(P-P_{\mathrm{ex}})\|_{\infty,S}
\le\Lambda_j\tau$ for $j=0,1$.
Adding these bounds proves Eq.~\eqref{eq:apx-local-vector-value-bounds}.

\noindent\textbf{Nested nodes.}
Choose a common $s>0$ such that $js$ satisfies
Eq.~\eqref{eq:apx-interpolation-spacing} at degree $j$
and exponent $\beta_j$, for every $1\le j\le p$.
There are finitely many positive upper bounds on $s$,
so such a choice exists.
Set
$X_j:=\{\boldsymbol c+s\boldsymbol\alpha:
\boldsymbol\alpha\in\mathbb Z_{\ge0}^d,\,
|\boldsymbol\alpha|\le j\}$.
These sets are nested, have the required cardinalities,
and satisfy $X_j=\boldsymbol c+(js)\Sigma_j$ for $j\ge1$.
The preceding construction and its error bound apply at
every degree with the corresponding physical Lagrange basis.
At degree zero,
$\|\widetilde F_{\boldsymbol c}-F(\boldsymbol y)\|
\le\|\widetilde F_{\boldsymbol c}-F(\boldsymbol c)\|
+\|F(\boldsymbol c)-F(\boldsymbol y)\|
\le\tau+L_0\rho^\beta$.
Its only basis function is the constant one, giving
$\Lambda_0=1$ and $\Lambda_1=0$.
All nodes depend only on the prescribed parameters and
the geometry of $S$, so all evaluations are nonadaptive.
\end{proof}

The stability bound (Eq.~\eqref{eq:apx-contracted-stability})
also explains the effect of node contraction.
Decreasing $\eta$ reduces the Taylor residual at the
nodes, but can amplify errors in the supplied values.
Thus Eq.~\eqref{eq:apx-local-vector-value-bounds} separates
the approximation error from the precision required
of the input data.

\subsubsection{Smooth Reconstruction from Function Value}
\label{subsec:smooth-reconstruction}
The preceding polynomial construction controls the error at a fixed
radius. To achieve smaller errors, we apply it on overlapping subsets
and combine the resulting polynomials using smooth weights.
The entire construction is linear in the supplied values, so its
sensitivity to inexact data can be described by stability
constants defined analogously to those in
Lemma~\ref{lem:interpolation-error-decomposition}.

For approximation budgets $\varepsilon_0,\varepsilon_1>0$, define the refinement parameter as
\begin{equation}
\label{eq:apx-function-radius}
 \sigma:=
 \max\left\{
 \left(
 \frac{2a_{p,\beta}L_p\rho^\nu}{\varepsilon_0}
 \right)^{\frac{1}{\nu}},
 \left(
 \frac{2a_{p-1,\beta}L_p\rho^{\nu-1}}{\varepsilon_1}
 \right)^{\frac{1}{\nu-1}}
 \right\}.
\end{equation}
For a reconstruction of the form
$W=\sum\limits_{\boldsymbol v\in V}
\widetilde F_{\boldsymbol v}L_{\boldsymbol v}$,
with fixed smooth scalar coefficient functions $L_{\boldsymbol v}$,
use the value and derivative stability constants
$\Lambda_j:=
\sup\limits_{\boldsymbol y\in S}
\sum\limits_{\boldsymbol v\in V}
\|D^jL_{\boldsymbol v}(\boldsymbol y)\|$ for $j=0,1$.
These constants are determined by the nodes and coefficient
functions, before the data are observed.

\begin{lemma}[\textbf{Smooth Reconstruction from Finite Evaluations}]
\label{lem:polynomial-reconstruction}
Suppose that $F\in C^{p,\beta}(L_p)$ on a neighborhood of $S$.
There is a constant $C_{p,\beta}>0$, depending only on $p,\beta$,
and a nonadaptively chosen node set $V\subset S$, containing
$\boldsymbol c$, with
\begin{equation}
\label{eq:apx-two-accuracy-queries}
 |V|
 \le
 \binom{d+p}{p}
 \left[
 \max\left\{
 1,\,
 C_{p,\beta}d^{\frac{1}{\nu-1}}(\sigma-1)
 \right\}
 \right]^d.
\end{equation}
There are explicit coefficient functions $L_{\boldsymbol v}$
that are $C^\infty$ on a neighborhood of $S$ and have
finite stability constants $\Lambda_0,\Lambda_1$.
For supplied values satisfying
$\|\widetilde F_{\boldsymbol v}-F(\boldsymbol v)\|\le\tau$
at every $\boldsymbol v\in V$, the model
$W:=\sum\limits_{\boldsymbol v\in V}
\widetilde F_{\boldsymbol v}L_{\boldsymbol v}$ satisfies
$W(\boldsymbol c)=\widetilde F_{\boldsymbol c}$ and
\begin{equation}
\label{eq:apx-vector-value-bounds}
 \|W-F\|_{\infty,S}
 \le\varepsilon_0+\Lambda_0\tau,
 \qquad
 \|DW-DF\|_{\infty,S}
 \le\varepsilon_1+\Lambda_1\tau.
\end{equation}
If $\sigma\le(\frac{4}{3})^{\frac{1}{\nu}}$, then
$\binom{d+p}{p}$ evaluations suffice, and each component
of $W$ belongs to $\Pi_p^d$.
All queries form one nonadaptive batch.
For exact data, $\tau=0$, the two budgets are attained and
$W(\boldsymbol c)=F(\boldsymbol c)$.
\end{lemma}

\begin{proof}
We first construct the model from exact values and denote it
by $W^{\mathrm{ex}}$.
We then identify its linear dependence on the data and bound
the effect of perturbing those data.
Set $\sigma_0:=(\frac{4}{3})^{\frac{1}{\nu}}>1$.

\noindent\textbf{A single polynomial.}
If $\sigma\le\sigma_0$, take $W^{\mathrm{ex}}=P$ from
Proposition~\ref{prop:local-reconstruction} with exact data.
The definition of $\sigma$ gives
\[
 \frac{\|W^{\mathrm{ex}}-F\|_{\infty,S}}{\varepsilon_0}
 \le\frac{3}{4}\sigma^\nu
 \le\frac{3}{4}\sigma_0^\nu=1,
 \qquad
 \frac{\|DW^{\mathrm{ex}}-DF\|_{\infty,S}}{\varepsilon_1}
 \le\frac{3}{4}\sigma^{\nu-1}
 \le\frac{3}{4}\sigma_0^{\nu-1}\le1.
\]
The last inequality uses $\sigma_0>1$ and $\nu-1<\nu$.
This case uses $\binom{d+p}{p}$ evaluations and includes
$L_p=0$.
Its reconstruction is the interpolation operator of
Proposition~\ref{prop:local-reconstruction}; its coefficient
functions are the corresponding polynomial Lagrange basis.

Suppose henceforth that $\sigma>\sigma_0$.
We construct a model at a scale $0<h\le\rho$, then choose
$h$ to meet the two approximation budgets.

\noindent\textbf{Covering and polynomial models.}
Choose a maximal set
$\{\boldsymbol c_1=\boldsymbol c,\ldots,\boldsymbol c_n\}
\subset\operatorname{int}(S)$
whose distinct points have distance at least $h$.
The open balls of radius $\frac{h}{2}$ about these centers
are disjoint and lie in
$B(\boldsymbol c,\rho+\frac{h}{2})$.
Comparing their volumes gives
$n\le(1+\frac{2\rho}{h})^d$.
This finite bound also ensures that successively adding
admissible centers terminates.

By maximality, every point of $\operatorname{int}(S)$
is within distance less than $h$ of some center.
Since $S$ is convex with nonempty interior,
$S=\overline{\operatorname{int}(S)}$.
Taking closures therefore shows that the closed $h$-balls
about the centers cover $S$.

For each $i$, set
$S_i:=S\cap\overline B(\boldsymbol c_i,2h)$.
This set is compact and convex,
$\boldsymbol c_i\in\operatorname{int}(S_i)$,
and its radius about $\boldsymbol c_i$ is at most $2h$.
Proposition~\ref{prop:local-reconstruction}, with exact data
and $\frac{3}{2}$ weakened to $2$, provides a polynomial map
$P_i$ satisfying
\begin{equation}
\label{eq:apx-patch-accuracy}
\begin{aligned}
 \|P_i-F\|_{\infty,S_i}
 &\le A_h:=2^{\nu+1}a_{p,\beta}L_ph^\nu,\\
 \|DP_i-DF\|_{\infty,S_i}
 &\le B_h:=2^\nu a_{p-1,\beta}L_ph^{\nu-1}.
\end{aligned}
\end{equation}
Let $V_i$ be its interpolation nodes and let $V$ be their
union, with any repeated nodes included only once.
Then $|V|\le n\binom{d+p}{p}$ and $\boldsymbol c\in V$.

\noindent\textbf{Smooth weights.}
Choose a smooth function $\psi:\mathbb R\to[0,1]$
equal to $1$ on $(-\infty,1]$ and $0$ on $[4,\infty)$,
with $\|\psi'\|_\infty\le\frac{1}{2}$.
One construction smooths the linear transition on
$[\frac{3}{2},\frac{7}{2}]$ with a nonnegative unit-mass
kernel supported in $(-\frac{1}{2},\frac{1}{2})$.
Define
$b_i(\boldsymbol y)
:=\psi(\|\boldsymbol y-\boldsymbol c_i\|_2^2/h^2)$,
$Z:=\sum\limits_i b_i^d$, and $\theta_i:=b_i^d/Z$.
At every point of $S$, the covering supplies an index with
$b_i=1$, so $Z\ge1$ on $S$.
Consequently, the $\theta_i$ are smooth on the open
neighborhood $\{Z>0\}$ of $S$, are nonnegative, and sum
to one.
Whenever $\theta_i$ or its gradient is nonzero at a point
of $S$, that point belongs to $S_i$.

At any point $\boldsymbol y$, at most $5^d$ of the
$b_i(\boldsymbol y)$ are positive.
Indeed, positivity requires
$\|\boldsymbol y-\boldsymbol c_i\|_2<2h$.
The corresponding disjoint balls of radius $\frac{h}{2}$
therefore lie in $B(\boldsymbol y,\frac{5h}{2})$,
and volume comparison gives the bound.

The chain rule and the support of $\psi'$ give
$\|\nabla b_i\|_2\le\frac{2}{h}$.
For $d\ge2$, H\"older's inequality yields
$\sum\limits_{b_i>0}b_i^{d-1}
\le(5^d)^{\frac{1}{d}}
(\sum\limits_i b_i^d)^{\frac{d-1}{d}}
=5Z^{1-\frac{1}{d}}$.
For $d=1$, the same bound follows from the number of
positive terms.
Since $b_i=0$ implies $\nabla b_i=0$, the quotient rule
gives, for $\boldsymbol y\in S$,
\begin{equation}
\label{eq:apx-partition-derivatives}
\begin{aligned}
 \sum\limits_i\|\nabla\theta_i\|_2
 &\le
 \frac{1}{Z}\sum\limits_i\|\nabla(b_i^d)\|_2
 +\frac{\|\nabla Z\|_2}{Z^2}\sum\limits_i b_i^d\\
 &\le
 \frac{2}{Z}\sum\limits_i\|\nabla(b_i^d)\|_2\\
 &\le
 \frac{4d}{hZ}\sum\limits_{b_i>0}b_i^{d-1}\\
 &\le\frac{20d}{h}Z^{-\frac{1}{d}}
 \le\frac{20d}{h}.
\end{aligned}
\end{equation}
The second inequality uses
$\|\nabla Z\|_2\le\sum\limits_i\|\nabla(b_i^d)\|_2$,
and the last uses $Z\ge1$.

\noindent\textbf{Errors for exact data.}
Set $\widetilde W^{\mathrm{ex}}
:=\sum\limits_i\theta_iP_i$ and $E_i:=P_i-F$.
Since the weights sum to one,
$\widetilde W^{\mathrm{ex}}-F
=\sum\limits_i\theta_iE_i$.
Every nonzero term uses an error bound on its own set $S_i$,
so $\|\widetilde W^{\mathrm{ex}}-F\|_{\infty,S}\le A_h$.
Differentiating gives
\[
 D\widetilde W^{\mathrm{ex}}-DF
 =\sum\limits_i\theta_iDE_i
 +\sum\limits_i E_i(\nabla\theta_i)^{\mathsf T}.
\]
The first sum has norm at most $B_h$.
For the second, the rank-one norm identity and
\eqref{eq:apx-partition-derivatives} give
$\|\sum\limits_i E_i(\nabla\theta_i)^{\mathsf T}\|
\le A_h\sum\limits_i\|\nabla\theta_i\|_2
\le\frac{20dA_h}{h}$.

Define
$W^{\mathrm{ex}}:=\widetilde W^{\mathrm{ex}}
+F(\boldsymbol c)-\widetilde W^{\mathrm{ex}}(\boldsymbol c)$.
The value $F(\boldsymbol c)$ is among the sampled values.
This correction enforces
$W^{\mathrm{ex}}(\boldsymbol c)=F(\boldsymbol c)$
and leaves derivatives unchanged.
Since
$\|F(\boldsymbol c)-\widetilde W^{\mathrm{ex}}(\boldsymbol c)\|
\le A_h$, it increases the value error by at most $A_h$.
Using $a_{p-1,\beta}=\nu a_{p,\beta}$ gives
\begin{equation}
\label{eq:apx-global-scale-bounds}
\begin{aligned}
 \|W^{\mathrm{ex}}-F\|_{\infty,S}
 &\le2A_h
 =2^{\nu+2}a_{p,\beta}L_ph^\nu,\\
 \|DW^{\mathrm{ex}}-DF\|_{\infty,S}
 &\le B_h+\frac{20dA_h}{h}
 =2^\nu a_{p,\beta}(\nu+40d)L_ph^{\nu-1}.
\end{aligned}
\end{equation}

\noindent\textbf{Scale and node count.}
Set
\[
 H_d:=\max\left\{
 2^{1+\frac{1}{\nu}},\,
 2\left(1+\frac{40d}{\nu}\right)^{\frac{1}{\nu-1}}
 \right\},
 \qquad h:=\frac{\rho}{H_d\sigma}.
\]
Then $0<h<\rho$ because $\sigma>\sigma_0>1$.
The definition of $\sigma$ implies
$a_{p,\beta}L_p\rho^\nu
\le\frac{\varepsilon_0}{2}\sigma^\nu$
and
$a_{p-1,\beta}L_p\rho^{\nu-1}
\le\frac{\varepsilon_1}{2}\sigma^{\nu-1}$.
Substitution into \eqref{eq:apx-global-scale-bounds} yields
\[
 \frac{\|W^{\mathrm{ex}}-F\|_{\infty,S}}{\varepsilon_0}
 \le\frac{2^{\nu+1}}{H_d^\nu}\le1,
 \qquad
 \frac{\|DW^{\mathrm{ex}}-DF\|_{\infty,S}}{\varepsilon_1}
 \le
 \frac{2^{\nu-1}(1+\frac{40d}{\nu})}{H_d^{\nu-1}}
 \le1.
\]

The node count is at most
$\binom{d+p}{p}(1+2H_d\sigma)^d$.
Since $d\ge1$ and
$1+\frac{40d}{\nu}\le d(1+\frac{40}{\nu})$,
we have $H_d\le H_1d^{\frac{1}{\nu-1}}$.
Moreover, $\sigma>\sigma_0$ implies
$\sigma-1\ge(1-\sigma_0^{-1})\sigma$.
Hence
\[
\begin{aligned}
 1+2H_d\sigma
 &\le(1+2H_1)d^{\frac{1}{\nu-1}}\sigma\\
 &\le
 \frac{1+2H_1}{1-\sigma_0^{-1}}
 d^{\frac{1}{\nu-1}}(\sigma-1).
\end{aligned}
\]
Taking $C_{p,\beta}:=\frac{1+2H_1}{1-\sigma_0^{-1}}$
and combining the two cases proves
\eqref{eq:apx-two-accuracy-queries}.

\noindent\textbf{Perturbing the supplied values.}
Each local polynomial is linear in its nodal values.
Write
$P_i=\sum\limits_{\boldsymbol v\in V_i}
F(\boldsymbol v)\ell_{i,\boldsymbol v}$
for its Lagrange representation, and extend the index set
by setting $\ell_{i,\boldsymbol v}=0$ when
$\boldsymbol v\notin V_i$.
Define
$\widetilde L_{\boldsymbol v}
:=\sum\limits_i\theta_i\ell_{i,\boldsymbol v}$ and
$L_{\boldsymbol v}(\boldsymbol y)
:=\widetilde L_{\boldsymbol v}(\boldsymbol y)
-\widetilde L_{\boldsymbol v}(\boldsymbol c)
+\mathbf 1_{\{\boldsymbol v=\boldsymbol c\}}$.
Expanding the blend and its center correction shows that
$W^{\mathrm{ex}}=\sum\limits_{\boldsymbol v\in V}
F(\boldsymbol v)L_{\boldsymbol v}$.
All coefficient functions are $C^\infty$ on $\{Z>0\}$
and satisfy
$L_{\boldsymbol v}(\boldsymbol c)
=\mathbf 1_{\{\boldsymbol v=\boldsymbol c\}}$.
In the single-polynomial case, take its Lagrange basis as
$L_{\boldsymbol v}$; these properties hold there as well.
Since $V$ is finite and $S$ is compact, both
$\Lambda_0$ and $\Lambda_1$
are finite.

Applying the same coefficients to the supplied values gives
$W:=\sum\limits_{\boldsymbol v\in V}
\widetilde F_{\boldsymbol v}L_{\boldsymbol v}$ and
$W(\boldsymbol c)=\widetilde F_{\boldsymbol c}$.
For $j=0,1$, the data error bound and the triangle inequality
give
\[
\begin{aligned}
 \|D^j(W-W^{\mathrm{ex}})\|_{\infty,S}
 &\le
 \sup\limits_{\boldsymbol y\in S}
 \sum\limits_{\boldsymbol v\in V}
 \|\widetilde F_{\boldsymbol v}-F(\boldsymbol v)\|
 \|D^jL_{\boldsymbol v}(\boldsymbol y)\|\\
 &\le\tau\Lambda_j.
\end{aligned}
\]
For $j=1$, the summands are rank-one derivatives, whose
operator norms are the displayed products.
Adding the corresponding exact-data error proves
\eqref{eq:apx-vector-value-bounds}.
Every covering center, interpolation node, and coefficient
function is fixed by the prescribed parameters and the
geometry of $S$, before any values are observed.
Thus all queries form one nonadaptive batch.
\end{proof}

To obtain final tolerances $\varepsilon_0,\varepsilon_1$
from inexact data, apply the construction with approximation
budgets $\frac{\varepsilon_0}{2},\frac{\varepsilon_1}{2}$
and require
$\Lambda_j\tau\le\frac{\varepsilon_j}{2}$
for $j=0,1$, using the stability constants of that construction.
The node bound then uses the corresponding value of $\sigma$.
This condition makes explicit the data accuracy needed to
attain the two final tolerances.

\subsubsection{Reconstruction from Gradient Data}
\label{subsubsec:gradient-reconstruction}
Applying Proposition~\ref{prop:local-reconstruction}
to $F=\nabla f$ produces a vector polynomial.
This polynomial need not itself be the gradient of a
scalar polynomial.
We obtain a scalar model by integrating it along segments
from a reference point.
The following conversion uses only the approximation
errors of the vector polynomial, so it applies to models
constructed from either exact or inexact gradient data.

\begin{lemma}[\textbf{Scalar Reconstruction from Vector Fields}]
\label{lem:radial-scalar-reconstruction}
Let $\boldsymbol x\in S$ and
$\rho_{\boldsymbol x}
:=\max\limits_{\boldsymbol y\in S}
\|\boldsymbol y-\boldsymbol x\|_2$.
Suppose that $f$ is $C^1$ on a neighborhood of $S$
and $W\in(\Pi_{p-1}^d)^d$.
Define
\[
 P(\boldsymbol x+\boldsymbol v)
 :=
 \int\limits_0^1
 \langle W(\boldsymbol x+t\boldsymbol v),
 \boldsymbol v\rangle\,dt.
\]
Then $P\in\Pi_p^d$, $P(\boldsymbol x)=0$, and
$\nabla P(\boldsymbol x)=W(\boldsymbol x)$.
If $\|W-\nabla f\|_{\infty,S}\le\eta_0$, then
$\|P-f+f(\boldsymbol x)\|_{\infty,S}
\le\rho_{\boldsymbol x}\eta_0$.

If additionally $f$ is $C^2$ on a neighborhood of $S$
and $\|DW-\nabla^2f\|_{\infty,S}\le\eta_1$, then
$\|\nabla P-\nabla f\|_{\infty,S}
\le\eta_0+\frac{1}{2}\rho_{\boldsymbol x}\eta_1$.
The construction uses only the coefficients of $W$
and requires no additional queries.
\end{lemma}

\begin{proof}
The integrand is a polynomial in $\boldsymbol v$
of degree at most $p$.
Integrating its coefficients in $t$ therefore gives
$P\in\Pi_p^d$.
Its constant term is zero, and its linear term is
$\langle W(\boldsymbol x),\boldsymbol v\rangle$.
Thus $P(\boldsymbol x)=0$ and
$\nabla P(\boldsymbol x)=W(\boldsymbol x)$.

Write $E:=W-\nabla f$ for the vector-field error and
$e:=P-f+f(\boldsymbol x)$ for the scalar error.
For $\boldsymbol x+\boldsymbol v\in S$, convexity
keeps the segment
$[\boldsymbol x,\boldsymbol x+\boldsymbol v]$ in $S$.
The fundamental theorem of calculus gives
$f(\boldsymbol x+\boldsymbol v)-f(\boldsymbol x)
=\int\limits_0^1
\langle\nabla f(\boldsymbol x+t\boldsymbol v),
\boldsymbol v\rangle\,dt$.
Subtracting this identity from the definition of $P$
yields
$e(\boldsymbol x+\boldsymbol v)
=\int\limits_0^1
\langle E(\boldsymbol x+t\boldsymbol v),
\boldsymbol v\rangle\,dt$.
Consequently,
$|e(\boldsymbol x+\boldsymbol v)|
\le\int\limits_0^1
\|E(\boldsymbol x+t\boldsymbol v)\|_2
\|\boldsymbol v\|_2\,dt
\le\eta_0\|\boldsymbol v\|_2
\le\rho_{\boldsymbol x}\eta_0$.

Under the additional differentiability assumptions,
differentiating with respect to $\boldsymbol v$ yields
\[
 \nabla e(\boldsymbol x+\boldsymbol v)
 =
 \int\limits_0^1
 \left[
 E(\boldsymbol x+t\boldsymbol v)
 +t\,DE(\boldsymbol x+t\boldsymbol v)^{\mathsf T}
 \boldsymbol v
 \right]dt.
\]
This differentiation is valid also when
$\boldsymbol x+\boldsymbol v$ lies on the boundary of $S$:
the segment is compact and lies in the open neighborhood
where $f$ is defined, so sufficiently small perturbations
of $\boldsymbol v$ keep the whole segment in that
neighborhood.
The integrand has norm at most
$\eta_0+t\eta_1\|\boldsymbol v\|_2$.
Integration therefore gives
$\|\nabla e(\boldsymbol x+\boldsymbol v)\|_2
\le\eta_0+\eta_1\|\boldsymbol v\|_2
\int\limits_0^1t\,dt
=\eta_0+\frac{1}{2}\eta_1\|\boldsymbol v\|_2
\le\eta_0+\frac{1}{2}\rho_{\boldsymbol x}\eta_1$.
\end{proof}

Combining this conversion with polynomial interpolation
gives a scalar reconstruction directly from supplied
gradient vectors.

\begin{corollary}[\textbf{Reconstruction from Gradient Data}]
\label{lem:gradient-only-scalar-reconstruction}
Suppose that $f\in C^{p,\beta}(L_p)$ on a neighborhood of $S$.
There is a nonadaptively chosen set of
$\binom{d+p-1}{p-1}$ nodes in $S$, including
$\boldsymbol c$, such that supplied gradients satisfying
$\|\widetilde{\boldsymbol g}_i
-\nabla f(\boldsymbol z_i)\|_2\le\tau$
determine $P\in\Pi_p^d$ with
$P(\boldsymbol c)=0$,
$\nabla P(\boldsymbol c)
=\widetilde{\boldsymbol g}_{\boldsymbol c}$, and
\begin{equation}
\label{eq:apx-gradient-only-bounds}
 \|P-f+f(\boldsymbol c)\|_{\infty,S}\le B_0,
 \qquad
 \|\nabla P-\nabla f\|_{\infty,S}\le B_1.
\end{equation}
Here $\widetilde{\boldsymbol g}_{\boldsymbol c}$ is the
supplied gradient at $\boldsymbol c$, and $\Lambda_0,\Lambda_1$
are the stability constants of the degree-$(p-1)$
interpolation at these nodes.
For $p\ge2$, the bounds are
\[
\begin{aligned}
 B_0
 &:=2a_{p-1,\beta}L_p\rho^\nu
 +\rho\Lambda_0\tau,\\
 B_1
 &:=(2a_{p-1,\beta}+a_{p-2,\beta})L_p\rho^{\nu-1}
 +\left(\Lambda_0+\frac{\rho}{2}\Lambda_1\right)\tau.
\end{aligned}
\]
For $p=1$, they are
$B_0:=a_{1,\beta}L_1\rho^{1+\beta}+\rho\tau$
and $B_1:=L_1\rho^\beta+\tau$, with
$\Lambda_0=1$ and $\Lambda_1=0$.

Moreover, for all $\boldsymbol x,\boldsymbol y\in S$,
\begin{equation}
\label{eq:apx-gradient-only-increments}
 \left|
 P(\boldsymbol y)-P(\boldsymbol x)
 -f(\boldsymbol y)+f(\boldsymbol x)
 \right|
 \le
 \min\left\{
 2B_0,\,
 B_1\|\boldsymbol y-\boldsymbol x\|_2
 \right\}.
\end{equation}
The term $2B_0$ can be replaced by $B_0$ when
$\boldsymbol x=\boldsymbol c$.
No function values are required.
For $\tau=0$, these are the exact-gradient bounds;
if also $L_p=0$, then $P=f-f(\boldsymbol c)$ on $S$.
\end{corollary}

\begin{proof}
Suppose first that $p\ge2$.
Pairing the vector output of $D^{p-1}\nabla f$
with a unit vector supplies the last argument of
the scalar $p$-linear form $D^pf$.
Taking the supremum over this unit vector and the other
$p-1$ unit arguments gives
$\|D^{p-1}\nabla f(\boldsymbol y)
-D^{p-1}\nabla f(\boldsymbol z)\|
=\|D^pf(\boldsymbol y)-D^pf(\boldsymbol z)\|$.
Hence $\nabla f\in C^{p-1,\beta}(L_p)$.

Apply Proposition~\ref{prop:local-reconstruction}
to $F=\nabla f$ at degree $p-1$ with the supplied
gradient vectors as its input values.
The resulting vector polynomial
$W\in(\Pi_{p-1}^d)^d$ uses
$\binom{d+p-1}{p-1}$ supplied gradients and satisfies
$W(\boldsymbol c)=\widetilde{\boldsymbol g}_{\boldsymbol c}$.
Weakening the factor $\frac{3}{2}$ to $2$, its errors obey
$\|W-\nabla f\|_{\infty,S}
\le2a_{p-1,\beta}L_p\rho^{p-1+\beta}+\Lambda_0\tau$
and
$\|DW-\nabla^2f\|_{\infty,S}
\le2a_{p-2,\beta}L_p\rho^{p-2+\beta}+\Lambda_1\tau$.

Use Lemma~\ref{lem:radial-scalar-reconstruction}
with $\boldsymbol x=\boldsymbol c$ to construct $P$.
It gives $P(\boldsymbol c)=0$ and
$\nabla P(\boldsymbol c)=W(\boldsymbol c)
=\widetilde{\boldsymbol g}_{\boldsymbol c}$.
Its error bounds yield
\[
\begin{aligned}
 \|P-f+f(\boldsymbol c)\|_{\infty,S}
 &\le
 \rho\left(
 2a_{p-1,\beta}L_p\rho^{p-1+\beta}
 +\Lambda_0\tau\right)
 =B_0,\\
 \|\nabla P-\nabla f\|_{\infty,S}
 &\le
 2a_{p-1,\beta}L_p\rho^{p-1+\beta}
 +\Lambda_0\tau\\
 &\quad+
 \frac{\rho}{2}\left(
 2a_{p-2,\beta}L_p\rho^{p-2+\beta}
 +\Lambda_1\tau\right)
 =B_1.
\end{aligned}
\]

For $p=1$, use only the supplied gradient
$\widetilde{\boldsymbol g}_{\boldsymbol c}$ and set
$P(\boldsymbol c+\boldsymbol v)
:=\langle\widetilde{\boldsymbol g}_{\boldsymbol c},
\boldsymbol v\rangle$.
Then $P(\boldsymbol c)=0$ and
$\nabla P=\widetilde{\boldsymbol g}_{\boldsymbol c}$.
The triangle inequality and H\"older continuity give
$\|\nabla P(\boldsymbol c+\boldsymbol v)
-\nabla f(\boldsymbol c+\boldsymbol v)\|_2
\le\|\widetilde{\boldsymbol g}_{\boldsymbol c}
-\nabla f(\boldsymbol c)\|_2
+\|\nabla f(\boldsymbol c)
-\nabla f(\boldsymbol c+\boldsymbol v)\|_2
\le\tau+L_1\|\boldsymbol v\|_2^\beta\le B_1$.
For the value error, integrate along the segment to obtain
\[
\begin{aligned}
 &\left|
 P(\boldsymbol c+\boldsymbol v)
 -f(\boldsymbol c+\boldsymbol v)+f(\boldsymbol c)
 \right|\\
 &\quad=
 \left|
 \int\limits_0^1
 \left\langle
 \widetilde{\boldsymbol g}_{\boldsymbol c}
 -\nabla f(\boldsymbol c+t\boldsymbol v),
 \boldsymbol v
 \right\rangle\,dt
 \right|\\
 &\quad\le
 \int\limits_0^1
 \left(\tau+L_1t^\beta\|\boldsymbol v\|_2^\beta\right)
 \|\boldsymbol v\|_2\,dt\\
 &\quad=
 \tau\|\boldsymbol v\|_2
 +a_{1,\beta}L_1\|\boldsymbol v\|_2^{1+\beta}
 \le B_0.
\end{aligned}
\]
The constant vector interpolant has the single basis
function $1$, so $\Lambda_0=1$ and $\Lambda_1=0$.

Finally, set $e:=P-f+f(\boldsymbol c)$.
The established bounds give
$\|e\|_{\infty,S}\le B_0$ and
$\|\nabla e\|_{\infty,S}\le B_1$.
Applying \eqref{eq:apx-increment-control} to $e$
proves \eqref{eq:apx-gradient-only-increments}.
Since $P(\boldsymbol c)=0$, we have $e(\boldsymbol c)=0$
even for inexact gradient data, so the centered bound
replaces $2B_0$ by $B_0$.
Setting $\tau=0$ recovers all exact-gradient conclusions,
including $\nabla P(\boldsymbol c)=\nabla f(\boldsymbol c)$.
If $L_p=\tau=0$, then $B_0=0$, which gives
$P=f-f(\boldsymbol c)$ throughout $S$.
The interpolation nodes are chosen nonadaptively, and
integration of the known vector polynomial requires
no further queries.
\end{proof}

The gradient construction meets prescribed value and
derivative tolerances whenever
\begin{equation}
\label{eq:apx-gradient-only-radius}
 B_0\le\varepsilon_0,
 \qquad B_1\le\varepsilon_1,
\end{equation}
respectively.
Either condition may be imposed independently when
only the corresponding accuracy is required.

\subsection{Approximate Global Minimization}
\label{subsec:approximation-global-minimization}

We now address Problem~\ref{prob:apx-global-minimization}.
Given sampled values of $f$ and a known continuous function
$q$, we reconstruct a model $\widehat f$ and approximately
minimize $\widehat f+q$.
The reconstruction error then controls the suboptimality
for $f+q$.
We count only queries to the unknown function $f$ while the
optimization of the reconstructed model plus the known
term $q$ is treated as postprocessing.

\paragraph{From reconstruction error to suboptimality.}
For a continuous scalar function $E$ on a nonempty compact
set $Y$, define its oscillation by
$\operatorname{osc}_Y(E)
:=\max\limits_{\boldsymbol y\in Y}E(\boldsymbol y)
-\min\limits_{\boldsymbol y\in Y}E(\boldsymbol y)$.
Applied to the reconstruction error $E:=\widehat f-f$,
this quantity measures how much the error can change
between two points.
It is unchanged by adding a constant to the model.

\begin{lemma}[\textbf{Suboptimality under Model Error}]
\label{lem:approximation-minimization-transfer}
Let $Y$ be nonempty and compact, let
$f,\widehat f,q\in C(Y)$, and let $\zeta\ge0$.
If $\widehat{\boldsymbol y}\in Y$ satisfies
$(\widehat f+q)(\widehat{\boldsymbol y})
\le\min\limits_{\boldsymbol y\in Y}
(\widehat f+q)(\boldsymbol y)+\zeta$, then
\[
 (f+q)(\widehat{\boldsymbol y})
 -\min\limits_{\boldsymbol y\in Y}(f+q)(\boldsymbol y)
 \le
 \operatorname{osc}_Y(\widehat f-f)+\zeta
 \le
 2\|\widehat f-f\|_{\infty,Y}+\zeta.
\]
\end{lemma}

\begin{proof}
Continuity and compactness give a minimizer
$\boldsymbol y^\star\in
\operatorname*{arg\,min}\limits_{\boldsymbol y\in Y}
(f+q)(\boldsymbol y)$.
Write $E:=\widehat f-f$.
Since
$\min\limits_Y(\widehat f+q)
\le(\widehat f+q)(\boldsymbol y^\star)$,
the assumed model suboptimality implies
$(\widehat f+q)(\widehat{\boldsymbol y})
-(\widehat f+q)(\boldsymbol y^\star)\le\zeta$.
Therefore,
\[
\begin{aligned}
 &(f+q)(\widehat{\boldsymbol y})
 -(f+q)(\boldsymbol y^\star)\\
 &\quad=
 (\widehat f+q)(\widehat{\boldsymbol y})
 -(\widehat f+q)(\boldsymbol y^\star)
 +E(\boldsymbol y^\star)-E(\widehat{\boldsymbol y})\\
 &\quad\le
 \zeta+\max\limits_{\boldsymbol y\in Y}E(\boldsymbol y)
 -\min\limits_{\boldsymbol y\in Y}E(\boldsymbol y)\\
 &\quad=
 \zeta+\operatorname{osc}_Y(E)
 \le\zeta+2\|E\|_{\infty,Y}.
\end{aligned}
\]
The last inequality follows from
$\max\limits_Y E\le\|E\|_{\infty,Y}$ and
$\min\limits_Y E\ge-\|E\|_{\infty,Y}$.
\end{proof}

\paragraph{Minimization from supplied values.}
We combine the smooth reconstruction in
Lemma~\ref{lem:polynomial-reconstruction} with the
preceding minimization-transfer lemma.
The following theorem specifies the sampling nodes and
data accuracy sufficient for a prescribed optimization
tolerance.

\begin{theorem}[\textbf{Global Minimization with a Known Term}]
\label{thm:global-minimization-known-term}
Let $Y:=\prod\limits_{i=1}^k[a_i,b_i]$, where
$a_i<b_i$, and suppose that $\operatorname{diam}(Y)\le D$.
Let $f\in C^{p,\beta}(L_p)$ on a neighborhood of $Y$,
where $p\ge1$, and let $q\in C(Y)$ be known.
Write $\nu:=p+\beta$.

For every $\delta>0$, there is a nonadaptively chosen
node set $V\subset Y$ with
\[
 |V|\le C_{k,p,\beta}
 \max\left\{
 1,\,
 D^k\left(\frac{L_p}{\delta}\right)^{\frac{k}{\nu}}
 \right\},
\]
where $C_{k,p,\beta}$ depends only on $k,p,\beta$.
There are explicit coefficient functions
$L_{\boldsymbol v}$, smooth on a neighborhood of $Y$,
with value stability constant
$\Lambda_0:=
\sup\limits_{\boldsymbol y\in Y}
\sum\limits_{\boldsymbol v\in V}
|L_{\boldsymbol v}(\boldsymbol y)|<\infty$,
such that supplied values satisfying
\[
 \max\limits_{\boldsymbol v\in V}
 |\widetilde f(\boldsymbol v)-f(\boldsymbol v)|
 \le\tau\le\frac{\delta}{8\Lambda_0}
\]
determine the model
$\widehat f:=
\sum\limits_{\boldsymbol v\in V}
\widetilde f(\boldsymbol v)L_{\boldsymbol v}$
with
$\|\widehat f-f\|_{\infty,Y}\le\frac{\delta}{4}$.

Consequently, any $\frac{\delta}{2}$-approximate
minimizer $\widehat{\boldsymbol y}\in Y$ of
$\widehat f+q$ satisfies
\begin{equation}
\label{eq:apx-global-minimum-accuracy}
 (f+q)(\widehat{\boldsymbol y})
 -\min\limits_{\boldsymbol y\in Y}(f+q)(\boldsymbol y)
 \le\delta.
\end{equation}
The model, its optimization, and the evaluation of
$\widehat f(\widehat{\boldsymbol y})$ require no further
queries to $f$.
\end{theorem}

\begin{proof}
Take $\boldsymbol c$ to be the center of $Y$ and set
$\rho:=\max\limits_{\boldsymbol y\in Y}
\|\boldsymbol y-\boldsymbol c\|_2\le D$.
Apply Lemma~\ref{lem:polynomial-reconstruction} with
value budget $\varepsilon_0:=\frac{\delta}{8}$.
Since no derivative accuracy is required here, choose
any $\varepsilon_1>0$ satisfying
$\varepsilon_1\ge
2a_{p-1,\beta}L_p\rho^{\nu-1}$.
The refinement parameter then satisfies
$\sigma\le
\max\{1,
(\frac{16a_{p,\beta}L_p\rho^\nu}{\delta})^{\frac1\nu}\}$.
The node bound in that lemma therefore gives
\[
\begin{aligned}
 |V|
 &\le C_{k,p,\beta}\max\{1,\sigma\}^k\\
 &\le C_{k,p,\beta}
 \max\left\{
 1,\,
 D^k\left(\frac{L_p}{\delta}\right)^{\frac{k}{\nu}}
 \right\},
\end{aligned}
\]
after enlarging $C_{k,p,\beta}$.
This also covers $L_p=0$.

The same lemma supplies the coefficient functions and
their finite stability constant $\Lambda_0$.
Under the assumed data accuracy, its value estimate gives
$\|\widehat f-f\|_{\infty,Y}
\le\frac{\delta}{8}+\Lambda_0\tau
\le\frac{\delta}{4}$.

Since $\widehat f+q$ is continuous on compact $Y$,
a $\frac{\delta}{2}$-approximate minimizer exists.
For any such point $\widehat{\boldsymbol y}$,
Lemma~\ref{lem:approximation-minimization-transfer} yields
\[
\begin{aligned}
 (f+q)(\widehat{\boldsymbol y})
 -\min\limits_{\boldsymbol y\in Y}(f+q)(\boldsymbol y)
 &\le
 2\|\widehat f-f\|_{\infty,Y}+\frac{\delta}{2}\\
 &\le
 2\cdot\frac{\delta}{4}+\frac{\delta}{2}
 =\delta.
\end{aligned}
\]
The coefficient functions are predetermined, and
$\widehat f$ is completely determined by the supplied
values. All subsequent operations use only
$\widehat f$ and the known function $q$.
\end{proof}

\section{Proof of Theorem~\ref{the:main1}}
\label{app:full-proof-main1}
\label{app:full-proof-main1-unconstrained}
\label{sec:constrained-case}

We analyze Algorithm~\ref{alg:main0} with either initialization from
Section~\ref{subsec:trapping-algorithm}.
We first establish the approximation and query guarantees of the
separator subroutine. We then prove correctness by maintaining a
boundary certificate until the region is small enough to apply
Lemma~\ref{lem:stationarity-from-model}. Finally, we bound the total
number of queries and the adaptive depth.
Throughout, write $\nu:=p+\beta_p>1$.
Constants denoted by $C_{d,p,\beta_p}$ may change between occurrences
and depend only on the indicated parameters.
We use the reconstruction and minimization results in
Appendix~\ref{app:approximation}.
Only function-value oracle queries are counted; optimization of
the explicit models requires no additional oracle queries.

\subsection{Guarantees of the Separator Subroutine}
\label{app:main1-separator}

Fix an iteration $t$ and write $\delta:=\varepsilon_0^t$.
We establish the approximation guarantee used in the correctness
proof, together with the query cost and the number of adaptive
rounds needed to obtain it.

Each separator $E_m^t$, $m\in\{1,2\}$, is a compact
$(d-1)$-dimensional hyperrectangle of diameter at most
$\sqrt{d-1}\,r^t$.
An affine isometric parameterization preserves the
$C^{p,\beta_p}(L_p)$ bound: its linear part has operator norm one,
so restriction of the $p$th derivative cannot increase its
H\"older constant.
Apply Theorem~\ref{thm:global-minimization-known-term} in these
intrinsic coordinates, with the common known term
$q_t(\boldsymbol y):=-f(\boldsymbol x^t)
+\varepsilon_t\|\boldsymbol y-\boldsymbol x^t\|_2$.
For each $m$, the theorem constructs $\widehat f_m$ satisfying
$\|\widehat f_m-f\|_{\infty,E_m^t}\le\delta/4$
and a point $\boldsymbol z_m\in E_m^t$ such that
$(\widehat f_m+q_t)(\boldsymbol z_m)
\le\min_{E_m^t}(\widehat f_m+q_t)+\delta/2$.

The algorithm chooses $\widehat m$ minimizing
$(\widehat f_m+q_t)(\boldsymbol z_m)$ and sets
$\boldsymbol z^\star:=\boldsymbol z_{\widehat m}$.
To verify its accuracy on the entire union $\mathcal E^t$, let
$\boldsymbol y_\star\in E_{m_\star}^t$ minimize the true trapping
function over $\mathcal E^t$.
Such a point exists because the union is compact and the trapping
function is continuous.
The model error, candidate selection, and model suboptimality give
\begin{equation}
\label{eq:app-main1-separator-accuracy}
\begin{aligned}
 \phi_{\boldsymbol x^t,\varepsilon_t}(\boldsymbol z^\star)
 &\le(\widehat f_{\widehat m}+q_t)(\boldsymbol z^\star)
       +\delta/4\\
 &\le(\widehat f_{m_\star}+q_t)(\boldsymbol z_{m_\star})
       +\delta/4\\
 &\le(\widehat f_{m_\star}+q_t)(\boldsymbol y_\star)
       +3\delta/4\\
 &\le\phi_{\boldsymbol x^t,\varepsilon_t}(\boldsymbol y_\star)
       +\delta.
\end{aligned}
\end{equation}
Thus $\boldsymbol z^\star$ is an
$\varepsilon_0^t$-approximate minimizer of the trapping function
over $\mathcal E^t$.

Both separator reconstructions can be queried in one nonadaptive
batch, since their sites depend only on the current region and
the prescribed accuracy.
After selecting $\boldsymbol z^\star$, one exact query at this point
determines its reachability and supplies the next pivot value if
the algorithm moves.
Including this query, the iteration cost satisfies
\begin{equation}
\label{eq:app-main1-separator-cost}
 Q_t\le C_{d,p,\beta_p}\left[
  1+(r^t)^{d-1}
  \left(\frac{L_p}{\varepsilon_0^t}\right)^{\frac{d-1}{\nu}}
 \right].
\end{equation}
Each iteration therefore uses at most two adaptive rounds.

\subsection{Correctness}
\label{app:main1-correctness}

We maintain a boundary certificate as the region shrinks.
At termination, this certificate and the accuracy of the local
polynomial model will provide the hypotheses of
Lemma~\ref{lem:stationarity-from-model}.

\paragraph{Boundary invariant.}
The invariant is that $\boldsymbol x^t\in R_t$ and
\[
\phi_{\boldsymbol x^t,\varepsilon_t}(\boldsymbol y)
=
f(\boldsymbol y)-f(\boldsymbol x^t)
+\varepsilon_t\|\boldsymbol y-\boldsymbol x^t\|_2
>0
\]
for every $\boldsymbol y\in\partial R_t$ in the unconstrained case,
or every
$\boldsymbol y\in\partial_{\mathrm{art}}R_t
=\partial_{[0,1]^d}R_t$
in the cube-constrained case.
We call these the certified boundaries in the respective settings.
We first verify the initialization and then prove preservation
under one iteration.

We first check that the invariant holds at initialization.
In the unconstrained case, $f\ge0$.
If $f(\boldsymbol x^0)=0$, then $\boldsymbol x^0$ is a global
minimizer, and differentiability on $\mathbb R^d$ gives
$\nabla f(\boldsymbol x^0)=\boldsymbol0$.
The algorithm therefore terminates correctly after its initial
query. For the rest of the unconstrained analysis, assume
$f(\boldsymbol x^0)>0$.
The initial region is a full-dimensional cube centered at
$\boldsymbol x^0$, with side length
$r^0=16f(\boldsymbol x^0)/\varepsilon$.
For every $\boldsymbol y\in\partial R_0$,
$\|\boldsymbol y-\boldsymbol x^0\|_2
\ge\|\boldsymbol y-\boldsymbol x^0\|_\infty=r^0/2$.
Since $\varepsilon_0=\varepsilon/4$, it follows that
$\phi_{\boldsymbol x^0,\varepsilon_0}(\boldsymbol y)
\ge-f(\boldsymbol x^0)+(\varepsilon/4)(r^0/2)
=f(\boldsymbol x^0)>0$.
In the cube-constrained case, $R_0=[0,1]^d$ contains the prescribed
pivot and has no artificial facets, so its boundary certificate
is vacuous. Thus both initializations satisfy the invariant.

We next prove the induction step. Suppose that the invariant holds
at iteration $t$. We first record the geometry needed to preserve it. Let $j$ be the selected coordinate and $F_t$ the newly created facet.
\begin{fact}
$\boldsymbol x^{t+1}\in R_{t+1}$ and
$\operatorname{dist}(\boldsymbol x^{t+1},F_t)\ge r^t/6$.  
\end{fact}
\begin{proof}
The updated pivot $\boldsymbol x^{t+1}$ is either
$\boldsymbol x^t$ or $\boldsymbol z^\star\in\mathcal E^t$,
and hence belongs to $R_t$.
If $[\boldsymbol x^{t+1}]_j\ge a_j^t+r^t/2$, the algorithm retains
$[a_j^t+r^t/3,b_j^t]$ and $F_t=E_1^t$.
The updated pivot lies in this interval, at distance
$[\boldsymbol x^{t+1}]_j-a_j^t-r^t/3\ge r^t/6$ from $F_t$.
If $[\boldsymbol x^{t+1}]_j<a_j^t+r^t/2$, the retained interval is
$[a_j^t,b_j^t-r^t/3]$ and $F_t=E_2^t$.
The updated pivot again lies in the retained interval, now at
distance $a_j^t+2r^t/3-[\boldsymbol x^{t+1}]_j>r^t/6$ from $F_t$.
These are the Euclidean distances to the corresponding separators:
orthogonal projection leaves the other coordinates unchanged and
within their original intervals.
Consequently, $\boldsymbol x^{t+1}\in R_{t+1}$ and
$\operatorname{dist}(\boldsymbol x^{t+1},F_t)\ge r^t/6$.   
\end{proof}

All side lengths remain positive, and every boundary point of
$R_{t+1}$ lies on a retained old facet or on $F_t$.
In the constrained case, a retained physical facet remains physical,
whereas $F_t$ is artificial because its selected coordinate lies
strictly between $0$ and $1$.
Thus
\begin{equation}
\label{eq:app-main1-boundary-decomposition}
\begin{aligned}
 \partial R_{t+1}
 &\subseteq(\partial R_t\cap R_{t+1})\cup F_t,\\
 \partial_{\mathrm{art}}R_{t+1}
 &\subseteq(\partial_{\mathrm{art}}R_t\cap R_{t+1})\cup F_t.
\end{aligned}
\end{equation}
The second inclusion also covers intersections of artificial and
physical facets, as required by the definition of
$\partial_{\mathrm{art}}$.
It remains to prove positivity on the retained certified boundary
and on the new facet.

For every $\boldsymbol y\in R_t$, we have 
\begin{equation}
\label{eq:app-main1-pivot-transfer}
\begin{aligned}
&\phi_{\boldsymbol x^{t+1},\varepsilon_{t+1}}(\boldsymbol y)\\
=&f(\boldsymbol y)-f(\boldsymbol x^{t+1})
  +\varepsilon_{t+1}\|\boldsymbol y-\boldsymbol x^{t+1}\|_2
  \\
=&f(\boldsymbol y)-f(\boldsymbol x^{t+1})
  +\varepsilon_t\|\boldsymbol y-\boldsymbol x^{t+1}\|_2
  +\Delta_t\|\boldsymbol y-\boldsymbol x^{t+1}\|_2\\
\ge& f(\boldsymbol y)-f(\boldsymbol x^{t+1})
  +\varepsilon_t\|\boldsymbol y-\boldsymbol x^t\|_2
  -\varepsilon_t\|\boldsymbol x^{t+1}-\boldsymbol x^t\|_2
  +\Delta_t\|\boldsymbol y-\boldsymbol x^{t+1}\|_2\\
=&\bigl[f(\boldsymbol y)-f(\boldsymbol x^t)
  +\varepsilon_t\|\boldsymbol y-\boldsymbol x^t\|_2\bigr]-\bigl[f(\boldsymbol x^{t+1})-f(\boldsymbol x^t)
  +\varepsilon_t\|\boldsymbol x^{t+1}-\boldsymbol x^t\|_2\bigr]
  +\Delta_t\|\boldsymbol y-\boldsymbol x^{t+1}\|_2\\
=&\phi_{\boldsymbol x^t,\varepsilon_t}(\boldsymbol y)
  -\phi_{\boldsymbol x^t,\varepsilon_t}(\boldsymbol x^{t+1})
  +\Delta_t\|\boldsymbol y-\boldsymbol x^{t+1}\|_2.
\end{aligned}
\end{equation}
For a retained point $\boldsymbol y$ of the certified old boundary,
we have $\phi_{\boldsymbol x^t,\varepsilon_t}(\boldsymbol y)>0$
by induction,
$\phi_{\boldsymbol x^t,\varepsilon_t}(\boldsymbol x^{t+1})\le0$
by reachability, and
$\Delta_t\|\boldsymbol y-\boldsymbol x^{t+1}\|_2\ge0$.
Thus \eqref{eq:app-main1-pivot-transfer} yields
\[
\begin{aligned}
\phi_{\boldsymbol x^{t+1},\varepsilon_{t+1}}(\boldsymbol y)
&\ge \phi_{\boldsymbol x^t,\varepsilon_t}(\boldsymbol y)
-\phi_{\boldsymbol x^t,\varepsilon_t}(\boldsymbol x^{t+1})
+\Delta_t\|\boldsymbol y-\boldsymbol x^{t+1}\|_2\\
&\ge \phi_{\boldsymbol x^t,\varepsilon_t}(\boldsymbol y)>0.
\end{aligned}
\]
For every $\boldsymbol y\in F_t\subseteq\mathcal E^t$,
\begin{equation}
\label{eq:app-main1-new-facet-margin}
\begin{aligned}
\phi_{\boldsymbol x^{t+1},\varepsilon_{t+1}}(\boldsymbol y)
&\ge \phi_{\boldsymbol x^t,\varepsilon_t}(\boldsymbol y)
-\phi_{\boldsymbol x^t,\varepsilon_t}(\boldsymbol x^{t+1})
+\Delta_t\|\boldsymbol y-\boldsymbol x^{t+1}\|_2\\
&\ge \phi_{\boldsymbol x^t,\varepsilon_t}(\boldsymbol z^\star)
-\varepsilon_0^t
-\phi_{\boldsymbol x^t,\varepsilon_t}(\boldsymbol x^{t+1})
+\Delta_t\|\boldsymbol y-\boldsymbol x^{t+1}\|_2\\
&\ge -\varepsilon_0^t
+\Delta_t\|\boldsymbol y-\boldsymbol x^{t+1}\|_2\\
&\ge -\frac{r^t\Delta_t}{12}+\frac{r^t\Delta_t}{6}\\
&=\frac{r^t\Delta_t}{12}>0.
\end{aligned}
\end{equation}
The first inequality follows from
\eqref{eq:app-main1-pivot-transfer}, and the second from
\eqref{eq:app-main1-separator-accuracy}.
The third uses
$\phi_{\boldsymbol x^t,\varepsilon_t}(\boldsymbol x^{t+1})
=\min\{0,\phi_{\boldsymbol x^t,\varepsilon_t}(\boldsymbol z^\star)\}
\le\phi_{\boldsymbol x^t,\varepsilon_t}(\boldsymbol z^\star)$.
The fourth uses $\varepsilon_0^t=r^t\Delta_t/12$ and
$\|\boldsymbol y-\boldsymbol x^{t+1}\|_2\ge r^t/6$.
The final strict inequality holds because $r^t>0$ and $\Delta_t>0$.
Together with \eqref{eq:app-main1-boundary-decomposition} and
$\boldsymbol x^{t+1}\in R_{t+1}$, this completes the induction
in both settings.


\paragraph{Terminal output.}
We now verify the boundary and model conditions of
Lemma~\ref{lem:stationarity-from-model}, and then apply it to the
terminal minimization problem.

First, the parameter increments have a bounded total.
The specified choice of $\alpha$ satisfies $0<\alpha<1$, and each
nonnegative integer occurs exactly $d$ times as $\lfloor t/d\rfloor$.
Thus, for every generated index $t$,
\begin{equation}
\label{eq:app-main1-budget-sum}
\begin{aligned}
 \varepsilon_t
 &=\frac{\varepsilon}{4}+\sum_{\ell=0}^{t-1}\Delta_\ell\le\frac{\varepsilon}{4}
   +\frac{\varepsilon(1-\alpha)}{4d}
       d\sum_{m=0}^{\infty}\alpha^m
 =\frac{\varepsilon}{2}.
\end{aligned}
\end{equation}
Thus
$\phi_{\boldsymbol x^T,\varepsilon/2}(\boldsymbol y)
=\phi_{\boldsymbol x^T,\varepsilon_T}(\boldsymbol y)
+(\varepsilon/2-\varepsilon_T)
\|\boldsymbol y-\boldsymbol x^T\|_2>0$
on the certified terminal boundary.

Let $\boldsymbol c_T$ be the center of $R_T$ and
$\rho_T:=\max_{\boldsymbol y\in R_T}
\|\boldsymbol y-\boldsymbol c_T\|_2\le\sqrt d\,r^T/2$.
For $L_p>0$, Proposition~\ref{prop:local-reconstruction}
and the stopping condition $r^T\le r_\varepsilon$ give
\begin{equation}
\label{eq:trapping-terminal-model-error}
\begin{aligned}
\|\nabla P_T-\nabla f\|_{\infty,R_T}
&\le \frac32 a_{p-1,\beta_p}L_p\rho_T^{\nu-1}\\
&\le \frac32 a_{p-1,\beta_p}L_p
\left(\frac{\sqrt d\,r_\varepsilon}{2}\right)^{\nu-1}
=\frac{\varepsilon}{8}.
\end{aligned}
\end{equation}

Finally, the explicit terminal objective
$\Psi_T(\boldsymbol y)=P_T(\boldsymbol y)-P_T(\boldsymbol x^T)
+(3\varepsilon/4)\|\boldsymbol y-\boldsymbol x^T\|_2$
is continuous on compact $R_T$, so it attains a minimum.
Let $\boldsymbol z^T$ be the minimizer selected by the algorithm.
All hypotheses of Lemma~\ref{lem:stationarity-from-model} have
now been verified.
In the unconstrained case, the lemma gives
$\boldsymbol z^T\in\operatorname{int}(R_T)$ and
$\|\nabla f(\boldsymbol z^T)\|_2\le\varepsilon$.
In the cube-constrained case, it gives
$\boldsymbol z^T\in\operatorname{int}_{[0,1]^d}(R_T)$ and
$\|G_\gamma(\boldsymbol z^T)\|_2\le\varepsilon$ for every
$\gamma>0$.

\subsection{Complexity}
\label{app:main1-complexity}

\paragraph{Adaptive depth.}
Starting from a cube, every $d$ cuts reduce each side exactly once:
within each block, a cut side is shorter than every uncut side.
Hence
\begin{equation}
\label{eq:app-main1-side-schedule}
r^t=r^0\left(\frac{2}{3}\right)^{\left\lfloor\frac{t}{d}\right\rfloor},
\end{equation}
and the aspect ratio remains at most $\frac{3}{2}$.
The stopping condition $r^t\le r_\varepsilon$
first holds at $T=dM_\varepsilon$, where
$M_\varepsilon
:=\left\lceil
\log_{\frac{3}{2}}\max\left\{1,\frac{r^0}{r_\varepsilon}\right\}
\right\rceil$.

By Subsection~\ref{app:main1-separator}, each iteration uses
at most two adaptive rounds.
The initial query and terminal reconstruction each require at most
one additional round, giving a total adaptive depth of at most $2T+2$.
The two initializations give
\begin{equation}
\label{eq:app-main1-initial-terminal-ratios}
\frac{r^0}{r_\varepsilon}
=
\begin{cases}
8\sqrt d\,f(\boldsymbol x^0)
(12a_{p-1,\beta_p}L_p)^{\frac{1}{\nu-1}}
\varepsilon^{-\frac{\nu}{\nu-1}},
&\text{unconstrained},\\
\frac{\sqrt d}{2}
\left(\frac{12a_{p-1,\beta_p}L_p}{\varepsilon}\right)^{
\frac{1}{\nu-1}},
&\text{cube-constrained}.
\end{cases}
\end{equation}
Thus the adaptive depth is
$\mathcal O\!\left(1+\log\frac{1}{\varepsilon}\right)$
for fixed problem parameters and $0<\varepsilon\le1$.

\paragraph{Total query complexity.}
Write $q:=\frac{d-1}{\nu}$.
For every $t<T$, we have $r^t>r_\varepsilon$ and
$\varepsilon_0^t\le\varepsilon r^t$, so
\begin{equation}
\label{eq:app-main1-iteration-absorption}
\begin{aligned}
(r^t)^{d-1}\left(\frac{L_p}{\varepsilon_0^t}\right)^q
&\ge \left(\frac{L_p}{\varepsilon}\right)^q
(r^t)^{d-1-q}\\
&\ge \left(\frac{L_p}{\varepsilon}\right)^q
r_\varepsilon^{d-1-q}\\
&=\left(\frac{2}{\sqrt d}\right)^{d-1-q}
(12a_{p-1,\beta_p})^{-q}>0.
\end{aligned}
\end{equation}
Thus the constant term in \eqref{eq:app-main1-separator-cost}
can be absorbed into the reconstruction term.
Substituting the error schedule and
\eqref{eq:app-main1-side-schedule} gives
\begin{equation}
\label{eq:app-main1-scale-cost}
Q_t\le C_{d,p,\beta_p}
\left(\frac{L_p}{\varepsilon}\right)^q
(r^0)^{d-1-q}
\left(
\frac{\left(\frac{2}{3}\right)^{\nu-1}}{\alpha}
\right)^{q\left\lfloor\frac{t}{d}\right\rfloor}.
\end{equation}
Since $\alpha>\left(\frac{2}{3}\right)^{\nu-1}$ and
$T=dM_\varepsilon$, summing over all cuts and including the
initial and terminal queries yields
\begin{equation}
\label{eq:app-main1-total-with-iterations}
\begin{aligned}
Q
&\le 1+\binom{d+p}{p}+\sum_{t=0}^{T-1}Q_t\\
&\le 1+\binom{d+p}{p}
+C_{d,p,\beta_p}d
\left(\frac{L_p}{\varepsilon}\right)^q
(r^0)^{d-1-q}
\sum_{m=0}^{M_\varepsilon-1}
\left(
\frac{\left(\frac{2}{3}\right)^{\nu-1}}{\alpha}
\right)^{qm}\\
&\le 1+\binom{d+p}{p}
+C_{d,p,\beta_p}
\left(\frac{L_p}{\varepsilon}\right)^q
(r^0)^{d-1-q}.
\end{aligned}
\end{equation}

Substituting $r^0=\frac{16f(\boldsymbol x^0)}{\varepsilon}$
and $r^0=1$ proves \eqref{eq:main1-total-cost-sketch},
with respective rates
$\mathcal O(\varepsilon^{-(d-1)})$ and
$\mathcal O(\varepsilon^{-\frac{d-1}{\nu}})$
for fixed problem parameters.
This completes the proof of both parts of Theorem~\ref{the:main1}.

\section{Proofs in Section~\ref{sec:4}}

\subsection{Proof of Lemma~\ref{lem:randomized-line-integration}}
\label{app:4.1}

We use the Taylor constants $a_{r,\beta}$ from
Appendix~\ref{app:approximation} and the univariate interpolation
Lebesgue constants $\lambda_m$ from
Section~\ref{subsec:approximation-global-minimization}, with
$\lambda_0=1$.

\begin{proof}
The case $\ell=0$ is immediate. Otherwise, condition on
$\mathcal F$, fixing the endpoints and accuracy parameters,
and write $M:=L_p\ell^\nu$ for the remainder of the proof.

\noindent\textbf{Reduction to scalar integration.}
Write $w:=v-u$ and $g(t):=\langle\nabla f(u+tw),w\rangle$,
$t\in[0,1]$. The fundamental theorem of calculus gives
$\int_0^1g(t)\,dt=f(v)-f(u)$, and one gradient query evaluates
$g$ at any prescribed point. For $0\le r\le p-1$, the chain
rule gives $g^{(r)}(t)=D^{r+1}f(u+tw)[w^{\otimes(r+1)}]$.
Consequently,
\[
 |g^{(p-1)}(t)-g^{(p-1)}(s)|
 \le L_p\|(t-s)w\|_2^{\beta_p}\|w\|_2^p
 =M|t-s|^{\beta_p}.
\]

\noindent\textbf{Local polynomial models.}
Fix an integer $n\ge1$, set $h:=1/n$, and partition $[0,1]$
into $J_j:=[a_j,a_j+h]$, where $a_j:=(j-1)h$.
Let interpolant polynomial as 
\[P_j :=
\begin{cases}
\text{the degree-$(p-1)$ interpolant of $g$ at }
\left\{a_j+\frac{hi}{p-1}\right\}_{i=0}^{p-1},
& p\ge2,\\
g(a_j),
& p=1.
\end{cases}\]
Write $R_j:=g-P_j$ and
$C_p:=1+\lambda_{p-1}$. 
For $p\geq 2$, applying Lemma~\ref{lem:interpolation-error-decomposition}
with polynomial degree $p-1$ and
$\operatorname{diam}(J_j)=h$ gives
\begin{equation}
\label{eq:apx-line-local-residual}
 \|R_j\|_{\infty,J_j} =  \|g-P_j\|_{\infty,J_j}
 \le C_pa_{p-1,\beta_p}Mh^{\nu-1}
 \le C_pMh^{\nu-1},
\end{equation}
where $a_{r,\beta_p}=\prod_{i=1}^r(\beta_p+i)^{-1}\le1$.
For $p=1$, the H\"older condition gives
$|R_j(t)|\le M|t-a_j|^{\beta_p}$, which also implies
\eqref{eq:apx-line-local-residual}.

\noindent\textbf{Deterministic rule.}
Define $I_n^{\mathrm{det}}:=\sum_{j=1}^n\int_{J_j}P_j(t)\,dt$.
For $p\ge2$, integration of
\eqref{eq:apx-line-local-residual} yields
\[
 |I_n^{\mathrm{det}}-(f(v)-f(u))|
 \le\sum_{j=1}^n\int_{J_j}|R_j(t)|\,dt
 \le nh\,C_pa_{p-1,\beta_p}Mh^{\nu-1}
 =C_pa_{p-1,\beta_p}Mn^{-(\nu-1)}.
\]
For $p=1$, the rule is $I_n^{\mathrm{det}}=h\sum_jg(a_j)$.
Integrating the sharper residual bound gives error 
\begin{align*}
\bigl|I_n^{\mathrm{det}}-(f(v)-f(u))\bigr|
&=\left|
\sum_{j=1}^n\int_{a_j}^{a_j+h}
\bigl(g(a_j)-g(t)\bigr)\,dt
\right|\\
&\le\sum_{j=1}^n\int_{a_j}^{a_j+h}
|g(t)-g(a_j)|\,dt\\
&\le M\sum_{j=1}^n\int_{a_j}^{a_j+h}
(t-a_j)^{\beta_p}\,dt\\
&=M\sum_{j=1}^n
\left[\frac{(t-a_j)^{1+\beta_p}}{1+\beta_p}
\right]_{a_j}^{a_j+h}\\
&=\frac{nMh^{1+\beta_p}}{1+\beta_p}\\
&=\frac{M}{1+\beta_p}\,
n\left(\frac{1}{n}\right)^{1+\beta_p}
=\frac{Mn^{-\beta_p}}{1+\beta_p}.
\end{align*}
Thus, with
$K_{1,\beta_p}:=\frac{1}{1+\beta_p}$ and
$K_{p,\beta_p}:=C_pa_{p-1,\beta_p}$ for $p\ge2$, both cases
satisfy
\[|I_n^{\mathrm{det}}-(f(v)-f(u))|
\le K_{p,\beta_p}Mn^{-(\nu-1)}.\]
Choosing
$n=n_{\mathrm{det}}:=\max\left\{1,
\left\lceil
\left(\frac{K_{p,\beta_p}M}{\varepsilon_0}\right)^{
\frac{1}{\nu-1}}
\right\rceil\right\}$
proves the claimed deterministic accuracy.

\noindent\textbf{Randomized rule.}
Let $Z_1,\ldots,Z_n$ be independent uniform random variables on
$[0,1]$, independent of $\mathcal F$, and set
$U_j:=a_j+hZ_j$. Define random estimation as 
\begin{equation}
\label{eq:apx-line-integration-estimator}
 \widehat I_n
 :=\sum_{j=1}^n
 \left[\int_{J_j}P_j(t)\,dt+hR_j(U_j)\right].
\end{equation}
Conditionally on $\mathcal F$, the objective and interpolation
nodes, and hence $P_j$, are fixed. Since $U_j$ is uniform on
$J_j$, we have $\mathbb E[hR_j(U_j)\mid\mathcal F]
=\int_{J_j}R_j(t)\,dt$. Substituting this identity in
\eqref{eq:apx-line-integration-estimator} proves conditional
unbiasedness.

Set $Y_j:=hR_j(U_j)$. These variables are conditionally
independent, with range lengths at most $2C_pMh^\nu$ by
\eqref{eq:apx-line-local-residual}. Since $|R_j(t)|\le C_pMh^{\nu-1}$, the random variable
$hR_j(U_j)$ takes values in
$[-C_pMh^\nu,C_pMh^\nu]$.
Centering does not change its range length.
Therefore, using $h=1/n$,
\begin{align*}
\sum_{j=1}^n
\left[
h\left(\sup_{t\in J_j}R_j(t)-\inf_{t\in J_j}R_j(t)\right)
\right]^2
&\le \sum_{j=1}^n(2C_pMh^\nu)^2\\
&=4C_p^2M^2nh^{2\nu}\\
&=4C_p^2M^2n\left(\frac{1}{n}\right)^{2\nu}\\
&=4C_p^2M^2n^{1-2\nu}.
\end{align*}
Moreover, 
\[\widehat I_n-(f(v)-f(u))
=\sum_j(Y_j-\mathbb E[Y_j\mid\mathcal F]).\]
For $M>0$ and every positive $\mathcal F$-measurable threshold
$\xi$, conditional Hoeffding's inequality gives
\begin{equation}
\label{eq:apx-line-integration-tail}
 \mathbb P\!\left(
 |\widehat I_n-(f(v)-f(u))|>\xi\mid\mathcal F\right)
 \le2\exp\!\left(
 -\frac{\xi^2n^{2\nu-1}}{2C_p^2M^2}\right).
\end{equation}
Choose
\[
 n=n_{\mathrm{rand}}
 :=\max\left\{1,\left\lceil
 \left[2C_p^2\left(\frac{M}{\varepsilon_0}\right)^2
 \log\frac{2}{\vartheta}\right]^{1/(2\nu-1)}
 \right\rceil\right\}.
\]
At $\xi=\varepsilon_0$, the exponent in
\eqref{eq:apx-line-integration-tail} is at least
$\log(2/\vartheta)$, proving
\eqref{eq:apx-line-integration-accuracy}. If $M=0$, the residual
bound gives $R_j\equiv0$, so both rules are exact.

\noindent\textbf{Queries.}
For $p\ge2$, sharing interpolation endpoints gives
$n(p-1)+1$ deterministic queries and at most $pn+1$
randomized queries.
For $p=1$, both rules use $n$ queries, since the randomized
estimator simplifies to $h\sum_{j=1}^n g(U_j)$.
Polynomial integration requires no additional queries.
Substituting $n_{\mathrm{det}}$ and $n_{\mathrm{rand}}$
yields the claimed bounds, with constants depending only on $p$.

\noindent\textbf{Nonadaptivity.}
Conditional on $\mathcal F$, the sample counts are fixed,
and all query points are selected before observing any new
oracle answer. Each rule thus uses one nonadaptive batch.
Taking $\mathcal F$ to be trivial recovers the guarantees
for deterministic inputs.
\end{proof}

\subsection{Correctness}
\label{subsec:gradient-only-analysis}

We first prove correctness whenever all edge estimates meet
their prescribed tolerances, and then establish this condition
for the deterministic and randomized implementations.
Write
$\phi_t:=\phi_{\boldsymbol x^t,\varepsilon_t}$ and
$\delta_t:=\varepsilon_0^t=\frac{r^t\Delta_t}{12}$.
Let $\mathcal H_t$ be the sigma-field generated by the transcript
before iteration $t$, and define
\begin{equation}
\label{eq:edge-success-event}
 \mathcal G_t:=
 \bigcap_{\ell=0}^{J_t}
 \bigcap_{e=(\boldsymbol u,\boldsymbol v)\in\mathcal A_{t,\ell}}
 \left\{
 \left|\widehat I_e-f(\boldsymbol v)+f(\boldsymbol u)\right|
 \le\xi_{t,\ell}
 \right\}.
\end{equation}

\begin{fact}
\label{fact:gradient-only-edge-accuracy}
In deterministic mode, $\mathcal G_t$ always holds.
In randomized mode,
$\mathbb P(\mathcal G_t\mid\mathcal H_t)\ge1-\rho_t$.
\end{fact}

\begin{proof}
The deterministic quadrature guarantee implies $\mathcal G_t$.
In randomized mode,
\[
\begin{aligned}
 \mathbb P(\mathcal G_t^c\mid\mathcal H_t)
 &\le
 \sum_{\ell=0}^{J_t}
 \sum_{e=(\boldsymbol u,\boldsymbol v)\in\mathcal A_{t,\ell}}
 \mathbb P\!\left(
 \left|\widehat I_e-f(\boldsymbol v)+f(\boldsymbol u)\right|
 >\xi_{t,\ell}
 \,\middle|\,\mathcal H_t
 \right)\\
 &\le
 \sum_{\ell=0}^{J_t}
 |\mathcal A_{t,\ell}|
 \frac{\rho_t}{|\mathcal A_t|}
 =\rho_t,
\end{aligned}
\]
where the first inequality is the conditional union bound,
and the second uses the edgewise failure guarantees.
\end{proof}

\begin{fact}[\bf Reconstruction and update certificates]
\label{lem:uniform-gradient-reconstruction}
\label{lem:gradient-only-update-certificate}
On $\mathcal G_t$, $ \|\widehat F_t-F_t\|_{\infty,\mathcal E^t}
 =
 \|\widehat\phi_t-\phi_t\|_{\infty,\mathcal E^t}
 \le\frac{\delta_t}{4},$
and $\widehat{\boldsymbol z}^t$ is a
$\delta_t$-approximate minimizer of $\phi_t$ on $\mathcal E^t$.
Moreover,
\begin{equation}
\label{eq:gradient-only-update-certificates}
\begin{aligned}
 \phi_t(\boldsymbol x^{t+1})&\le0,\\
 \phi_t(\boldsymbol y)-\phi_t(\boldsymbol x^{t+1})
 &\ge-\delta_t,
 \qquad\forall \boldsymbol y\in\mathcal E^t.
\end{aligned}
\end{equation}
\end{fact}

\begin{proof}
Assume $\mathcal G_t$ holds.
For a finest-grid node $\boldsymbol z$, let
$\mathcal P_t(\boldsymbol z)$ denote its root path.
Writing $e=(\boldsymbol u_e,\boldsymbol v_e)$, we obtain
\[
\begin{aligned}
 \left|\widehat F_t(\boldsymbol z)-F_t(\boldsymbol z)\right|
 &=
 \left|
 \sum_{e\in\mathcal P_t(\boldsymbol z)}
 \bigl(\widehat I_e-f(\boldsymbol v_e)+f(\boldsymbol u_e)\bigr)
 \right|\\
 &\le
 \sum_{e\in\mathcal P_t(\boldsymbol z)}
 \left|\widehat I_e-f(\boldsymbol v_e)+f(\boldsymbol u_e)\right|\\
 &\le\sum_{\ell=0}^{J_t}\xi_{t,\ell}
 =\frac{\delta_t}{8\Lambda_{d-1,p}},
\end{aligned}
\]
where the equality follows by telescoping, and the last
inequality uses $\mathcal G_t$ and the fact that each root path
contains at most one edge from each creation level.
Consequently,
\[
\begin{aligned}
 \|\widehat\phi_t-\phi_t\|_{\infty,\mathcal E^t}
 &=
 \|\widehat F_t-F_t\|_{\infty,\mathcal E^t}\\
 &\le
 A_{d-1,p,\beta_p}L_p
 \left(\frac{\sqrt{d-1}\,r^t}{2^{J_t}}\right)^\nu
 +\Lambda_{d-1,p}
  \frac{\delta_t}{8\Lambda_{d-1,p}}\\
 &\le\frac{\delta_t}{8}+\frac{\delta_t}{8}
 =\frac{\delta_t}{4},
\end{aligned}
\]
where the first inequality follows from the approximation lemma
with cell diameter at most
$\frac{\sqrt{d-1}\,r^t}{2^{J_t}}$
and nodal error at most
$\frac{\delta_t}{8\Lambda_{d-1,p}}$,
and the second uses Eq.~
\eqref{eq:gradient-only-refinement-level}, which gives
$A_{d-1,p,\beta_p}L_p
\left(\frac{\sqrt{d-1}\,r^t}{2^{J_t}}\right)^\nu
\le\frac{\varepsilon_0^t}{8}
=\frac{\delta_t}{8}$.

For every $\boldsymbol y\in\mathcal E^t$,
\begin{equation}
\label{eq:gradient-only-separator-suboptimality}
\begin{aligned}
 \phi_t(\widehat{\boldsymbol z}^t)
 &\le
 \widehat\phi_t(\widehat{\boldsymbol z}^t)
 +\frac{\delta_t}{4}\\
 &\le
 \min_{\mathcal E^t}\widehat\phi_t
 +\frac{\delta_t}{2}+\frac{\delta_t}{4}\\
 &\le
 \widehat\phi_t(\boldsymbol y)+\frac{3\delta_t}{4}\\
 &\le\phi_t(\boldsymbol y)+\delta_t,
\end{aligned}
\end{equation}
where the first and fourth inequalities use the uniform error
bound, and the second uses the surrogate minimization tolerance.
This proves the claimed approximate optimality.

If $\widehat\phi_t(\widehat{\boldsymbol z}^t)
\le-\frac{\delta_t}{4}$, then
$\boldsymbol x^{t+1}=\widehat{\boldsymbol z}^t$, and
\[
\begin{aligned}
 \phi_t(\boldsymbol x^{t+1})
 &\le
 \widehat\phi_t(\widehat{\boldsymbol z}^t)
 +\frac{\delta_t}{4}
 \le0,\\
 \phi_t(\boldsymbol y)-\phi_t(\boldsymbol x^{t+1})
 &=\phi_t(\boldsymbol y)
   -\phi_t(\widehat{\boldsymbol z}^t)
 \ge-\delta_t
 \qquad(\boldsymbol y\in\mathcal E^t),
\end{aligned}
\]
where the first line uses the acceptance condition,
and the second uses
\eqref{eq:gradient-only-separator-suboptimality}.

Otherwise, $\boldsymbol x^{t+1}=\boldsymbol x^t$, so
$\phi_t(\boldsymbol x^{t+1})=0$, and
\begin{equation}
\label{eq:gradient-only-retention-certificate}
\begin{aligned}
 \phi_t(\boldsymbol y)-\phi_t(\boldsymbol x^{t+1})
 &=\phi_t(\boldsymbol y)\\
 &\ge\widehat\phi_t(\boldsymbol y)-\frac{\delta_t}{4}\\
 &\ge
 \widehat\phi_t(\widehat{\boldsymbol z}^t)
 -\frac{\delta_t}{2}-\frac{\delta_t}{4}\\
 &>-\delta_t,
 \qquad \boldsymbol y\in\mathcal E^t,
\end{aligned}
\end{equation}
where the inequalities use the uniform error bound,
the surrogate minimization tolerance, and the rejection
condition, respectively.
\end{proof}

\paragraph{Preservation of the trapping invariant.}
Let $\Gamma_t:=\partial R_t$ in the unconstrained setting and
$\Gamma_t:=\partial_{[0,1]^d}R_t$ in the constrained setting.
We prove
\begin{equation}
\label{eq:gradient-only-trapping-invariant}
 \boldsymbol x^t\in R_t,
 \qquad
 \phi_t(\boldsymbol y)>0
 \quad(\boldsymbol y\in\Gamma_t)
\end{equation}
whenever $\mathcal G_0,\ldots,\mathcal G_{t-1}$ hold.

Initially, $\boldsymbol x^0\in R_0$ in both settings.
In the unconstrained case, for every
$\boldsymbol y\in\partial R_0$,
\[
\begin{aligned}
 \phi_0(\boldsymbol y)
 &=f(\boldsymbol y)-f(\boldsymbol x^0)
   +\frac{\varepsilon}{4}
      \|\boldsymbol y-\boldsymbol x^0\|_2\\
 &\ge-\Delta+\frac{\varepsilon}{4}\frac{r^0}{2}\\
 &=-\Delta+\frac{\varepsilon}{4}
              \frac{8\Delta}{\varepsilon}
 =\Delta>0,
\end{aligned}
\]
where the inequality uses the gap bound and the geometry of
the initial cube.
In the constrained case, $R_0=[0,1]^d$, so its boundary
relative to $[0,1]^d$ is empty:
$\Gamma_0=\partial_{[0,1]^d}R_0=\varnothing$.

Suppose that Eq.~\eqref{eq:gradient-only-trapping-invariant}
holds at iteration $t$ and that $\mathcal G_t$ holds.
Let $j$ be the selected coordinate and let
$\mathcal F_t\subseteq\mathcal E^t$ denote the newly created facet.
The cutting rule retains $\boldsymbol x^{t+1}$ and gives
\[
\begin{aligned}
 \operatorname{dist}(\boldsymbol x^{t+1},\mathcal F_t)
 &=
 \begin{cases}
 [\boldsymbol x^{t+1}]_j-a_j^t-\frac{r^t}{3},
 &[\boldsymbol x^{t+1}]_j\ge a_j^t+\frac{r^t}{2},\\[2mm]
 a_j^t+\frac{2r^t}{3}-[\boldsymbol x^{t+1}]_j,
 &[\boldsymbol x^{t+1}]_j<a_j^t+\frac{r^t}{2}
 \end{cases}\\
 &\ge\frac{r^t}{6},
\end{aligned}
\]
where $\mathcal F_t=E_1^t$ in the first case and
$\mathcal F_t=E_2^t$ in the second.
Thus
\[
 \boldsymbol x^{t+1}\in R_{t+1},
 \qquad
 \Gamma_{t+1}
 \subseteq(\Gamma_t\cap R_{t+1})\cup\mathcal F_t.
\]
In the constrained setting, the new facet is artificial,
whereas every retained physical facet remains physical.

For every $\boldsymbol y\in R_t$,
\begin{equation}
\label{eq:gradient-only-inherited-potential}
\begin{aligned}
 \phi_{t+1}(\boldsymbol y)
 &=f(\boldsymbol y)-f(\boldsymbol x^{t+1})
   +(\varepsilon_t+\Delta_t)
      \|\boldsymbol y-\boldsymbol x^{t+1}\|_2\\
 &\ge
 f(\boldsymbol y)-f(\boldsymbol x^{t+1})
 +\varepsilon_t\|\boldsymbol y-\boldsymbol x^t\|_2\\
 &\qquad
 -\varepsilon_t\|\boldsymbol x^{t+1}-\boldsymbol x^t\|_2
 +\Delta_t\|\boldsymbol y-\boldsymbol x^{t+1}\|_2\\
 &=\phi_t(\boldsymbol y)-\phi_t(\boldsymbol x^{t+1})
   +\Delta_t\|\boldsymbol y-\boldsymbol x^{t+1}\|_2,
\end{aligned}
\end{equation}
where the inequality is the triangle inequality.

For $\boldsymbol y\in\Gamma_t\cap R_{t+1}$,
\[
\begin{aligned}
 \phi_{t+1}(\boldsymbol y)
 &\ge
 \phi_t(\boldsymbol y)-\phi_t(\boldsymbol x^{t+1})
 +\Delta_t\|\boldsymbol y-\boldsymbol x^{t+1}\|_2\\
 &\ge\phi_t(\boldsymbol y)>0,
\end{aligned}
\]
where the second inequality uses
$\phi_t(\boldsymbol x^{t+1})\le0$ and $\Delta_t>0$,
and the strict inequality is the induction hypothesis.
For $\boldsymbol y\in\mathcal F_t$,
\[
\begin{aligned}
 \phi_{t+1}(\boldsymbol y)
 &\ge
 \phi_t(\boldsymbol y)-\phi_t(\boldsymbol x^{t+1})
 +\Delta_t\|\boldsymbol y-\boldsymbol x^{t+1}\|_2\\
 &\ge-\delta_t+\Delta_t\frac{r^t}{6}\\
 &=-\frac{r^t\Delta_t}{12}
   +\frac{r^t\Delta_t}{6}
 =\delta_t>0,
\end{aligned}
\]
where the second inequality uses
Eq.~\eqref{eq:gradient-only-update-certificates}
and the distance bound.
This proves Eq.~\eqref{eq:gradient-only-trapping-invariant}
at iteration $t+1$.

\begin{lemma}[\bf Gradient-only terminal output]
\label{lem:gradient-only-terminal-output}
Suppose that the trapping invariant holds at termination.
Then the terminal procedure terminates after finitely many
offline operations and returns a point
$\boldsymbol z^T\in R_T$ satisfying
\[
\begin{cases}
 \|\nabla f(\boldsymbol z^T)\|_2\le\varepsilon,
 &\text{in the unconstrained case},\\
 \|G_\gamma(\boldsymbol z^T)\|_2\le\varepsilon
 \quad\text{for every }\gamma>0,
 &\text{in the cube-constrained case}.
\end{cases}
\]
\end{lemma}

\begin{proof}
Write $R:=R_T$, $\boldsymbol x:=\boldsymbol x^T$,
$P:=P_T$, $\Psi:=\Psi_T$, and
$\zeta:=\zeta_T=\frac{\varepsilon r^T}{16}>0$.

\noindent\textbf{Finite termination of the terminal procedure.}
The coefficients of $P$ give computable constants
\[
 K\ge
 \max\left\{1,\sup_R\|\nabla P\|_2+\frac{3\varepsilon}{4}\right\},
 \qquad
 H\ge
 \max\left\{1,\sup_R\|\nabla^2P\|_{\mathrm{op}}\right\}.
\]
Choose a finite grid $\mathcal Q\subset R$ with covering radius
at most $\frac{\zeta}{K}$, and set
$\boldsymbol y_0\in
\operatorname*{argmin}_{\mathcal Q\cup\{\boldsymbol x\}}\Psi$.
Then
\[
\begin{aligned}
 \Psi(\boldsymbol y_0)&\le\Psi(\boldsymbol x)=0,\\
 \Psi(\boldsymbol y_0)-\min_R\Psi
 &\le\min_{\mathcal Q}\Psi-\min_R\Psi
 \le K\frac{\zeta}{K}
 =\zeta,
\end{aligned}
\]
where the second line uses the $K$-Lipschitz continuity of
$\Psi$ and the covering radius of $\mathcal Q$.

Starting from $\boldsymbol y_0$, define
\begin{equation}
\label{eq:gradient-only-offline-proximal-step}
 \boldsymbol y_{k+1}:=
 \operatorname*{argmin}_{\boldsymbol y\in R}
 \left\{
 \langle\nabla P(\boldsymbol y_k),
          \boldsymbol y-\boldsymbol y_k\rangle
 +H\|\boldsymbol y-\boldsymbol y_k\|_2^2
 +\frac{3\varepsilon}{4}
       \|\boldsymbol y-\boldsymbol x\|_2
 \right\}.
\end{equation}
Each auxiliary problem has a unique minimizer, computable
by finite algebraic computation in the exact-real model.%
\footnote{
Completing the square reduces the problem to minimizing
$H\|\boldsymbol w-\boldsymbol c\|_2^2+a\|\boldsymbol w\|_2$
over $B:=R-\boldsymbol x$, where
$a:=\frac{3\varepsilon}{4}$ and
$\boldsymbol c:=\boldsymbol y_k-\boldsymbol x
-\frac{\nabla P(\boldsymbol y_k)}{2H}$.
For a nonzero minimizer, convex optimality gives
$\boldsymbol w=\operatorname{proj}_B(\tau\boldsymbol c)$,
where
$\tau=\frac{2H\|\boldsymbol w\|_2}
{2H\|\boldsymbol w\|_2+a}\in(0,1)$.
For each clipping pattern,
$\|\boldsymbol w\|_2^2=D+C\tau^2$
with known $C,D\ge0$, so
$a^2\tau^2=4H^2(1-\tau)^2(D+C\tau^2)$.
Enumerate the finitely many clipping patterns and real roots
of these polynomial equations, retaining those satisfying
the unsquared equation, $0<\tau<1$, and the clipping inequalities.
Test $\boldsymbol w=0$ separately using
$\operatorname{dist}(2H\boldsymbol c,N_B(\boldsymbol0))\le a$.
These tests identify the unique minimizer.
}
Set $\boldsymbol d_k:=\boldsymbol y_{k+1}-\boldsymbol y_k$.
Then
\[
\begin{aligned}
 \Psi(\boldsymbol y_{k+1})-\Psi(\boldsymbol y_k)
 &\le
 \langle\nabla P(\boldsymbol y_k),\boldsymbol d_k\rangle
 +\frac{H}{2}\|\boldsymbol d_k\|_2^2\\
 &\qquad
 +\frac{3\varepsilon}{4}
 \bigl(
 \|\boldsymbol y_{k+1}-\boldsymbol x\|_2
 -\|\boldsymbol y_k-\boldsymbol x\|_2
 \bigr)\\
 &\le-\frac{H}{2}\|\boldsymbol d_k\|_2^2,
\end{aligned}
\]
where the first inequality uses the $H$-Lipschitz continuity
of $\nabla P$, and the second compares the auxiliary minimizer
with $\boldsymbol y_k$.
Consequently,
\[
 \Psi(\boldsymbol y_k)\le0,
 \qquad
 \Psi(\boldsymbol y_k)-\min_R\Psi\le\zeta
 \qquad(k\ge0).
\]

Convex optimality supplies
$\boldsymbol v_k\in
\partial\|\cdot-\boldsymbol x\|_2(\boldsymbol y_{k+1})$
and $\boldsymbol n_k\in N_R(\boldsymbol y_{k+1})$ such that
\[
 \nabla P(\boldsymbol y_k)+2H\boldsymbol d_k
 +\frac{3\varepsilon}{4}\boldsymbol v_k+\boldsymbol n_k
 =\boldsymbol0.
\]
Hence
\[
\begin{aligned}
 &\left\|
 \nabla P(\boldsymbol y_{k+1})
 +\frac{3\varepsilon}{4}\boldsymbol v_k+\boldsymbol n_k
 \right\|_2\\
 &\quad=
 \|\nabla P(\boldsymbol y_{k+1})
   -\nabla P(\boldsymbol y_k)-2H\boldsymbol d_k\|_2\\
 &\quad\le
 \|\nabla P(\boldsymbol y_{k+1})
   -\nabla P(\boldsymbol y_k)\|_2
 +2H\|\boldsymbol d_k\|_2\\
 &\quad\le3H\|\boldsymbol d_k\|_2,
\end{aligned}
\]
where the last inequality uses the Hessian bound for $P$.

Stop at the first step with
$\|\boldsymbol d_k\|_2\le\frac{\varepsilon}{24H}$.
If the first $N\ge1$ steps were all nonterminating, then
\[
\begin{aligned}
 \frac{N\varepsilon^2}{1152H}
 &<
 \frac{H}{2}\sum_{k=0}^{N-1}\|\boldsymbol d_k\|_2^2\\
 &\le
 \sum_{k=0}^{N-1}
 \bigl(\Psi(\boldsymbol y_k)-\Psi(\boldsymbol y_{k+1})\bigr)\\
 &=\Psi(\boldsymbol y_0)-\Psi(\boldsymbol y_N)\\
 &\le\Psi(\boldsymbol y_0)-\min_R\Psi
 \le\zeta,
\end{aligned}
\]
where the first inequality uses
$\|\boldsymbol d_k\|_2>\frac{\varepsilon}{24H}$,
and the second is the descent bound.
Thus termination occurs within
$\left\lfloor\frac{1152H\zeta}{\varepsilon^2}\right\rfloor+1$
steps.

At the terminating step, return
$\boldsymbol z^T:=\boldsymbol y_{k+1}$ and set
$\boldsymbol v:=\boldsymbol v_k$,
$\boldsymbol n:=\boldsymbol n_k$.
The preceding estimates give
\[
\begin{aligned}
 \Psi(\boldsymbol z^T)&\le0,\\
 \Psi(\boldsymbol z^T)-\min_R\Psi&\le\zeta,\\
 \left\|
 \nabla P(\boldsymbol z^T)
 +\frac{3\varepsilon}{4}\boldsymbol v+\boldsymbol n
 \right\|_2
 &\le3H\|\boldsymbol d_k\|_2
 \le\frac{\varepsilon}{8}.
\end{aligned}
\]
This implements the prescribed terminal step using finitely
many offline operations and no additional oracle queries
after constructing $P$.

\noindent\textbf{Correctness of the output.}
The terminal reconstruction satisfies
\[
\begin{aligned}
 \|\nabla P-\nabla f\|_{\infty,R}
 &\le
 c_{p,\beta_p}L_p
 \left(\max_{\boldsymbol y\in R}
       \|\boldsymbol y-\boldsymbol c_T\|_2\right)^{\nu-1}\\
 &\le
 c_{p,\beta_p}L_p
 \left(\frac{\sqrt d\,r^T}{2}\right)^{\nu-1}\\
 &\le\frac{\varepsilon}{8},
\end{aligned}
\]
where the inequalities follow from
Corollary~\ref{lem:gradient-only-scalar-reconstruction},
the geometry of $R$, and
Eq.~\eqref{eq:gradient-only-terminal-radius}, respectively.

For every $\boldsymbol y\in R$,
\[
\begin{aligned}
 &\left|P(\boldsymbol y)-P(\boldsymbol x)
       -f(\boldsymbol y)+f(\boldsymbol x)\right|\\
 &\quad=
 \left|
 \int_0^1
 \left\langle
 (\nabla P-\nabla f)
   (\boldsymbol x+u(\boldsymbol y-\boldsymbol x)),
 \boldsymbol y-\boldsymbol x
 \right\rangle\,du
 \right|\\
 &\quad\le
 \|\nabla P-\nabla f\|_{\infty,R}
 \|\boldsymbol y-\boldsymbol x\|_2\\
 &\quad\le
 \frac{\varepsilon}{8}\|\boldsymbol y-\boldsymbol x\|_2,
\end{aligned}
\]
where the segment lies in $R$ by convexity.
For $\boldsymbol y\in\Gamma_T$, this gives
\begin{equation}
\label{eq:gradient-only-terminal-barrier}
\begin{aligned}
 \Psi(\boldsymbol y)
 &=P(\boldsymbol y)-P(\boldsymbol x)
   +\frac{3\varepsilon}{4}
      \|\boldsymbol y-\boldsymbol x\|_2\\
 &\ge
 f(\boldsymbol y)-f(\boldsymbol x)
 +\frac{5\varepsilon}{8}
      \|\boldsymbol y-\boldsymbol x\|_2\\
 &=\phi_T(\boldsymbol y)
   +\left(\frac{5\varepsilon}{8}-\varepsilon_T\right)
      \|\boldsymbol y-\boldsymbol x\|_2\\
 &\ge
 \phi_T(\boldsymbol y)
 +\frac{\varepsilon}{8}\|\boldsymbol y-\boldsymbol x\|_2
 >0,
\end{aligned}
\end{equation}
where the last line uses
$\varepsilon_T\le\frac{\varepsilon}{2}$ and the trapping
invariant.
Since $\Psi(\boldsymbol z^T)\le0$, we obtain
$\boldsymbol z^T\notin\Gamma_T$.

Moreover,
\begin{equation}
\label{eq:gradient-only-terminal-stationarity}
\begin{aligned}
 \|\nabla f(\boldsymbol z^T)+\boldsymbol n\|_2
 &\le
 \|\nabla f(\boldsymbol z^T)-\nabla P(\boldsymbol z^T)\|_2\\
 &\qquad+
 \left\|\nabla P(\boldsymbol z^T)
       +\frac{3\varepsilon}{4}\boldsymbol v
       +\boldsymbol n\right\|_2
 +\frac{3\varepsilon}{4}\|\boldsymbol v\|_2\\
 &\le
 \frac{\varepsilon}{8}
 +\frac{\varepsilon}{8}
 +\frac{3\varepsilon}{4}
 =\varepsilon,
\end{aligned}
\end{equation}
where the second inequality uses the model error,
the terminal residual bound, and
$\|\boldsymbol v\|_2\le1$.

In the unconstrained case,
$\boldsymbol z^T\notin\partial R$ implies
$\boldsymbol z^T\in\operatorname{int}R$.
Hence $\boldsymbol n=\boldsymbol0$, and
\eqref{eq:gradient-only-terminal-stationarity} gives
$\|\nabla f(\boldsymbol z^T)\|_2\le\varepsilon$.

In the cube-constrained case,
$\boldsymbol z^T\notin\Gamma_T$ implies
$\boldsymbol z^T\in\operatorname{int}_{[0,1]^d}R$.
Therefore
$N_R(\boldsymbol z^T)=N_{[0,1]^d}(\boldsymbol z^T)$,
as established in Section~\ref{sec:3}.
For every $\gamma>0$,
\[
\begin{aligned}
 \|G_\gamma(\boldsymbol z^T)\|_2
 &=
 \frac{1}{\gamma}
 \left\|
 \boldsymbol z^T-
 \operatorname{proj}_{[0,1]^d}
       (\boldsymbol z^T-\gamma\nabla f(\boldsymbol z^T))
 \right\|_2\\
 &=
 \frac{1}{\gamma}
 \left\|
 \operatorname{proj}_{[0,1]^d}
       (\boldsymbol z^T+\gamma\boldsymbol n)
 -
 \operatorname{proj}_{[0,1]^d}
       (\boldsymbol z^T-\gamma\nabla f(\boldsymbol z^T))
 \right\|_2\\
 &\le\|\nabla f(\boldsymbol z^T)+\boldsymbol n\|_2\\
 &\le\varepsilon,
\end{aligned}
\]
where the second equality follows from
$\boldsymbol n\in N_{[0,1]^d}(\boldsymbol z^T)$,
and the first inequality is nonexpansiveness of projection.
\end{proof}

\paragraph{Deterministic and randomized guarantees.}
Starting from a cube, each block of $d$ longest-side cuts
multiplies every side length by $\frac{2}{3}$.
Thus the stopping rule is reached after a finite deterministic
number $T$ of cuts.

In deterministic mode, every $\mathcal G_t$ holds.
The trapping invariant and
Lemma~\ref{lem:gradient-only-terminal-output}
therefore guarantee a correct output.

In randomized mode,
\[
\begin{aligned}
 \mathbb P\!\left(
 \bigcup_{t=0}^{T-1}\mathcal G_t^c
 \right)
 &\le\sum_{t=0}^{T-1}\mathbb P(\mathcal G_t^c)\\
 &=\sum_{t=0}^{T-1}
   \mathbb E\!\left[
   \mathbb P(\mathcal G_t^c\mid\mathcal H_t)
   \right]\\
 &\le\sum_{t=0}^{T-1}\rho_t\\
 &\le\sum_{t=0}^{\infty}
       \frac{6\rho}{\pi^2(t+1)^2}
 =\rho,
\end{aligned}
\]
where the first inequality is the union bound, the equality
is the tower property, and the next inequality follows from
Lemma~\ref{lem:uniform-gradient-reconstruction}.
On $\bigcap_{t=0}^{T-1}\mathcal G_t$, the trapping invariant
holds through termination, so
Lemma~\ref{lem:gradient-only-terminal-output}
certifies the output.
The randomized implementation is therefore correct with
probability at least $1-\rho$.

\subsection{Gradient-Query Complexity}
\label{subsubsec:gradient-only-complexity}

Write $s:=d-1$ and
\[
 B_t:=\frac{L_p(r^t)^\nu}{\delta_t}
     =\frac{12L_p(r^t)^{\nu-1}}{\Delta_t}.
\]
Throughout this subsection, $C$ denotes a positive constant
depending only on $d,p,\beta_p$, whose value may change between
occurrences. All query bounds hold for every run, irrespective
of reconstruction success.

\paragraph{Adaptive depth.}
Starting from a cube, every block of $d$ longest-side cuts
reduces each side exactly once: within a block, a cut side
is shorter than every uncut side. Hence
\begin{equation}
\label{eq:gradient-only-side-schedule}
 r^t=r^0
 \left(\frac{2}{3}\right)^{\left\lfloor\frac{t}{d}\right\rfloor}.
\end{equation}
For $L_p>0$, the stopping condition
$r^t\le r_\varepsilon^\nabla$ therefore first holds at
\[
 T=d\left\lceil
 \log_{\frac{3}{2}}
 \max\left\{1,\frac{r^0}{r_\varepsilon^\nabla}\right\}
 \right\rceil.
\]
Moreover,
\[
\begin{aligned}
 \left(\frac{r^0}{r_\varepsilon^\nabla}\right)^{\nu-1}
 &=
 8c_{p,\beta_p}
 \left(\frac{\sqrt d}{2}\right)^{\nu-1}
 \frac{L_p(r^0)^{\nu-1}}{\varepsilon} = \frac{c_{p,\beta_p}(1-\alpha)}{6d}
 \left(\frac{\sqrt d}{2}\right)^{\nu-1}B_0,
\end{aligned}
\]
where the equalities use
\eqref{eq:gradient-only-terminal-radius} and
$\Delta_0=\frac{\varepsilon(1-\alpha)}{4d}$, respectively.
Consequently,
\begin{equation}
\label{eq:gradient-only-iteration-bound}
\begin{aligned}
 T
 &\le
 d+\frac{d}{\log\frac{3}{2}}
 \log\max\left\{1,\frac{r^0}{r_\varepsilon^\nabla}\right\}\le C[1+\log(1+B_0)].
\end{aligned}
\end{equation}

Initialization requires no oracle queries.
Each outer iteration uses one nonadaptive batch, and the
terminal reconstruction uses one additional batch.
Thus the adaptive depth is at most
$T+1=O_{d,p,\beta_p}(1+\log(1+B_0))$.

\begin{fact}[\bf Geometry of the shared-root tree]
\label{lem:shared-root-tree-geometry}
The edge counts satisfy
\[
\begin{aligned}
 |\mathcal A_{t,0}|
 &=2(p+1)^{d-1},\\
 |\mathcal A_{t,\ell}|
 &=2\bigl[
 (p2^\ell+1)^{d-1}-(p2^{\ell-1}+1)^{d-1}
 \bigr],
 \qquad 1\le\ell\le J_t,\\
 |\mathcal A_t|
 &=2(p2^{J_t}+1)^{d-1}.
\end{aligned}
\]
For every $0\le\ell\le J_t$ and
$e=(\boldsymbol u,\boldsymbol v)\in\mathcal A_{t,\ell}$,
the edge length satisfies $ h_e:=\|\boldsymbol v-\boldsymbol u\|_2
 \le\frac{C_{d,p}r^t}{2^\ell}.$
Moreover, for every exponent $a>0$,
\[
 \sum_{e\in\mathcal A_{t,\ell}}h_e^a
 \le C_{d,p,a}(r^t)^a2^{(d-1-a)\ell}.
\]
\end{fact}

\begin{proof}
At level $\ell$, each free coordinate of a separator has
$p2^\ell+1$ grid points. Since each separator has $d-1$
free coordinates, its grid contains $(p2^\ell+1)^{d-1}$ nodes.
Each newly introduced node contributes exactly one parent edge,
including a zero-length edge if it coincides with the root.
Therefore,
\[
\begin{aligned}
 |\mathcal A_{t,0}|
 &=2(p+1)^{d-1},\\
 |\mathcal A_{t,\ell}|
 &=2\bigl[
 (p2^\ell+1)^{d-1}-(p2^{\ell-1}+1)^{d-1}
 \bigr],
 \qquad 1\le\ell\le J_t,
\end{aligned}
\]
where the factor $2$ accounts for the two separators,
and the difference counts newly introduced nodes because
the grids are nested.
Consequently,
\[
\begin{aligned}
 |\mathcal A_t|=\sum_{\ell=0}^{J_t}|\mathcal A_{t,\ell}|=2(p+1)^{d-1}
 +2\sum_{\ell=1}^{J_t}
 \bigl[
 (p2^\ell+1)^{d-1}-(p2^{\ell-1}+1)^{d-1}
 \bigr]=2(p2^{J_t}+1)^{d-1},
\end{aligned}
\]
where the first equality uses the disjointness of the
levelwise edge sets, and the last equality follows by telescoping.

For a level-zero edge $e=(\boldsymbol x^t,\boldsymbol v)$,
both endpoints lie in
$R_t=\prod_{i=1}^d[a_i^t,b_i^t]$, so
\[
\begin{aligned}
 h_e
 &=\|\boldsymbol v-\boldsymbol x^t\|_2\le
 \left(\sum_{i=1}^d(b_i^t-a_i^t)^2\right)^{\frac{1}{2}}\le\sqrt d\,r^t,
\end{aligned}
\]
where the first inequality bounds each coordinate difference
by the corresponding side length, and the second uses
$b_i^t-a_i^t\le r^t$.

Now let $\ell\ge1$ and
$e=(\boldsymbol u,\boldsymbol v)\in\mathcal A_{t,\ell}$.
Let $j$ be the fixed coordinate of the separator.
For each $i\ne j$, write the child coordinate as
\[
 [\boldsymbol v]_i
 =a_i^t+\frac{k_i(b_i^t-a_i^t)}{p2^\ell}.
\]
The parent rule assigns index
$\left\lfloor\frac{k_i}{2}\right\rfloor$ at level $\ell-1$.
Hence
\[
\begin{aligned}
 |[\boldsymbol v]_i-[\boldsymbol u]_i|
 &=
 \left|
 \frac{k_i}{p2^\ell}
 -
 \frac{\left\lfloor\frac{k_i}{2}\right\rfloor}
      {p2^{\ell-1}}
 \right|(b_i^t-a_i^t)=
 \frac{k_i-2\left\lfloor\frac{k_i}{2}\right\rfloor}
      {p2^\ell}(b_i^t-a_i^t)\le\frac{b_i^t-a_i^t}{p2^\ell}\le\frac{r^t}{p2^\ell},
\end{aligned}
\]
where the numerator in the second line belongs to $\{0,1\}$,
and the last inequality uses the definition of $r^t$.
Since the parent and child lie on the same separator,
$[\boldsymbol u]_j=[\boldsymbol v]_j$, and therefore
\[
\begin{aligned}
 h_e
 &=
 \left(
 \sum_{i\ne j}|[\boldsymbol v]_i-[\boldsymbol u]_i|^2
 \right)^{\frac{1}{2}}\le
 \left(
 (d-1)\left(\frac{r^t}{p2^\ell}\right)^2
 \right)^{\frac{1}{2}}=\frac{\sqrt{d-1}\,r^t}{p2^\ell}\le\frac{\sqrt d\,r^t}{2^\ell},
\end{aligned}
\]
where the first inequality uses the coordinate bounds,
and the last uses $p\ge1$.
Together with the level-zero estimate, this proves the
edge-length bound with $C_{d,p}=\sqrt d$.

For every $0\le\ell\le J_t$,
\[
\begin{aligned}
 |\mathcal A_{t,\ell}|
 &\le2(p2^\ell+1)^{d-1}\le2(p+1)^{d-1}2^{(d-1)\ell},
\end{aligned}
\]
where the first inequality counts all level-$\ell$ grid nodes,
and the second uses
$p2^\ell+1\le(p+1)2^\ell$.
Thus, for every $a>0$,
\[
\begin{aligned}
 \sum_{e\in\mathcal A_{t,\ell}}h_e^a
 &\le
 |\mathcal A_{t,\ell}|
 \left(\frac{\sqrt d\,r^t}{2^\ell}\right)^a\\
 &\le
 2(p+1)^{d-1}2^{(d-1)\ell}
 \frac{d^{\frac{a}{2}}(r^t)^a}{2^{a\ell}}\\
 &=
 2(p+1)^{d-1}d^{\frac{a}{2}}
 (r^t)^a2^{(d-1-a)\ell},
\end{aligned}
\]
where the first inequality uses the edge-length bound
and the monotonicity of $u\mapsto u^a$ on $[0,\infty)$,
and the second uses the edge-count bound.
The claimed estimate follows with
$C_{d,p,a}=2(p+1)^{d-1}d^{\frac{a}{2}}$.
\end{proof}

Recall
\[
 \kappa_{\mathrm{det}}=\frac{1}{\nu-1},
 \qquad
 \kappa_{\mathrm{rand}}=\frac{2}{2\nu-1},
 \qquad
 \chi_{\mathsf m}
 =\max\left\{\frac{d-1}{\nu},\kappa_{\mathsf m}\right\}.
\]
For $B\ge0$, define
\begin{equation}
\label{eq:gradient-only-deterministic-cost-function}
 \mathcal D(B):=
 \begin{cases}
 1+B^{\chi_{\mathrm{det}}},
 &d-1\ne\frac{\nu}{\nu-1},\\[1mm]
 1+B^{\frac{1}{\nu-1}}
       [1+\log(1+B)]^{\frac{\nu}{\nu-1}},
 &d-1=\frac{\nu}{\nu-1}.
 \end{cases}
\end{equation}

\begin{lemma}[\bf Gradient queries in one iteration]
\label{lem:gradient-queries-one-iteration}
For $C=C_{d,p,\beta_p}\ge2$, the issued query counts satisfy
\begin{equation}
\label{eq:gradient-only-one-iteration-cost}
\begin{aligned}
 Q_t^{\mathrm{det}}
 &\le C\mathcal D(B_t),\\
 Q_t^{\mathrm{rand}}
 &\le
 C(1+B_t^{\chi_{\mathrm{rand}}})
 \left[
 \log\frac{C(1+B_t)^{\frac{d-1}{\nu}}}{\rho_t}
 \right]^{\frac{1}{2\nu-1}}.
\end{aligned}
\end{equation}
These bounds hold irrespective of reconstruction success.
\end{lemma}

\begin{proof}
We first bound the refinement level $J_t$ and the number of edges $|\mathcal A_t|$.
If $B_t=0$, the refinement condition already holds at level
zero, so $J_t=0$.
If $J_t\ge1$, its minimality gives
\[
\begin{aligned}
 \frac{\delta_t}{8}
 &<
 A_{d-1,p,\beta_p}L_p
 \left(\frac{\sqrt{d-1}\,r^t}{2^{J_t-1}}\right)^\nu\\
 &=
 \frac{
 A_{d-1,p,\beta_p}(d-1)^{\frac{\nu}{2}}
 \delta_t B_t}
 {2^{\nu(J_t-1)}},
\end{aligned}
\]
where the inequality says that level $J_t-1$ fails
\eqref{eq:gradient-only-refinement-level}, and the equality
uses $L_p(r^t)^\nu=\delta_t B_t$.
Since $\delta_t>0$, rearranging yields
\[
 2^{J_t}
 <
 2\left(
 8A_{d-1,p,\beta_p}(d-1)^{\frac{\nu}{2}}B_t
 \right)^{\frac{1}{\nu}}.
\]
Including the case $J_t=0$, we obtain
\begin{equation}
\label{eq:gradient-only-refinement-cost-bounds}
\begin{aligned}
 2^{J_t}
 &\le C\max\left\{1,B_t^{\frac{1}{\nu}}\right\},\\
 J_t+1
 &\le
 1+\frac{\log C}{\log2}
 +\frac{\log\max\{1,B_t\}}{\nu\log2}\le C[1+\log(1+B_t)],\\
 |\mathcal A_t|
 &=2(p2^{J_t}+1)^{d-1}\le2(p+1)^{d-1}2^{(d-1)J_t}\le C(1+B_t^{\frac{d-1}{\nu}}).
\end{aligned}
\end{equation}
Here the bound on $J_t+1$ follows by taking logarithms
and using $\max\{1,B_t\}\le1+B_t$.
The edge-count identity is
Fact~\ref{lem:shared-root-tree-geometry} and
the following inequalities use
$p2^{J_t}+1\le(p+1)2^{J_t}$ and the bound on $2^{J_t}$.

Fix a mode and write $\kappa:=\kappa_{\mathsf m}$.
Set $ \mathcal L_t:=
 \begin{cases}
 1,&\mathsf m=\mathrm{det},\\
 \left[
 \log\frac{2|\mathcal A_t|}{\rho_t}
 \right]^{\frac{1}{2\nu-1}},
 &\mathsf m=\mathrm{rand}.
 \end{cases}$
In both modes, $\mathcal L_t\ge1$ while
in randomized mode this follows from
$|\mathcal A_t|\ge4$ and $0<\rho_t<1$.
Lemma~\ref{lem:randomized-line-integration} bounds the cost
of an edge $e\in\mathcal A_{t,\ell}$ by
\[
 C_p\left[
 1+\mathcal L_t
 \left(\frac{L_ph_e^\nu}{\xi_{t,\ell}}\right)^\kappa
 \right].
\]
Recall from Eq.~\eqref{eq:edge-error-budget} that
\[
 S_t=
 \sum_{\ell=0}^{J_t}
 2^{\frac{(d-1-\nu\kappa)\ell}{1+\kappa}},
 \qquad
 w_{t,\ell}
 =
 \frac{2^{\frac{(d-1-\nu\kappa)\ell}{1+\kappa}}}{S_t},
 \qquad
 \xi_{t,\ell}
 =
 \frac{\delta_t w_{t,\ell}}{8\Lambda_{d-1,p}}.
\]
Summing the edge costs gives
\begin{equation}
\label{eq:gradient-only-weighted-level-sum}
\begin{aligned}
 Q_t
 &\le
 C|\mathcal A_t|
 +C\mathcal L_t
 \left(\frac{L_p}{\delta_t}\right)^\kappa
 \sum_{\ell=0}^{J_t}w_{t,\ell}^{-\kappa}
 \sum_{e\in\mathcal A_{t,\ell}}h_e^{\nu\kappa}\\
 &\le
 C|\mathcal A_t|
 +C\mathcal L_t B_t^\kappa
 \sum_{\ell=0}^{J_t}
 2^{(d-1-\nu\kappa)\ell}w_{t,\ell}^{-\kappa}\\
 &=
 C|\mathcal A_t|
 +C\mathcal L_t B_t^\kappa S_t^\kappa
 \sum_{\ell=0}^{J_t}
 2^{\frac{(d-1-\nu\kappa)\ell}{1+\kappa}}\\
 &=
 C|\mathcal A_t|
 +C\mathcal L_t B_t^\kappa S_t^{1+\kappa}.
\end{aligned}
\end{equation}
The first inequality substitutes the edge tolerances into
the quadrature bound.
The second uses
Fact~\ref{lem:shared-root-tree-geometry} with $a=\nu\kappa$
and
$\left(\frac{L_p}{\delta_t}\right)^\kappa
(r^t)^{\nu\kappa}=B_t^\kappa$.
The equalities substitute $w_{t,\ell}$ and then use the
definition of $S_t$.

If $d-1<\nu\kappa$, then
\[
\begin{aligned}
 S_t
 &\le
 \sum_{\ell=0}^{\infty}
 2^{\frac{(d-1-\nu\kappa)\ell}{1+\kappa}}=
 \frac{1}{
 1-2^{\frac{d-1-\nu\kappa}{1+\kappa}}},\\
 B_t^\kappa S_t^{1+\kappa}
 &\le CB_t^\kappa,
\end{aligned}
\]
where the infinite series converges because
$2^{\frac{d-1-\nu\kappa}{1+\kappa}}\in(0,1)$.

If $d-1>\nu\kappa$, the finite geometric sum satisfies
\[
\begin{aligned}
 S_t
 &=
 \frac{
 2^{\frac{(d-1-\nu\kappa)(J_t+1)}{1+\kappa}}-1}
 {2^{\frac{d-1-\nu\kappa}{1+\kappa}}-1}\le
 \frac{
 2^{\frac{d-1-\nu\kappa}{1+\kappa}}}
 {2^{\frac{d-1-\nu\kappa}{1+\kappa}}-1}
 2^{\frac{(d-1-\nu\kappa)J_t}{1+\kappa}}\le C2^{\frac{(d-1-\nu\kappa)J_t}{1+\kappa}},
\end{aligned}
\]
where the denominator is positive and depends only on the
fixed parameters.
Consequently,
\[
\begin{aligned}
 B_t^\kappa S_t^{1+\kappa}
 &\le
 CB_t^\kappa2^{(d-1-\nu\kappa)J_t}\\
 &\le
 CB_t^\kappa
 \max\left\{1,B_t^{\frac{d-1-\nu\kappa}{\nu}}\right\}\\
 &=C\max\left\{
 B_t^\kappa,B_t^{\frac{d-1}{\nu}}
 \right\}\\
 &\le C(1+B_t^{\frac{d-1}{\nu}}),
\end{aligned}
\]
where the second inequality uses
\eqref{eq:gradient-only-refinement-cost-bounds}.
The equality uses
$\kappa+\frac{d-1-\nu\kappa}{\nu}=\frac{d-1}{\nu}$.
For the last inequality, use $B_t^\kappa\le1$ when
$B_t\le1$, and
$B_t^\kappa\le B_t^{\frac{d-1}{\nu}}$ when $B_t\ge1$.

If $d-1=\nu\kappa$, every summand defining $S_t$ equals one.
Thus
\[
\begin{aligned}
 S_t&=J_t+1,\\
 B_t^\kappa S_t^{1+\kappa}
 &=B_t^\kappa(J_t+1)^{1+\kappa}\\
 &\le
 CB_t^\kappa[1+\log(1+B_t)]^{1+\kappa},
\end{aligned}
\]
where the inequality uses
\eqref{eq:gradient-only-refinement-cost-bounds}.

Combining these estimates with
\eqref{eq:gradient-only-weighted-level-sum} gives
\[
 Q_t\le C\mathcal L_t
 \begin{cases}
 1+B_t^{\chi_{\mathsf m}},
 &d-1\ne\nu\kappa,\\
 1+B_t^\kappa[1+\log(1+B_t)]^{1+\kappa},
 &d-1=\nu\kappa.
 \end{cases}
\]
Here the edge-count term is absorbed using
$\mathcal L_t\ge1$ and
$B_t^{\frac{d-1}{\nu}}\le1+B_t^{\chi_{\mathsf m}}$
in the noncritical cases.
In the critical case,
$\frac{d-1}{\nu}=\kappa$ and
$[1+\log(1+B_t)]^{1+\kappa}\ge1$.

For deterministic integration,
\[
 \mathcal L_t=1,
 \qquad
 \nu\kappa_{\mathrm{det}}
 =1+\kappa_{\mathrm{det}}
 =\frac{\nu}{\nu-1}.
\]
Substitution into the preceding bound gives
$Q_t^{\mathrm{det}}\le C\mathcal D(B_t)$.

For randomized integration,
$\nu\kappa_{\mathrm{rand}}
=\frac{2\nu}{2\nu-1}\in(1,2)$.
Since $d-1$ is a positive integer, the critical case cannot occur.
Therefore,
\[
 Q_t^{\mathrm{rand}}
 \le C\mathcal L_t(1+B_t^{\chi_{\mathrm{rand}}}).
\]
Finally,
\[
\begin{aligned}
 \mathcal L_t
 &=
 \left[
 \log\frac{2|\mathcal A_t|}{\rho_t}
 \right]^{\frac{1}{2\nu-1}}\le
 \left[
 \log\frac{C(1+B_t^{\frac{d-1}{\nu}})}{\rho_t}
 \right]^{\frac{1}{2\nu-1}}\le
 \left[
 \log\frac{C(1+B_t)^{\frac{d-1}{\nu}}}{\rho_t}
 \right]^{\frac{1}{2\nu-1}},
\end{aligned}
\]
where the first inequality uses the edge-count bound,
and the second uses
$1+B_t^{\frac{d-1}{\nu}}
\le2(1+B_t)^{\frac{d-1}{\nu}}$,
absorbing the factor $2$ into $C$.
This proves \eqref{eq:gradient-only-one-iteration-cost}.
\end{proof}

\paragraph{Total query complexity.}
Set
\begin{equation}
\label{eq:gradient-only-geometric-parameters}
 \theta:=
 \frac{\left(\frac{2}{3}\right)^{\nu-1}}{\alpha},
 \qquad
 B_0=
 \frac{48dL_p(r^0)^{\nu-1}}{(1-\alpha)\varepsilon}.
\end{equation}
The parameter choice
$\left(\frac{2}{3}\right)^{\nu-1}<\alpha<1$
ensures $0<\theta<1$.

\begin{proposition}[\bf Total gradient-query complexity]
\label{prop:total-gradient-query-complexity}
For $C=C_{d,p,\beta_p}\ge2$, the total query counts,
including the terminal reconstruction, satisfy
\begin{equation}
\label{eq:gradient-only-total-explicit}
\begin{aligned}
 Q^{\mathrm{det}}
 &\le N_{d,p}^\nabla+C\mathcal D(B_0),\\
 Q^{\mathrm{rand}}
 &\le N_{d,p}^\nabla+
 C(1+B_0^{\chi_{\mathrm{rand}}})
 \left[
 \log\frac{C(1+B_0)}{\rho}
 \right]^{\frac{1}{2\nu-1}}.
\end{aligned}
\end{equation}
These bounds hold for every run.
\end{proposition}

\begin{proof}

The radius and error schedules give
\[
\begin{aligned}
 B_t
 &=
 \frac{12L_p(r^t)^{\nu-1}}{\Delta_t}\\
 &=
 \frac{48dL_p(r^0)^{\nu-1}}{(1-\alpha)\varepsilon}
 \left(
 \frac{\left(\frac{2}{3}\right)^{\nu-1}}{\alpha}
 \right)^{\left\lfloor\frac{t}{d}\right\rfloor}\\
 &=B_0\theta^{\left\lfloor\frac{t}{d}\right\rfloor},
\end{aligned}
\]
where the second equality substitutes
\eqref{eq:gradient-only-side-schedule} and
$\Delta_t=
\frac{\varepsilon(1-\alpha)}{4d}
\alpha^{\left\lfloor\frac{t}{d}\right\rfloor}$.

For every fixed exponent $a>0$,
\begin{equation}
\label{eq:gradient-only-geometric-sum}
\begin{aligned}
 \sum_{t=0}^{T-1}B_t^a
 &=
 B_0^a\sum_{t=0}^{T-1}
 \theta^{a\left\lfloor\frac{t}{d}\right\rfloor}\\
 &\le
 dB_0^a\sum_{m=0}^{\infty}\theta^{am}\\
 &=\frac{dB_0^a}{1-\theta^a},
\end{aligned}
\end{equation}
where the inequality uses the fact that each value of
$\left\lfloor\frac{t}{d}\right\rfloor$ occurs at most $d$ times,
and the last equality uses $0<\theta^a<1$.

We also have
\begin{equation}
\label{eq:gradient-only-constant-cost-absorption}
 T\le C_a(1+B_0^a),
\end{equation}
where $C_a$ may additionally depend on $a$.
For $B_0\le1$, this follows directly from
\eqref{eq:gradient-only-iteration-bound}.
For $B_0>1$,
\[
\begin{aligned}
 \log(1+B_0)
 &\le\log2+\log B_0\\
 &=\log2+\int_1^{B_0}\frac{1}{u}\,du\\
 &\le\log2+\int_1^{B_0}u^{a-1}\,du\\
 &=\log2+\frac{B_0^a-1}{a}.
\end{aligned}
\]
Here the first inequality uses $1+B_0\le2B_0$,
and the second uses $u^a\ge1$ for $u\ge1$.
Substituting this estimate into
$T\le C[1+\log(1+B_0)]$
proves \eqref{eq:gradient-only-constant-cost-absorption}.

\noindent\textbf{Deterministic integration.}
If $d-1\ne\frac{\nu}{\nu-1}$, then
\[
\begin{aligned}
 Q^{\mathrm{det}}
 &\le
 N_{d,p}^\nabla+
 C\sum_{t=0}^{T-1}(1+B_t^{\chi_{\mathrm{det}}})\\
 &\le
 N_{d,p}^\nabla+C(T+B_0^{\chi_{\mathrm{det}}})\\
 &\le
 N_{d,p}^\nabla+C(1+B_0^{\chi_{\mathrm{det}}})\\
 &=N_{d,p}^\nabla+C\mathcal D(B_0),
\end{aligned}
\]
where the first inequality sums
Lemma~\ref{lem:gradient-queries-one-iteration}
and adds the terminal batch.
The second uses
\eqref{eq:gradient-only-geometric-sum}
with $a=\chi_{\mathrm{det}}$;
the third uses
\eqref{eq:gradient-only-constant-cost-absorption}
with the same exponent.

If $d-1=\frac{\nu}{\nu-1}$, then
\[
\begin{aligned}
 Q^{\mathrm{det}}
 &\le
 N_{d,p}^\nabla+
 C\left[
 T+\sum_{t=0}^{T-1}
 B_t^{\frac{1}{\nu-1}}
 [1+\log(1+B_t)]^{\frac{\nu}{\nu-1}}
 \right]\\
 &\le
 N_{d,p}^\nabla+
 C\left[
 T+
 [1+\log(1+B_0)]^{\frac{\nu}{\nu-1}}
 \sum_{t=0}^{T-1}B_t^{\frac{1}{\nu-1}}
 \right]\\
 &\le
 N_{d,p}^\nabla+
 C\left[
 T+B_0^{\frac{1}{\nu-1}}
 [1+\log(1+B_0)]^{\frac{\nu}{\nu-1}}
 \right]\\
 &\le N_{d,p}^\nabla+C\mathcal D(B_0),
\end{aligned}
\]
where the first inequality uses the critical case of
Lemma~\ref{lem:gradient-queries-one-iteration}.
The second uses $B_t\le B_0$.
The third uses
\eqref{eq:gradient-only-geometric-sum}
with $a=\frac{1}{\nu-1}$.
The last uses
\eqref{eq:gradient-only-constant-cost-absorption}
with the same exponent and
$[1+\log(1+B_0)]^{\frac{\nu}{\nu-1}}\ge1$.

\noindent\textbf{Randomized integration.}
For every $t<T$,
\[
\begin{aligned}
 \log\frac{C(1+B_t)^{\frac{d-1}{\nu}}}{\rho_t}
 &\le
 \log\frac{C(1+B_0)^{\frac{d-1}{\nu}}(t+1)^2}{\rho}\\
 &\le
 \log\frac{
 C(1+B_0)^{\frac{d-1}{\nu}}
 [1+\log(1+B_0)]^2}{\rho}\\
 &\le
 \log\frac{C(1+B_0)^{\frac{d-1}{\nu}+2}}{\rho}\\
 &=
 \log\frac{C}{\rho}
 +\left(\frac{d-1}{\nu}+2\right)\log(1+B_0)\\
 &\le
 \left(\frac{d-1}{\nu}+2\right)
 \log\frac{C(1+B_0)}{\rho}.
\end{aligned}
\]
The first inequality uses $B_t\le B_0$ and
$\rho_t=\frac{6\rho}{\pi^2(t+1)^2}$.
The second uses
$t+1\le T\le C[1+\log(1+B_0)]$.
The third uses $\log(1+B_0)\le B_0$.
The last uses
$\frac{d-1}{\nu}+2\ge1$ and
$\log\frac{C}{\rho}>0$.

Absorbing the fixed prefactor into $C$, we obtain
\[
\begin{aligned}
 Q^{\mathrm{rand}}
 &\le
 N_{d,p}^\nabla+
 C\left[
 \log\frac{C(1+B_0)}{\rho}
 \right]^{\frac{1}{2\nu-1}}
 \sum_{t=0}^{T-1}(1+B_t^{\chi_{\mathrm{rand}}})\\
 &\le
 N_{d,p}^\nabla+
 C\left[
 \log\frac{C(1+B_0)}{\rho}
 \right]^{\frac{1}{2\nu-1}}
 (T+B_0^{\chi_{\mathrm{rand}}})\\
 &\le
 N_{d,p}^\nabla+
 C(1+B_0^{\chi_{\mathrm{rand}}})
 \left[
 \log\frac{C(1+B_0)}{\rho}
 \right]^{\frac{1}{2\nu-1}},
\end{aligned}
\]
where the first inequality sums the single-iteration bound
and uses the uniform logarithmic estimate.
The second uses
\eqref{eq:gradient-only-geometric-sum}
with $a=\chi_{\mathrm{rand}}$,
and the third uses
\eqref{eq:gradient-only-constant-cost-absorption}
with the same exponent.
\end{proof}

\begin{proof}[Proof of Theorem~\ref{the:main_gradient}]
Substituting $r^0=\frac{16\Delta}{\varepsilon}$
in the unconstrained case and $r^0=1$ in the
cube-constrained case gives
\[
 B_0=
 \begin{cases}
 \displaystyle
 \frac{48d\,16^{\nu-1}}{1-\alpha}
 L_p\Delta^{\nu-1}\varepsilon^{-\nu},
 &\text{unconstrained},\\[2mm]
 \displaystyle
 \frac{48d}{1-\alpha}\frac{L_p}{\varepsilon},
 &\text{cube-constrained}.
 \end{cases}
\]
Thus $B_0$ is a fixed multiple of $\mathcal B$.
Proposition~\ref{prop:total-gradient-query-complexity}
gives the parameter-dependent query bounds, including the
critical deterministic logarithmic factor.

For fixed problem parameters, $B_0$ scales as
$\varepsilon^{-\nu}$ in the unconstrained case.
Hence the accuracy exponents are
\[
\begin{aligned}
 \nu\chi_{\mathrm{det}}
 &=
 \nu\max\left\{
 \frac{d-1}{\nu},\frac{1}{\nu-1}
 \right\}
 =
 \max\left\{d-1,\frac{\nu}{\nu-1}\right\},\\
 \nu\chi_{\mathrm{rand}}
 &=
 \nu\max\left\{
 \frac{d-1}{\nu},\frac{2}{2\nu-1}
 \right\}
 =
 \max\left\{d-1,\frac{2\nu}{2\nu-1}\right\}.
\end{aligned}
\]
In the cube-constrained case, $B_0$ scales as
$\varepsilon^{-1}$, giving exponents
\[
 \chi_{\mathrm{det}}
 =\max\left\{
 \frac{d-1}{\nu},\frac{1}{\nu-1}
 \right\},
 \qquad
 \chi_{\mathrm{rand}}
 =\max\left\{
 \frac{d-1}{\nu},\frac{2}{2\nu-1}
 \right\}.
\]
The terminal cost $N_{d,p}^\nabla$ is constant for fixed $d,p$.

Finally, the same substitutions into
$T+1\le C[1+\log(1+B_0)]$
give $O(1+\log\frac{1}{\varepsilon})$ adaptive rounds
for fixed parameters and $0<\varepsilon\le1$.
Deterministic correctness and randomized success probability
at least $1-\rho$ follow from
Subsection~\ref{subsec:gradient-only-analysis}.
\end{proof}

\section{Proofs and Additional Material for Section~\ref{sec:trapping-or-descent}}
\label{app:trapping-or-descent}

We prove the framework, reconstruction, and complexity statements
used in Section~\ref{sec:trapping-or-descent}.
The polynomial reconstruction estimates themselves are supplied
by Appendix~\ref{app:approximation}.
In the geometric proofs, point and vector variables are written
without boldface for brevity.

\subsection{The Trapping Framework and Finite Localization}
\label{app:trapping-framework-proofs}



We use the notation of Section~5.1 and assume that $f$ is
continuously differentiable on a neighborhood of $Q$.
Since every $y\in Q$ with $\|y-x\|_2<r$ belongs to
$\operatorname{int}_XQ$, we have $\partial_XQ\subseteq\Gamma$.
If $Q\ne X$, a segment from $x$ to a point of $X\setminus Q$
intersects $\Gamma$ by convexity of $X$.
Thus $\Gamma=\varnothing$ implies $Q=X$ and
$\partial_XQ=\varnothing$.

\begin{lemma}[\bf Trapping-or-descent dichotomy]
Suppose that $\Gamma\ne\varnothing$ and that a continuous
surrogate $\widehat\phi_x$ on $\Gamma$ satisfies
$\|\widehat\phi_x-\phi_{x,\frac{\varepsilon}{4}}\|_{\infty,\Gamma}
\le\frac{\varepsilon r}{16}$.
Let $\widehat y\in\Gamma$ satisfy
$\widehat\phi_x(\widehat y)
\le\min_\Gamma\widehat\phi_x+\frac{\varepsilon r}{32}$.
Then the following alternatives are exhaustive and disjoint.
\begin{enumerate}
\item If $\widehat\phi_x(\widehat y)>\frac{\varepsilon r}{32}$, then
every point of $\partial_XQ$ is $\frac{3\varepsilon}{8}$-unreachable
from $x$.
\item If $\widehat\phi_x(\widehat y)\le\frac{\varepsilon r}{32}$, then
$f(\widehat y)\le f(x)-\frac{5\varepsilon r}{32}
\le f(x)-\frac{\varepsilon}{8}$.
\end{enumerate}
If $\Gamma=\varnothing$, the trapping certificate in the first
alternative holds vacuously.
\end{lemma}

\begin{proof}[Proof of Lemma~\ref{lem:box-trapping-or-descent}]
If $\Gamma=\varnothing$, then $\partial_XQ=\varnothing$,
so the trapping claim holds vacuously. Assume $\Gamma\ne\varnothing$.

If $\widehat\phi_x(\widehat y)>\frac{\varepsilon r}{32}$,
then, for every $y\in\Gamma$,
\[
\begin{aligned}
 \phi_{x,\frac{3\varepsilon}{8}}(y)
 &=\phi_{x,\frac{\varepsilon}{4}}(y)
   +\frac{\varepsilon r}{8}\ge\widehat\phi_x(y)+\frac{\varepsilon r}{16}\ge\widehat\phi_x(\widehat y)+\frac{\varepsilon r}{32}>\frac{\varepsilon r}{16}>0,
\end{aligned}
\]
where the equality uses $\|y-x\|_2=r$, and the inequalities
use the surrogate error, approximate minimality, and the
threshold condition, respectively.
Since $\partial_XQ\subseteq\Gamma$, the first alternative follows.

Otherwise, $\widehat y\in\Gamma\subseteq X$ satisfies
\[
\begin{aligned}
 f(\widehat y)-f(x)
 &=\phi_{x,\frac{\varepsilon}{4}}(\widehat y)
   -\frac{\varepsilon r}{4}\le\widehat\phi_x(\widehat y)-\frac{3\varepsilon r}{16}\le-\frac{5\varepsilon r}{32}
 \le-\frac{\varepsilon r}{8},
\end{aligned}
\]
where the equality uses $\|\widehat y-x\|_2=r$,
and the first two inequalities use the surrogate error and
$\widehat\phi_x(\widehat y)\le\frac{\varepsilon r}{32}$.
This proves the second alternative.
\end{proof}

\begin{lemma}[\bf Correctness of terminal localization]
Let $Q:=Q_x(r)$.
Suppose that $\phi_{x,\frac{3\varepsilon}{8}}>0$ on
$\partial_XQ$, and that the known scalar polynomial $P_x$
satisfies
$\|\nabla P_x-\nabla f\|_{\infty,Q}
\le\frac{\varepsilon}{16}$.
Then Algorithm~\ref{alg:trap-localize} terminates after
finitely many offline operations, without further oracle queries,
and returns a point $z\in Q$ satisfying
\[
\begin{cases}
 \|\nabla f(z)\|_2\le\varepsilon,
 &X=\mathbb R^d,\\
 \|G_\gamma(z)\|_2\le\varepsilon
 \quad\text{for every }\gamma>0,
 &X=[0,1]^d.
\end{cases}
\]
\end{lemma}

\begin{proof}[Proof of Lemma~\ref{lem:trapped-ball-localization}]
Write $P:=P_x$, $\Psi(y):=P(y)-P(x)
+\frac{\varepsilon}{2}\|y-x\|_2$, and
$\zeta:=\frac{\varepsilon r}{32}$.

\noindent\textbf{Finite termination.}
Termwise bounds on the derivatives of the known polynomial
over $x+[-r,r]^d$ give computable constants
\[
 M_P\ge\sup_{y\in Q}\|\nabla P(y)\|_2+\frac{\varepsilon}{2},
 \qquad
 L_P\ge\sup_{y\in Q}\|\nabla^2P(y)\|_{\mathrm{op}}.
\]
Thus $M_P>0$, $\Psi$ is $M_P$-Lipschitz on $Q$, and
$\nabla P$ is $L_P$-Lipschitz on $Q$.
Set $\tau:=\frac{1}{L_P+1}$, as in
Algorithm~\ref{alg:trap-localize}.

A finite $\frac{\zeta}{M_P}$-net $\mathcal V\subseteq Q$
containing $x$ can be constructed by projecting a Cartesian
grid in $x+[-r,r]^d$ onto $Q$ and adding $x$.
Indeed, choose coordinate spacing at most
$\frac{\zeta}{M_P\sqrt d}$.
For every $y\in Q$, some grid point $a$ then satisfies
\[
 \|\Pi_Q(a)-y\|_2
 =\|\Pi_Q(a)-\Pi_Q(y)\|_2
 \le\|a-y\|_2
 \le\frac{\zeta}{M_P},
\]
where the inequality uses nonexpansiveness of projection.
Let $z^0\in\operatorname*{argmin}_{\mathcal V}\Psi$.
If $y^\star\in\operatorname*{argmin}_Q\Psi$, choose
$a^\star\in\mathcal V$ within distance
$\frac{\zeta}{M_P}$ of $y^\star$. Then
\[
\begin{aligned}
 \Psi(z^0)&\le\Psi(x)=0,\\
 \Psi(z^0)
 &\le\Psi(a^\star)
 \le\Psi(y^\star)+M_P\|a^\star-y^\star\|_2
 \le\min_Q\Psi+\zeta,
\end{aligned}
\]
where the first line uses $x\in\mathcal V$, and the second
uses the choice of $z^0$ and Lipschitz continuity.

The algorithm subsequently computes
\begin{equation}
\label{eq:apx-trapping-proximal-refinement}
 z^{k+1}:=\operatorname*{argmin}_{y\in Q}
 \left\{
 \langle\nabla P(z^k),y-z^k\rangle
 +\frac{\|y-z^k\|_2^2}{2\tau}
 +\frac{\varepsilon}{2}\|y-x\|_2
 \right\}.
\end{equation}
Each minimizer exists and is unique because $Q$ is nonempty,
compact, and convex, and the objective is continuous and
strongly convex.

Both these minimizers and the grid projections admit finite
real-algebraic computation in the exact-real query model.
The cap is described by $\|y-x\|_2^2\le r^2$ and, in the
constrained case, $0\le y_i\le1$.
The norm term can be represented by an auxiliary variable
$u\ge0$ with $u^2\ge\|y-x\|_2^2$ and its positive coefficient
forces equality at a minimizer.
Thus the optimization problems have polynomial objectives
and constraints with known coefficients, and finite
real-algebraic elimination and root selection determine their
unique minimizers.

Put $s_k:=z^{k+1}-z^k$.
The update satisfies
\begin{equation}
\label{eq:apx-trapping-proximal-decrease}
\begin{aligned}
 \Psi(z^{k+1})-\Psi(z^k)
 &\le
 \langle\nabla P(z^k),s_k\rangle
 +\frac{L_P}{2}\|s_k\|_2^2
 +\frac{\varepsilon}{2}
 \bigl(\|z^{k+1}-x\|_2-\|z^k-x\|_2\bigr)\\
 &\le
 -\left(\frac{1}{2\tau}-\frac{L_P}{2}\right)
 \|s_k\|_2^2\\
 &=-\frac{1}{2}\|s_k\|_2^2,
\end{aligned}
\end{equation}
where the first inequality uses the Hessian bound,
the second compares the proximal objective at $z^{k+1}$
and $z^k$, and the equality uses
$\frac{1}{\tau}=L_P+1$.
Consequently, every iterate satisfies
$\Psi(z^k)\le0$ and $\Psi(z^k)\le\min_Q\Psi+\zeta$.

Convex optimality gives
$v^{k+1}\in\partial\|\cdot-x\|_2(z^{k+1})$ and
$n^{k+1}\in N_Q(z^{k+1})$ such that
\[
 \nabla P(z^k)+\frac{s_k}{\tau}
 +\frac{\varepsilon}{2}v^{k+1}+n^{k+1}=0.
\]
The subdifferential sum rule applies because the norm is
continuous everywhere.
Hence the residual monitored by the algorithm satisfies
\[
\begin{aligned}
 e^{k+1}
 &:=\nabla P(z^{k+1})-\nabla P(z^k)-\frac{s_k}{\tau}\\
 &=\nabla P(z^{k+1})
   +\frac{\varepsilon}{2}v^{k+1}+n^{k+1},\\
 \|e^{k+1}\|_2
 &\le\left(L_P+\frac{1}{\tau}\right)\|s_k\|_2
 =(2L_P+1)\|s_k\|_2,
\end{aligned}
\]
where the inequality uses Lipschitz continuity of $\nabla P$.

At every step at which the algorithm does not return,
$\|e^{k+1}\|_2>\frac{\varepsilon}{16}$, so
\[
\begin{aligned}
 \Psi(z^k)-\Psi(z^{k+1})
 &\ge\frac{1}{2}\|s_k\|_2^2\\
 &>\frac{1}{2}
 \left(\frac{\varepsilon}{16(2L_P+1)}\right)^2
 =\frac{\varepsilon^2}{512(2L_P+1)^2}.
\end{aligned}
\]
Since $\Psi(z^0)-\min_Q\Psi\le\zeta$, the algorithm terminates
after at most
$\left\lfloor
\frac{512(2L_P+1)^2\zeta}{\varepsilon^2}
\right\rfloor+1$
proximal solves.
Its output $z$ therefore satisfies Eq.~\eqref{eq:trapping-terminal-certificates} and
$\Psi(z)\le\min_Q\Psi+\zeta$.
All computations use the known polynomial and require
no additional oracle queries.

\noindent\textbf{Correctness of the output.}
For every $y\in Q$, convexity of $Q$ and the model accuracy give
\[
\begin{aligned}
 |P(y)-P(x)-f(y)+f(x)|&=
 \left|
 \int_0^1
 \left\langle
 (\nabla P-\nabla f)(x+t(y-x)),y-x
 \right\rangle\,dt
 \right|\\
 &\le
 \|\nabla P-\nabla f\|_{\infty,Q}\|y-x\|_2\\
 &\le\frac{\varepsilon}{16}\|y-x\|_2.
\end{aligned}
\]
Thus, for every $y\in\partial_XQ$,
\[
\begin{aligned}
 \Psi(y)
 &\ge f(y)-f(x)+\frac{7\varepsilon}{16}\|y-x\|_2\\
 &=\phi_{x,\frac{3\varepsilon}{8}}(y)
   +\frac{\varepsilon}{16}\|y-x\|_2\\
 &>\frac{\varepsilon r}{16}>0,
\end{aligned}
\]
where the last inequality uses the trapping hypothesis and
$\partial_XQ\subseteq\Gamma$.
Since $\Psi(z)\le0$, we obtain $z\in\operatorname{int}_XQ$.

It follows that $N_Q(z)=N_X(z)$.
Indeed, $Q\subseteq X$ gives $N_X(z)\subseteq N_Q(z)$.
Conversely, for $n\in N_Q(z)$ and $y\in X$, relative interior
and convexity of $X$ imply $z+t(y-z)\in Q$ for sufficiently
small $t>0$. Hence
$0\ge\langle n,z+t(y-z)-z\rangle=t\langle n,y-z\rangle$,
which proves $n\in N_X(z)$.

For the vectors $v,n$ supplied by the terminal certificate,
\[
\begin{aligned}
 \|\nabla f(z)+n\|_2
 &\le
 \|\nabla f(z)-\nabla P(z)\|_2
 +\left\|\nabla P(z)+\frac{\varepsilon}{2}v+n\right\|_2
 +\frac{\varepsilon}{2}\|v\|_2\\
 &\le
 \frac{\varepsilon}{16}
 +\frac{\varepsilon}{16}
 +\frac{\varepsilon}{2}
 =\frac{5\varepsilon}{8}<\varepsilon,
\end{aligned}
\]
where the second inequality uses the model accuracy,
the residual stopping condition, and $\|v\|_2\le1$.

When $X=\mathbb R^d$, we have $n=0$, giving the required
gradient bound.
When $X=[0,1]^d$, the inclusion $n\in N_X(z)$ implies
$\Pi_X(z+\gamma n)=z$ for every $\gamma>0$. Therefore,
\[
\begin{aligned}
 \|G_\gamma(z)\|_2
 &=
 \frac{1}{\gamma}
 \left\|
 \Pi_X(z+\gamma n)
 -\Pi_X(z-\gamma\nabla f(z))
 \right\|_2\\
 &\le\|\nabla f(z)+n\|_2
 \le\frac{5\varepsilon}{8}<\varepsilon,
\end{aligned}
\]
where the first inequality uses nonexpansiveness of projection.
The same output $z$ satisfies this bound for every $\gamma>0$.
\end{proof}


\begin{proposition}[\bf Outer-loop guarantee]
\label{prop:unconstrained-trapping-outer}
\label{prop:box-constrained-guarantee}
Under the above assumptions,
Algorithm~\ref{alg:unconstrained-boundary-lazy}
terminates and returns $z\in X$ satisfying
$\|\nabla f(z)\|_2\le\varepsilon$ when $X=\mathbb R^d$,
or $\|G_\gamma(z)\|_2\le\varepsilon$ for every $\gamma>0$
when $X=[0,1]^d$.
The number of accepted moves satisfies
\begin{equation}
\label{eq:outer-move-count}
 N_{\mathrm{mov}}
 \le\left\lfloor\frac{8\Delta}{\varepsilon r}\right\rfloor.
\end{equation}
The algorithm makes exactly $N_{\mathrm{mov}}+1$ calls to
construct $(P_{\boldsymbol x},\widehat\phi_{\boldsymbol x})$
using the selected fresh or lazy procedure, and one call to
$\mathsf{TrapLocalize}$ (Algorithm~\ref{alg:trap-localize}).
\end{proposition}

\begin{proof}
Every accepted move remains in $X$.
By Lemma~\ref{lem:boundary-trapping-or-descent},
after $N$ accepted moves,
\[
\begin{aligned}
 N\frac{\varepsilon r}{8}
 &\le\sum_{t=0}^{N-1}\bigl(f(x^t)-f(x^{t+1})\bigr)\\
 &=f(x^0)-f(x^N)\\
 &\le f(x^0)-\inf_Xf
 \le\Delta,
\end{aligned}
\]
which proves Eq.~\eqref{eq:outer-move-count}.
Each boundary minimization concerns a known polynomial on
a compact semialgebraic set and admits finite real-algebraic
computation. Hence the algorithm reaches a trapping branch
after finitely many iterations.

At the terminal center,
$\phi_{x,\frac{3\varepsilon}{8}}>0$ on $\partial_XQ$
by Lemma~\ref{lem:boundary-trapping-or-descent}.
When $\Gamma=\varnothing$, this condition holds vacuously.
Lemma~\ref{lem:trapped-ball-localization} therefore guarantees
finite termination of localization and correctness of the output.

There is one model call at the initial center and after each
accepted move, giving exactly $N_{\mathrm{mov}}+1$ calls.
Terminal localization is invoked exactly once.
\end{proof}

\subsection{Fresh Reconstruction and Query Complexity}
\label{app:trapping-fresh-proofs}

Throughout this subsection, assume $p\ge2$, $L_p>0$, and
$\Delta>0$, and write $\nu:=p+\beta_p$.

\begin{proposition}[\bf Fresh query complexity]
\label{prop:function-value-query-complexity}
\label{prop:gradient-only-query-complexity}
Under the standing assumptions, use the radius
$r=\frac{1}{2}
\left(\frac{\varepsilon}{48L_p}\right)^{\frac{1}{\nu-1}}$
from Eq.~\eqref{eq:function-value-radius}.
The fresh implementations of
Algorithm~\ref{alg:unconstrained-boundary-lazy}
satisfy the output guarantees of
Proposition~\ref{prop:unconstrained-trapping-outer}.
Their deterministic query complexities are
\[
\begin{aligned}
 Q_{\mathrm{val}}
 &=\mathcal O_{p,\beta_p}\!\left(
 d^p\left[
 1+L_p^{\frac{1}{\nu-1}}
 \Delta\varepsilon^{-\frac{\nu}{\nu-1}}
 \right]\right),\\
 Q_{\mathrm{grad}}
 &=\mathcal O_{p,\beta_p}\!\left(
 d^{p-1}\left[
 1+L_p^{\frac{1}{\nu-1}}
 \Delta\varepsilon^{-\frac{\nu}{\nu-1}}
 \right]\right).
\end{aligned}
\]
All queries lie in $X$.
\end{proposition}

\begin{proof}
At a pivot $\boldsymbol x$, let $Q:=Q_{\boldsymbol x}(r)$
and choose $\boldsymbol c\in\operatorname{int}(Q)$.
For $X=\mathbb R^d$, take $\boldsymbol c=\boldsymbol x$;
for the cube, a sufficiently small move from $\boldsymbol x$
toward $(\frac{1}{2},\ldots,\frac{1}{2})$ gives such a point.
Then
$\rho:=\max_{\boldsymbol y\in Q}
\|\boldsymbol y-\boldsymbol c\|_2\le2r$.

Write $N_{d,q}:=\binom{d+q}{q}$.
Proposition~\ref{prop:local-reconstruction} and
Corollary~\ref{lem:gradient-only-scalar-reconstruction}
construct $P_{\boldsymbol x}$ using one nonadaptive batch of
$N_{d,p}$ function-value queries or $N_{d,p-1}$ gradient queries,
respectively, all in $Q$.
In both cases,
\[
\begin{aligned}
 \|\nabla P_{\boldsymbol x}-\nabla f\|_{\infty,Q}
 &\le3L_p\rho^{\nu-1}\le3L_p(2r)^{\nu-1}
 =\frac{\varepsilon}{16},
\end{aligned}
\]
where the first inequality uses $a_{j,\beta_p}\le1$,
and the equality follows from the choice of $r$.
For every $\boldsymbol y\in Q$, integrating this error
along $[\boldsymbol x,\boldsymbol y]\subseteq Q$ gives
\[
\begin{aligned}
 &|P_{\boldsymbol x}(\boldsymbol y)
   -P_{\boldsymbol x}(\boldsymbol x)
   -f(\boldsymbol y)+f(\boldsymbol x)|\le
 \frac{\varepsilon}{16}\|\boldsymbol y-\boldsymbol x\|_2
 \le\frac{\varepsilon r}{16}.
\end{aligned}
\]
Thus \eqref{eq:boundary-approximation} and
\eqref{eq:terminal-ball-gradient-accuracy} hold,
and the same polynomial serves at termination.

Proposition~\ref{prop:unconstrained-trapping-outer}
therefore gives correctness and at most
$1+\left\lfloor\frac{8\Delta}{\varepsilon r}\right\rfloor$
model calls.
Since terminal localization requires no additional queries,
\[
\begin{aligned}
 Q_{\mathrm{val}}
 &\le N_{d,p}
 \left(1+\left\lfloor\frac{8\Delta}{\varepsilon r}\right\rfloor\right),\\
 Q_{\mathrm{grad}}
 &\le N_{d,p-1}
 \left(1+\left\lfloor\frac{8\Delta}{\varepsilon r}\right\rfloor\right).
\end{aligned}
\]
Finally,
\[
 \frac{1}{\varepsilon r}
 =
 2\,48^{\frac{1}{\nu-1}}
 L_p^{\frac{1}{\nu-1}}
 \varepsilon^{-\frac{\nu}{\nu-1}}.
\]
Substituting this identity and
$N_{d,q}=\mathcal O_q(d^q)$ with $q=p,p-1$
proves the stated bounds.
\end{proof}

\subsection{Proofs for Lazy Boundary Reuse}
\label{app:trapping-lazy-proofs}

We first consider $X=\mathbb R^d$ under the standing
assumptions, with $p\ge2$, $L_p>0$, and $\Delta>0$.
Fix an integer $m\ge1$ and use $r=r_m$ from
Eq.~\eqref{eq:lazy-radius}, so that
$96m^{\beta_p}L_pr^{\nu-1}=\varepsilon$.
Recall that, on $\mathbb B$,
\[
 A_{\boldsymbol x}(\boldsymbol u)
 :=f(\boldsymbol x+r\boldsymbol u)-f(\boldsymbol x),
 \qquad
 B_{\boldsymbol x}(\boldsymbol u)
 :=r\nabla f(\boldsymbol x+r\boldsymbol u)
 =\nabla A_{\boldsymbol x}(\boldsymbol u).
\]

\noindent\textbf{Fixed interpolation operators.}
Use the nested node sets
$U_0\subset\cdots\subset U_p\subset\mathbb B$,
with $U_0=\{\boldsymbol0\}$ and
$|U_j|=N_{d,j}:=\binom{d+j}{j}$, from
Proposition~\ref{prop:local-reconstruction}.
Choose a common contraction of the nodes for the finitely
many degree--exponent pairs $(j,\beta_p)$ and $(j,1)$.
The resulting interpolation operators $I_j$ are fixed
throughout the run, are linear, and interpolate at $\boldsymbol0$.
For $j\ge1$, $\gamma\in\{\beta_p,1\}$, and scalar- or
vector-valued $G\in C^{j,\gamma}(L)$,
\begin{equation}
\label{eq:apx-lazy-fixed-interpolation}
 \|D^i(I_jG-G)\|_{\infty,\mathbb B}
 \le\frac{3}{2}L,
 \qquad i=0,1,
\end{equation}
where we used $a_{j-i,\gamma}\le1$.
For $I_0G\equiv G(\boldsymbol0)$, if
$G\in C^1$ and $\|DG\|_{\infty,\mathbb B}\le L$, then
\[
 \|I_0G-G\|_{\infty,\mathbb B}\le L,
 \qquad
 \|D(I_0G-G)\|_{\infty,\mathbb B}\le L.
\]
The first bound follows by integrating $DG$ along
$[\boldsymbol0,\boldsymbol u]$; the second uses $D(I_0G)=0$.
This covers the degree-zero updates when $p=2$.

\begin{lemma}[\bf Accuracy of lazy reuse]
\label{lem:lazy-function-value-accuracy}
\label{lem:lazy-gradient-accuracy}
With the fixed interpolation operators above, every call to
Algorithm~\ref{alg:lazy-approx-boundary} returns
$P_{\boldsymbol x}\in\Pi_p^d$ and
$\widehat\phi_{\boldsymbol x}$ satisfying
\[
\begin{aligned}
 \left\|
 \widehat\phi_{\boldsymbol x}
 -\phi_{\boldsymbol x,\frac{\varepsilon}{4}}
 \right\|_{\infty,\Gamma_{\boldsymbol x}(r)}
 &\le\frac{\varepsilon r}{16},\\
 \left\|\nabla P_{\boldsymbol x}-\nabla f\right\|
 _{\infty,Q_{\boldsymbol x}(r)}
 &\le\frac{\varepsilon}{16}.
\end{aligned}
\]
The same polynomial can be used for terminal localization
without further oracle queries.
\end{lemma}

\begin{proof}
\noindent\textbf{Phase identities.}
Let $\boldsymbol x_0=\boldsymbol a,\ldots,
\boldsymbol x_k=\boldsymbol x$ be the centers visited
within the current phase, where $k\le m-1$.
Since each accepted move has length $r$,
$\|\boldsymbol x-\boldsymbol a\|_2\le kr\le(m-1)r$.
The update rules and linearity give
\begin{equation}
\label{eq:apx-lazy-scalar-phase-identity}
\begin{aligned}
 \widehat A_{\boldsymbol x}
 &=I_pA_{\boldsymbol a}
   +\sum_{j=1}^k
    I_{p-1}(A_{\boldsymbol x_j}-A_{\boldsymbol x_{j-1}})\\
 &=I_pA_{\boldsymbol a}
   +I_{p-1}(A_{\boldsymbol x}-A_{\boldsymbol a}),\\
 V_{\boldsymbol x}
 &=I_{p-1}B_{\boldsymbol a}
   +\sum_{j=1}^k
    I_{p-2}(B_{\boldsymbol x_j}-B_{\boldsymbol x_{j-1}})\\
 &=I_{p-1}B_{\boldsymbol a}
   +I_{p-2}(B_{\boldsymbol x}-B_{\boldsymbol a}),
\end{aligned}
\end{equation}
where the second equality in each pair follows by telescoping.

\noindent\textbf{Model error.}
The chain rule gives
$A_{\boldsymbol a}\in C^{p,\beta_p}(L_pr^\nu)$ and
$B_{\boldsymbol a}\in C^{p-1,\beta_p}(L_pr^\nu)$.
Moreover, for every $\boldsymbol u\in\mathbb B$,
\begin{equation}
\label{eq:apx-lazy-aggregate-derivative}
\begin{aligned}
 \|D^p(A_{\boldsymbol x}-A_{\boldsymbol a})
       (\boldsymbol u)\|_{\mathrm{op}}
 &=
 \|D^{p-1}(B_{\boldsymbol x}-B_{\boldsymbol a})
       (\boldsymbol u)\|_{\mathrm{op}}\\
 &=r^p
 \|D^pf(\boldsymbol x+r\boldsymbol u)
   -D^pf(\boldsymbol a+r\boldsymbol u)\|_{\mathrm{op}}\\
 &\le L_pr^p
       \|\boldsymbol x-\boldsymbol a\|_2^{\beta_p},
\end{aligned}
\end{equation}
where the first equality uses
$B_{\boldsymbol x}=\nabla A_{\boldsymbol x}$ and
Euclidean duality, and the inequality uses H\"older continuity.

Integrating these derivatives along segments in $\mathbb B$
shows that $A_{\boldsymbol x}-A_{\boldsymbol a}$ and
$B_{\boldsymbol x}-B_{\boldsymbol a}$ have Lipschitz
derivatives of orders $p-1$ and $p-2$, respectively, with
constant $L_pr^p\|\boldsymbol x-\boldsymbol a\|_2^{\beta_p}$.
For $p=2$, Eq.~\eqref{eq:apx-lazy-aggregate-derivative}
also supplies the derivative bound required for $I_0$.

Subtracting $A_{\boldsymbol x}$ and $B_{\boldsymbol x}$
from the phase identities and applying the interpolation
bounds therefore gives, for $i=0,1$,
\begin{equation}
\label{eq:apx-lazy-joint-invariants}
\begin{aligned}
 &\max\left\{
 \|D^i(\widehat A_{\boldsymbol x}-A_{\boldsymbol x})
       \|_{\infty,\mathbb B},
 \|D^i(V_{\boldsymbol x}-B_{\boldsymbol x})
       \|_{\infty,\mathbb B}
 \right\}\\
 &\quad\le
 \frac{3}{2}L_pr^\nu
 +\frac{3}{2}L_pr^p
       \|\boldsymbol x-\boldsymbol a\|_2^{\beta_p}\\
 &\quad\le
 \frac{3}{2}L_pr^\nu
 \left[1+(m-1)^{\beta_p}\right]\\
 &\quad\le3m^{\beta_p}L_pr^\nu
 =\frac{\varepsilon r}{32},
\end{aligned}
\end{equation}
where the second inequality uses the displacement bound,
the third uses $1+(m-1)^{\beta_p}\le2m^{\beta_p}$,
and the equality follows from the choice of $r$.

\noindent\textbf{Boundary and terminal guarantees.}
In function-value mode, set
$P_{\boldsymbol x}(\boldsymbol y)
:=\widehat A_{\boldsymbol x}
(\frac{\boldsymbol y-\boldsymbol x}{r})$.
Since every $I_j$ interpolates at $\boldsymbol0$,
$\widehat A_{\boldsymbol x}(\boldsymbol0)=0$.
The preceding error bounds and the chain rule give
\[
\begin{aligned}
 \|P_{\boldsymbol x}-f+f(\boldsymbol x)\|
       _{\infty,Q_{\boldsymbol x}(r)}
 &\le\frac{\varepsilon r}{32},\\
 \|\nabla P_{\boldsymbol x}-\nabla f\|
       _{\infty,Q_{\boldsymbol x}(r)}
 &\le\frac{1}{r}\frac{\varepsilon r}{32}
 =\frac{\varepsilon}{32}.
\end{aligned}
\]

In gradient-only mode, define
\begin{equation}
\label{eq:apx-lazy-radial-potential}
 S_{\boldsymbol x}(\boldsymbol u)
 :=\int_0^1
 \langle V_{\boldsymbol x}(t\boldsymbol u),
         \boldsymbol u\rangle\,dt,
 \qquad
 P_{\boldsymbol x}(\boldsymbol y)
 :=S_{\boldsymbol x}
   \left(\frac{\boldsymbol y-\boldsymbol x}{r}\right).
\end{equation}
Apply Lemma~\ref{lem:radial-scalar-reconstruction} to
$A_{\boldsymbol x}$ and $V_{\boldsymbol x}$ on $\mathbb B$,
using Eq.~\eqref{eq:apx-lazy-joint-invariants}.
It gives $S_{\boldsymbol x}\in\Pi_p^d$,
$S_{\boldsymbol x}(\boldsymbol0)=0$, and
\[
\begin{aligned}
 \|P_{\boldsymbol x}-f+f(\boldsymbol x)\|
       _{\infty,Q_{\boldsymbol x}(r)}
 &\le\frac{\varepsilon r}{32},\\
 \|\nabla P_{\boldsymbol x}-\nabla f\|
       _{\infty,Q_{\boldsymbol x}(r)}
 &\le
 \frac{1}{r}
 \left(\frac{\varepsilon r}{32}
       +\frac{1}{2}\frac{\varepsilon r}{32}\right)\\
 &=\frac{3\varepsilon}{64}
 \le\frac{\varepsilon}{16}.
\end{aligned}
\]
In both modes, $P_{\boldsymbol x}(\boldsymbol x)=0$,
so the value bounds imply
Eq.~\eqref{eq:boundary-approximation}, while the gradient bounds
give Eq.~\eqref{eq:terminal-ball-gradient-accuracy}.
Both scalar polynomials are computed from the retained models
without further oracle queries.
\end{proof}

\begin{proposition}[\bf Lazy boundary-reuse complexity]
\label{prop:lazy-boundary-reuse-complexity}
Under the standing assumptions, take $m=d$ and use
$r=r_m$ from Eq.~\eqref{eq:lazy-radius}.
Both lazy implementations deterministically return an
$\varepsilon$-stationary point, with query complexities
\begin{equation}
\label{eq:optimized-lazy-complexities}
\begin{aligned}
 Q_{\mathrm{val}}
 &=\mathcal O_{p,\beta_p}\!\left(
 d^p+d^{p-1+\frac{\beta_p}{\nu-1}}
 L_p^{\frac{1}{\nu-1}}
 \Delta\varepsilon^{-\frac{\nu}{\nu-1}}
 \right),\\
 Q_{\mathrm{grad}}
 &=\mathcal O_{p,\beta_p}\!\left(
 d^{p-1}+d^{p-2+\frac{\beta_p}{\nu-1}}
 L_p^{\frac{1}{\nu-1}}
 \Delta\varepsilon^{-\frac{\nu}{\nu-1}}
 \right).
\end{aligned}
\end{equation}
Each model call uses one nonadaptive batch, and terminal
localization requires no additional queries.
\end{proposition}

\begin{proof}
We first keep $m\ge1$ arbitrary.
Lemma~\ref{lem:lazy-function-value-accuracy} verifies the
model requirements of
Proposition~\ref{prop:unconstrained-trapping-outer}.
Hence correctness follows, and the number $K$ of model calls
satisfies
\begin{equation}
\label{eq:apx-lazy-number-of-calls}
\begin{aligned}
 K
 &\le1+\left\lfloor\frac{8\Delta}{\varepsilon r}\right\rfloor\\
 &\le1+C_{p,\beta_p}
 m^{\frac{\beta_p}{\nu-1}}
 L_p^{\frac{1}{\nu-1}}
 \Delta\varepsilon^{-\frac{\nu}{\nu-1}},
\end{aligned}
\end{equation}
where the second inequality substitutes Eq.~\eqref{eq:lazy-radius}.

For function values, put $q=p$; for gradients, put $q=p-1$.
A refresh queries the nodes in
$\boldsymbol x+rU_q$ and retains the data on
$\boldsymbol x+rU_{q-1}$.
A rolling update queries only
$\boldsymbol x+rU_{q-1}$, since the corresponding data at
the preceding center are already cached.
Thus refreshes and rolling updates use $N_{d,q}$ and
$N_{d,q-1}$ queries, respectively.

Among $K$ calls, exactly
$\left\lceil\frac{K}{m}\right\rceil$ are refreshes.
Counting the prescribed stencil queries, the total is
\begin{equation}
\label{eq:apx-lazy-exact-query-count}
\begin{aligned}
 Q
 &=
 \left\lceil\frac{K}{m}\right\rceil N_{d,q}
 +\left(K-\left\lceil\frac{K}{m}\right\rceil\right)
   N_{d,q-1}\\
 &\le
 N_{d,q}
 +K\left(\frac{N_{d,q}}{m}+N_{d,q-1}\right),
\end{aligned}
\end{equation}
where the inequality uses
$\left\lceil\frac{K}{m}\right\rceil\le1+\frac{K}{m}$.
Combining this with Eq.~\eqref{eq:apx-lazy-number-of-calls}
and $N_{d,q}=\mathcal O_q(d^q)$ gives
\begin{equation}
\label{eq:apx-lazy-general-complexity}
 Q\le C_{p,\beta_p}\left[
 d^q+
 \left(
 d^q m^{-\frac{p-1}{\nu-1}}
 +d^{q-1}m^{\frac{\beta_p}{\nu-1}}
 \right)
 L_p^{\frac{1}{\nu-1}}
 \Delta\varepsilon^{-\frac{\nu}{\nu-1}}
 \right].
\end{equation}
Here the terms independent of $\Delta$ are at most
$3N_{d,q}$ because $m\ge1$ and $N_{d,q-1}\le N_{d,q}$,
and we used
$\frac{\beta_p}{\nu-1}-1=-\frac{p-1}{\nu-1}$.

For $m=d$, both dimension powers in parentheses equal
$q-1+\frac{\beta_p}{\nu-1}$.
Taking $q=p$ and $q=p-1$ proves
Eq.~\eqref{eq:optimized-lazy-complexities}.
Each translated node set is fixed before its answers are
observed, so every model call uses one nonadaptive batch.
The terminal polynomial is already supplied by
Lemma~\ref{lem:lazy-function-value-accuracy}.
\end{proof}

\noindent\textbf{The cube-constrained case.}
Let $X=[0,1]^d$ and $X_r:=X+r\mathbb B$.
Assume that the oracle is available on $X_r$ and that
$f\in C^{p,\beta_p}(L_p)$ on a neighborhood of $X_r$.
All translated query sets and segments used above lie in
the convex set $X_r$.
Restricting the models to $Q_{\boldsymbol x}(r)$ preserves
their value and gradient error bounds.
Proposition~\ref{prop:unconstrained-trapping-outer}
therefore gives the same query complexities, with
$f(\boldsymbol x^0)-\min_Xf\le\Delta$, and an output
$\boldsymbol z\in X$ satisfying
$\|G_\gamma(\boldsymbol z)\|_2\le\varepsilon$
for every $\gamma>0$.
The query points lie in $X_r$.

\section{Proofs in Section~\ref{sec:lb}}
\label{app:lb}

\subsection{Proof of Lemma~\ref{lem:drhs-one-phase}:
One-Phase Delayed-RHS Hardness}
\label{subsec:drhs-one-phase}

In this section we prove Lemma~\ref{lem:drhs-one-phase} first under the
unconditioned law $\mathsf P_0$.
Section~\ref{subsubsec:drhs-preliminaries} collects the spectral
and Krylov ingredients, which
Section~\ref{subsubsec:drhs-fixed-time} combines to establish
the lower bound for a single output at a fixed disclosure time.
Section~\ref{subsubsec:drhs-adaptive-completion} then handles
adaptive disclosure and multiple outputs, and passes to the
conditioned law $\mathsf P_{\mathrm{ph}}$.

Throughout, we write
$\rho=\frac{\varepsilon}{\mu}$,
$\kappa=\frac{\bar L}{\mu}$, $B_0=512$, and
$b=B_0\varepsilon\xi$, where $\xi$ is uniform on the unit sphere.

\subsubsection{Useful Facts}
\label{subsubsec:drhs-preliminaries}
The first lemma provides a spectrum on which low-degree
polynomials cannot approximate the inverse, with sufficiently
large multiplicities in every eigenspace.
\begin{lemma}[\bf Replicated hard spectrum]
\label{lem:drhs-spectrum}
There are numerical constants \(c_0>0\) and \(\kappa_0\ge4\) such that the
following holds.  Suppose \(\kappa\ge\kappa_0\), set
\(t:=\lfloor c_0\sqrt\kappa\rfloor\ge1\) and \(K:=t+2\), and let
\(M\ge2\kappa\) be an integer.  Then there are distinct numbers
\(\mu\le\lambda_1<\cdots<\lambda_K\le\bar L\) and positive integers
\(n_1,\ldots,n_K\), with \(\sum_{i=1}^K n_i=M\), such that, for every
real polynomial \(q\) of degree at most \(t\),
\begin{equation}
\label{eq:drhs-flat-spectrum}
 n_i\ge\frac{M}{2\kappa}\quad(i\in[K]),
 \qquad
 \sum_{i=1}^K\frac{n_i}{M}
 \left|\frac1{\lambda_i}-q(\lambda_i)\right|^2
 \ge\frac{3}{512\mu^2}.
\end{equation}
Consequently,
\(D:=\operatorname{diag}(\lambda_1I_{n_1},\ldots,\lambda_KI_{n_K})\)
satisfies \(\mu I\preceq D\preceq\bar L I\).
\end{lemma}
The proof is given in Appendix~\ref{app:drhs-spectrum}.

\begin{remark}[\bf Relation to classical inverse approximation]
The uniform approximation of \(x\mapsto x^{-1}\) on a positive interval is classical.  Exact minimax errors are given in \cite[Theorem~2.1 and Corollary~2.2]{kraus2012bestinverse} and \cite[Theorem~3.4]{plonka2026reciprocal}; in particular, for every fixed \(C>0\), degree \(m\leq C\sqrt\kappa\) entails uniform error \(\Omega_C(\mu^{-1})\) on \([\mu,\bar L]\).  Lemma~\ref{lem:drhs-spectrum} converts this uniform obstruction into a weighted-\(L_2\) certificate on exactly \(t+2\) spectral points and rounds its weights to integer multiplicities bounded below by \(M/(2\kappa)\).  Related discretization and replication arguments appear in \cite[Appendix~A.4--A.5]{derezinski2026matrixvector}, which constructs two moment-matching spectral measures on a common finite support and rounds both to integer multiplicities; this paired-measure construction does not directly yield the single-measure weighted-\(L_2\) bound in \eqref{eq:drhs-flat-spectrum}.
\end{remark}

We next isolate the random linear-system problem underlying the one-phase
lower bound.

\paragraph{Problem and objective.}
Fix \(D=\operatorname{diag}(\lambda_1,\ldots,\lambda_M)\succ0\), draw
\(U\sim\operatorname{Haar}(O(M))\), and set \(A:=U^\top D U\).
Independently, draw
\(\xi\sim\operatorname{Unif}(\mathbb S^{M-1})\).
Given matrix–vector product access \(v\mapsto Av\), the goal is to estimate
\(x^\star:=A^{-1}\xi\) under the loss
\(\ell_{A,\xi}(y):=\|y-x^\star\|_2\).
In the phase construction, \(b=B_0\varepsilon\xi\) and
\(A^{-1}b=B_0\varepsilon x^\star\), so the result transfers directly by
scaling.

\paragraph{Two-stage matrix–vector product algorithms.}
Let \(q_0,q_1\in\mathbb N_{+}\) and \(n:=q_0+q_1\).
A deterministic two-stage algorithm is specified by Borel maps
\(\psi_1,\ldots,\psi_n,\psi_{\mathrm{out}}\).
Starting from \(\mathsf T_0=\varnothing\), it chooses
\(v_i=\psi_i(\mathsf T_{i-1})\) for \(i\leq q_0\), receives
\(r_i:=Av_i\), and sets
\(\mathsf T_i:=(\mathsf T_{i-1},v_i,r_i)\).
After \(\xi\) is disclosed, it chooses
\(v_i=\psi_i(\mathsf T_{i-1},\xi)\) for \(q_0<i\leq n\), receives the
same response \(r_i=Av_i\), and returns
\(y:=\psi_{\mathrm{out}}(\mathsf T_n,\xi)\).
The output need not have been queried.

The next lemma shows that the component of \(\xi\) unexplored during the
first \(q_0\) queries enters the Krylov representation only at disclosure.
Hence preprocessing cannot retroactively apply powers of \(A\) to this
fresh component.

\begin{lemma}[\bf Delayed-input Krylov representation]
\label{lem:drhs-delayed-krylov}
Let \(q_0,q_1\in\mathbb N_{+}\), \(j_\star:=q_0+1\), and
\(m:=q_0+q_1+2\), and suppose \(m^2<M\).
Every deterministic two-stage algorithm above admits a
joint-law-preserving orthogonal coupling with
\(\widehat A=\widehat U^\top D\widehat U\), where
\(\widehat U\sim\operatorname{Haar}(O(M))\), and with
\(h_1,\ldots,h_m\stackrel{\mathrm{iid}}{\sim}N(0,I_M)\), independent of
\(\widehat U\).

Write \(\widetilde\xi,\widetilde y\) for the coupled input and output, and
set
\(\mathcal K_{\mathrm{pre}}
 :=\operatorname{span}\{\widehat A^a h_j:
 1\leq j\leq q_0,\ a\in\mathbb Z_{\geq0},\
 a+j\leq q_0+1\}\).
Then, almost surely,
\begin{align}
 \widetilde\xi
 &=
 P_{\mathcal K_{\mathrm{pre}}}\widetilde\xi
 +\beta
 \frac{(I-P_{\mathcal K_{\mathrm{pre}}})h_{j_\star}}
      {\|(I-P_{\mathcal K_{\mathrm{pre}}})h_{j_\star}\|_2},
 \qquad
 \beta:=\|(I-P_{\mathcal K_{\mathrm{pre}}})\widetilde\xi\|_2>0,
 \label{eq:drhs-flat-fresh-input}\\
 \widetilde y
 &\in
 \operatorname{span}
 \{\widehat A^a h_j:
 1\leq j\leq m,\ a\in\mathbb Z_{\geq0},\ a+j\leq m\},
 \label{eq:drhs-flat-output-span}\\
 \widehat A^a h_{j_\star}
 \text{ appearing in \eqref{eq:drhs-flat-output-span}}
 &\quad\Longrightarrow\quad
 a\leq m-j_\star=q_1+1.
 \label{eq:drhs-flat-degree}
\end{align}
\end{lemma}

Lemma~\ref{lem:drhs-delayed-krylov} gives a time-respecting Krylov
representation of the two-stage algorithm.  Equation
\eqref{eq:drhs-flat-fresh-input} decomposes the disclosed input into a
component contained in the preprocessing space
\(\mathcal K_{\mathrm{pre}}\) and an orthogonal component represented by
the fresh Gaussian column \(h_{j_\star}\).  Its index
\(j_\star=q_0+1\) records that this column appears only when the input is
disclosed.  Accordingly, the triangular constraint in
\eqref{eq:drhs-flat-output-span} permits \(h_{j_\star}\) only through
powers \(\widehat A^a\) with \(a\leq q_1+1\).  Thus the representation
preserves the distinction between pre- and post-disclosure matrix–vector products. The proof is given in Appendix~\ref{app:drhs-delayed-krylov}.

\begin{remark}[\bf Relation to adaptive block-Krylov lifting]
The adaptive block-Krylov lifting of
\cite[Lemma~6.3 and Appendix~A]{bakshi2023krylov} represents an adaptive
matrix–vector products transcript in a triangular block-Krylov space.  It does not directly
specify where a vector disclosed only after preprocessing enters this
representation.  Lemma~\ref{lem:drhs-delayed-krylov} supplies this
time-respecting extension by assigning the fresh component of the delayed
input to its disclosure slot, while preserving the full disclosed vector
and allowing an unqueried terminal output.  
\end{remark}

The final lemma ensures that the fresh component has
constant norm with high probability.

\begin{lemma}[\bf Fresh mass after adaptive preprocessing]
\label{lem:drhs-fresh-mass}
In the setting of Lemma~\ref{lem:drhs-delayed-krylov}, let \(\mathcal S\)
be the analyst-only augmented span immediately before \(\xi\) is disclosed,
and let \(\beta=\|(I-P_{\mathcal S})\xi\|_2\).  If the procedure makes
\(q_0\) pre-disclosure products, then
\begin{equation}
\label{eq:drhs-preprocessing-span-dimension}
 \dim(\mathcal S)
 \le\frac{q_0(q_0+3)}2
 \le(q_0+1)^2.
\end{equation}
Consequently, if \(M\ge8(q_0+1)^2\), then
\begin{equation}
\label{eq:drhs-flat-beta-tail}
 \Pr(\beta<1/2)\le2e^{-M/32}.
\end{equation}
\end{lemma}

\begin{proof}
In the analyst-only augmentation, preprocessing slot \(j\) contributes at
most the vectors
\(v_j,Av_j,\ldots,A^{q_0+1-j}v_j\), and hence at most
\(q_0+2-j\) dimensions.  Summing over the preprocessing slots gives
\[
 \dim(\mathcal S)
 \le
 \sum_{j=1}^{q_0}(q_0+2-j)
 =\frac{q_0(q_0+3)}2
 \le(q_0+1)^2,
\]
which proves the first part (Eq.~\eqref{eq:drhs-preprocessing-span-dimension}).

Let \(\mathcal F_{\mathrm{pre}}\) be the sigma-field generated by the
augmented preprocessing transcript.  Conditional on
\(\mathcal F_{\mathrm{pre}}\), the subspace \(\mathcal S\) is fixed, while
\(\xi\) remains an independent uniform unit vector.  Write
\(\xi=g/\|g\|_2\) with \(g\sim N(0,I_M)\), put
\(r:=\dim(\mathcal S)\), and define
\[
 X:=\|P_{\mathcal S}g\|_2^2,
 \qquad
 Y:=\|(I-P_{\mathcal S})g\|_2^2.
\]
Conditionally on \(\mathcal F_{\mathrm{pre}}\), the variables \(X\) and
\(Y\) are independent with \(X\sim\chi_r^2\) and
\(Y\sim\chi_{M-r}^2\), and \(\beta^2=Y/(X+Y)\).  Hence
\(\beta<1/2\) implies \(X>3Y\), so
\[
 \Pr(\beta<1/2\mid\mathcal F_{\mathrm{pre}})
 \le
 \Pr\left(Y<\frac{M-r}{2}\right)
 +
 \Pr\left(X>\frac{3(M-r)}2\right).
\]
The chi-square moment-generating function gives
\[
 \Pr\left(Y<\frac{M-r}{2}\right)
 \le e^{-(M-r)/16},
 \qquad
 \Pr\left(X>\frac{3(M-r)}2\right)
 \le e^{-3(M-r)/8}2^{r/2}.
\]
If \(M\ge8(q_0+1)^2\), then
Eq.~\eqref{eq:drhs-preprocessing-span-dimension} gives \(r\le M/8\).  The two
terms in the last display are then at most \(e^{-M/32}\).  Thus
\(\Pr(\beta<1/2\mid\mathcal F_{\mathrm{pre}})\le2e^{-M/32}\), and
averaging over \(\mathcal F_{\mathrm{pre}}\) proves the second part 
(Eq.~\eqref{eq:drhs-flat-beta-tail}).
\end{proof}

\subsubsection{Fixed-Time Exclusion for a Single Output}
\label{subsubsec:drhs-fixed-time}

We now combine the three preceding lemmas to bound the success
probability of a single output at a fixed disclosure time.
Set $\delta_\star:=\frac{\delta}{2(Q+1)\mathsf L}$,
where $Q$ and $\mathsf L$ are the query and candidate-vector
budgets in Lemma~\ref{lem:drhs-one-phase}.
This allocation allows a subsequent union bound over the
$Q+1$ possible disclosure times and at most $\mathsf L$ candidates.


\begin{lemma}[\bf Fixed-time single-output exclusion]
\label{lem:drhs-fixed-time}
Assume the hypotheses of Lemma~\ref{lem:drhs-one-phase},
with $C_0$ in \eqref{eq:drhs-one-phase-dimension}
sufficiently large.
Fix $r\in\{0,\ldots,Q\}$.
Consider any deterministic algorithm that makes exactly $r$
queries to $v\mapsto Av$ before receiving $b$, then makes at
most $t-1$ further queries and returns $y\in\mathbb R^M$.
Then
\begin{equation}
\label{eq:drhs-flat-fixed-time-exclusion}
 \Pr_{(A,b)\sim\mathsf P_0}
 \left(\|y-A^{-1}b\|_2<3\rho\right)
 \le\delta_\star.
\end{equation}
\end{lemma}

\begin{proof}
\noindent\textbf{Step 1: Obtain the Krylov representation.}
Store the output $y$ and pad the execution with zero-vector
queries until exactly $t-1$ post-disclosure queries have been
made. The stored output is unchanged.
Since $b=B_0\varepsilon\xi$ and $B_0\varepsilon$ is known,
the algorithm can recover $\xi=\frac{b}{B_0\varepsilon}$.
Thus the padded algorithm has
$q_0=r$ and $q_1=t-1$.

Set $m:=r+t+1$ and $j_\star:=r+1$.
Since $r,t\le Q$ and $Q\ge1$,
\begin{equation}
\label{eq:drhs-flat-dimension-consequences}
 m^2\le(2Q+1)^2\le9Q^2<M,
\end{equation}
where the last inequality follows from the $Q^2$ term in
Eq.~\eqref{eq:drhs-one-phase-dimension}, for sufficiently large $C_0$.
Lemma~\ref{lem:drhs-delayed-krylov} therefore applies.

Use its coupled matrix $\widehat A$, input $\widetilde\xi$,
output $\widetilde y$, and Gaussian vectors $h_1,\ldots,h_m$.
Writing
$\sigma_\star:=
\|(I-P_{\mathcal K_{\mathrm{pre}}})h_{j_\star}\|_2$,
the input and output satisfy, almost surely,
\[
\begin{aligned}
 \widetilde\xi
 &=
 P_{\mathcal K_{\mathrm{pre}}}\widetilde\xi
 +\frac{\beta}{\sigma_\star}
  (I-P_{\mathcal K_{\mathrm{pre}}})h_{j_\star},\\
 \widetilde y
 &=\sum_{j=1}^m p_j(\widehat A)h_j,
 \qquad \deg p_j\le m-j,
\end{aligned}
\]
for some real polynomials $p_j$, where $\beta>0$ and
$\sigma_\star>0$.
In particular,
$q_y:=p_{j_\star}$ has degree at most
$m-j_\star=t$.

\medskip
\noindent\textbf{Step 2: Isolate the fresh direction.}
For each $i\in[K]$, let $E_i$ be the
$\lambda_i$-eigenspace of $\widehat A$, and let $P_i$ be
its orthogonal projector.
Define $R_i$ as the orthogonal projector onto
\[
 E_i\cap
 \operatorname{span}
 \{P_i h_j:1\le j\le m,\ j\ne j_\star\}^{\perp}.
\]
For every $j\ne j_\star$ and $a\ge0$,
$R_i\widehat A^a h_j=\lambda_i^aR_i h_j=0$.
Thus $R_i$ annihilates all output terms except the one
generated by $h_{j_\star}$.
It also annihilates $\mathcal K_{\mathrm{pre}}$, whose
generators have indices $j\le r<j_\star$.
Applying $R_i$ to the representations from Step~1 gives
\begin{equation}
\label{eq:drhs-flat-projected-representation}
 R_i\widetilde\xi
 =\frac{\beta}{\sigma_\star}R_i h_{j_\star},
 \qquad
 R_i\widetilde y
 =q_y(\lambda_i)R_i h_{j_\star}.
\end{equation}
The same polynomial $q_y=p_{j_\star}$ occurs for every $i$;
only its evaluation point $\lambda_i$ changes.

\medskip
\noindent\textbf{Step 3: Reduce the error to polynomial approximation.}
Define
$q(\lambda):=
\frac{\sigma_\star}{B_0\varepsilon\beta}q_y(\lambda)$,
so that $\deg q\le t$.
Since $R_i\widehat A^{-1}=\frac{1}{\lambda_i}R_i$,
Eq.~\eqref{eq:drhs-flat-projected-representation} yields
\[
 R_i\left(
 \widetilde y-B_0\varepsilon\widehat A^{-1}\widetilde\xi
 \right)
 =
 \frac{B_0\varepsilon\beta}{\sigma_\star}
 \left(q(\lambda_i)-\frac{1}{\lambda_i}\right)
 R_i h_{j_\star}.
\]
The ranges of the $R_i$ lie in mutually orthogonal eigenspaces.
Consequently,
\begin{equation}
\label{eq:drhs-flat-error-reduction}
\begin{aligned}
 &\|\widetilde y-
 B_0\varepsilon\widehat A^{-1}\widetilde\xi\|_2^2\ge
 \sum_{i=1}^K
 \left\|R_i\left(
 \widetilde y-B_0\varepsilon\widehat A^{-1}\widetilde\xi
 \right)\right\|_2^2=
 B_0^2\varepsilon^2
 \frac{\beta^2}{\sigma_\star^2}
 \sum_{i=1}^K
 \|R_i h_{j_\star}\|_2^2
 \left|\frac{1}{\lambda_i}-q(\lambda_i)\right|^2.
\end{aligned}
\end{equation}
It remains to lower-bound the random prefactor and weights
in this expression.

\medskip
\noindent\textbf{Step 4: Control the random prefactor and weights.}
Let $\mathcal G_\star$ be the event
\begin{equation}
\label{eq:drhs-flat-good-event}
 \beta\ge\frac{1}{2},
 \qquad
 \|h_{j_\star}\|_2^2\le2M,
 \qquad
 \|R_i h_{j_\star}\|_2^2\ge\frac{n_i}{4}
 \quad(i\in[K]).
\end{equation}
We bound the failure probability of each condition.

First, the $Q^2$ term in the dimension assumption ensures
$M\ge8(Q+1)^2\ge8(r+1)^2$.
Lemma~\ref{lem:drhs-fresh-mass} therefore gives
$\Pr(\beta<\frac{1}{2})\le2e^{-\frac{M}{32}}$.

Second, $\|h_{j_\star}\|_2^2\sim\chi_M^2$.
Markov's inequality and the chi-square moment-generating
function give
$\Pr(\|h_{j_\star}\|_2^2>2M)
\le e^{-\frac{M}{2}}2^{\frac{M}{2}}
\le e^{-\frac{M}{8}}$,
where the last inequality uses $\log2\le\frac{3}{4}$.

Third, the $\kappa Q$ term in the dimension assumption
ensures $M\ge8\kappa Q$.
By Lemma~\ref{lem:drhs-spectrum},
$n_i\ge\frac{M}{2\kappa}\ge4Q\ge2(m-1)$.
The definition of $R_i$ removes at most $m-1$ dimensions
from $E_i$, so
$\operatorname{rank}(R_i)\ge n_i-(m-1)\ge\frac{n_i}{2}$.
Conditional on $\widehat A$ and $\{h_j:j\ne j_\star\}$,
the projector $R_i$ is fixed and $h_{j_\star}$ remains
standard Gaussian.
Hence $\|R_i h_{j_\star}\|_2^2$ is chi-square with
$\operatorname{rank}(R_i)$ degrees of freedom.
The chi-square lower-tail bound used in the proof of
Lemma~\ref{lem:drhs-fresh-mass}, followed by averaging over
the conditioning variables, gives
\[
 \Pr\left(
 \|R_i h_{j_\star}\|_2^2<\frac{n_i}{4}
 \right)
 \le e^{-\frac{n_i}{32}}
 \le e^{-\frac{M}{64\kappa}}.
\]

Combining these bounds by a union bound, we obtain
\begin{equation}
\label{eq:drhs-flat-tail-budget}
\begin{aligned}
 \Pr(\mathcal G_\star^c)
 &\le
 2e^{-\frac{M}{32}}
 +e^{-\frac{M}{8}}
 +K e^{-\frac{M}{64\kappa}}\\
 &\le
 (K+3)e^{-\frac{M}{64\kappa}}
 \le\delta_\star.
\end{aligned}
\end{equation}
The second inequality uses $\kappa\ge1$.
For the last inequality, recall that $K=t+2\le Q+2$ and
$\delta_\star=\frac{\delta}{2(Q+1)\mathsf L}$.
Since $K+3\le2(Q+2)$ and $\log(Q+1)\le Q$,
\[
 \log\frac{K+3}{\delta_\star}
 =
 \log\frac{2(K+3)(Q+1)\mathsf L}{\delta}
 \le
 Q+\log\frac{4(Q+2)\mathsf L}{\delta}.
\]
The $\kappa Q$ and logarithmic terms in
Eq.~\eqref{eq:drhs-one-phase-dimension} therefore ensure
$\frac{M}{64\kappa}\ge
\log\frac{K+3}{\delta_\star}$ for sufficiently large $C_0$.
Thus $\mathcal G_\star$ has probability at least
$1-\delta_\star$.

\medskip
\noindent\textbf{Step 5: Apply the spectral lower bound.}
On $\mathcal G_\star$,
$\sigma_\star^2\le\|h_{j_\star}\|_2^2\le2M$ and
$\beta^2\ge\frac{1}{4}$, so
$\frac{\beta^2}{\sigma_\star^2}\ge\frac{1}{8M}$.
Substituting this and
$\|R_i h_{j_\star}\|_2^2\ge\frac{n_i}{4}$
into Eq.~\eqref{eq:drhs-flat-error-reduction} gives
\[
\begin{aligned}
 \|\widetilde y-
 B_0\varepsilon\widehat A^{-1}\widetilde\xi\|_2^2
 &\ge
 \frac{B_0^2\varepsilon^2}{32}
 \sum_{i=1}^K\frac{n_i}{M}
 \left|\frac{1}{\lambda_i}-q(\lambda_i)\right|^2\\
 &\ge
 \frac{3B_0^2\varepsilon^2}{16384\mu^2}\\
 &=48\rho^2>9\rho^2.
\end{aligned}
\]
The second inequality is Lemma~\ref{lem:drhs-spectrum},
applicable because $\deg q\le t$.
Its bound holds for every such polynomial, including the
polynomial determined by the coupled realization.
The equality uses $B_0=512$ and
$\rho=\frac{\varepsilon}{\mu}$.

The orthogonal coupling preserves the distance:
$\|y-A^{-1}b\|_2
=\|\widetilde y-
B_0\varepsilon\widehat A^{-1}\widetilde\xi\|_2$.
Thus $\|y-A^{-1}b\|_2<3\rho$ implies
$\mathcal G_\star^c$.
Taking probabilities and using
Eq.~\eqref{eq:drhs-flat-tail-budget} proves the lemma with
Eq.~\eqref{eq:drhs-flat-fixed-time-exclusion}.
\end{proof}

\subsubsection{Proof of Lemma~\ref{lem:drhs-one-phase}}
\label{subsubsec:drhs-adaptive-completion}

\begin{proof}
\noindent\textbf{Adaptive disclosure and multiple candidates.}
Let $\tau\in\{0,\ldots,Q\}$ be the disclosure index.
If the procedure terminates without disclosure, append
zero-vector queries up to index $Q$, disclose $b$, and terminate
without naming further candidates.
Enumerate the candidates as $y_1,\ldots,y_{\mathsf L}$,
padding missing entries with zeros.

For each $(r,\ell)\in\{0,\ldots,Q\}\times[\mathsf L]$,
construct a procedure $\mathsf S_{r,\ell}$ that receives $b$
after exactly $r$ responses.
It simulates the original procedure, recording its $\ell$th
candidate whenever named, including before disclosure.
If the original rule requests disclosure before $r$, or
does not request it at $r$, the simulation switches to
zero-vector queries and a zero output.
Otherwise, it continues the original simulation and returns
the recorded candidate, or zero if none is named.
Pad every post-disclosure stage to $t-1$ queries.

These decisions use only available information.
Writing $y^{r,\ell}$ for the resulting output, we have
$y^{r,\ell}=y_\ell$ on $\{\tau=r\}$, and hence
\begin{equation}
\label{eq:drhs-flat-dummy-inclusion}
 \{\tau=r,\ \|y_\ell-A^{-1}b\|_2<3\rho\}
 \subseteq
 \{\|y^{r,\ell}-A^{-1}b\|_2<3\rho\}.
\end{equation}
Let $H$ be the event that some named vector lies within
$3\rho$ of $A^{-1}b$.
By Eq.~\eqref{eq:drhs-flat-dummy-inclusion} and
Lemma~\ref{lem:drhs-fixed-time},
\begin{equation}
\label{eq:drhs-flat-unconditioned-exclusion}
\begin{aligned}
 \Pr_{\mathsf P_0}(H)
 &\le
 \sum_{r=0}^Q\sum_{\ell=1}^{\mathsf L}
 \Pr_{\mathsf P_0}
 \bigl(\tau=r,\ \|y_\ell-A^{-1}b\|_2<3\rho\bigr)\\
 &\le
 \sum_{r=0}^Q\sum_{\ell=1}^{\mathsf L}
 \Pr_{\mathsf P_0}
 \bigl(\|y^{r,\ell}-A^{-1}b\|_2<3\rho\bigr)\\
 &\le (Q+1)\mathsf L\delta_\star
 =\frac{\delta}{2}.
\end{aligned}
\end{equation}

\medskip
\noindent\textbf{Conditioning on admissibility.}
Write $\xi=\frac{g}{\|g\|_2}$, where $g\sim N(0,I_M)$
is independent of $A$, and let $P_i$ project onto the
$\lambda_i$-eigenspace of $A$.
Define
\[
 \mathcal E:=
 \left\{
 \|g\|_2^2\le2M,\quad
 \|P_i g\|_2^2\ge\frac{n_i}{2}
 \text{ for every }i\in[K]
 \right\}.
\]
The chi-square tail bounds, $n_i\ge\frac{M}{2\kappa}$,
and Eq.~\eqref{eq:drhs-one-phase-dimension} give
\[
\begin{aligned}
 \Pr_{\mathsf P_0}(\mathcal E^c)
 &\le e^{-\frac{M}{8}}
      +\sum_{i=1}^K e^{-\frac{n_i}{16}}\le e^{-\frac{M}{8}}
      +K e^{-\frac{M}{32\kappa}}
 \le\frac{1}{8}.
\end{aligned}
\]
On $\mathcal E$,
\[
\begin{aligned}
 \|A^{-1}b\|_2^2
 &=
 B_0^2\varepsilon^2
 \frac{\sum_{i=1}^K
       \lambda_i^{-2}\|P_i g\|_2^2}
      {\|g\|_2^2}\ge
 \frac{B_0^2\varepsilon^2}{4}
 \sum_{i=1}^K\frac{n_i}{M}\lambda_i^{-2}\ge
 \frac{3B_0^2\varepsilon^2}{2048\mu^2}
 \ge(16\rho)^2,
\end{aligned}
\]
where the second inequality uses
Lemma~\ref{lem:drhs-spectrum} with $q\equiv0$,
and the last uses $B_0=512$ and
$\rho=\frac{\varepsilon}{\mu}$.
Moreover,
$\|A^{-1}b\|_2
\le\|A^{-1}\|_{\mathrm{op}}\|b\|_2
\le B_0\rho$ deterministically, as do
$\mu I\preceq A\preceq\bar L I$ and
$\|b\|_2=B_0\varepsilon$.
Thus the admissibility event $\mathcal G$ in
Eq.~\eqref{eq:drhs-phase-conditioning-event}
contains $\mathcal E$, so
$\mathsf P_0(\mathcal G)\ge\frac{7}{8}$.

The law $\mathsf P_{\mathrm{ph}}
=\mathsf P_0(\,\cdot\mid\mathcal G)$ therefore has the
required support. By
Eq.~\eqref{eq:drhs-flat-unconditioned-exclusion},
\[
\begin{aligned}
 \Pr_{\mathsf P_{\mathrm{ph}}}(H)
 &=
 \frac{\Pr_{\mathsf P_0}(H\cap\mathcal G)}
      {\mathsf P_0(\mathcal G)}\le
 \frac{\Pr_{\mathsf P_0}(H)}
      {\mathsf P_0(\mathcal G)}
 \le
 \frac{\frac{\delta}{2}}{\frac{7}{8}}
 =\frac{4\delta}{7}<\delta.
\end{aligned}
\]
The construction of $D$, $\mathsf P_0$, and $\mathcal G$
is independent of the procedure.
A procedure designed for $\mathsf P_{\mathrm{ph}}$ remains
a fixed family of Borel maps, to which the bound under
$\mathsf P_0$ applies.
\end{proof}

\subsection{Proof of Lemma~\ref{lem:drhs-sequential}:
Sequential Accumulation}
\label{subsec:proof-drhs-sequential}

We work under the right-hand-side disclosure protocol of
Remark~\ref{remark:str_oracle}.  The algorithm itself continues to make
only value-and-gradient queries and all matrix–vector products introduced below belong to an
auxiliary one-phase simulator.  We may pad the algorithm to \(T-1\) oracle
calls without changing its output: after freezing its original output, it
makes fixed dummy queries, ignores their responses, and finally returns the
stored point.  Denote the resulting query points by
\(z^{(0)},\ldots,z^{(T-2)}\), append the response-free terminal candidate
\(z^{(T-1)}:=z^{\mathrm{out}}\), and set \(\tau_0:=0\).  For
\(s\in[N]\), define
\[
 \tau_s:=
 \min\Bigl(
  \bigl\{j\in\{0,\ldots,T-1\}:
   \|z_s^{(j)}-x_s^\star\|_2<3\rho\bigr\}
  \cup\{+\infty\}
 \Bigr).
\]
Thus \(\tau_s\) records the first phase-\(s\) outer hit by either a query
point or the terminal output.  The proof reduces the fast-transition event
\(\{\tau_s<\tau_{s-1}+t\}\) to the success event of a legal one-phase
delayed-input procedure.

\begin{proposition}[\bf Phasewise delayed-input reduction]
\label{prop:drhs-phasewise-reduction}
Fix \(s\in[N]\) and condition on
\(\mathcal F_{-s}:=\sigma((A_r,b_r):r\neq s)\).  For almost every
realization of \(\mathcal F_{-s}\), there is a deterministic Borel
delayed-input procedure \(\mathsf S_s\) for \((A_s,b_s)\) that uses at
most \(T-1\) products with \(A_s\) before \(b_s\) is disclosed, fewer
than \(t\) products afterward, and names at most \(T\) vectors.  Under a
coupling with the conditioned chain execution,
\begin{equation}
\label{eq:drhs-seq-fast-reduction}
 \{\tau_s<\tau_{s-1}+t\}
 \subseteq
 \left\{
  \min\limits_{y\ {\rm named\ by}\ \mathsf S_s}
  \|y-A_s^{-1}b_s\|_2<3\rho
 \right\},
 \qquad s\in[N],
\end{equation}
where \(+\infty+t:=+\infty\).
\end{proposition}

\begin{proof}
\noindent\textbf{Step 1: Construct the simulator.}
Fix $s$ and a realization of $\mathcal F_{-s}$.
The other phase data are available to the simulator for reconstructing
responses, but are not supplied to the original algorithm.
Whenever the simulated algorithm selects a candidate
$\widehat z^{(j)}$, the simulator freezes it and records its
$s$th block $\widehat z_s^{(j)}$.

For $s=1$, disclose $b_1$ before selecting the first candidate and
set $\sigma_1:=0$.
For $s\ge2$, let $\sigma_s$ be the first candidate index satisfying
$\|\widehat z_{s-1}^{(j)}-x_{s-1}^\star\|_2<3\rho$.
At this index, disclose $b_s$ after the candidate has been frozen
and recorded, but before returning any response.
If no such index occurs, set $\sigma_s:=T$ and disclose $b_s$
after the terminal candidate.
In every case, stop after recording candidate
\[
 \ell_s:=\min\{\sigma_s+t-1,T-1\},
\]
without requesting its response.

At each earlier candidate, request
$v^{(j)}:=A_s\widehat z_s^{(j)}$.
This supplies the current quadratic gradient and value,
respectively,
$v^{(j)}$ and
$\frac{1}{2}\langle\widehat z_s^{(j)},v^{(j)}\rangle$.
All terms involving only other phases are directly computable.

For $s\ge2$, the incoming link $-G_{s-1}\Phi_s$ contributes zero
before disclosure, since both $G_{s-1}$ and $\nabla G_{s-1}$
vanish there. After disclosure, $\Phi_s$ and its derivatives
are computable from $b_s$.
For $s=1$, $b_1$ is available from the start.
For $s<N$, delete the entire outgoing link $-G_s\Phi_{s+1}$,
including its value and both gradient contributions, and suppress
the corresponding disclosure of $b_{s+1}$.
All other disclosures follow the original protocol.
Thus no response computation requires the unknown center
$x_s^\star=A_s^{-1}b_s$.

\medskip
\noindent\textbf{Step 2: Verify validity and query budgets.}
For $s\ge2$, the disclosure test uses only the current frozen
candidate and the known center $x_{s-1}^\star$.
It can therefore be performed before returning the current
response, using only information already available.
All decisions and response computations are Borel measurable.

If $\sigma_s<T$, the pre-disclosure products have indices
$0,\ldots,\sigma_s-1$, while the post-disclosure products have
indices $\sigma_s,\ldots,\ell_s-1$.
Their numbers are therefore $\sigma_s\le T-1$ and
$\ell_s-\sigma_s\le t-1$, respectively.
This includes $s=1$, for which $\sigma_1=0$ and $\ell_1=t-1$.
If $\sigma_s=T$, there are $T-1$ pre-disclosure products and
none afterward.
In every case, at most $T$ candidates are recorded.
In particular, if the disclosure candidate is also the stopping
candidate, no product is requested at that index.

\medskip
\noindent\textbf{Step 3: Establish agreement up to the first hit.}
Couple the simulated and true executions using the same phase data.
We claim that
\begin{equation}
\label{eq:drhs-seq-coupling}
 \widehat z^{(j)}=z^{(j)}
 \qquad
 \bigl(0\le j\le\min\{\tau_s,\ell_s\}\bigr).
\end{equation}

We induct on the complete information returned to the algorithm,
including both oracle responses and right-hand-side disclosures.
If the preceding records agree, determinism gives the same
current candidate in both executions.
For $j<\tau_s$, the true candidate satisfies
$G_s(z_s^{(j)})=0$ and $\nabla G_s(z_s^{(j)})=0$.
Hence the deleted outgoing link contributes zero to the true
value and gradient, and $b_{s+1}$ has not yet been disclosed.
The incoming link is either zero before disclosure or exactly
reconstructed afterward; all remaining terms and disclosures
also agree.
Thus, whenever a response is returned at such an index, the
complete records remain identical.

This proves Eq.~\eqref{eq:drhs-seq-coupling}.
At the first phase-$s$ hit, only agreement of the frozen
candidate is needed: its $s$th block has already been recorded,
regardless of whether the subsequent responses agree.

\medskip
\noindent\textbf{Step 4: Capture every fast transition.}
Suppose that $\tau_s<\tau_{s-1}+t$.
Then $\tau_s$ is finite and $\tau_s\le T-1$.
It suffices to show $\tau_s\le\ell_s$, since
Eq.~\eqref{eq:drhs-seq-coupling} then guarantees that the
successful vector $z_s^{(\tau_s)}$ is recorded.

For $s=1$, the fast-transition condition gives
$\tau_1\le t-1=\ell_1$.

Now let $s\ge2$.
If $\tau_s<\tau_{s-1}$, then $\sigma_s>\tau_s$.
Indeed, otherwise $\sigma_s\le\tau_s$ and
$\sigma_s\le\ell_s$, so Eq.~\eqref{eq:drhs-seq-coupling}
would give a true predecessor hit at index
$\sigma_s\le\tau_s<\tau_{s-1}$, a contradiction.
The definition of $\ell_s$ therefore gives
$\tau_s\le\ell_s$, and the successful vector is recorded
before disclosure.

Otherwise, put $r:=\tau_{s-1}\le\tau_s$.
Here $r$ is finite.
The same argument excludes $\sigma_s<r$, so
$\sigma_s\ge r$ and consequently $\ell_s\ge r$.
Equation~\eqref{eq:drhs-seq-coupling} then applies at index $r$,
where the true predecessor is first hit, and forces
$\sigma_s=r$.
By integrality and the fast-transition condition,
\[
 \tau_s
 \le\min\{r+t-1,T-1\}
 =\ell_s.
\]
Thus the successful vector is again recorded within the
allowed query budget.
This includes $\tau_s=\tau_{s-1}$, when the common candidate
is recorded before $b_s$ is disclosed.
The required inclusion
Eq.~\eqref{eq:drhs-seq-fast-reduction} follows.
\end{proof}

\begin{proof}[Proof of Lemma~\ref{lem:drhs-sequential}]
Fix \(s\in[N]\) and condition on \(\mathcal F_{-s}\).  Independence of
the phases gives
\(\mathcal L((A_s,b_s)\mid\mathcal F_{-s})=\mathsf P_{\mathrm{ph}}\).
Apply Lemma~\ref{lem:drhs-one-phase} to the procedure supplied by
Proposition~\ref{prop:drhs-phasewise-reduction}, with
\(Q=\mathsf L=T\) and \(\delta=1/(12N)\).  Here \(Q=T=Nt\geq t\),
\(\mathsf L=T\geq1\), and \(\delta\leq1/12<1/4\).  Moreover,
\[
 \log\frac{(Q+2)\mathsf L}{\delta}
 =
 \log\bigl(12NT(T+2)\bigr),
\]
so the assumed lower bound on \(M\) is exactly the dimension hypothesis
of the one-phase lemma.  Proposition~\ref{prop:drhs-phasewise-reduction}
and that lemma therefore give, almost surely,
\[
 \Pr\bigl(\tau_s<\tau_{s-1}+t\mid\mathcal F_{-s}\bigr)
 \leq
 \Pr\left(
  \min_{y\ {\rm named\ by}\ \mathsf S_s}
  \|y-A_s^{-1}b_s\|_2<3\rho
  \,\middle|\,\mathcal F_{-s}
 \right)
 \leq\frac1{12N}.
\]
Taking expectations over \(\mathcal F_{-s}\) yields
\begin{equation}
\label{eq:drhs-seq-phasewise-delay}
 \Pr(\tau_s<\tau_{s-1}+t)
 \leq\frac1{12N},
 \qquad s\in[N].
\end{equation}

Let
\(\mathcal B:=\bigcup_{s=1}^N
\{\tau_s<\tau_{s-1}+t\}\).  By
\eqref{eq:drhs-seq-phasewise-delay} and the union bound,
\(\Pr(\mathcal B)\leq1/12\).  On \(\mathcal B^c\), we have
\(\tau_s\geq\tau_{s-1}+t\) for every \(s\in[N]\).  Since
\(\tau_0=0\), induction gives \(\tau_s\geq st\), and in particular
\(\tau_N\geq Nt=T\).

If the terminal output lies within \(3\rho\) of every phase center, then
its \(N\)th block gives \(\tau_N\leq T-1\).  Hence
\[
 \left\{
  \|z_s^{\mathrm{out}}-x_s^\star\|_2<3\rho
  \text{ for every }s\in[N]
 \right\}
 \subseteq\mathcal B,
 \qquad
 \Pr\left(
  \|z_s^{\mathrm{out}}-x_s^\star\|_2<3\rho
  \text{ for every }s\in[N]
 \right)
 \leq\frac1{12}.
\]
The ambient term \(\mu\|u\|_2^2/2\) is independent of the phase data, and
its value and gradient can be reconstructed without any product with
\(A_s\).  It therefore does not affect the simulator or the argument.
\end{proof}

\subsection{Proof of Lemma~\ref{lem:drhs-certification}:
Stationarity Certification}
\label{subsec:drhs-certification-proof}

The proof has two steps.  We first show that an active phase lying outside
its inner ball already has a gradient block of norm larger than
\(\varepsilon\).  We then apply this local statement to the first phase
that is not inner-complete; all preceding phases are inner-complete, so its
incoming forcing is fully active.

For \(s\in[N]\), set \(\alpha_1(z):=1\),
\(\alpha_s(z):=G_{s-1}(z_{s-1})\) for \(s\geq2\), and
\[
 \beta_s(z):=
 \begin{cases}
  \Phi_{s+1}(z_{s+1})\nabla G_s(z_s),&s<N,\\
  0,&s=N.
 \end{cases}
\]
Then \eqref{eq:drhs-hard-gradient} becomes
\begin{equation}
\label{eq:drhs-cert-block-gradient}
 \nabla_sF_{\mathbf A,\mathbf b}(z)
 =A_sz_s-\alpha_s(z)\nabla\Phi_s(z_s)-\beta_s(z).
\end{equation}

\begin{proposition}[\bf Active-phase gradient separation]
\label{prop:drhs-active-phase-gradient}
Fix an admissible phase realization, \(z\in(\mathbb R^M)^N\), and
\(s\in[N]\).  If \(\alpha_s(z)=1\) and
\(\|z_s-x_s^\star\|_2\geq2\rho\), then
\[
 \|\nabla_sF_{\mathbf A,\mathbf b}(z)\|_2>\varepsilon.
\]
\end{proposition}

\begin{proof}
The construction gives \(\Phi_s>0\) in
\eqref{eq:drhs-forcing-positive} and
\(\|x_s^\star\|_2+3\rho<R\) in
\eqref{eq:drhs-clipping-separation}.  We shall also use
\[
 \sup_w\|\nabla\Phi_s(w)\|_2
 \leq a_\chi\|b_s\|_2
 =a_\chi B_0\varepsilon,
\]
which follows from
\(\nabla\Phi_s=(D\pi_R)^\top b_s\) and the previously established bound
\(\sup_w\|D\pi_R(w)\|_{\mathrm{op}}\leq a_\chi\).

Put \(e_s:=z_s-x_s^\star\).  We distinguish whether clipping is active.
Suppose first that \(\|z_s\|_2\leq R\).  Then \(\pi_R\) is the identity
at \(z_s\), and hence
\(\nabla\Phi_s(z_s)=b_s=A_sx_s^\star\).  Using
\(\alpha_s(z)=1\) in \eqref{eq:drhs-cert-block-gradient}, we obtain
\[
 \nabla_sF_{\mathbf A,\mathbf b}(z)=A_se_s-\beta_s(z).
\]
The gate derivative is
\(\nabla G_s(z_s)=2\rho^{-2}
\eta'(\|e_s\|_2^2/\rho^2)e_s\).  Therefore, when \(s<N\),
\[
 \left\langle e_s,
 -\Phi_{s+1}(z_{s+1})\nabla G_s(z_s)\right\rangle
 =
 -\frac{2\Phi_{s+1}(z_{s+1})}{\rho^2}
 \eta'\!\left(\frac{\|e_s\|_2^2}{\rho^2}\right)
 \|e_s\|_2^2
 \geq0,
\]
where we used \eqref{eq:drhs-forcing-positive} and \(\eta'\leq0\).
For \(s=N\), this term is absent.  Thus
\[
 \|e_s\|_2\,
 \|\nabla_sF_{\mathbf A,\mathbf b}(z)\|_2
 \geq
 \langle e_s,\nabla_sF_{\mathbf A,\mathbf b}(z)\rangle
 \geq e_s^\top A_se_s
 \geq\mu\|e_s\|_2^2.
\]
Because \(\|e_s\|_2\geq2\rho>0\), division by \(\|e_s\|_2\) yields
\[
 \|\nabla_sF_{\mathbf A,\mathbf b}(z)\|_2
 \geq\mu\|e_s\|_2
 \geq2\mu\rho
 =2\varepsilon.
\]

Suppose now that \(\|z_s\|_2>R\).  Admissibility and
\(R=C_R\rho\) imply
\[
 \|z_s-x_s^\star\|_2
 \geq\|z_s\|_2-\|x_s^\star\|_2
 >R-B_0\rho
 =(C_R-B_0)\rho
 >3\rho.
\]
The outgoing gate is therefore in its exact dead zone, so
\(\beta_s(z)=0\).  Equation~\eqref{eq:drhs-cert-block-gradient}
reduces to
\(\nabla_sF_{\mathbf A,\mathbf b}(z)
=A_sz_s-\nabla\Phi_s(z_s)\).  Using
\(A_s\succeq\mu I\), the uniform forcing-gradient bound above, and
\(\mu R=C_R\varepsilon\), we obtain
\[
 \begin{aligned}
 \|\nabla_sF_{\mathbf A,\mathbf b}(z)\|_2
 &\geq\|A_sz_s\|_2-\|\nabla\Phi_s(z_s)\|_2\\
 &\geq\mu\|z_s\|_2-a_\chi B_0\varepsilon\\
 &>(C_R-a_\chi B_0)\varepsilon\\
 &=(B_0+4)\varepsilon
 >\varepsilon.
 \end{aligned}
\]
This proves the proposition in both regimes.
\end{proof}

\begin{proof}[Proof of Lemma~\ref{lem:drhs-certification}]
Suppose, toward the contrapositive, that at least one phase is not
inner-complete, and let
\[
 s:=\min\{r\in[N]:
 \|z_r-x_r^\star\|_2\geq2\rho\}.
\]
If \(s=1\), then \(\alpha_s(z)=1\) by definition.  If \(s\geq2\),
minimality gives
\(\|z_{s-1}-x_{s-1}^\star\|_2<2\rho\), so
\(G_{s-1}(z_{s-1})=1\) and again \(\alpha_s(z)=1\).
Proposition~\ref{prop:drhs-active-phase-gradient} therefore gives
\(\|\nabla_sF_{\mathbf A,\mathbf b}(z)\|_2>\varepsilon\).

For the ambient extension,
\[
 \|\nabla F_{\mathbf A,\mathbf b}^{(d)}(z,u)\|_2^2
 =
 \sum_{r=1}^N
 \|\nabla_rF_{\mathbf A,\mathbf b}(z)\|_2^2
 +\mu^2\|u\|_2^2
 >\varepsilon^2.
\]
Thus failure of inner completion implies
\(\|\nabla F_{\mathbf A,\mathbf b}^{(d)}(z,u)\|_2>\varepsilon\).
Taking the contrapositive proves the lemma.
\end{proof}

\subsection{Proof of Lemma~\ref{lem:drhs-regularity}:
Regularity and Initial Gap}
\label{subsec:drhs-regularity-proof}

The proof separates local calculus from product-space aggregation.  The
first proposition gives the derivative scales of one gate and one clipped
forcing.  The second uses the nearest-neighbor structure of the chain to
aggregate these bounds without a factor depending on \(N\).  We then
deduce the two Lipschitz estimates, coercivity, and the initial-gap bound.

\begin{proposition}[\bf Local derivative bounds]
\label{prop:drhs-local-derivatives}
For every integer \(j\geq1\), there are constants
\(a_j,b_j<\infty\), depending only on \(j\) and the fixed cutoff
profile, such that, for every phase \(s\),
\begin{equation}
\label{eq:drhs-reg-local-derivatives}
 \sup_w\|D^jG_s(w)\|_{\mathrm{op}}
 \leq a_j\rho^{-j},
 \qquad
 \sup_w\|D^j\Phi_s(w)\|_{\mathrm{op}}
 \leq b_j\varepsilon\rho^{1-j}.
\end{equation}
The corresponding zero-order bounds are
\begin{equation}
\label{eq:drhs-reg-forcing-range}
 0\leq G_s(w)\leq1,
 \qquad
 2B_0\varepsilon R
 \leq\Phi_s(w)\leq6B_0\varepsilon R,
 \qquad w\in\mathbb R^M.
\end{equation}
All constants are independent of \(M\), \(s\), and the phase
realization.
\end{proposition}

\begin{proof}
Define the unscaled maps
\(\Gamma(y):=\eta(\|y\|_2^2)\) and
\(\Pi(y):=\chi(\|y\|_2)y\).  Then
\[
 G_s(w)=\Gamma\!\left(\frac{w-x_s^\star}{\rho}\right),
 \qquad
 \pi_R(w)=R\Pi\!\left(\frac{w}{R}\right).
\]

Let \(r_2(y):=\|y\|_2^2\).  Its only nonzero derivatives are
\[
 Dr_2(y)[h]=2\langle y,h\rangle,
 \qquad
 D^2r_2(y)[h_1,h_2]=2\langle h_1,h_2\rangle.
\]
Thus the multilinear chain rule gives
\[
 D^j\Gamma(y)[h_1,\ldots,h_j]
 =
 \sum_{\mathcal P}
 \eta^{(|\mathcal P|)}(r_2(y))
 \prod_{B\in\mathcal P}
 D^{|B|}r_2(y)[h_i:i\in B],
\]
where the sum is over the set partitions of \([j]\) whose blocks have
size one or two.  Every summand is zero unless \(\|y\|_2\leq3\).
Consequently, when \(\|h_i\|_2\leq1\), every factor is bounded by a
constant depending only on \(j\), and the number of summands also depends
only on \(j\).  Hence
\[
 \sup_y\|D^j\Gamma(y)\|_{\mathrm{op}}\leq a_j,
\]
uniformly in \(M\).  Scaling by \(\rho\) proves the gate bound in
\eqref{eq:drhs-reg-local-derivatives}.

For the clipping map, \(\Pi(y)=y\) on \(\|y\|_2\leq1\) and
\(\Pi(y)=0\) on \(\|y\|_2\geq2\).  It remains to control the annulus
\(1<\|y\|_2<2\).  Write
\(r(y):=\|y\|_2\) and \(e(y):=y/r(y)\).  Then
\[
 Dr(y)[h]=\langle e(y),h\rangle,
 \qquad
 De(y)[h]
 =r(y)^{-1}(I-e(y)e(y)^\top)h.
\]
Repeated differentiation shows that \(D^kr(y)\) is a finite sum of
contractions among \(e(y),h_1,\ldots,h_k\), multiplied by powers of
\(r(y)^{-1}\).  Since \(1<r(y)<2\), for every fixed \(k\geq1\),
\[
 \sup_{1<\|y\|_2<2}
 \|D^kr(y)\|_{\mathrm{op}}<\infty,
\]
with a bound independent of \(M\).  The product and chain rules now imply
\(\sup_y\|D^j\Pi(y)\|_{\mathrm{op}}\leq d_j\).  Smoothness across the
radii one and two follows because \(\chi\) agrees with a constant on
both sides and all positive-order derivatives match there.

The scaling identity
\(D^j\pi_R(w)=R^{1-j}D^j\Pi(w/R)\) therefore gives
\[
 \sup_w\|D^j\pi_R(w)\|_{\mathrm{op}}
 \leq d_jR^{1-j}.
\]
For \(j\geq1\), contracting with \(b_s\) and using
\(\|b_s\|_2=B_0\varepsilon\), we obtain
\[
 \sup_w\|D^j\Phi_s(w)\|_{\mathrm{op}}
 \leq B_0d_j\varepsilon R^{1-j}
 =B_0d_jC_R^{1-j}\varepsilon\rho^{1-j}.
\]
For the zero-order estimate,
\(\|\pi_R(w)\|_2\leq2R\) gives
\[
 |\langle b_s,\pi_R(w)\rangle|
 \leq2B_0\varepsilon R.
\]
Together with \(H=4B_0\varepsilon R\), this proves
\eqref{eq:drhs-reg-forcing-range}; the gate bound follows directly from
\(0\leq\eta\leq1\).
\end{proof}

\begin{proposition}[\bf Dimension-free aggregation]
\label{prop:drhs-link-aggregation}
Let
\[
 \mathcal P(z):=
 \Phi_1(z_1)
 +\sum_{s=2}^N
 G_{s-1}(z_{s-1})\Phi_s(z_s).
\]
For every integer \(r\geq2\), there is a constant \(K_r<\infty\),
depending only on \(r\) and the fixed cutoff profile, such that
\begin{equation}
\label{eq:drhs-reg-global-derivative}
 \sup_z\|D^r\mathcal P(z)\|_{\mathrm{op}}
 \leq K_r\varepsilon\rho^{1-r}.
\end{equation}
The constant \(K_r\) is independent of \(M\), \(N\), and the phase
realization.
\end{proposition}

\begin{proof}
Fix directions
\(v^{(1)},\ldots,v^{(r)}\in(\mathbb R^M)^N\), where
\(v^{(i)}=(v_1^{(i)},\ldots,v_N^{(i)})\).  For one link, the
multilinear Leibniz rule gives
\[
 \begin{aligned}
 &D^r\!\left[
   G_{s-1}(z_{s-1})\Phi_s(z_s)
 \right][v^{(1)},\ldots,v^{(r)}]=
 \sum_{I\subseteq[r]}
 D^{|I|}G_{s-1}(z_{s-1})
 [v_{s-1}^{(i)}:i\in I]\,
 D^{r-|I|}\Phi_s(z_s)
 [v_s^{(i)}:i\notin I].
 \end{aligned}
\]
By the derivative and zero-order bounds of
Proposition~\ref{prop:drhs-local-derivatives}, the summand indexed by
\(I\) has absolute value at most
\[
 K_r'\varepsilon\rho^{1-r}
 \prod_{i\in I}\|v_{s-1}^{(i)}\|_2
 \prod_{i\notin I}\|v_s^{(i)}\|_2.
\]
Indeed, the gate contributes \(\rho^{-|I|}\), while the forcing
contributes
\(\varepsilon\rho^{1-r+|I|}\), so their product is always
\(\varepsilon\rho^{1-r}\), independently of \(I\).

For a fixed \(I\), generalized H\"older's inequality and
\(\|a\|_{\ell_r}\leq\|a\|_{\ell_2}\) for \(r\geq2\) give
\[
 \begin{aligned}
 \sum_{s=2}^N
 \prod_{i\in I}\|v_{s-1}^{(i)}\|_2
 \prod_{i\notin I}\|v_s^{(i)}\|_2&\quad\leq
 \prod_{i\in I}
 \left(
  \sum_{s=2}^N\|v_{s-1}^{(i)}\|_2^r
 \right)^{1/r}
 \prod_{i\notin I}
 \left(
  \sum_{s=2}^N\|v_s^{(i)}\|_2^r
 \right)^{1/r}\\
 &\quad\leq
 \prod_{i=1}^r\|v^{(i)}\|_2.
 \end{aligned}
\]
Summing over the links and the \(2^r\) subsets \(I\) therefore costs
only an \(r\)-dependent factor.  The initial term satisfies
\[
 \left|
 D^r\Phi_1(z_1)
 [v_1^{(1)},\ldots,v_1^{(r)}]
 \right|
 \leq
 b_r\varepsilon\rho^{1-r}
 \prod_{i=1}^r\|v^{(i)}\|_2.
\]
Taking the supremum over unit directions proves
\eqref{eq:drhs-reg-global-derivative}.
\end{proof}

\begin{proof}[Proof of Lemma~\ref{lem:drhs-regularity}]
Write
\[
 \mathcal Q(z,u):=
 \frac12\sum_{s=1}^Nz_s^\top A_sz_s
 +\frac{\mu}{2}\|u\|_2^2,
 \qquad
 F_{\mathbf A,\mathbf b}^{(d)}
 =\mathcal Q-\mathcal P,
\]
with the \(u\)-term omitted when \(d=NM\).  The first proposition shows
that all local components, and hence \(F_{\mathbf A,\mathbf b}^{(d)}\),
are \(C^\infty\).  We regard \(\mathcal P\) as independent of \(u\);
the ambient operator norm of each derivative is then the same as its
operator norm on the hard coordinates.

The Hessian of \(\mathcal Q\) is
\(\operatorname{diag}(A_1,\ldots,A_N,\mu I)\), whose operator norm is at
most \(\bar L\).  Proposition~\ref{prop:drhs-link-aggregation} with
\(r=2\), together with \(\varepsilon/\rho=\mu\), gives
\[
 \sup_x
 \|D^2F_{\mathbf A,\mathbf b}^{(d)}(x)\|_{\mathrm{op}}
 \leq\bar L+K_2\mu.
\]
For arbitrary \(x,y\in\mathbb R^d\),
\[
 \nabla F(x)-\nabla F(y)
 =
 \int_0^1
 D^2F(y+\theta(x-y))[x-y]\,d\theta.
\]
Taking norms proves
\[
 \operatorname{Lip}
 \bigl(\nabla F_{\mathbf A,\mathbf b}^{(d)}\bigr)
 \leq\bar L+C_{\mathrm{sm}}\mu,
 \qquad
 C_{\mathrm{sm}}:=K_2.
\]

Now fix \(q\geq2\).  Since \(D^{q+1}\mathcal Q=0\), the second
proposition with \(r=q+1\) gives
\begin{equation}
\label{eq:drhs-reg-qplusone}
 \sup_x
 \|D^{q+1}F_{\mathbf A,\mathbf b}^{(d)}(x)\|_{\mathrm{op}}
 \leq
 K_{q+1}\varepsilon\rho^{-q}
 =
 K_{q+1}\frac{\mu^q}{\varepsilon^{q-1}}.
\end{equation}
Applying the line-segment identity to \(D^qF\) and using
\eqref{eq:drhs-reg-qplusone},
\[
 \begin{aligned}
 \|D^qF(x)-D^qF(y)\|_{\mathrm{op}}
 &\leq
 \int_0^1
 \|D^{q+1}F(y+\theta(x-y))\|_{\mathrm{op}}
 \|x-y\|_2\,d\theta\\
 &\leq
 K_{q+1}\frac{\mu^q}{\varepsilon^{q-1}}
 \|x-y\|_2.
 \end{aligned}
\]
Thus
\[
 \operatorname{Lip}
 \bigl(D^qF_{\mathbf A,\mathbf b}^{(d)}\bigr)
 \leq
 C_q\frac{\mu^q}{\varepsilon^{q-1}},
 \qquad
 C_q:=K_{q+1}.
\]

It remains to prove coercivity and the gap estimate.  By
\eqref{eq:drhs-reg-forcing-range},
\(0\leq\mathcal P(z)\leq6NB_0\varepsilon R\).  Since
\(A_s\succeq\mu I\),
\begin{equation}
\label{eq:drhs-reg-coercive-lower}
 F_{\mathbf A,\mathbf b}^{(d)}(z,u)
 \geq
 \frac{\mu}{2}
 \left(
  \sum_{s=1}^N\|z_s\|_2^2+\|u\|_2^2
 \right)
 -6NB_0\varepsilon R.
\end{equation}
The right-hand side tends to \(+\infty\) with \(\|(z,u)\|_2\), so the
function is coercive.

At the origin, admissibility gives
\(\|x_s^\star\|_2\geq16\rho>3\rho\), and hence \(G_s(0)=0\).
Also \(\pi_R(0)=0\), so
\[
 F_{\mathbf A,\mathbf b}^{(d)}(0)
 =-\Phi_1(0)
 =-H
 =-4B_0\varepsilon R.
\]
Equation~\eqref{eq:drhs-reg-coercive-lower} gives
\(\inf_xF_{\mathbf A,\mathbf b}^{(d)}(x)
\geq-6NB_0\varepsilon R\).  Therefore
\[
 \begin{aligned}
 F_{\mathbf A,\mathbf b}^{(d)}(0)
 -\inf_xF_{\mathbf A,\mathbf b}^{(d)}(x)
 &\leq(6N-4)B_0\varepsilon R\\
 &\leq
 6B_0C_RN\frac{\varepsilon^2}{\mu},
 \end{aligned}
\]
where \(R=C_R\rho=C_R\varepsilon/\mu\).  Hence the gap bound holds with
\(C_{\mathrm{gap}}:=6B_0C_R\).  All constants above depend only on the
relevant derivative order and the fixed cutoff profile, and not on
\(M,N,d,\mu,\bar L,\varepsilon\), or the phase realization.
\end{proof}

\subsection{Proof of Lemma~\ref{lem:drhs-spectrum}}
\label{app:drhs-spectrum}
\begin{lemma}[\bf Replicated hard spectrum]
There are numerical constants \(c_0>0\) and \(\kappa_0\ge4\) such that the
following holds.  Suppose \(\kappa\ge\kappa_0\), set
\(t:=\lfloor c_0\sqrt\kappa\rfloor\ge1\) and \(K:=t+2\), and let
\(M\ge2\kappa\) be an integer.  Then there are distinct numbers
\(\mu\le\lambda_1<\cdots<\lambda_K\le\bar L\) and positive integers
\(n_1,\ldots,n_K\), with \(\sum_{i=1}^K n_i=M\), such that, for every
real polynomial \(q\) of degree at most \(t\),
\begin{equation}
 n_i\ge\frac{M}{2\kappa}\quad(i\in[K]),
 \qquad
 \sum_{i=1}^K\frac{n_i}{M}
 \left|\frac1{\lambda_i}-q(\lambda_i)\right|^2
 \ge\frac{3}{512\mu^2}.
\end{equation}
Consequently,
\(D:=\operatorname{diag}(\lambda_1I_{n_1},\ldots,\lambda_KI_{n_K})\)
satisfies \(\mu I\preceq D\preceq\bar L I\).
\end{lemma}
\begin{proof}

Let $\Pi_t$ denote the space of real polynomials of degree at most $t$,
and write
\[
 E_t
 :=
 \inf_{q\in\Pi_t}
 \sup_{\lambda\in[\mu,\bar L]}
 \left|\frac{1}{\lambda}-q(\lambda)\right|.
\]

\begin{fact}[\bf Uniform approximation lower bound]
\label{fact:drhs-uniform-approximation}
Suppose that $\kappa=\bar L/\mu\geq4$ and
$t\leq c_0\sqrt{\kappa}$, where $c_0>0$ is a sufficiently small
numerical constant. Then $E_t\geq\frac{1}{8\mu}$.
\end{fact}

\begin{proof}
Suppose, toward a contradiction, that there exists $q\in\Pi_t$
such that
$\|\lambda^{-1}-q(\lambda)\|_{\infty,[\mu,\bar L]}
<\frac{1}{8\mu}$.
Since $\kappa\geq4$, both $\mu$ and $4\mu$ belong to
$[\mu,\bar L]$. Evaluating the approximation bound at these points gives
$q(\mu)>\frac{7}{8\mu}$ and
$q(4\mu)<\frac{3}{8\mu}$.
Thus $q(\mu)-q(4\mu)>\frac{1}{2\mu}$, and the mean-value theorem yields
$\|q'\|_{\infty,[\mu,\bar L]}>\frac{1}{6\mu^2}$.

On the other hand, the same approximation bound implies
$\|q\|_{\infty,[\mu,\bar L]}\leq\frac{9}{8\mu}$.
Markov's inequality for polynomials, rescaled from $[-1,1]$
to $[\mu,\bar L]$, therefore gives
\[
 \begin{aligned}
 \|q'\|_{\infty,[\mu,\bar L]}
 &\leq
 \frac{2t^2}{\bar L-\mu}
 \|q\|_{\infty,[\mu,\bar L]}\\
 &\leq
 \frac{9t^2}{4(\kappa-1)\mu^2}
 \leq
 \frac{9c_0^2\kappa}{4(\kappa-1)\mu^2}
 \leq
 \frac{3c_0^2}{\mu^2},
 \end{aligned}
\]
where the last inequality uses
$\frac{\kappa}{\kappa-1}\leq\frac{4}{3}$.
Choosing $c_0$ so that $3c_0^2<\frac{1}{6}$ contradicts
the preceding derivative lower bound.
Hence $E_t\geq\frac{1}{8\mu}$.
\end{proof}

The uniform approximation bound does not yet specify a fixed spectral
measure. We next construct positive weights for which a best uniform
approximant is also a best weighted least-squares approximant.

\begin{fact}[\bf Discrete weighted approximation lower bound]
\label{fact:drhs-weighted-approximation}
Let $K:=t+2$. There exist distinct points
$\mu\leq\lambda_1<\cdots<\lambda_K\leq\bar L$
and positive weights $\theta_1,\ldots,\theta_K$, with
$\sum_{i=1}^K\theta_i=1$, such that
\[
 \sum_{i=1}^K\theta_i
 \left|\frac{1}{\lambda_i}-q(\lambda_i)\right|^2
 \geq E_t^2
 \qquad\text{for every }q\in\Pi_t.
\]
\end{fact}

\begin{proof}
Let $q^\star\in\Pi_t$ be a best uniform approximant to
$\lambda\mapsto\lambda^{-1}$ on $[\mu,\bar L]$.
By the Chebyshev alternation theorem, there exist
$K=t+2$ ordered points $\lambda_1<\cdots<\lambda_K$ such that
the residuals
$e_i:=\lambda_i^{-1}-q^\star(\lambda_i)$
satisfy $|e_i|=E_t$ and alternate in sign.

For each $i\in[K]$, define
$c_i:=\prod_{j\neq i}(\lambda_i-\lambda_j)^{-1}$.
We first observe that
\[
 \sum_{i=1}^K c_i s(\lambda_i)=0
 \qquad\text{for every }s\in\Pi_t.
\]
Indeed, the $i$th Lagrange basis polynomial
$\ell_i(x):=\prod_{j\neq i}\frac{x-\lambda_j}{\lambda_i-\lambda_j}$
has leading coefficient $c_i$.
Comparing the coefficients of $x^{K-1}$ in the interpolation identity
$s(x)=\sum_{i=1}^K s(\lambda_i)\ell_i(x)$
gives the claim, since $\deg(s)\leq t=K-2$.

Set $C:=\sum_{j=1}^K|c_j|$ and $\theta_i:=\frac{|c_i|}{C}$.
These weights are positive and sum to one.
Since $\operatorname{sgn}(c_i)=(-1)^{K-i}$, the signs of the $c_i$
alternate just as those of the $e_i$ do.
Consequently, there is a common sign $\sigma\in\{-1,1\}$ such that
$e_i=\sigma E_t\operatorname{sgn}(c_i)$ for every $i$.
Thus $\theta_i e_i=\frac{\sigma E_t}{C}c_i$, and the preceding
interpolation identity yields
$\sum_{i=1}^K\theta_i e_i s(\lambda_i)=0$
for every $s\in\Pi_t$.

Now fix $q\in\Pi_t$ and put $h:=q-q^\star\in\Pi_t$.
Since $\lambda_i^{-1}-q(\lambda_i)=e_i-h(\lambda_i)$,
expanding the square gives
\[
 \begin{aligned}
 \sum_{i=1}^K\theta_i
 \left|\frac{1}{\lambda_i}-q(\lambda_i)\right|^2
 &=
 \sum_{i=1}^K\theta_i e_i^2
 -2\sum_{i=1}^K\theta_i e_i h(\lambda_i)
 +\sum_{i=1}^K\theta_i|h(\lambda_i)|^2\\
 &=
 E_t^2+\sum_{i=1}^K\theta_i|h(\lambda_i)|^2
 \geq E_t^2.
 \end{aligned}
\]
Here the cross term vanishes by the orthogonality relation with
$s=h$, and the first term equals $E_t^2$ because
$|e_i|=E_t$ and $\sum_i\theta_i=1$.
\end{proof}

The weights from Fact~\ref{fact:drhs-weighted-approximation}
may be arbitrarily small and need not correspond to integer
multiplicities. The following fact addresses both issues.

\begin{fact}[\bf Regularizing and rounding the weights]
\label{fact:drhs-integer-multiplicities}
Suppose that $K\leq\frac{\kappa}{4}$ and that $M\geq2\kappa$
is an integer. For any positive weights
$\theta_1,\ldots,\theta_K$ summing to one, there exist positive
integers $n_1,\ldots,n_K$, summing to $M$, such that
\[
 n_i\geq\frac{M}{2\kappa},
 \qquad
 \frac{n_i}{M}\geq\frac{3}{8}\theta_i
 \qquad(i\in[K]).
\]
\end{fact}

\begin{proof}
Define
$\omega_i:=\left(1-\frac{K}{\kappa}\right)\theta_i
+\frac{1}{\kappa}$.
Then $\sum_{i=1}^K\omega_i=1$.
Moreover, $K\leq\frac{\kappa}{4}$ implies both
$\omega_i\geq\frac{1}{\kappa}$ and
$\omega_i\geq\frac{3}{4}\theta_i$.

Start with $n_i=\lfloor M\omega_i\rfloor$ and distribute the
remaining $M-\sum_i\lfloor M\omega_i\rfloor$ copies arbitrarily
among the $K$ indices.
The resulting integers sum to $M$.
Since $M\omega_i\geq\frac{M}{\kappa}\geq2$, they satisfy
\[
 n_i
 \geq\lfloor M\omega_i\rfloor
 \geq M\omega_i-1
 \geq\frac{M\omega_i}{2}.
\]
In particular, every $n_i$ is positive.
Combining this bound with
$\omega_i\geq\frac{1}{\kappa}$ gives
$n_i\geq\frac{M}{2\kappa}$, while combining it with
$\omega_i\geq\frac{3}{4}\theta_i$ gives
$\frac{n_i}{M}\geq\frac{3}{8}\theta_i$.
\end{proof}

\begin{proof}[Proof of Lemma~\ref{lem:drhs-spectrum}]
Choose $c_0>0$ as in
Fact~\ref{fact:drhs-uniform-approximation}.
Since $K=t+2\leq c_0\sqrt{\kappa}+2$, we may choose
$\kappa_0\geq4$ sufficiently large that
$K\leq\frac{\kappa}{4}$ whenever $\kappa\geq\kappa_0$.

Take the spectral points $\lambda_i$ and weights $\theta_i$
from Fact~\ref{fact:drhs-weighted-approximation}, and apply
Fact~\ref{fact:drhs-integer-multiplicities} to obtain
positive integers $n_i$ summing to $M$.
These integers satisfy $n_i\geq\frac{M}{2\kappa}$.
Furthermore, for every $q\in\Pi_t$,
\[
 \begin{aligned}
 \sum_{i=1}^K\frac{n_i}{M}
 \left|\frac{1}{\lambda_i}-q(\lambda_i)\right|^2
 &\geq
 \frac{3}{8}\sum_{i=1}^K\theta_i
 \left|\frac{1}{\lambda_i}-q(\lambda_i)\right|^2\\
 &\geq\frac{3}{8}E_t^2
 \geq\frac{3}{512\mu^2},
 \end{aligned}
\]
where the last two inequalities follow from
Facts~\ref{fact:drhs-weighted-approximation}
and~\ref{fact:drhs-uniform-approximation}, respectively.
This proves \eqref{eq:drhs-flat-spectrum}.

Finally, every $\lambda_i$ belongs to $[\mu,\bar L]$, so
$D:=\operatorname{diag}(\lambda_1I_{n_1},\ldots,\lambda_KI_{n_K})$
satisfies $\mu I\preceq D\preceq\bar L I$.
\end{proof}

\end{proof}

\subsection{Proof of Lemma~\ref{lem:drhs-delayed-krylov}}
\label{app:drhs-delayed-krylov}
\begin{lemma}[\bf Delayed-input Krylov representation]
Let \(q_0,q_1\in\mathbb N_{+}\), \(j_\star:=q_0+1\), and
\(m:=q_0+q_1+2\), and suppose \(m^2<M\).
Every deterministic two-stage algorithm above admits a
joint-law-preserving orthogonal coupling with
\(\widehat A=\widehat U^\top D\widehat U\), where
\(\widehat U\sim\operatorname{Haar}(O(M))\), and with
\(h_1,\ldots,h_m\stackrel{\mathrm{iid}}{\sim}N(0,I_M)\), independent of
\(\widehat U\).

Write \(\widetilde\xi,\widetilde y\) for the coupled input and output, and
set
\(\mathcal K_{\mathrm{pre}}
 :=\operatorname{span}\{\widehat A^a h_j:
 1\leq j\leq q_0,\ a\in\mathbb Z_{\geq0},\
 a+j\leq q_0+1\}\).
Then, almost surely,
\begin{align}
 \widetilde\xi
 &=
 P_{\mathcal K_{\mathrm{pre}}}\widetilde\xi
 +\beta
 \frac{(I-P_{\mathcal K_{\mathrm{pre}}})h_{j_\star}}
      {\|(I-P_{\mathcal K_{\mathrm{pre}}})h_{j_\star}\|_2},
 \qquad
 \beta:=\|(I-P_{\mathcal K_{\mathrm{pre}}})\widetilde\xi\|_2>0,
\\
 \widetilde y
 &\in
 \operatorname{span}
 \{\widehat A^a h_j:
 1\leq j\leq m,\ a\in\mathbb Z_{\geq0},\ a+j\leq m\},
\\
 \widehat A^a h_{j_\star}
 \text{ appearing in \eqref{eq:drhs-flat-output-span}}
 &\quad\Longrightarrow\quad
 a\leq m-j_\star=q_1+1.
\end{align}
\end{lemma}

Recall that $j_\star=q_0+1$ and $m=q_0+q_1+2$.
We first state the adaptive block-Krylov lifting result used
in the proof.

For a matrix $B$ and vectors $v_1,\ldots,v_k$, let
$H_k(B;v_1,\ldots,v_k)$ denote the ordered family
\[
 H_k(B;v_1,\ldots,v_k)
 :=
 \bigl(B^a v_j:
 1\leq j\leq k,\quad a\in\mathbb Z_{\geq0},\quad
 a+j\leq k+1\bigr),
\]
where the indices follow a fixed deterministic ordering.
Write $H_k$ when the arguments are clear, and set
$\mathcal S_k:=\operatorname{span}(H_k)$.
We take $H_0$ to be empty and $\mathcal S_0:=\{0\}$.
For a reference matrix $\widehat B$ and vectors
$h_1,\ldots,h_m$, define $\mathcal K_0:=\{0\}$ and
\[
 \mathcal K_k
 :=
 \operatorname{span}\bigl\{
 \widehat B^a h_j:
 1\leq j\leq k,\quad a\in\mathbb Z_{\geq0},\quad
 a+j\leq k+1
 \bigr\},
 \qquad 1\leq k<m.
\]

\begin{lemma}[\bf Adaptive block-Krylov lifting
{\cite[Lemma~6.3 and Appendix~A.3]{bakshi2023krylov}}]
\label{lem:drhs-adaptive-gaussian-lifting}
Let $B=U^\top D U$, where $D$ is a fixed positive diagonal
$M\times M$ matrix and $U$ is Haar distributed on $O(M)$.
Suppose that $m^2<M$.
For each $1\leq k\leq m$, let $v_k$ be a unit vector in
$\mathcal S_{k-1}^\perp$, chosen as a deterministic Borel
function of $H_{k-1}$.

The joint law of $(B,v_1,\ldots,v_m)$ admits a coupling with
a Haar matrix $\widehat U$, standard Gaussian vectors
$h_1,\ldots,h_m$, and an orthogonal matrix $R$ such that
$\widehat U,h_1,\ldots,h_m$ are mutually independent and,
with $\widehat B:=\widehat U^\top D\widehat U$, almost surely,
\[
 \widehat B=RBR^\top,
 \qquad
 Rv_k=
 \frac{(I-P_{\mathcal K_{k-1}})h_k}
      {\|(I-P_{\mathcal K_{k-1}})h_k\|_2}
 \quad(1\leq k\leq m),
 \qquad
 R\mathcal S_k=\mathcal K_k
 \quad(0\leq k<m).
\]
In particular, every terminal vector
$y\in\mathcal S_{m-1}+\operatorname{span}\{v_m\}$ satisfies
\[
 Ry\in
 \operatorname{span}\bigl\{
 \widehat B^a h_j:
 1\leq j\leq m,\quad a\in\mathbb Z_{\geq0},\quad
 a+j\leq m
 \bigr\}.
\]
No response to the final innovation $v_m$ is needed.
The construction can be chosen jointly Borel in any fixed
auxiliary parameter on which the innovation rules depend
jointly Borel measurably.
\end{lemma}

The following fact converts arbitrary candidates into the
unit innovations required by the lifting lemma and explains
how the augmented transcript determines each genuine response.

\begin{fact}[\bf Triangular augmentation and response reconstruction]
\label{fact:drhs-triangular-augmentation}
Fix $B\in\mathbb R^{M\times M}$ and set
$\mathcal S_0:=\{0\}$.
For $1\leq k\leq m$, let $x_k$ be a candidate and write
$s_k:=P_{\mathcal S_{k-1}}x_k$,
$r_k:=x_k-s_k$, and $\alpha_k:=\|r_k\|_2$.
If $\alpha_k>0$, set $v_k:=r_k/\alpha_k$;
otherwise, choose a unit vector
$v_k\in\mathcal S_{k-1}^\perp$ by a fixed Borel rule.
For $k<m$, form $H_k$ and $\mathcal S_k$ as above.

Then, for every $1\leq k<m$,
\[
 \mathcal S_{k-1}+\operatorname{span}\{x_k,Bx_k\}
 \subseteq\mathcal S_k,
 \qquad
 \dim(\mathcal S_k)\leq\frac{k(k+3)}{2}.
\]
Moreover, $x_k$ and the ordered augmented transcript $H_k$
determine $Bx_k$ measurably.
The terminal candidate satisfies
$x_m\in\mathcal S_{m-1}+\operatorname{span}\{v_m\}$.
Under $m^2<M$, every dummy direction required in the
construction exists.
\end{fact}

\begin{proof}
For $k<m$, the definition of $H_k$ gives
$\mathcal S_{k-1}\subseteq\mathcal S_k$ and
$B\mathcal S_{k-1}\subseteq\mathcal S_k$:
every preceding generator remains present, together with
one additional power of $B$.
The vectors $v_k$ and $Bv_k$ also belong to $\mathcal S_k$.
Since $x_k=s_k+\alpha_kv_k$, including when $\alpha_k=0$,
both $x_k$ and $Bx_k=Bs_k+\alpha_kBv_k$ belong to
$\mathcal S_k$.

To reconstruct the response, let $W_{k-1}$ be the matrix
whose columns are the ordered vectors in $H_{k-1}$.
Choose measurably coefficients $c_k$ satisfying
$W_{k-1}c_k=s_k$, for example using the Moore--Penrose
pseudoinverse.
Every column of $BW_{k-1}$ occurs in $H_k$, so the augmented
transcript determines
$Bs_k=(BW_{k-1})c_k$.
It also contains $Bv_k$, and therefore determines
$Bx_k=Bs_k+\alpha_kBv_k$.
For $k=1$, the term $Bs_k$ is zero.

For each $j\leq k$, the family $H_k$ contains $k+2-j$
powers generated by $v_j$.
Its total number of vectors is therefore
$\sum_{j=1}^k(k+2-j)=k(k+3)/2$, which bounds
$\dim(\mathcal S_k)$.
Consequently, before every slot $k\leq m$,
$\dim(\mathcal S_{k-1})\leq (k-1)(k+2)/2<m^2<M$.
Thus the required orthogonal complements are nontrivial,
and measurable dummy directions can be selected by a fixed
Gram--Schmidt rule.
Finally, $x_m=s_m+\alpha_mv_m$ proves the terminal inclusion.
\end{proof}

\begin{proof}[Proof of Lemma~\ref{lem:drhs-delayed-krylov}]
Temporarily fix the delayed input at $\xi=z$.
We construct an augmented experiment with $m$ candidate slots.
The candidates at slots $1,\ldots,q_0$ are the original
preprocessing queries.
At slot $j_\star=q_0+1$, the candidate is $z$.
The candidates at slots $j_\star+1,\ldots,m-1$ are the
$q_1$ original post-disclosure queries.
The candidate at slot $m$ is the terminal output $y$.

At every slot $k<m$, the analyst applies
Fact~\ref{fact:drhs-triangular-augmentation} with $B=A$,
obtaining the innovation $v_k$ and the augmented transcript
$H_k$.
For a genuine query, the analyst reconstructs $Ax_k$ from
$H_k$ and returns only this response to the original procedure.
All additional Krylov vectors remain private.
At the forced slot $j_\star$, the analyst privately forms
the next triangular layer but returns no matrix-vector
response.
Only at this slot is $z$ disclosed to the procedure.
At the terminal slot, the analyst forms $v_m$ from the
residual of $y$ without exposing any additional matrix
products.

For fixed $z$, each innovation $v_k$ is a deterministic
Borel function of $H_{k-1}$.
Indeed, the ordinary transcript can be reconstructed
recursively from $z$ and the augmented transcript.
Given the preceding ordinary transcript, the original
query rule determines the next genuine candidate;
Fact~\ref{fact:drhs-triangular-augmentation} then recovers
its response from the next augmented transcript.
The forced candidate is the fixed vector $z$, and the
terminal candidate is determined by the completed ordinary
transcript.
Projection, normalization away from zero, and the fixed
dummy-direction rule are all Borel.
Thus the innovation rules satisfy the hypotheses of
Lemma~\ref{lem:drhs-adaptive-gaussian-lifting}.

Apply that lemma for each fixed $z$, using its construction
jointly Borel in $z$.
We may therefore sample $\xi$ uniformly on the unit sphere,
independently of $\widehat U,h_1,\ldots,h_m$, and use the
coupling corresponding to $z=\xi$.
For each fixed $z$, the original-coordinate matrix and
augmented transcript have the same joint law as in the
original experiment with input $z$.
Integrating this identity over the spherical law of $z$
preserves the full joint law of $A$, $\xi$, the ordinary
transcript, and $y$.
In this coupling, write
$\widehat A:=\widehat U^\top D\widehat U$,
$\widetilde\xi:=R\xi$, and $\widetilde y:=Ry$.

Let $\mathcal S:=\mathcal S_{q_0}$ and
$\mathcal K_{\mathrm{pre}}:=\mathcal K_{q_0}$.
The preprocessing query rules do not use $\xi$.
We also choose all preprocessing dummy directions using
only the preceding augmented transcript.
Hence $\mathcal S$ is a function of $A$ alone and is
independent of $\xi$ in the original experiment.
The coupling preserves this independence.
Furthermore,
$\dim(\mathcal S)\leq q_0(q_0+3)/2<M$.
Since a proper linear subspace has zero spherical measure,
$\beta:=\|(I-P_{\mathcal S})\xi\|_2>0$ almost surely.

The innovation at the forced slot is therefore
$v_{j_\star}=(I-P_{\mathcal S})\xi/\beta$.
The lifting lemma gives
$R\mathcal S=\mathcal K_{\mathrm{pre}}$ and identifies
$Rv_{j_\star}$ with the normalized Gaussian residual
associated with $h_{j_\star}$.
Since
$RP_{\mathcal S}=P_{\mathcal K_{\mathrm{pre}}}R$,
rotating the decomposition
$\xi=P_{\mathcal S}\xi+\beta v_{j_\star}$ yields
\[
 \widetilde\xi
 =
 P_{\mathcal K_{\mathrm{pre}}}\widetilde\xi
 +
 \beta
 \frac{(I-P_{\mathcal K_{\mathrm{pre}}})h_{j_\star}}
      {\|(I-P_{\mathcal K_{\mathrm{pre}}})h_{j_\star}\|_2}.
\]
Orthogonality of $R$ also gives
$\beta=\|(I-P_{\mathcal K_{\mathrm{pre}}})
\widetilde\xi\|_2$.
This proves \eqref{eq:drhs-flat-fresh-input}.

At the terminal slot,
$y\in\mathcal S_{m-1}+\operatorname{span}\{v_m\}$.
The terminal assertion of the lifting lemma therefore gives
\[
 \widetilde y\in
 \operatorname{span}\bigl\{
 \widehat A^a h_j:
 1\leq j\leq m,\quad a\in\mathbb Z_{\geq0},\quad
 a+j\leq m
 \bigr\},
\]
which is \eqref{eq:drhs-flat-output-span}.
For the distinguished column $j_\star=q_0+1$, the triangular
constraint gives $a\leq m-j_\star=q_1+1$, proving
\eqref{eq:drhs-flat-degree}.
Finally, $\widehat A=RAR^\top$ and
$\widetilde\xi=R\xi$ imply
$\widehat A^{-1}\widetilde\xi=RA^{-1}\xi$.
\end{proof}
\subsection{Proof of Theorem~\ref{thm:drhs-without-l1}}
\label{app:drhs-without-l1}

\begin{proof}
We apply the delayed-information-chain construction directly,
choosing its upper spectral scale as part of the hard instance.
This allows all constants below to depend only on $p$.
Set
\[
 R_\varepsilon
 :=
 \frac{\Delta^{\frac{p}{p+1}}L_p^{\frac1{p+1}}}{\varepsilon},
 \qquad
 \alpha_p
 :=
 \begin{cases}
  \frac74,&p=2,\\[1mm]
  \frac53,&p\geq3.
 \end{cases}
\]
We will obtain a query lower bound of order
$R_\varepsilon^{\alpha_p}$ in dimension
$O_p(R_\varepsilon^{1+\frac1p+2\alpha_p})$.
By decreasing the accuracy constant in the theorem, we may
assume throughout that $R_\varepsilon\geq2$.

Let $C_{\mathrm{reg},p}$ and $C_{\mathrm{gap}}$ be the constants
in Lemma~\ref{lem:drhs-regularity}.
Choose positive constants $c_{\mu,p}$ and $c_N$ such that
$C_{\mathrm{reg},p}c_{\mu,p}^p\leq1$ and
$C_{\mathrm{gap}}c_N\leq1$.
Choose also
$C_{\kappa,p}\geq\max\{\kappa_0,4/c_0^2\}$.
Define
\[
 \begin{aligned}
 \mu&:=c_{\mu,p}(L_p\varepsilon^{p-1})^{\frac1p},
 &\qquad
 \rho&:=\frac{\varepsilon}{\mu},\\
 \kappa&:=C_{\kappa,p}
 R_\varepsilon^{\,2(\alpha_p-1-\frac1p)},
 &
 \bar L&:=\mu\kappa,
 \end{aligned}
\]
and let
$N:=\lfloor c_N\Delta\mu/\varepsilon^2\rfloor$,
$t:=\lfloor c_0\sqrt{\kappa}\rfloor$, and $T:=Nt$.
The choice of $\mu$ gives
$\Delta\mu/\varepsilon^2
=c_{\mu,p}R_\varepsilon^{1+\frac1p}$.
Thus, after decreasing the accuracy constant further,
$c_N\Delta\mu/\varepsilon^2\geq2$.
Since $\alpha_p>1+\frac1p$, our choice of $C_{\kappa,p}$
also ensures $\kappa\geq\kappa_0$ and
$c_0\sqrt{\kappa}\geq2$.
Using $\frac a2\leq\lfloor a\rfloor\leq a$ for $a\geq2$,
we obtain
\[
 N\asymp_p R_\varepsilon^{1+\frac1p},
 \qquad
 t\asymp_p R_\varepsilon^{\alpha_p-1-\frac1p},
 \qquad
 T\asymp_p R_\varepsilon^{\alpha_p}.
\]

Define the phase dimension $M$ and the total hard dimension
$d_0:=NM$ as in \eqref{eq:drhs-calibrated-dimension}, with
$C_{\mathrm{dim}}\geq\max\{C_0,2\}$.
This ensures $M\geq2\kappa$ and the dimension hypothesis of
Lemma~\ref{lem:drhs-sequential}.
To bound $d_0$, first note that
$\kappa/T=O_p(R_\varepsilon^{\alpha_p-2-\frac2p})=O_p(1)$,
because $\alpha_p<2+\frac2p$.
Moreover, $N\leq T$, so
$\log(12NT(T+2))\leq\log(12T^2(T+2))=O(T)$.
Substituting these estimates into
\eqref{eq:drhs-calibrated-dimension} gives
\[
 M=O_p(T^2),
 \qquad
 d_0=O_p(NT^2)
     =O_p\!\left(
       R_\varepsilon^{1+\frac1p+2\alpha_p}
     \right).
\]
Hence a sufficiently large constant $C_p$ ensures
$d\geq d_0$ under the dimension assumption of the theorem.

For such a dimension $d$, construct the hard chain with
spectral interval $[\mu,\bar L]$.
Let $F$ denote an unshifted realization.
The estimates in Lemma~\ref{lem:drhs-regularity} are uniform
in $\bar L$ and give
\[
 \begin{aligned}
 \operatorname{Lip}(D^pF)
 &\leq
 C_{\mathrm{reg},p}
 \frac{\mu^p}{\varepsilon^{p-1}}
 =
 C_{\mathrm{reg},p}c_{\mu,p}^pL_p
 \leq L_p,\\
 F(0)-\inf F
 &\leq
 C_{\mathrm{gap}}N\frac{\varepsilon^2}{\mu}
 \leq
 C_{\mathrm{gap}}c_N\Delta
 \leq\Delta.
 \end{aligned}
\]
Every realization is also $C^\infty$ and coercive.

To ensure nonnegativity, use the pointwise lower bound from
the proof of Lemma~\ref{lem:drhs-regularity}:
$F(x)\geq-C_{\mathrm{sh}}N\varepsilon^2/\mu$ for all $x$,
where $C_{\mathrm{sh}}:=6B_0C_R$.
Define
$f:=F+C_{\mathrm{sh}}N\varepsilon^2/\mu$.
Then $f\geq0$, all derivatives are unchanged, and
$f(0)-\inf f=F(0)-\inf F$.
Consequently, the law $\mathcal D_d^+$ of these shifted
functions is supported on
$\mathcal F_p^{(d)}(\Delta,L_p)$.
The shift depends only on the calibrated parameters and
not on the phase realization.
An exact value-and-gradient oracle for $f$ can therefore
be simulated using an oracle for $F$ by adding this fixed
constant to each returned function value and leaving the
gradient unchanged.

Fix a deterministic algorithm making at most $T-1$ oracle
calls.
Apply Lemmas~\ref{lem:drhs-certification} and
\ref{lem:drhs-sequential} to the corresponding algorithm
on the unshifted chain.
Writing $z_s^{\mathrm{out}}$ for the $s$th hard block of
its output, we obtain
\[
 \Pr_{f\sim\mathcal D_d^+}
 \bigl(\|\nabla f(x^{\mathrm{out}})\|_2\leq\varepsilon\bigr)
 \leq
 \Pr\bigl(
   \|z_s^{\mathrm{out}}-x_s^\star\|_2<3\rho
   \text{ for every }s\in[N]
 \bigr)
 \leq\frac1{12}.
\]
The distribution $\mathcal D_d^+$ is independent of the
algorithm.
For a randomized algorithm with the same hard query budget,
condition on its internal seed and apply the preceding
bound to each resulting deterministic algorithm.
Averaging over the seed shows that its average success
probability under $\mathcal D_d^+$ is at most $\frac1{12}$.
Thus some fixed
$f\in\mathcal F_p^{(d)}(\Delta,L_p)$ has success probability
at most $\frac1{12}$, and hence
\[
 T_{\varepsilon,\mathrm{FO}}^{\mathrm{rand}}
 \bigl(\mathcal F_p^{(d)}(\Delta,L_p)\bigr)
 \geq T
 \gtrsim_p R_\varepsilon^{\alpha_p}.
\]

For $p=2$, we have
$\alpha_p=\frac74$ and
$1+\frac1p+2\alpha_p=5$.
For $p\geq3$, we have
$\alpha_p=\frac53$ and
$1+\frac1p+2\alpha_p=\frac{13}{3}+\frac1p$.
Substituting the definition of $R_\varepsilon$ and adjusting
$c_p$ and $C_p$ proves the two claimed bounds.
\end{proof}

\end{document}